\documentclass[12pt]{book}
\usepackage{amssymb,amsfonts,amsthm,amsmath,extarrows,yfonts,mathabx}

\usepackage{fancyhdr}
\usepackage{graphicx}

\usepackage[font={small,sl}]{caption}
\usepackage[all]{xy}

\newcommand{\Nak}{\cM_{\theta,\zeta}(\bv,\bw)} 
\newcommand{\Mrn}{\cM(r,n)}
\newcommand{\Mr}[1]{\cM(r,#1)}

\newcommand{\cMs}{\widetilde{\cM}}
\newcommand{\tmu}{\widetilde{\mu}}

\newcommand{\Vs}{\widetilde{V}}

\newcommand{\Pp}{\mathbb{P}}
\newcommand{\cF}{\mathcal{F}}
\newcommand{\cG}{\mathcal{G}}
\newcommand{\cO}{\mathcal{O}}

\newcommand{\C}{\mathbb{C}}
\newcommand{\Z}{\mathbb{Z}}
\newcommand{\Q}{\mathbb{Q}}
\newcommand{\R}{\mathbb{R}}
\newcommand{\N}{\mathbb{N}}
\newcommand{\bbH}{\mathbb{H}}

\newcommand{\se}{\mathsf{e}}
\newcommand{\bL}{\mathsf{L}}

\newcommand{\bd}{\mathsf{d}}
\newcommand{\bC}{\mathsf{C}}
\newcommand{\sP}{\mathsf{P}}

\newcommand{\bG}{\mathsf{G}}
\newcommand{\bGt}{\mathsf{G}^{\sim}}
\newcommand{\bQ}{\mathsf{Q}}

\newcommand{\bF}{\mathsf{F}}
\newcommand{\bS}{\mathsf{S}}
\newcommand{\bc}{\mathsf{c}}
\newcommand{\bE}{\mathsf{E}}
\newcommand{\bY}{\mathsf{Y}}

\newcommand{\bT}{\mathsf{T}}
\newcommand{\bA}{\mathsf{A}}

\newcommand{\cM}{\mathcal{M}}
\newcommand{\cU}{\mathcal{U}}
\newcommand{\cL}{\mathcal{L}}
\newcommand{\cV}{\mathcal{V}}

\newcommand{\cN}{\mathcal{N}}

\newcommand{\cD}{\mathcal{D}}
\newcommand{\cW}{\mathcal{W}}
\newcommand{\vac}{\mathsf{vac}}

\newcommand{\bv}{\mathsf{v}} 
\newcommand{\bw}{\mathsf{w}} 
\newcommand{\ev}{\mathsf{ev}} 
\newcommand{\tX}{\widetilde{X}} 
\newcommand{\Xt}{X^{\sim}} 
\newcommand{\ddD}{\frac{\partial}{\partial D}}

\newcommand{\fa}{\mathfrak{a}} 
\newcommand{\fg}{\mathfrak{g}} 
\newcommand{\fz}{\mathfrak{z}} 
 
\newcommand{\fC}{\mathfrak{C}} 
\newcommand{\fH}{\mathfrak{H}} 
\newcommand{\fZ}{\mathfrak{Z}} 
 
\newcommand{\reg}{\textup{reg}} 
\newcommand{\red}{\textup{red}} 
\newcommand{\vir}{\textup{vir}} 
\newcommand{\edge}{\textup{edge}} 
 
\newcommand{\fhb}{\overline{\mathfrak{h}}} 
\newcommand{\fG}{\mathfrak{G}} 
 
\newcommand{\bal}{\boldsymbol\alpha} 
\newcommand{\ttau}{\boldsymbol\varsigma} 
 
\newcommand{\bphi}{\boldsymbol\phi} 
\newcommand{\bkap}{\boldsymbol\kappa} 
\newcommand{\bPhi}{\boldsymbol\Phi} 
\newcommand{\bOm}{\boldsymbol\Omega} 
\newcommand{\bbeta}{\boldsymbol\beta} 
\newcommand{\bmu}{\boldsymbol\mu}

\newcommand{\fHeis}{\mathfrak{Heis}} 
\newcommand{\alsn}{\alpha_{\textup{log}}} 
 
\newcommand{\fD}{\mathfrak{D}} 
\newcommand{\Qb}{{\overline{Q}}} 
\newcommand{\Cb}{{\overline{\bC}}} 
\newcommand{\Ib}{\overline{I}}

\newcommand{\lang}{\left\langle}
\newcommand{\rang}{\right\rangle}
\newcommand{\vacv}[1]{\left|#1\rang}
\newcommand{\vacd}[1]{\lang#1\right|}

\newcommand{\fVir}{\mathfrak{Vir}}
\newcommand{\fh}{\mathfrak{h}}
\newcommand{\fB}{\mathfrak{B}}
\newcommand{\fb}{\mathfrak{b}}
\newcommand{\bsi}{\varepsilon}
\newcommand{\bV}{\mathsf{V}}

\newcommand{\Th}{T^{\mathbf{1/2}}}

\newtheorem{Theorem}{Theorem}[section]
\newtheorem{Lemma}[Theorem]{Lemma}
\newtheorem{Proposition}[Theorem]{Proposition}

\newtheorem{Corollary}[Theorem]{Corollary}
\newtheorem{Conjecture}{Conjecture}
\newtheorem{Question}{Question}

\newcommand{\pt}{\mathsf{pt}}
\newcommand{\Fix}{\mathsf{Fix}}

\newcommand{\fF}{\mathfrak{F}} 
 
\newcommand{\ft}{\mathfrak{t}} 
\newcommand{\bR}{\mathbf{R}} 
\newcommand{\bRh}{\widehat{\bR}} 
 
\newcommand{\hzeta}{\widehat{\zeta}} 
\newcommand{\br}{\mathbf{r}} 
\newcommand{\bbr}{\overline{\br}}

\newcommand{\nui}{\rm{\textsf{(i)}}}
\newcommand{\nuii}{\rm{\textsf{(ii)}}}
\newcommand{\nuiii}{\rm{\textsf{(iii)}}}

\newcommand{\hd}{{\displaystyle\boldsymbol{\cdot}}}
\newcommand{\nord}{\!:\!}
\newcommand{\vT}{\boldsymbol{T}}
\newcommand{\vL}{\bL}

\newcommand{\bQh}{\widehat{\bQ}_{\textup{cl}}}

\newcommand{\fl}{\textup{flop}}
\newcommand{\bk}{\Bbbk}
\newcommand{\bbY}{\mathbb{Y}}
\newcommand{\bP}{\mathsf{P}}
\newcommand{\ga}{{\bG_\bA}}
\newcommand{\gat}{{\bG_{\bA_2}}}
\newcommand{\gal}{\fg_\bA}
\newcommand{\bK}{\mathbb{K}}
\newcommand{\ct}{\mathsf{c_2}}
\newcommand{\cb}{\mathbf{c}}
\newcommand{\db}{\mathbf{d}}
\newcommand{\bI}{\mathbf{I}} 
\newcommand{\bdel}{\boldsymbol\delta} 
\newcommand{\cVt}{\widehat{\mathcal{V}}}
\newcommand{\tN}{\widehat{N}}
\newcommand{\tR}{\widehat{R}}
\newcommand{\tT}{\widetilde{\bT}}
\newcommand{\sing}{\textup{sing}}
\newcommand{\bOmh}{\widehat{\boldsymbol\Omega}}

\newcommand{\tLie}{\textup{Lie}}
\newcommand{\vHeis}{\mathfrak{V}} 
\newcommand{\vHeisr}{\mathfrak{V}_{r}} 
\newcommand{\Mb}{\overline{M}}
\newcommand{\fS}{\mathfrak{S}}

\newcommand{\cT}{\mathcal{T}}
\newcommand{\gl}{\mathfrak{gl}}
\newcommand{\fsl}{\mathfrak{sl}}
\newcommand{\Grl}{\mathsf{Gr}}
\newcommand{\bZ}{\mathsf{Z}}

\newcommand{\bchi}{\boldsymbol\chi}
\newcommand{\cDh}{\widehat{\mathcal{D}}}
\newcommand{\bJ}{\mathsf{J}}
\newcommand{\sI}{\mathsf{I}}
\newcommand{\rsc}{\textup{\small{\textfrak{R}}}_\textup{sc}} 
\newcommand{\gsc}{\textup{\large{\textfrak{g}}}_\textup{sc}} 
\newcommand{\ve}{\varepsilon}

\newcommand{\ffZ}{\textup{\textfrak{Z}}}

\DeclareMathOperator{\Mat}{Mat}
\DeclareMathOperator{\Path}{Path}
\DeclareMathOperator{\Euler}{Euler}
\DeclareMathOperator{\qdet}{qdet}
\DeclareMathOperator{\Center}{Center}
\DeclareMathOperator{\crn}{corners}

\newcommand{\glh}{\widehat{\mathfrak{gl}(1)}}

\theoremstyle{definition}

\newtheorem{Definition}[Theorem]{Definition}
\newtheorem{Example}[Theorem]{Example}
\newtheorem{Remark}[Theorem]{Remark}

\DeclareMathOperator{\Hilb}{Hilb}
\DeclareMathOperator{\Eff}{Eff}
\DeclareMathOperator{\supp}{supp}
\DeclareMathOperator{\Stab}{Stab}

\DeclareMathOperator{\Ext}{Ext}
\DeclareMathOperator{\Lie}{Lie}
\DeclareMathOperator{\Aut}{Aut}
\DeclareMathOperator{\Pic}{Pic}

\DeclareMathOperator{\rk}{rk}
\DeclareMathOperator{\tr}{tr}
\DeclareMathOperator{\virdim}{vir\,\,dim}
\DeclareMathOperator{\diag}{diag}
\DeclareMathOperator{\Spec}{Spec}
\DeclareMathOperator{\Proj}{Proj}
\DeclareMathOperator{\codim}{codim}
\DeclareMathOperator{\Span}{Span}
\DeclareMathOperator{\sgn}{sgn}
\DeclareMathOperator{\End}{End}
\DeclareMathOperator{\Cochar}{Cochar}
\DeclareMathOperator{\LAttr}{\Attr}
\DeclareMathOperator{\Ker}{Ker}

\DeclareMathOperator{\Attr}{Attr}
\newcommand{\Bran}{\Attr^f}
\newcommand{\Branch}{\Bran_\fC} 
\DeclareMathOperator{\gr}{gr}
\DeclareMathOperator{\Res}{Res}

\DeclareMathOperator{\ch}{ch}
\DeclareMathOperator{\Hom}{Hom}
\DeclareMathOperator{\ad}{ad}
\DeclareMathOperator{\Taut}{Taut}

\DeclareMathOperator{\Rep}{Rep}

\DeclareMathOperator{\Img}{Im}
\DeclareMathOperator{\Gr}{Gr}
\DeclareMathOperator{\Coh}{Coh}

\begin{document}
\title{Quantum Groups and Quantum Cohomology}
\author{Davesh Maulik and Andrei Okounkov}
\date{}
\maketitle

\setcounter{tocdepth}{1}

\chapter*{\centering \begin{normalsize} \sc Abstract\end{normalsize}}
\vspace{-1cm}
\begin{quotation}
\noindent 
In this paper, we study the classical and quantum equivariant cohomology of Nakajima
quiver varieties for a general quiver $Q$.  Using a geometric $R$-matrix formalism, we
construct a Hopf algebra $\bY_Q$, the Yangian of $Q$, acting on the cohomology of these
varieties, and show several results about their basic structure theory.  We prove a
formula for quantum multiplication by divisors in terms of this Yangian action.  The
quantum connection can be identified with the trigonometric Casimir connection for $\bY_Q$;
equivalently, the divisor operators correspond to certain elements of Baxter
subalgebras of $\bY_Q$.  A key role is played by geometric shift operators which 
can be identified with the quantum KZ difference connection.

In the second part, we give an extended example of the general theory for moduli spaces
of sheaves on $\C^2$, framed at infinity.  Here, the Yangian action is analyzed
explicitly in terms of a free field realization; the corresponding $R$-matrix is closely
related to the reflection operator in Liouville field theory.  We show that divisor
operators generate the quantum ring, which is identified with the full Baxter
subalgebras.  As a corollary of our construction, we obtain an action of the W-algebra $\cW\big(\mathfrak{gl}(r)\big)$
on the equivariant cohomology of rank $r$ moduli spaces, which implies certain conjectures of Alday, Gaiotto, and
Tachikawa.
\end{quotation}

\tableofcontents

\chapter{Introduction}

\lhead[\thechapter\quad \leftmark]{\thepage}
\rhead[\thepage]{\thesection\quad \rightmark}
\pagestyle{fancy}

\section{Fundamental structures and  conjectures}

\subsection{} 

This paper is about the equivariant quantum cohomology of Nakajima 
quiver varieties \cite{Nak94,Nak98}. We see it as part of 
a larger project \cite{FRG} which connects equivariant 
quantum cohomology of symplectic resolutions with their 
quantizations and derived autoequivalences. These connections, 
however, will not be discussed here. 

Here we develop a general structural theory for quantum 
cohomology of Nakajima quiver varieties associated to an 
arbitrary quiver $Q$. 
We formulate our answer in terms of a certain Hopf algebra $\bY_Q$,
called the Yangian of $Q$,
which acts on the cohomology of Nakajima quiver varieties.  

The construction of $Y_Q$ and an analysis of its basic structure theory 
is another objective of this paper and occupies the bulk of its first half.
In the case when $Q$ has no loops, this construction is related to work of
Varagnolo \cite{Var} and Nakajima \cite{Nak01}, who construct
a certain subalgebra of $\bY_Q$ via generators and relations.
In this paper, we give an alternative approach which we will describe shortly.

In the second half of the paper, we work 
out explicitly what our theory means for the quiver with 
one vertex and one loop. In other words, we work out explicitly the
quantum cohomology of the moduli 
spaces $\cM(r,n)$ of framed rank $r$ torsion free sheaves on $\C^2$, 
generalizing the previous work \cite{OPhilb,MO} on the Hilbert schemes
of points. 

\subsection{} 

Let $X$ be a smooth quasi-projective variety with an action 
of a reductive group $\bG$. 
Quantum cohomology is a commutative associative deformation of 
ordinary multiplication in equivariant cohomology $H^\hd_\bG(X)$
defined by 
\begin{equation}
(\gamma_1 * \gamma_2, \gamma_3) = \sum_{\beta>0} q^\beta \lang 
\gamma_1,\gamma_2,\gamma_3 \rang_\beta 
\label{def*}
\end{equation}
where 
$(\gamma_1,\gamma_2)=\int_X \gamma_1 \cup \gamma_2$ is the standard bilinear 
form on $H^\hd_\bG(X)$, $\beta$ ranges over the cone 
of effective classes
in $H_2(X,\Z)$, $q^\beta$ denotes the corresponding element 
of the semigroup algebra of the effective cone,  
and 
$$
\lang \gamma_1,\gamma_2,\gamma_3 \rang_\beta\in H^\hd_\bG(\pt,\Q) 
$$
is 
the virtual count of rational curves of degree $\beta$ meeting 
cycles Poincar\'e dual to $\gamma_1,\gamma_2,\gamma_3$. See 
e.g.\ \cite{CoxKatz, Mirror} for an introduction. 

As defined by \eqref{def*}, 
the structure constants of quantum multiplication are 
formal power series in $q^\beta$. However, one conjectures
that for all \emph{equivariant symplectic resolutions}, 
and Nakajima quiver varieties in particular, the series
in \eqref{def*} represents a rational function of $q^\beta$. 
We will prove a slightly weaker statement below. Thus 
we get a family of commutative associative 
multiplications on $H^\hd_\bG(X)$. 

Note that working in equivariant cohomology is crucial as 
all nonequivariant counts $\lang \gamma_1,\gamma_2,\gamma_3 \rang_\beta$
vanish for trivial reasons for $\beta\ne 0$. 
\subsection{}
A basic property of quantum multiplication is that 
\begin{equation}
1 * \gamma  = \gamma \,, \quad \forall \gamma \in H^\hd_\bG(X) \,. 
\label{1*}
\end{equation}
For any structure of a 
commutative associative algebra with unit on a vector space $H$, 
the operators of multiplication form a maximal commutative
subalgebra of $\End(H)$. 

In particular, the operators
of quantum multiplication, for different values of the quantum 
parameters $q$, form a $b_2(X)$-dimensional family 
of maximal commutative subalgebras in the algebra that they 
all generate. For brevity, we call these subalgebras the 
algebras of quantum multiplication. For $q=0$, they specialize
to the algebra of classical multiplication in $H^\hd_\bG(X)$. 

Not much is known or conjectured
about this algebraic structure for general $X$. 
For Nakajima quiver varieties, by contrast, one 
expects the following very strong link 
with much-studied structures in representation theory 
and mathematical physics. 

\subsection{}

The Nakajima quiver varieties $\Nak$ with 
parameters 
$$
\bv,\bw \in \N^I\,, \quad \theta\in \R^I\,, \quad \zeta\in \C^I 
$$
are associated to a quiver 
$Q$ with the vertex set $I$. The quiver $Q$ may have 
loops and multiple edges. Nakajima varieties have 
large groups $\bG$ of automorphism that preserve (or scale, for 
$\zeta=0$) their natural symplectic 
form\footnote{Note the quantum product is trivial unless
$\zeta=0$ because all curve contributions are proportional 
to the weight $\hbar$ of the symplectic form.}. By construction, the space 
$$
H(\bw) = \bigoplus_\bv H^\hd_\bG\left(\Nak\right)
$$
will be a module over the Yangian $\bY_Q$.
 By construction, operators
of cup product by characteristic classes of universal 
bundles form a commutative subalgebra in $\bY_Q$. 

\subsection{}

The algebras $\bY_Q$ generalize Yangians of simple finite-dimensional 
Lie algebras, as defined by Drinfeld \cite{Drin}. 
Their origins lie in the theory of quantum integrable systems, see 
e.g.\  \cite{Fad,JM,Kor,Slav} for an introduction. 

A powerful correspondence between quantum integrable systems and moduli of 
vacua in supersymmetric gauge theories (of which Nakajima varieties
are examples) was proposed in the work of Nekrasov and Shatashvili
\cite{NS1,NS2,NS3}. In particular, 
quantum group actions on their cohomology or $K$-theory  
constructed by Varagnolo and Nakajima fit into this framework. 

For us, the main prediction of Nekrasov and Shatashvili is 
a conjectural identification of algebras of quantum multiplication 
with \emph{Baxter subalgebras}\footnote{Also known as 
\emph{Bethe subalgebras}.}
in the Yangian $\bY_Q$.

\subsection{}

Independently, Bezrukavnikov conjectured a relation between 
the monodromy of the quantum differential equation, see  \eqref{QDE} 
below, and autoequivalences of
$D^b \Coh_\bG X$
 for symplectic 
resolutions $X$, see Section \ref{QDEmonod}. This was inspired, in part, 
by the work of T.~Bridgeland \cite{Bridg1,Bridg2}, see also \cite{ABM}. 

Towards the end of the special 2007/08 year at IAS, it was realized 
this conjecture is naturally a composition of two more basic ones. 
The first, which is proven in this paper for Nakajima 
varieties, identifies the quantum 
differential equation with 
the \emph{trigonometric Casimir connection} for a certain 
Lie algebra $\fg_Q$.  A related conjecture about quantum 
cohomology of Laumon spaces was made in \cite{FFFR}. 

For finite-dimensional Lie algebras, trigonometric Casimir 
connections were defined and studied by Toledano Laredo 
in \cite{Valerio}. 
As explained there, they are very closely related to the Yangians
of the same Lie algebras. This links the conjectural frameworks
of Nekrasov-Shatashvili and Bezrukavnikov. 
The trigonometric Casimir connection generalizes the rational 
Casimir connection studied in \cite{FMTV,MilTL,Valerio1}
and also by C.~De Concini (unpublished). 

After this, the second step of Bezrukavnikov's conjecture becomes 
a geometric description of the monodromy of trigonometric 
Casimir connections. This could be viewed as a natural extension 
of the monodromy conjecture made in \cite{Valerio}.

\subsection{}

It appears the ideas of both Nekrasov-Shatashvili and Bezrukavnikov
may apply more generally than just for symplectic resolutions. 
For example, Laumon spaces discussed in \cite{FFFR} 
have a natural Poisson structure which is not symplectic. 

Similarly, the most general moduli of vacua considered by Nekrasov
and Shatashvili fail all key property of Nakajima varieties: 
they may not be smooth,  not symplectic, and not resolutions of singularities. 

In this  paper, we use the existence of a symplectic form and of a 
proper map to an affine variety in an essential way. It is would 
be very interesting to make our constructions work in greater 
generality.

\section{Baxter subalgebras and quantum multiplication}

\subsection{}

The construction of $\bY_Q$ and the notion of a Baxter subalgebra are best explained in 
the original language of quantum inverse scattering method. 
The main ingredient there is 
an \emph{$R$-matrix}, that is, a collection of 
vector spaces $F_i$ and  operator-valued
functions
\begin{equation}
R_{F_i,F_j}(u) \in \End(F_i \otimes F_j)
\label{Rmatrdef}
\end{equation}
which satisfy the \emph{Yang-Baxter equation}
\begin{equation}
R_{12}(u) \, R_{13}(u+v) \, R_{23}(v ) = 
R_{23}(v ) \, R_{13}(u+v) \, R_{12}(u) \,, 
\label{YBe}
\end{equation}
as operators in $F_i\otimes F_j\otimes F_k$.  Here 
$$
R_{12} = R_{F_i,F_j} \otimes 1_{F_k} \in \End(F_i\otimes F_j\otimes F_k) \,,
$$
et cetera. In principle, the argument $u$ could be taken from an arbitrary 
abelian group; the case $u\in \C$ corresponds
to Yangians.

For $m\in \End(F)$ and all $W\in \{F_i\}$, consider the operators 
$$
T_F(m,u) = \tr_F \,(m\otimes 1) \, R_{F,W}(u) \in \End(W) \,,
$$
where the trace is taken over the first tensor factor. 
In the formalism of  Faddeev, Reshetikhin, and Takhtajan 
\cite{FRT}, these operators generate the Yangian $\bY$ associated to $R$.

\subsection{}\label{s_Baxter_sub}

Let $\fG\subset \prod GL(F_i)$ be the centralizer of all $R$-matrices
and take $g\in \fG$. 
It follows at once from the Yang-Baxter equation and 
invertibility of $R$ that 
\begin{equation}
\left[T_{F_1}(g,u_1),T_{F_2}(g,u_2)\right] = 0\,. 
\label{e_Baxter}
\end{equation}
A pictorial proof of this is given in Figure \ref{f_Baxter}. 
This means the operators 
$T_{F}(g,u)$, for fixed $g\in\fG$ and all $F\in \{F_i\}$, $u\in \C$ 
generate a commutative 
subalgebra of the Yangian. This is what is 
called a Baxter (or Bethe) subalgebra. 
\begin{figure}[!hbtp]
  \centering
   \includegraphics[scale=0.44]{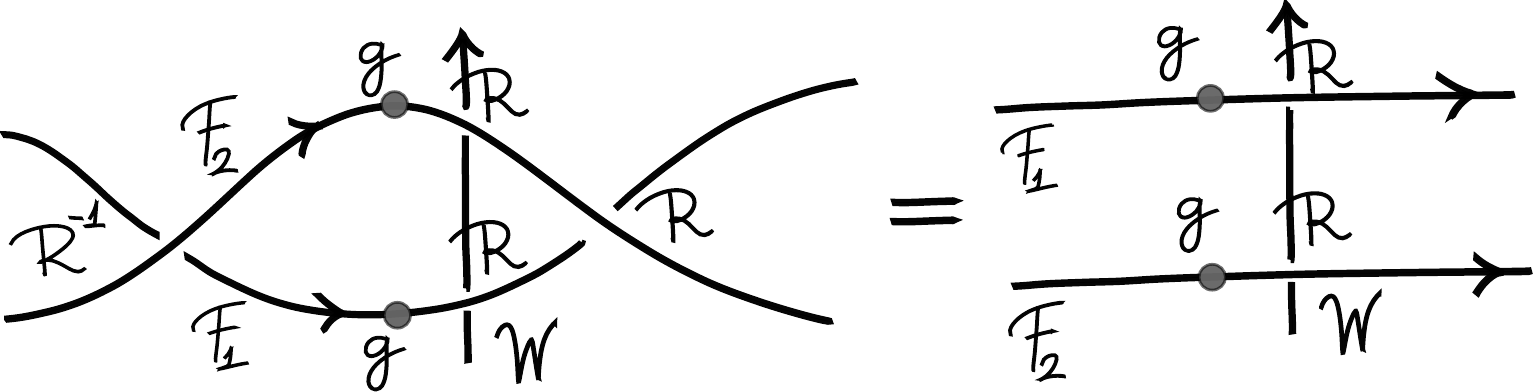}
 \caption{{}From the YB equation and $\left[g\otimes g, R\right]=0$ we
 deduce that $R_{F_2,F_1}$ conjugates
$g_{F_2} R_{F_2,W} g_{F_1} R_{F_1,W}$ to the product in the opposite
order. Taking the trace over $F_2 \otimes F_1$ gives
\eqref{e_Baxter}. }
  \label{f_Baxter}
\end{figure}

\subsection{}

Assuming for simplicity that 
$\fG$ is 
connected,
%
%
a natural parameter set for Baxter subalgebras 
is a maximal torus 
$$
\fH\subset \fG \big/ \textup{Centralizer}(\bY) \,. 
$$ 
It may be compactified to $\overline{\fH}\supset \fH$ 
by considering limits of 
Baxter subalgebras as $g$ degenerates. To connect with 
quantum cohomology, we need a map 
\begin{equation}
\fH \to H^2(X,\C) \big/ 2\pi i \, H^2(X,\Z)\,, 
\label{fHtoH2}
\end{equation}
that extends to 
$$
\overline{\fH} \to \textup{K\"ahler moduli space of $X$} \,. 
$$

\subsection{}\label{s_theta_dis}

There is a small, but essential detail in this identification, 
namely a shift of origin, 
$$
\fH \owns 1 \mapsto \pi i \bkap_X \in H^2(X,\C) \big/ 2\pi i \, H^2(X,\Z)
$$
for a certain class 
$$
\bkap_X \in H^2(X,\Z/2)
$$
that we call the canonical theta characteristic. 

When $X=T^*Y$ then 
$\bkap_X$ is the pull-back of the canonical class $K_Y$ to $X$. 
Nakajima varieties are cotangent bundles only in sense of stacks, but
still $\bkap_X$ is well-defined, see Section \ref{s_theta_char}. 

\subsection{}
It is very convenient to incorporate the shift 
\begin{equation}
q^\beta \mapsto (-1)^{(\beta,\bkap)} q^\beta
\label{modif*}
\end{equation}
into the definition of the quantum product. We call it 
the \emph{modified} quantum product. 

With this modification, we can use the map 
\begin{equation}
H_2(X,\Z) \owns q^\beta \mapsto e^{( \beta, \, \cdot\, )}
\in \fH^\wedge \,,
\label{H_2weight}
\end{equation}
dual to \eqref{fHtoH2}, 
to identify operators $T_F(g,u)$ with operators of quantum 
multiplication. 
Note that a trace over an auxiliary space is an 
element in the group algebra $\C[\fH^\wedge]$, or 
its completion if the auxiliary space is 
infinite-dimensional.

\subsection{}

To turn this into a practical description of the
 quantum product, one 
needs an $R$-matrix construction of the Yangian $\bY_Q$.

The main geometric idea is simple and 
uses the embedding 
\begin{equation}
\bigsqcup_{\bv_1+\bv_2 = \bw} 
\cM_{\theta,\zeta}(\bv_1,\bw_1) \times 
\cM_{\theta,\zeta}(\bv_2,\bw_2) 
\hookrightarrow 
\cM_{\theta,\zeta}(\bv_1+\bv_2,\bw) 
\label{MtoM}
\end{equation}
as a fixed point set of a $\C^\times$-action. 
This embedding is, of course, well-known and played a central 
role in the work of M.~Varagnolo and  E.~Vasserot \cite{Vass_tens,
Vass_tens2}, H.~Nakajima \cite{Nak_tens}, and A.~Malkin \cite{Mal}.
See in particular the paper \cite{Nak_tens2} for further 
developments in this direction, closely related to our 
construction.

\subsection{}

Suppose a torus $\bA$ acts on a holomorphic symplectic variety $X$ 
preserving the symplectic form. Then under fairly general 
hypotheses listed in Chapter \ref{s_envelopes}, one can define a 
collection of maps, called \emph{stable envelopes}, 
$$
\Stab_\fC: H^\hd_{\bG_\bA}(X^\bA) \to H^\hd_{\bG_\bA}(X) 
$$
parameterized by certain chambers $\fC$ in $\Lie(\bA)$. 
Here $\bG_\bA$ denotes the centralizer of $\bA$ in $\bG$. 
Stable envelopes enjoy a number of remarkable 
geometric properties, see Chapters \ref{s_envelopes} and 
\ref{s_properties}. 

For $\bA=\C^\times$ with fixed points \eqref{MtoM}, 
there are just two chambers $\pm \fC$ and one defines 
$$
R(u) = \Stab_{-\fC}^{-1} \circ \Stab_{\fC} 
$$
where $u\in \C = \Lie \bA$ is the equivariant parameter for $\bA$. 
The Yang-Baxter equation and other expected properties of $R$-matrices
follow easily from general properties of stable 
envelopes. Thus, we have $R$-matrices \eqref{Rmatrdef} for 
$$
\{F_i\} = \{H(\bw)\}_{\bw\in \N^I} \,. 
$$
See Chapter \ref{s_Yang} for a precise definition of the 
corresponding Yangian $\bY_Q$ and Chapter \ref{s_Further} 
for further discussion of its properties.

\subsection{}
Our $R$-matrices have the form 
$$
R(u) = 1 + \frac{\hbar}{u}\, \br  + O(u^{-2}) \,,
$$
where $\hbar\in H^2_\bG(\pt)$ is the weight of the symplectic
form and 
$$
\br \in S^2 \fg_Q\,,
$$
is an invariant tensor for a certain Lie 
algebra $\fg_Q$ which contains the Kac-Moody Lie algebra 
associated to the quiver $Q$. In particular, the action 
of $\fg_Q$ on $H(\bw)$ generalizes the construction of Nakajima
\cite{Nak94,Nak98}. 
The action of $\fg_Q$ commutes with $R$-matrices. 

If $Q$ is a quiver of finite type then, modulo
center, $\fg_Q$ is 
the corresponding Kac-Moody Lie algebra, but in general it is 
larger. For example, it may not be finitely 
generated like $\fg_Q \cong \glh$ for the quiver with one 
vertex and one loop. We expect the assignment 
$$
Q \mapsto \fg_Q
$$
to behave well with respect to the natural operations on 
quivers. In particular, the results of Section \ref{s_covers} 
relate $\fg_{Q/\Gamma}$ and $\fg_Q^\Gamma$, where 
$$
Q \to Q'=Q/\Gamma
$$
is a covering of quivers 
corresponding to $\Gamma\subset \pi_1(Q')$. An example of 
this is the well-known relation between $\glh$ and infinite 
Toeplitz matrices.

\subsection{}

A maximal torus $\fhb_Q\subset \fg_Q$ is identified with 
\begin{equation}
\fhb_Q=\fh \oplus \fz\,, \quad \fh,\fz \cong \C^I\,, 
\end{equation}
where $\fh$ and $\fz$ act on $H^\hd_\bG(\Nak)$ by multiplication 
by linear functions of $\bv$ and $\bw$, respectively. Note 
that $\fz$ is central in $\fg_Q$ and $\bY_Q$.

\subsection{}

The Lie algebra $\fg_Q$ acts on $H(\bw)$ by correspondences of the 
following shape. Let $0\ne \alpha \in \N^I$ be a dimension vector
and choose $\bw_0 \in \N^I$ so that $\bw_0 \cdot \alpha \ne 0$. 
For example, one can take $\bw_0 = \delta_i$ for $i\in \supp 
\alpha$. 

For all $\bv,\bw$, there is a canonical 
Lagrangian cycle 
$$
\br_{\bv,\bw,\alpha,\bw_0} \subset \cM(\bv+\alpha,\bw) \times 
\cM(\bv,\bw) \times \cM(\alpha,\bw_0) \,.
$$
One can view this cycle as a correspondence between the second 
and the first factor in which the third factor is a parameter. 
This gives a map 
\begin{equation}
H^\hd_\bG(\cM(\alpha,\bw_0)) \to \left(\fg_Q\right)_\alpha \,,
\label{Htog}
\end{equation}
which is surjective unless $\alpha<0$, 
see Proposition \ref{p_maps_inj}. Here
\begin{equation}
\fg_Q = \fhb_Q \oplus \bigoplus_{\alpha} \left(\fg_Q\right)_\alpha
\label{rootfg}
\end{equation}
is the root decomposition of $\fg_Q$, that is, the decomposition 
into the eigenspaces of the adjoint action of $\fh$. 
Reading the same correspondence $\br_{\bv,\bw,\alpha,\bw_0}$ 
in the opposite direction produces operators 
in $ \left(\fg_Q\right)_{-\alpha}$.

\subsection{}\label{s_surj} 
{}From the construction of $X=\Nak$ as a quotient by 
the action of $GL(\bv)=\prod_{i\in I} GL(\bv_i)$, one has tautological bundles
$\cV_i$ on $X$ of ranks $\bv_i$ for $i\in I$. 
The corresponding map 
$$
\Z^I \to \Pic(X)\cong H^2(X,\Z)
$$
given by $\det(\cV_i)$, $i\in I$, is expected to be 
surjective for all $\bv$ and an isomorphism for $\bv$
sufficiently large (see Section \ref{s_update} below). Dually, we have 
$$
H_2(X,\Z) \hookrightarrow \fH^\wedge 
$$
where $\fH\cong (\C^\times)^I$ is the torus with the Lie algebra
$\fh$. Since this matches \eqref{H_2weight}, we can state the following 
precise version of the Nekrasov-Shatashvili 
principle:  

\begin{Conjecture}\label{con1} The Baxter subalgebras of $\bY_Q$
corresponding to $g\in\fH$ are the algebras of modified
equivariant quantum multiplication for Nakajima 
varieties. 
\end{Conjecture}

\section{Quantum multiplication by divisors}

\subsection{}

Conjecture \ref{con1} may be approached in two steps, the first 
one being the identification of operators of quantum 
multiplication by divisors, that is, 
elements of $H^2(\cM)$. 

The Yangian 
$\bY_Q$ has a grading which after doubling 
corresponds to cohomological degree.
In this paper, we prove the following 

\begin{Theorem}\label{t_Bax} 
The operators of cohomological degree $2$ in the Baxter subalgebra 
are 
the operators of modified quantum multiplication by elements of 
$H^2_\bG(\cM)_\textup{taut}$
\end{Theorem}

\noindent
Here 
\begin{equation}
  \label{TautH}
  H^\hd_\bG(\cM)_\textup{taut} \subset H^\hd_\bG(\cM)
\end{equation}
is the subalgebra spanned by the characteristic classes
of the tautological bundles. An equality in \eqref{TautH}
is expected.

\subsection{}

Theorem \ref{t_Bax} means the following concrete formula
for quantum multiplication by 
$$
c_1(\lambda) = \sum \lambda_i \, c_1(\cV_i) \,. 
$$
The Lie algebra
$\fg_Q$ has an invariant bilinear form such that 
$$
(\fg_Q)_\alpha \perp  (\fg_Q)_\beta\,, \quad \alpha + \beta \ne 0 \,,
$$
see Theorem \ref{t_Lie}. 
Abusing notation, we denote by
$$
\{e_\alpha \} \subset (\fg_Q)_\alpha, 
\{e_{-\alpha} \} \subset (\fg_Q)_{-\alpha}, 
$$
dual bases of root subspaces. Note the dimensions of the 
root spaces, known as root multiplicities, are finite by 
surjectivity in \eqref{Htog}.

Theorem \ref{t_Bax} is equivalent to the following

\begin{Theorem}\label{t_c1*}
We have
\begin{equation}
c_1(\lambda) *_\textup{modif} 
\,\cdot\,= c_1(\lambda) \cup \,\cdot\, - 
\hbar \sum_{\theta\cdot\alpha>0} (\lambda,\alpha) \, \frac{q^\alpha}{1-q^\alpha} \, 
\, e_{\alpha} e_{-\alpha} + \dots\,,
\label{c_1*}
\end{equation}
where modified quantum product means the substitution 
\eqref{modif*}, the sum is over roots of $\fg_Q$ with 
multiplicity, and dots stand for a multiple of the identity.
\end{Theorem}

\noindent 
The multiple of the identity left as dots in \eqref{c_1*} is 
uniquely fixed from equation \eqref{1*}.

The operator $c_1(\lambda) \, \cup$ lies in the Yangian $\bY$ if
$\theta>0$ or in a certain completion of the Yangian for general 
$\theta$, see Section \ref{s_cl_div}. Changing $\theta$ corresponds
to flops of Nakajima varieties and formula \eqref{c_1*} has 
the expected flop-covariance. 

One can compare \eqref{c_1*} to the more 
abstract structural statement for quantum 
multiplication by divisors derived in  \cite{BMO}.

\subsection{}
For $\lambda\in H^2_\bG(X)$ consider the operators
\begin{equation}
\nabla_\lambda = \frac{d}{d\lambda} - \lambda \, * 
\label{QDE}
\end{equation}
acting in $H^\hd_\bG(X) \otimes 
\Q(q^\beta)$ by 
$$
\frac{d}{d\lambda}\,  q^\beta = (\lambda,\beta) \, q^\beta \,. 
$$
Note that $\frac{d}{d\lambda} =0$ 
if $\lambda$ is purely equivariant, that is, $\lambda$ 
comes from $H^2_\bG(\pt)$. It is known very generally that 
$$
\left[\nabla_\lambda,\nabla_\mu\right] = 0 
$$
for all $\lambda,\mu\in H^2_\bG(X)$. Hence any section of the 
projection $H^2_\bG(X) \to H^2(X)$ defines a flat connection 
on a trivial $H^2_\bG(X)$-bundle over $H^2(X)$. 
This connection is known 
as the quantum differential equation or Dubrovin connection.

Formula \eqref{c_1*} precisely means that the quantum 
differential equation for Nakajima varieties is a 
trigonometric Casimir connection in the sense of \cite{Valerio}. 
To be precise, we prove this for 
$H^2(X)_\textup{taut}$, which is expected to be the whole $H^2(X)$.

\subsection{}

Conjecture \ref{con1} would be implied by the affirmative
answer to the following 

\begin{Question}
Do the operators \eqref{c_1*} have a simple joint spectrum ? 
Equivalently, is quantum cohomology of Nakajima varieties
generated by tautological divisors ? 
\end{Question}

\noindent 
In this paper we treat the following special case. 

\begin{Theorem}\label{t_gen_div} 
The quantum cohomology of the moduli space of framed 
torsion-free sheaves on $\Pp^2$ is generated by the divisor. 
\end{Theorem}

\noindent 
These moduli spaces are Nakajima varieties 
associated the quiver of with one 
vertex and one loop. Our proof of Theorem 
\ref{t_gen_div} is based on an explicit 
representation of quantum multiplication 
by divisor in terms of Heisenberg operators.

\section{Shift operators and qKZ} 

\subsection{} 

For simplicity, let us replace the group $\bG$ by its maximal 
torus $\bT$. By construction, the elements of 
$H^\hd_\bT(X) \otimes \Q(q^\beta)$ are functions on 
$$
\ft \times H^2(X) \,,
$$
where $\ft = \Lie \bT$. The operators \eqref{QDE} define 
a flat connection along the $H^2(X)$-directions. In fact, 
this is a part of a flat difference-differential connection, 
in which the difference part corresponds to the lattice 
$$
\Cochar(\bT) \subset \ft  \,.
$$
The corresponding operators 
$$
\bS(\sigma) \in \End 
H^\hd_\bT(X) \otimes \Q[[q^\beta]]
$$
are known as \emph{shift operators} because they shift 
the values of the equivariant parameters in $\nabla_\lambda$. 
They are constructed geometrically as follows.

\subsection{}

Let 
$$
\sigma: \C^\times \to \bT 
$$
be a cocharacter of $\bT$. To it, one associates a nontrivial 
$X$-bundle $p$ 
$$
\xymatrix{
    X
 \ar@{^{(}->}[r] & \Xt \ar@{->}^p[d] \\ 
& \Pp^1}
$$
over $\Pp^1$, see Chapter \ref{S_Int}. 
 By definition, rational curves in $\Xt$ that map 
to the base $\Pp^1$ with degree $1$ are the $\sigma$-twisted 
rational curves in $X$. Their enumerative geometry is closely 
related to the Gromov-Witten theory of $X$. 
In particular, the shift operator $\bS(\sigma)$ is constructed
from the virtual count of twisted 2-pointed 
rational curves with 
marked points in $p^{-1}(0),p^{-1}(\infty) \cong X$, 
see Section \ref{s_bS(sigma)}. 

The flatness condition 
$$
\left[\nabla_\lambda,e^{-\frac{d}{d\sigma}} \, \bS(\sigma)\right] =  0
$$
is the $\varepsilon=1$ specialization of Proposition \ref{p_intertw}. 
Here $e^{\frac{d}{d\sigma}}$ is the translation by $\sigma\in \ft$. 

\subsection{}

The key step in our proof of Theorem \ref{t_c1*} is an 
explicit computation of the shift operators $\bS(\sigma)$ for 
certain special cocharacters $\sigma$.

An action of $\C^\times$ on a symplectic resolution $X$ is 
called \emph{minuscule} if $H^0(\cO_X)$ is generated by 
functions of weight $0,\pm 1$. One easily shows, see 
Section \ref{s_minusc}, that the $\C^\times$-action from 
\eqref{MtoM} is minuscule. For minuscule $\sigma$, the 
operators $\bS(\sigma)$ may be computed in term of 
$R$-matrices as follows. 

\subsection{}

A $\sigma$-fixed point $x\in X^\sigma$ defines a section 
$\zeta_x$ of $p$. The classes of these sections 
$$
[\zeta_x] \in H_2(\Xt,\Z)
$$
lie in a single $H_2(X,\Z)$-coset. Thus, up-to an overall multiple,
$q^\zeta$ is a well-defined function from the set of components 
of $X^\sigma$ to the group algebra of $H_2(X,\Z)$. In fact, 
for Nakajima varieties, there is a preferred way to fix the 
ambiguity, see Section \ref{dis_zeta}. 

Recall the stable envelope maps 
$$
\Stab_{\pm} : H^\hd_\bT(X^\sigma) \to H^\hd_\bT(X)
$$
and their ratio $R_\sigma = \Stab_-^{-1} \circ \Stab_+$. Define 
$$
\nabla^\sigma_\lambda = \Stab_+^{-1} \circ \nabla_\lambda 
\circ \Stab_+ \,. 
$$
Theorem \ref{T_int_R} in Section \ref{s_S_R} is equivalent to the 
following 

\begin{Theorem}\label{t_qKZ}
For minuscule $\sigma$, $\nabla^\sigma_\lambda$ 
commutes with the difference connection 
\begin{equation}
\Psi(t+\sigma) = (-1)^{(\zeta,\bkap_X)} \,  q^\zeta 
\, R_\sigma \, \Psi(t) 
\label{qKZ}
\end{equation}
where we consider $\Psi \in H^\hd_\bT(X^\sigma) \otimes \Q(q^\beta)$
as a function of $t\in \ft$.  
\end{Theorem}

\noindent 
Here $\bkap_X$ is the canonical theta characteristic discussed in 
Section \ref{s_theta_dis}\,. 

\subsection{}

In the case of \eqref{MtoM}, we have 
$$
q^\zeta = q^{\bv_1} = q^{\bv} \otimes 1 
$$
where $q^\bv$ lies in the torus $\fH$ with Lie algebra $\fh$. We thus 
recognize in \eqref{qKZ} the 
quantum Knizhnik-Zamolodchikov equation of Frenkel and Reshetikhin, 
see \cite{FrenkelReshetikhin}. 

\subsection{} 

It follows from Theorem \ref{t_qKZ} that the operator 
$(-1)^{(\zeta,\bkap_X)} \,  q^\zeta 
\, R_\sigma$ commutes with operators of quantum 
multiplication for minuscule $\sigma$. This plays 
a key role in the proof of Theorem \ref{t_c1*}. 
In other words, we determine the quantum connection 
$\nabla_\lambda$ through the commuting difference 
connection.

For this to work, it is important to relate Nakajima 
varieties with different framing vectors $\bw$ as in \eqref{MtoM}. 
For instance, quantum cohomology of the moduli spaces of framed torsion 
free sheaves on $\C^2$ is a object of significant geometric 
interest, see below. From our perspective, it is easier to 
determine it for general rank then just in the special case 
of Hilbert schemes.

\section{Yangian of  $\glh$ and instanton moduli}

\subsection{}

In the second half of the paper, we make the general 
theory explicit in the case of the quiver $Q$ with one 
vertex and one loop. Denote 
$$
r = \bw_1 \,, \quad n = \bv_1 \,. 
$$
The corresponding Nakajima variety 
$$
\cM(r,n) = \cM_{1,0}(\bv,\bw)
$$
 is the moduli space of framed rank $r$ 
torsion-free sheaves $\cF$ on $\Pp^2$ with 
$c_2(\cF)=n$. A framing of a sheaf $\cF$, 
by definition, is a choice of an isomorphism 
$$
\phi: \cF\big|_{L_\infty} \to \cO_{L_\infty}^{\oplus r}
$$
where $L_\infty\subset\Pp^2$ is a fixed line. Usually, the line $L_\infty$ is viewed as the line at infinity of $\C^2\subset \Pp^2$. 
The group 
$$
\bG=GL(2)\times GL(r)
$$ 
acts naturally on $\Mrn$, the first factor acting on $\C^2$ while the second acts by changing the framing.

See, for example, \cite{NakBook} for an introduction to the geometry of $\Mrn$. It plays an important role in Donaldson theory \cite{DonKron} and 
in mathematical approaches to supersymmetric quantum gauge theories, particularly in the work of Nekrasov \cite{NekrSW}. By a theorem of Donaldson, a dense
open subset of $\cM(r,n)$, $r>1$, that parameterizes locally free sheaves
is the moduli space of framed $U(r)$-instantons of charge $n$.

\subsection{}

For $r=1$, $\Mrn$ becomes the Hilbert scheme of points, the quantum cohomology of which was determined in 
\cite{OPhilb}, a result that found applications to the enumerative theories of curves in threefolds \cite{MOOP}. 

Theorem \ref{t_c1*} gives a new proof of this 
result and extends it to higher rank.  We expect it to play 
a role in the higher rank Donaldson-Thomas theory of threefolds. 
In fact, higher rank DT theory of threefolds was one of the 
main motivations for the present work. 

\subsection{} 

In Chapter \ref{s_g_fr} we relate the Lie algebra 
$\fg_Q$ to the Heisenberg algebra $\glh$ that 
acts on the cohomology of $\cM(r,n)$ by the 
work of Nakajima \cite{NakHilb}, 
Grojnowski \cite{Groj},  and Baranovsky \cite{Bar}. 

To be precise, for an arbitrary quiver we discuss 
two versions of the Yangian: 
the Yangian $\bY$ mentioned above and another, 
more economical, algebra $\bbY$ which we call the 
\emph{core Yangian}. They correspond to different 
normalization of $R$-matrices: those for $\bY$ 
fix vacuum vectors while those for $\bbY$ 
scale them by certain $\Gamma$-factors, see 
Section \ref{s_Gamma_gen}. 

For $\cM(r,n)$, $\glh \subset \bbY$, while $\fg_Q \subset
\bY$ is the quotient of $\glh$ by the constant loops
$\mathfrak{gl}(1) \subset \glh$. 

\subsection{}

Recall that Nakajima's Heisenberg algebra acts 
irreducibly on the cohomology $H(1)$ of 
$$
\cM(1) = \bigsqcup_n \cM(1,n)\,,
$$
and this identifies $H(1)$ with the standard Fock 
space of one free boson. Stable envelopes
give a map 
$$
H(1)^{\otimes r} \to H(r)\,, 
$$
which makes it possible to describe $H(r)$, and 
the Yangian action on it, in terms of $r$ 
free bosons. 

In this way, one recovers and generalizes
many familiar objects of conformal field theory. 
For example, the Yangian of $\glh$ contains the 
Virasoro algebra in the Feigin-Fuchs representation. 

The quantum integrable system given by the classical, 
that is $q=0$, product in cohomology, is a certain 
generalization of the second-quantized trigonometric 
quantum Calogero-Sutherland system to $r$ 
interacting bosonic fields, see 
Section \ref{s_cup_div}. More generally, a connection 
between the quantum, that is $q\ne 0$, product in 
cohomology and a quantum intermediate long-wave equation is 
explored in \cite{NOILW}. 

\subsection{}

In the literature, one can find many different ways to 
construct and study algebras that may be called a Yangian 
of $\glh$, see for example \cite{Etingof1,EGR,FHHSY,FFJMM,LevSud,Miki,SV1,SV2}. Perhaps one of the 
advantages of our approach is that our $\bY(\glh)$ is obtained
by a general procedure, applicable to 
an arbitrary quiver.

\subsection{}

For us, $R$-matrices are the main objects of study and those 
for $\cM(r,n)$ turn out to be related to very interesting operators
in CFT. Namely, in Section \ref{s_R_vir} we relate the 
$R$-matrix for $\bY(\glh)$ to the reflection operator in Liouville 
theory. As far as we know, the Yang-Baxter equation satisfied by $R$ 
has not been previously explored in the conformal field theory 
context. 

Recall that Theorems \ref{t_Bax} and \ref{t_gen_div} 
identify the algebra of quantum multiplication for $\cM(r,n)$ 
as a Baxter subalgebra in $\bY(\glh)$. The identification of 
$R$ gives a mechanical procedure to write the corresponding 
commuting operators in terms of free bosons. 

\subsection{}

During the workshop at the Simons Center in January 2010, 
we were asked by Nakajima and Tachikawa whether our theory 
can help with some of the questions raised in the work of Alday, 
Gaiotto, and Tachikawa \cite{AGT}.

The connection is, indeed, 
very strong and some simple applications are  
immediate. For example, it is easy to describe the 
image of $\bY(\glh)$ in its representation on $H(1)^{\otimes r}$ 
in terms of the vertex algebra $\cW\big(\mathfrak{gl}(r)\big)$.  
This is discussed in Section \ref{s_W}. We anticipate many further 
applications in this direction.  Similar results have recently been obtained by Schiffman-Vasserot 
\cite{SV2}.

Although applications to the conjectures of \cite{AGT} appear at the end 
of the paper, they require very little of the 
preceding machinery.  In particular, this 
is about purely classical cohomology of $\cM(r,n)$, quantum 
products play no role.  

Classical limits of the formula from which this 
discussion with Nakajima and Tachikawa started were
later independently found in \cite{EPSS} and 
also \cite{SWY}.

\section{Further directions}

We conclude this Introduction with a brief discussion of some 
natural directions in which one can pursue the results of this paper.

\subsection{K-theory}

In \cite{Nak01}, Nakajima constructs an action of 
$\cU_q(\widehat{\fg_\textup{\tiny{KM}}})$ on the equivariant 
$K$-theory of quiver 
varieties. Here $\fg_\textup{\tiny{KM}}$ is a Kac-Moody Lie algebra 
and $\cU_q(\widehat{\fg_\textup{\tiny{KM}}})$ is the quantized 
universal enveloping of the loop Lie algebra of $\fg_\textup{\tiny{KM}}$. 
These algebras are defined by explicit generators and relations, see
\cite{Nak01}. 

A natural extension of the present work to $K$-theory would produce
a larger Hopf algebra $\cU_q(\widehat{\fg_Q})$, defined in 
the style of \cite{FRT} and  acting naturally on $K_\bG(\cM_Q)$. 
At least for quiver varieties, one can construct a $K$-theoretic 
analog of stable envelopes, which we expect to be the key 
ingredient for such project. 

For the Jordan quiver, the $K$-theoretic $R$-matrix was computed in 
\cite{OS}. As expected, it is closely related to the 
results of \cite{FeiTsy,SV1}. 

\subsection{Monodromy of QDE and categorification}\label{QDEmonod}

The quantum differential equation \ref{QDE} defines a connection 
$\nabla$ with regular singularities on the K\"ahler moduli 
space $\overline{\fH}$ of $\cM_Q$, which is a compactification of 
the torus $\fH \cong (\C^\times)^I$.  Consider the regular points 
$$
\fH_\textup{reg} = \left\{ q\in \fH \, \big| \, \forall \alpha 
\,\,  q^\alpha \ne 1
\right\}
$$
of this connection. The monodromy of $\nabla$ defines a 
homomorphism 
$$
B = \pi_1(\fH_\textup{reg}) \to  \overline{\bY(\fg_Q)}\,,
$$
where bar denotes a certain completion. 

A generalization of the Toledano Laredo's monodromy 
conjecture for trigonometric Casimir 
connections \cite{Valerio,GT} identifies $B$ with 
what should be called the quantum Weyl group of 
$\cU_q(\widehat{\fg_Q})$. It was further conjectured by 
Bezrukavnikov that this action of $B$ lifts to 
$$
B \to \Aut D^b \Coh_\bG \cM_Q \,.
$$
This is known in a handful of cases, in particular for 
Hilbert schemes of points of $\C^2$, see \cite{BO}. 
Perhaps a categorical version of stable envelopes, obtained
from the parabolic induction functors for quantizations of 
Nakajima varieties, 
is the proper technical tool to attack these problems. 

\subsection{Higher rank Donaldson-Thomas theory}

The quantum cohomology of Hilbert scheme of points of a symplectic 
surface $S$ is closely related to the Donaldson-Thomas theory of threefolds
fibered in $S$ over a curve. In particular, in 
the case of $A_n$ surfaces, this point of view lead 
to an explicit description of DT invariants of toric threefolds \cite{MOOP}. 

For higher rank sheaves on ADE surfaces, there is again a close
connection with DT theory, via Diaconescu's work on 
ADHM-sheaves \cite{diaconescu-adhm}, see also \cite{CFKM}.
Arguments parallel to those in this paper should give an 
effective determination of the virtual invariants
of the moduli of ADHM sheaves on a smooth projective curve
 in terms of our operators of quantum multiplication.  

Using
a Beilinson-type construction, as in section 7 of \cite{diaconescu-adhm}, 
the ADHM moduli spaces can be identified 
with a certain moduli space of higher-rank framed complexes on 
ADE-fibrations over curves.  

For general quivers, Theorem \ref{t_c1*} implies an identification (up to a scalar
function) between the small $J$-function and $I$-function in these geometries (as defined in \cite{CFKM}),without any
change of variables required.

\subsection{Hilbert Schemes of points of general
  surfaces} \label{s_gen_surf} 

For a general surface $S$, quantum cohomology of the Hilbert schemes
of points and DT theory of $S$-fibrations will diverge and we 
expect the latter to have a better structure. However, we 
expect the classical cohomology $\Hilb(S)$ to be described 
as a $q=0$ Baxter subalgebra for a certain $R$-matrix. 
In fact, this $R$-matrix should be the 
reflection operator $R$ associated in Section \ref{s_reflect}
to the Frobenius algebra $\bbH=H^\hd_\bG(S)$. 
This is a joint 
project with Vivek Shende and one of its potential goals could be 
a better structural understanding of some of  the many mysterious 
universal generating series in the theory through representation theory 
of Yangians. 

\subsection{K-theoretic DT theory}

Perhaps one of the most challenging projects for the future  would be 
to upgrade the connection with DT theory of 3-folds to the 
level of K-theory. K-theoretic DT invariants are a subject of 
interest in both mathematics and theoretical physics, due
to their M-theoretic interpretation \cite{Mth} and their connection 
to the motivic DT invariants \cite{Kth}. 

\section{Update}\label{s_update}

This work reflects what we knew in 2010, with some improvements to
exposition made during 2010--12. 
As we revise it in
the early 2017, it seems necessary to add a certain bare minimum of
references to subsequent developments, in particular, 
in connection with directions for further researched 
outlined above. We decided to limit all such updates
to this section. 

A survey of the progress since 2012 may be found 
in \cite{pcmi, slc}. In particular, lectures 
\cite{pcmi} explain the extension of the present 
work to equivariant K-theory, including application 
to K-theoretic Donaldson-Thomas theory. In K-theory, 
the quantum differential equations studied here become 
quantum \emph{difference} equation. 
Those were determined in 
\cite{OS2} for all Nakajima varieties. 

The monodromy problem for the quantum 
difference equations was analyzed in \cite{AO}. 
This analysis may be directly linked to Bezrukavnikov's
quantization in characteristic $p\gg 0$, to the 
monodromy conjectures above \cite{BO}, and to the 
categorical stable envelopes \cite{HLMO}. 

K.~McGerty and T.~Nevins proved in \cite{mn} that the 
equivariant cohomology of Nakajima varieties is generated
by characteristic classes of tautological bundles (which is  
a property also called \emph{Kirwan surjectivity}). At 
several points, e.g.\ in Section \ref{s_surj} or 
in the statement of Theorem \ref{t_Bax} we had to work our 
way around Kirwan surjectivity which was not known at the time. 
The results of \cite{mn} make these workarounds unnecessary.

The goals stated in Section \ref{s_gen_surf} 
were achieved by N.~Arbesfeld \cite{Arb}.

\section{Acknowledgments} 

\subsection{} 

It has taken us a few years to complete this project and, in the process,
we received a great deal of help from many people.

At the beginning, the initial motivation came from 
unpublished conjectures made by Nikita Nekrasov and Samson Shatashvili, on the one 
hand, and Roma Bezrukavnikov, on the other, and we are grateful to them 
for sharing their insight with us.
Both of us are novices in geometric representation theory and have 
learned a great deal from conversations with Victor Ginzburg and Hiraku Nakajima.
We also thank Sasha Braverman, Pavel Etingof, Valerio Toledano Laredo, and other members of the
FRG group who have helped us crystallize many of the ideas here.  

We had many discussions with Eric Vasserot and Olivier 
Schiffmann, whose work \cite{SV1,SV2} has several parallel 
aspects, in particular, in applications to \cite{AGT}. We 
thank Edik Frenkel, Davide Gaiotto, Victor Kac, and Yuji Tachikawa for 
sharing their knowledge of vertex algebras with us. 

Some of our formulas were rediscovered in the literature
on the AGT conjectures. In particular, our free-field 
formulas for cup product by $c_1(\cO(1))$ for $\cM(r,n)$ 
were also found in \cite{EPSS}. 
It is a pleasure to thank Vincent Pasquier and Didina Serban 
for very interesting 
discussions.

Additionally, we are grateful to Sabin Cautis, 
Sachin Gautam, 
Johan de Jong, David Kazhdan, Dmitry Kaledin, 
Michael McBreen, Andrei Negut, Rahul Pandharipande, 
Nikolai Reshetikhin, 
Vivek Shende, Daniel Shenfeld,  Catharina Stroppel, and 
Zhiwei Yun for many related conversations.

We feel deeply indebted to the anonymous referee for his dedication to
the very arduous task of working through these pages. Many places in
our narrative have significantly gained in clarity thanks to his
attentive and thoughtful remarks.

\subsection{}

As stated at the beginning, this paper is a part of a larger 
joint project with Bezrukavnikov, Braverman, Etingof, Finkelberg, 
Toledano Laredo, Losev, and others. This larger project will surely 
have a nonempty intersection with the ongoing work of Braden, 
Licata, Proudfoot, Webster, and their collaborators, although 
we don't know whether quantum cohomology currently plays a role 
in what they do. Eventually, we hope quantum cohomology will be an 
important part of the unified geometric and representation-theoretic study of equivariant symplectic 
resolutions. 

\subsection{}

We thank NSF for supporting our research. DM has been partially supported by a Clay Research Fellowship.

We thank Simons Center
for Geometry and Physics for being the place where many of our 
results were first presented or written down. Another important 
venue where these results were presented and discussed 
was the 2010 Midrasha Mathematicae in Jerusalem. We thank its 
organizers for the invitation and hospitality.

\part{General Theory}

\chapter{Nakajima varieties}\label{s_Nak}

In this chapter, we recall definitions and basic facts on the geometry of Nakajima quiver varieties.  
There is a large literature on the subject, although most of what we need can be found in the original references \cite{Nak94,Nak98}
and papers of Crawley-Boevey \cite{CBgeom,CBnormal}.  We also explain some results on natural group actions on 
Nakajima quiver varieties. 

\section{Definition}

\subsection{} 
Let $Q$ be a quiver, i.e.\ an oriented multigraph, with finite vertex set $I$. We allow 
loops and multiple edges in $Q$. The quiver data is simply 
the adjacency matrix 
$$
Q = (q_{ij})_{i,j\in I} 
$$
where 
$$
q_{ij} = 
\big|\{\textup{edges from $i$ to $j$}\}\big| \,. 
$$
For what follows, we can assume that 
multiple edges have the same orientation in $Q$. 
We also consider quivers $\Qb$ and $\vec Q$ with vertex set
given by the union
$I\sqcup \Ib$ of two 
copies of the set $I$ and with
adjacency matrices
$$
\Qb = 
\begin{pmatrix}
Q+Q^T & \mathrm{id} \\
\mathrm{id} & 0 
\end{pmatrix}\,, \quad 
\vec Q = 
\begin{pmatrix}
Q & 0 \\
\mathrm{id} & 0 
\end{pmatrix}\,. 
$$

\subsection{}

A representation of a quiver is an assignment of a coordinate vector 
space to each vertex and of a linear map to each arrow. 
The dimension of a representation is an element of $\N^I$, where
$\N=\Z_{\ge 0}$.

For $\bv,\bw\in \N^I$, denote by 
$\Rep_{\Qb}(\bv,\bw)$ the space of representations of the quiver
$\Qb$ of dimensions $\bv_i$ for $i\in I$ and $\bw_i$ for $i\in \Ib$. 
Using the trace pairing, we can write 
\begin{equation}
\Rep_{\Qb} = \Rep_{\vec Q} \oplus \left(\Rep_{\vec Q}\right)^* \,, 
\label{Rep1/2}
\end{equation}
which gives this linear space a symplectic form $\omega$. This symplectic
form is preserved by the action of 
$$
G_\bv = \prod GL(\bv_i)\,, \quad  G_\bw = \prod GL(\bw_i)\,.
$$

We can also define an action of the group
$$
\prod_i Sp(2 q_{ii}) \prod_{i\ne j} GL(q_{ij}) \,. 
$$
as follows.
Given a vertex $i$, loops at this vertex contribute a factor
$$
\End(\C^{\bv_i})^{\oplus q_{ii}} \oplus 
\textup{its dual} \cong  \End(\C^{\bv_i}) \otimes \C^{2 q_{ii}} \,, 
$$
to $\Rep_{\Qb}$
where the symplectic form is induced by the symmetric trace pairing on the first factor and the 
standard symplectic form on the second. 
The factor $Sp_2(2q_{ii})$ acts naturally on the second factor.
Similarly, given distinct vertices $i,j$, the contribution of edges between these vertices is naturally identified with
$$\left(\Hom(\C^{\bv_i},\C^{\bv_j})\otimes \C^{q_{ij}}\right)\oplus \textup{its dual} $$
and the factor $GL(q_{ij})$ acts in the natural way.
By  construction, these groups also preserves the symplectic form $\omega$.

\subsection{} 

The symplectic form $\omega$ is scaled by the action of $\C^\times$
scaling the second summand in \eqref{Rep1/2}. We denote
by $\hbar$ its $\C^\times$-weight. When there are other $\C^\times$'s 
around, we denote this one by $\C^\times_\hbar$. 

We set 
$$G_\edge = \prod_i Sp(2 q_{ii}) \prod_{i\ne j} GL(q_{ij}) \times \C^\times_\hbar.$$
As we shall see, this group will act uniformly on all families of quiver varieties associated to $Q$.

\subsection{Weight convention}\label{weightconvention}

In this paper, we embed group weights into 
Lie algebra weights.   For example, we will also use $\hbar$ to denote
the generator of the equivariant cohomology of $\C^\times_{\hbar}$.

\subsection{}
Sometimes it is convenient to consider, following Crawley-Boevey, 
representations of the quiver $Q_\infty$ with vertex 
set $I\sqcup\{\infty\}$ and 
adjacency matrix 
\begin{equation}
Q_\infty = 
\begin{pmatrix}
Q+Q^T & \bw \\
\bw^{T} & 0 
\end{pmatrix} \,. 
\label{CBQ}
\end{equation}
Note that we have a natural identification
$$
\Rep_{\Qb}(\bv,\bw) \cong \Rep_{Q_\infty}((\bv,1)).
$$
Furthermore, this isomorphism is equivariant with the natural action of the groups above.
For the action of $G_\bw$ on the right-hand side, we define an edge group $G^{\infty}_{\edge}$ analogously to the last section, and it contains a copy of both $G_\bw$ and $G_\edge$.

\subsection{}

Consider the moment map 
$$
\mu: \Rep_{\Qb}(\bv,\bw) \to \fg_\bv^*\,,
$$
for the action of $G_\bv$, where $\fg_\bv = \Lie G_\bv$.  Denote by 
$$
\fz_\bv = [\fg_\bv,\fg_\bv]^\perp \subset \fg_\bv^* \,,
$$
the fixed points of coadjoint action.  If we identify $\fg_\bv^*$ with $\fg_\bv$ 
via the trace pairing, $\fz_\bv$ corresponds to scalar matrices, i.e. a copy of $\C$
for every $i\in \bI$ such that $\bv_i \ne 0$.
We consider the preimage
 $$\fZ = \mu^{-1}(\fz_\bv).$$
 In general, this may be reducible and nonreduced.
 
\subsection{}

Note for any $x\in \Rep_{Q_\infty}$, its stabilizers in $G_\bv$ is 
the quotient of units by scalars for some associative algebra over $\C$. 
Hence the $G_\bv$-stabilizer of $x$ is finite if and only if it is trivial.

\subsection{}
Given $\theta\in\Z^I$, it defines a character of $G_\bv$ by the
convention
$$(g_i) \mapsto \prod (\mathrm{det} g_i)^{\theta_i}\in \C^\times.$$
We define 
\begin{align*}
\cMs_\theta &= \fZ /\!\!/_\theta \, G_\bv \,,\\
&= \Proj 
\bigoplus_{n\ge 0} \, \C\left[\fZ\right]_{n \theta} \,, 
\end{align*}
where the  subscript $n\theta$ denotes the 
corresponding $G_\bv$-isotypic component. The map
$\mu$ descends to a map 
$$
\tmu: \cMs_\theta  \to \fz_\bv \,.
$$
\begin{Definition}
A Nakajima quiver variety is a fiber of this map:
$$
\Nak = \tmu^{-1}(\zeta)\,, \quad \zeta\in \fz_\bv \,. 
$$
\end{Definition}

\section{Basic properties}

\subsection{}

The following result is proven in \cite{Nak94}

\begin{Proposition}\label{basicproperties1}
For any $Q, \bv$, and $\bw$ there 
exits a finite set $\{\alpha_i\}\subset \N^I$ such that 
$\Nak$ contains a strictly semistable point only if 
$$
\alpha_i \cdot \theta = \alpha_i \cdot \zeta = 0 \, 
$$
for some $i$.
\end{Proposition}

These hyperplanes are closely related to the roots of the Lie algebra $\fg_Q$ 
that will be associated to the quiver $Q$ in Section \ref{s_fgQ}.
One corollary of this proposition is that, for $\theta$ in the complement of these hyperplanes, 
the natural map
$$\tmu: \cMs_\theta  \to \fz_\bv \,$$
is smooth, although it is possible that the domain is empty.

We also state the following result, which is well-known.  Since we do not use it in the paper, it can be safely skipped.
However, we sketch its proof briefly.
\begin{Proposition}\label{basicproperties2}
If there exists a free $G_\bv$-orbit contained in $\fZ$,
then $\tmu$ is surjective for all values of $\theta$.  The generic fiber
is smooth and affine.
\end{Proposition}
\begin{proof}
If there exists a free orbit, then the moment map $\mu$ is smooth at any point of this orbit and, in particular, the
image of $\fZ$ contains a dense, Zariski-open set $U$ of $\fz_\bv$.  By Theorem $1.2$ of \cite{CBgeom}, after further shrinking,
the entire fiber of any point of $U$ consists of simple representations of $Q_\infty$.  These are $\theta$-stable for all
stability conditions $\theta$; consequently, the GIT quotient for any choice of $\theta$ equals the categorical quotient of the fiber, which is affine.  This proves the second statement.

For the first statement, we use the definition of quiver varieties via hyperkahler reduction, as in \cite{Nak94}.  Let $U_\bv$ denote the maximal compact subgroup of $G_\bv$.  If we take the hyperkahler moment map, 
then the image of the locus of free $U_\bv$-orbits contains $\R^I \times U \subset \R^I \times \fz_\bv$.  
Since it is stable with respect to multiplication by unit quaternions, it contains $\{\theta\}\times \fz_\bv$ for any suitably generic $\theta$.
Consequently, $\tmu$ is surjective for general $\theta$.
 Finally, if $\theta$ lies on a wall on the space of stability conditions,there is a factorization
$$\cMs_{\theta'} \to \cMs_\theta \to \fz_\bv$$
where $\theta'$ is a nearby stability condition.  We can assume $\tmu$ is surjective for $\theta'$ which implies $\tmu$ is surjective for all $\theta$.
\end{proof}

In this paper, we are mainly interested in the case where $\theta$ is generic in the sense of Proposition \ref{basicproperties1} and when $\zeta = 0$.  We say $\theta > 0$ if $\theta_i > 0$ for all $i$.  This condition implies that $\theta$ is generic in the above sense, for
arbitrary quiver $Q$ and dimension vectors $\bv, \bw$.

\subsection{Group actions}\label{s_bG}

By construction, the group
\begin{equation}
  \label{defbG}
  \bG =
  \begin{cases}
    G_\bw \times G_\edge\,, \quad &\zeta= 0 \,, \\
       G_\bw \times \prod_i Sp(2 q_{ii}) \prod_{i\ne j} GL(q_{ij})\,, \quad &\zeta\ne 0  
  \end{cases}
\end{equation}
acts on $\Nak$.
The larger group also acts on $\cMs_\theta$ and  the map 
$$
\tmu: \cMs_\theta \to \fz_\bv \otimes \hbar^{-1}
$$
is $\bG$-equivariant.

The action of $\bG$ is not faithful on $\Nak$.
The center $Z(G_\bv)$ of $G_\bv$ has a natural map 
$$
\rho_Q: Z(G_\bv) \to G_\edge \,. 
$$
There is also a map
$$\tau_Q: \textup{Ker}(\rho_Q) \to G_\bw$$
given by constants acting by multiplication on $\C^{\bw_i}$.

The images of these maps act trivially on $\Nak$, and we could work with the corresponding
quotient groups
$$ 
G'_\edge = G_\edge/\mathrm{Im}(\rho_Q)\,, \quad G'_\bw =
G_\bw/\mathrm{Im}(\tau_Q)
$$
and their product $\bG'$.

However, it is sometimes convenient to 
 work with the larger group $\bG$ since the tautological bundles
considered shortly admit a natural $\bG$-equivariant structure.  
In practice, most of the geometric calculations and constructions considered later (e.g. $R$-matrices, quantum operators) will
naturally take values in $\bG'$-equivariant cohomology.

\subsection{Symplectic resolutions}

By construction, Nakajima varieties have an algebraic Poisson 
structure which is symplectic on their smooth locus.
The group $\bG$ preserves this symplectic form when $\zeta \ne 0$ and scales it by the
character $\hbar$ when $\zeta = 0$.

Furthermore, they come with a projective map 
$$
\pi: \Nak \to \cM_{0,\zeta}(\bv,\bw) = \Spec \C[\mu^{-1}(\zeta)]^{G_\bv} 
$$
to an affine algebraic variety.

Although $\pi$ is not always birational, it follows from section $10.3$ of \cite{Nak94} that
 it is birational onto its image.  In particular, for $\theta$ generic in the sense of Proposition \ref{basicproperties1}, 
 $\Nak$ is an equivariant symplectic resolution.
 When $\zeta = 0$ it carries a natural torus action that scales $\omega$ and is an example of the general
 theory considered, for example, in \cite{Kal}.

\subsection{Tautological bundles}

As $G_\bv$-quotients, Nakajima varieties have tautological 
bundles $\cV_i$ of ranks $\bv_i$, $i\in I$, associated
to representations
$$
G_\bv \to GL(\C^{\bv_i}) \,. 
$$
For uniformity, we consider the (topologically trivial) 
bundles $\cW_i$, $i\in I$, of ranks $\bw_i$ on a similar 
footing. Since these bundles carry a representation of 
$G_\bw$, their equivariant Chern classes capture the 
framing weights.

\subsection{Equivariant lifts}

The matrix elements of the matrices
$$Q + Q^T\,, \quad \Qb$$
are dimensions of vector spaces
which naturally carry representations of $\bG$, essentially by the definition of the group $\bG$.
As a result, we have a natural lift of $Q + Q^T$ and $\Qb$ to matrices with values in the representation ring
$K_\bG(\pt)$. 
Recall from Section \ref{weightconvention} that 
we embed group weights into 
Lie algebra weights. Here we treat $\hbar$ etc.\ 
as elements of $K_\bG(\pt)$. 

If we endow $K_\bG(\pt)$ with the involution given by taking duals, the Hermitian transpose
of $\Qb$ satisfies the relation
\begin{equation}
\left(\Qb\right)^* = \hbar \otimes \Qb \,.
\label{Qbtrans}
\end{equation}
where $\hbar$ denotes the character of $\bG$ associated to $\C^{\times}_{\hbar}$.

The Cartan matrix of $Q$ admits an equivariant lift
$$\bC = 1 + \hbar^{-1} - (Q + Q^T).$$
We also set
$$
\Cb = 
\begin{pmatrix}
- \bC &  \hbar^{-1}\\
1  & 0 
\end{pmatrix}\\,
$$
and define the Hermitian forms
\begin{align}\label{ringel} 
  (\bv,\bv')_Q &= \bv^* \, \bC \, \bv'\,\\
  ((\bv,\bw),(\bv',\bw'))_{\Qb} &= 
(\bv,\bw)^* \, \Cb \, (\bv',\bw') \,. 
\notag 
\end{align}
for $\bv,\bw,\dots \in K_\bG(\pt)^I$.

Given an arbitrary $\bG$-variety $X$ and 
$$
\bv,\bw,\bv',\bw'\in K_\bG(X)^I\,,
$$
the forms \eqref{ringel} still make sense
and takes values in $K_\bG(X)$.
Of course, very often, one takes just the 
nonequivariant specialization of \eqref{ringel}.  

\subsection{Tangent bundle}

Given $\theta$ generic, if $\Nak$ is nonempty, its dimension is given by 
$$
\dim \Nak = \| (\bv,\bw) \|_{\Qb}^2\,,
$$
with respect to the nonequivariant version of \eqref{ringel}. Using
the equivariant lifts described above,
we can identify the $K$-theory class of the tangent bundle as follows.

\begin{Lemma}
For $\theta$ generic, we have the identification
\begin{equation}
T \Nak = \| (\cV,\cW) \|_{\Qb}^2 \,,
\label{T_Nak}
\end{equation}
in $K_\bG(\Nak)$, where
$$
\cV,\cW \in K_\bG(\Nak)^I
$$
are vectors of tautological bundles. 
\end{Lemma}
\begin{proof}
On the affine space of representations of $\Qb$, the tangent bundle is given by
$$T_{\Rep_{\Qb}(\bv,\bw)} = (\cV,\cW)^* \, \Qb \, (\cV,\cW).$$
Since the moment map is submersive, 
the tangent bundle on $\Nak$ is obtained by subtracting off
$$\fg_\bv^*\otimes \hbar^{-1} - \fg_\bv$$
which gives the result.
\end{proof}

\subsection{Splitting of tangent bundle}\label{tangentsplitting}

Using the orientation of $Q$, we can define a virtual bundle 
$$
\Th  = \sum_{i,j} (Q_{i,j} - \delta_{i,j}) \Hom(\cV_i,\cV_j) + 
\sum_{i} \Hom(\cW_i,\cV_i) \in K(\Nak).
$$
If $H \subset \bG$ denotes the subgroup preserving the decomposition
\eqref{Rep1/2}, then the expression lifts to $K_{H}(\Nak)$ where it satisfies the identity
\begin{equation}
T \Nak = \Th + \hbar^{-1}\otimes\left(\Th\right)^\vee
\label{half_tangent}
\end{equation}
Nakajima varieties may be viewed as open substacks of the cotangent stacks
stacks 
$$
\Nak \approx T^*\left(
\Rep_{\vec Q}\big/G_\bv \right) 
$$
and the virtual bundle $\Th$ is the pullback of the tangent bundle from 
the base in this sense.

\subsection{Theta characteristic} \label{s_theta_char}

One notes that 
\begin{equation}
  \bkap_\cM = c_1\left(\Th\right) \!\!\!\!\!\mod 2  \quad \in H^2(\cM,\Z/2)
\label{def_kappa}
\end{equation}
is independent of the orientation of $Q$. We call it  
the canonical \emph{theta characteristic} of $\Nak$. 
It will be responsible for signs in the formulas for 
quantum multiplication.

\subsection{} 
Alternatively, Nakajima varieties may be defined using representation 
of the quiver $Q_\infty$ and parameters
$$
\hzeta_\infty = - \sum_{i\in I} \bv_i \, \hzeta_i \,.   
$$
This is because diagonal scalars in $\prod_{i\in I\sqcup \{\infty\}} GL(\bv_i)$
act trivially on representations of $Q_\infty$. 

\subsection{}
Note that for $\theta$ generic,
\begin{equation}
  \label{w0empt}
  \cM_{\theta,\zeta}(\bv,0) = \emptyset
\end{equation}
because when $\bw=0$ the action of $G_\bv$ cannot be free.

\section{Torus-fixed points}\label{Nak_fix}

In this section, unless stated explicitly, we assume throughout that $\theta$ is generic
in the sense of Proposition \ref{basicproperties1}, so $\Nak$ is in particular a smooth holomorphic symplectic
variety.

\subsection{}
Let 
\begin{equation}
\bA\subset \mathrm{Ker} \hbar \subset G_\edge \times G_\bw
\label{bAtoG'}
\end{equation}
be a torus. Since $\bA$ preserves $\omega$, 
its fixed locus $\Nak^\bA$ is a smooth holomorphic symplectic 
variety. In fact, it is a union of product of smaller 
Nakajima varieties, which can be seen as follows. 

\subsection{}
Take $x\in \Nak^\bA$ and let $X\in \Rep_{\Qb}(\bv,\bw)$ be a point 
above it. The subgroup 
$$
G^x \subset G_\bv \times G_\edge \times G_\bw 
$$
such that 
$$
1 \to G_\bv \to G^x \to \bA \to 1 
$$
acts on the orbit of $X$. Since the $G_\bv$ action is free, 
we get a map $G^x \to G_\bv$ that splits the above sequence. 
This gives homomorphisms
\begin{equation}
\bA \xrightarrow{\,\,\phi\,\,} 
G_\bv \times G_\edge \times G_\bw \to \bA 
\label{splitA}
\end{equation}
with identity composition and such that $X$ is fixed
by $\phi(\bA)$. 

\subsection{}
A homomorphism $\phi$ is equivalent to a lift of $\bv$,
$\bw$, and $Q$ to vectors and matrices with 
values in $K_\bA(\pt)$, consistent
with the embedding \eqref{bAtoG'}. 
To this, one associates a new  quiver $Q_\phi$ as follows. 
We set
$$
I_\phi = I \times \bA^\wedge 
$$
where $\bA^\wedge$ is the character group of $\bA$, and 
$$
\left(Q_\phi\right)_{(i,\lambda),(j,\nu)} = 
\textup{coefficient of $\nu/\lambda$ in} \, \, Q_{ij} \,, 
$$
where $\lambda,\nu\in \bA^\wedge$. 
This is an infinite quiver with a free action of the 
group $\bA^\wedge$ by automorphisms. We take dimension 
vectors
$$
\left(\bv_\phi
\right)_{(i,\lambda)} = \textup{coefficient of $\lambda$ in} \, \bv_i 
$$
and similarly for $\bw_\phi$. These have finite support, which 
may be disconnected. Clearly, representations of quivers factor 
over connected components of supports. Finally, 
$$
G_{\bv_\phi} = \left(G_\bv\right) ^{\phi(\bA)} \subset G_\bv 
$$
and this 
defines the pull-back $(\theta_\phi,\zeta_\phi)$ of $(\theta,\zeta)$.

\subsection{}
We consider two lifts $\phi_1$ and $\phi_2$ 
in \eqref{splitA} equivalent if they define the 
same action of $\bA$ on $\Rep_{\Qb}$. 

\begin{Proposition}\label{p_Nak_fix}
We have 
$$
\Nak^\bA = \bigsqcup_{\phi/\sim} \cM_\phi\,,
$$
where $\cM_\phi$ is the Nakajima variety associated to the quiver
$Q_\phi$ and the data $\bv_\phi,\bw_\phi,\theta_\phi,\zeta_\phi$ 
above. 
\end{Proposition}

\begin{proof}
It is clear that 
$$
\Rep_{\Qb}(\bv,\bw)^{\phi(\bA)} = 
\Rep_{\Qb_\phi}(\bv_\phi,\bw_\phi) \,. 
$$
The moment map $\mu$ 
takes this fixed locus to 
$$
\left(\fg_\bv^* \right)^{\phi(\bA)} = \fg_{\bv_\phi}^* 
$$
and coincides with $\mu_\phi$. It remains to check that 
$$
\textup{$\theta$-stability} \, 
\Leftrightarrow \, 
\textup{$\theta_\phi$-stability}\,.
$$
The $\Rightarrow$ implication is trivial. 
The set of all $\theta$-destabilizing 
subrepresentations is a projective variety with an 
action of $\bA$. If nonempty, it has an $\bA$-fixed 
point which is a $\theta_\phi$-destabilizing 
subrepresentation. 
\end{proof}

\subsection{}\label{s_abelian_cover}

As a first example, take $\bA$ to be the maximal 
torus of $G'_\edge$. Recall that $G'_\edge$  is largest 
quotient of $G_\edge$ that acts nontrivially. We have 
$$
\bA^\wedge = H_1(Q,\Z)
$$
and 
$$
Q_\phi\to Q_\phi\big/ \bA^\wedge \cong Q 
$$ 
is the universal abelian cover of $Q$. In 
particular, for any $Q$, 
$Q_\phi$ is a quiver without loops at vertices. 

\subsection{}

The restriction of the tangent bundle of $\Nak$ 
to the $\bA$-fixed locus is given by the 
same formula \eqref{T_Nak}, but interpreted
in the $\bA$-equivariant $K$-theory via the 
map $\phi$. 

Expanding \eqref{T_Nak} in characters of $\bA$, 
one expresses the $\bA$-eigensubundles in the 
normal bundle to $\Nak^\bA$ in terms of the tautological 
bundles of $\cM_\phi$.

\subsection{} 
Because the splitting \eqref{half_tangent} is equivariant with 
respect to all group actions, we have 
\begin{equation}
  c_1(N_{\pm}) \!\!\!\!\!\mod 2  = \bkap_{\cM} + \bkap_{\cM^\bA}
\label{parity_N}
\end{equation}
in $H^2(\cM^\bA,\Z/2)$ for any torus $\bA$ that preserves
the symplectic form.

\section{Tensor product of Nakajima varieties}\label{s_bw=}

\subsection{}

For this paper, the main example of the above fixed-point construction arises as follows.

Take a decomposition 
$$
\bw = \bw' + \bw'' 
$$
and define 
$$
\bA \cong \C^\times \subset G_\bw
$$
as the subgroup that scales the first term in 
\begin{equation}
\C^{\bw_i} = \C^{\bw'_i} \oplus \C^{\bw''_i} \,, 
\quad i\in I \,, 
\label{minuscA}
\end{equation}
with weight $1$. In other words, we take 
$$
\bw = z \, \bw' + \bw'' \in K_{\C^\times}(\pt)^I 
$$
where $z$ is the defining representation. Then 
the fixed points are precisely  
\begin{equation}
\bigsqcup_{\bv'+\bv'' = \bv} 
\cM_{\theta,\zeta}(\bv',\bw') \times 
\cM_{\theta,\zeta}(\bv'',\bw'') 
\hookrightarrow 
\cM_{\theta,\zeta}(\bv,\bw) 
\label{Mtens}
\end{equation}
as in \eqref{MtoM}. 
Indeed, the fixed points in \eqref{Mtens} correspond
to 
$$
\bv = z \, \bv' + \bv''
$$
and all other ones are empty because of \eqref{w0empt}. 

The embedding \eqref{Mtens} will play a key role in this paper 
and we call it \emph{tensor product} of Nakajima varieties. 
See Section \ref{s_Tens} for a discussion of this term. 

\subsection{} \label{s_N_} 

For a tensor product of Nakajima varieties, 
 the 
normal bundle to the fixed locus is 
$$
N = z N_+ \oplus z^{-1} N_-
$$
where $z^{\pm 1}$ is the torus weight, 
\begin{align}
 N_-  = &\sum \Hom(\cW'_i,\cV''_i) + 
\sum \Hom(\cV'_i,\cW''_i)\otimes \hbar^{-1} \notag \\
&- \sum \bC_{ij} \Hom(\cV'_i,\cV''_j) \label{charN_}
\end{align}
in the $K$-theory of the fixed locus, where
$\bC_{ij}$ denotes the equivariant Cartan matrix and 
$$
N_+ = \hbar^{-1} \otimes N_-^\vee   \,. 
$$

\section{Slices}

\subsection{}

Recall the affine quotient
$$
\cM_{0,\zeta} = \mu^{-1}(\zeta) / G_\bv\,.  
$$
Its closed points are the closed 
$G_\bv$-orbits in $\mu^{-1}(\zeta)\subset \Rep_{\Qb}$, and those 
correspond to 
isomorphism classes of semisimple representations of $\Qb$ or $Q_\infty$.
 
The natural map
\begin{equation}
\pi: \cM_{\theta,\zeta} \to \cM_{0,\zeta} \,.
\label{inclM0}
\end{equation}
takes a $\theta$-semistable representation to its 
semisimplification, see Proposition 3.20 in \cite{Nak98}. 

\subsection{}

Given $X\in \cM_{0,\zeta}$, it natural to study 
$\pi^{-1}(X)$, bearing in mind that it may be empty. 
Following Nakajima, see Section 6 in \cite{Nak94}, 
$\pi^{-1}(X)$ may be described as $(\pi')^{-1}(0)$ 
for a different quiver $Q'$. Here $0\in \cM'_{0,0}$ is the 
zero representation.

See Proposition 3.2.2 in \cite{Nak01} and Section 4 in \cite{CBnormal}
for the proof of the following 

\begin{Theorem}[\cite{Nak94,Nak01,CBnormal}]\label{t_slice} 
For any $X\in \cM_{0,\zeta}(\bv,\bw)$ there exist a quiver 
$Q'$ and dimension vectors $(\bv',\bw')$ such that:
\begin{itemize}
\item an analytic neighborhood $U$ of $X$ in $\cM_{0,\zeta}(\bv,\bw)$ is 
isomorphic to an analytic neighborhood $U'$ of $0$ in 
$\cM'_{0,0}(\bv',\bw')\times \C^{k}$ and 
\item this isomorphism may be lifted to an isomorphism $\Sigma_X$
between 
$(\pi')^{-1}(U')$ and $\pi^{-1}(U)$ that preserves
the fibers of $\pi$. 
\end{itemize}
These isomorphism are equivariant 
with respect to the stabilizer $\bG'\subset \bG$ of the representation $X$.
\end{Theorem}

We call the maps $\Sigma_X$ slices and for brevity write them 
as rational maps 
$$
\Sigma_X: 
\cM'(\bv',\bw') \times \C^k  \dashrightarrow 
\cM(\bv,\bw) 
$$
even though this is not what is claimed in Theorem \ref{t_slice}. 
The integer $k$ that appears here is the difference in dimensions,
see also \eqref{formk} below.

\subsection{}\label{ssslice} 
The data $Q',\bv',\bw'$ are constructed as follows. 
As a representation of $Q_\infty$, $X$ has a unique decomposition 
$$
X = X_\infty \oplus \bigoplus_{i\in I'} X_i^{\oplus \bv'_i} 
$$
into nonisomorphic simples $X_i$ with multiplicities $\bv'_i$. 
We denote by 
$$
\bd(X)_{ij} = \left(\dim X_j \right)_i \,, \quad i\in I \sqcup \{\infty\}\,,
j \in I' \sqcup \{\infty\}\,, 
$$ 
the matrix of their dimension vectors. 
The subgroup
$$
GL(1)\times G_{\bv'} \subset GL(1) \times G_\bv
$$
is the stabilizer of $X\in\Rep_{Q_\infty}$ and the matrix 
$\bd(X)$ describes its subgroup conjugacy class.

The 
representation $X_\infty$ is distinguished from the rest by 
$$
\bd(X)_{\infty,\infty} = 1 
$$
and then 
$$
\bd(X)_{\infty,j} = 0\,, \quad j\in I'\,, 
$$
because $(\dim X)_\infty  = 1$.

\subsection{}

By definition, $I'\sqcup \{\infty\}$ is the vertex
set for the new quiver $Q'_\infty$ and $\bv'$  is the 
new dimension vector. We use the matrix 
$$
\bd: \Z^{I' \sqcup \{\infty\}} \to \Z^{I \sqcup \{\infty\}}
$$
to transfer the other quiver data to $I'\sqcup \{\infty\}$. 
For example, we set 
$$
\hzeta' = \bd^T \hzeta \,. 
$$
It follows that 
$$
\zeta'=0\,, \quad \bv' \cdot \theta' = 0 \,, 
$$
because $\bv = \bd(X) \cdot \bv'$ and 
$$
\sum_{i\in I \sqcup \{\infty\}} \zeta_i \, \bd(X)_{i,j} = 0\,, 
\quad \forall j \,, 
$$ 
by the moment map equation. 

\subsection{}

The adjacency matrix of $Q'_\infty$, and in particular, 
the new framing vector $\bw'$ is found from the formula
\begin{equation}
  (a,b)_{Q'_\infty} = (\bd(X)\, a, \bd(X)\,b)_{Q_\infty} \,,
\label{new_ringel}
\end{equation}
see \eqref{CBQ} for the the matrix of this quadratic form. 

In the course of the proof, one uses reductivity to write 
$$
\fg^*_\bv = \fg^*_{\bv'} \oplus  \fg_{\bv'}^\perp 
$$
and 
identifies $d\mu^{-1}(\fg^*_{\bv'}) = \left(\fg_\bv \cdot X \right)^\angle$ 
and 
$$
\Rep_{Q'_\infty} \cong \left(\fg_\bv \cdot X \right)^\angle \Big/ \fg_\bv \cdot X 
$$
as $G_{\bv'}\times \bG'$-modules, where $\angle$ 
denotes the symplectic perpendicular. This leads to 
\eqref{new_ringel}.  

\subsection{}
Note that $Q'_\infty$ may have loops at the distinguished 
vertex $\infty$, in fact
\begin{equation}
\# \{\textup{loops at $\infty$}\} = k  = \left\| (\dim X_\infty,\bw)
\right\|^2_\Qb \label{formk}
\end{equation}
where $k$ is the number from Theorem \ref{t_slice}. 
These loops contribute a vector space factor to $\Rep_{Q'_\infty}$
because $\bv'_\infty=1$. Note that \eqref{formk} also describes
this vector space as a $\bG'$-module.

\subsection{}

The following is immediate: 

\begin{Proposition}\label{p_slice_w} 
If $X_\infty$ is the only nonzero representation in $X$ then 
$Q'$ is isomorphic to the subquiver of $Q$ formed by the 
support of $\bv'=\bv - \dim X_\infty$ and 
\begin{equation}
  \label{neww}
  \bw' = \bw - \hbar \, \bC \dim X_\infty \,. 
\end{equation}
\end{Proposition}

\noindent 
This also covers the trivial case when $X=0$ and $\bd(X)_{ij}=\delta_{ij}$.

\subsection{Example}\label{s_slice_ex1}

Consider the $A_n$-quiver,
that is, that is the quiver with 
$$
\bC = 
\begin{pmatrix}
1+ \hbar^{-1} & - \hbar^{-1} \\
-1 & 1+ \hbar^{-1} & - \hbar^{-1}  \\
 & \ddots & \ddots & \ddots \\ 
&& -1 & 1+ \hbar^{-1} \\
\end{pmatrix} \,. 
$$
We fix $1\le i < j \le n$ and take 
$$
\bw =  \hbar\, a \, \delta_i + a \delta_j\,,
$$
where $a$ is a weight of $G_\bw$. For 
$$
\dim X_\infty = a \sum_{k=i}^j \delta_k
$$
there is a torus fixed representation $X_\infty$ with such dimension. 
It takes the framing vector at the $j$th vertex, applies
the arrow in $Q$ to it $(j-i)$ times, and sends it 
to the framing vector at the $i$th vertex. Note that the 
final map in $\Hom(V_i,W_i)$ has torus weight $ \hbar^{\otimes -1}$ and 
the framing weight $\hbar a$ compensates for this. 

If the other $X_i$'s are zero, we get 
$$
\bw'  = a \, \delta_{i-1} + \hbar \, 
a \, \delta_{j+1}
$$
from formula \eqref{neww}. 

\subsection{Example}\label{s_slice_ex2}

Take the quiver with one vertex and one loop, for which 
$\bC$ is a $1\times 1$ matrix 
$$
\bC = (1-t_1)(1-t_2) \,, \quad t_1 \otimes t_2 = \hbar \,,
$$
where $t_1$ and $t_2$ are the weights of $G_\edge$. 
For 
$$
\bw = a+a \, t_1^{-n} \, t_2^{-1}\,, 
$$
there is a torus-fixed representation $X_\infty$ with 
$$
\dim X_\infty  = a (1 + t_1^{-1} + \dots + t_1^{1-n}) \,. 
$$
Just like in the previous example, it takes a framing 
vector of weight $a$ and applies the $t_1$-arrow to it $(n-1)$-times
(the weights have go change by $t_1^{-1}$ every time to compensate
for the $t_1$ weight of the arrow). We find
$$
\bw' = a t_1^{-n} + a t_2^{-1} \,. 
$$
This and the previous example are special cases of 
slices considered in Section \ref{s_Slices_Int}.

\subsection{}

Equivariance in Theorem \ref{t_slice} means that slices commute
with taking fixed points. That is, if $\bA'\subset \bG'$ 
is a torus preserving the symplectic form then 
$$
\left(\Sigma_X\right)^{\bA'}: 
\left(\cM'(\bv',\bw') \times \C^k\right)^{\bA'}  \dashrightarrow 
\cM(\bv,\bw)^{\bA'} 
$$
is an isomorphism of open subsets of quiver varieties (the fixed points 
are quiver varieties by Proposition \ref{p_Nak_fix}). 

In particular, slices are compatible with tensor products, 
in the sense that the following diagram 
commutes
\begin{equation}
  \label{slice_tensor}
  \xymatrix{
    \cM(\bw_0+\bw')\times \C^{\dots}
 \ar@{-->}[rr]^{\Sigma_X} && \cM(\bw_0+\bw) \\ 
    \cM(\bw_0) \times \cM(\bw') \times \C^{\dots}
\ar@{^{(}->}[u]
\ar@{-->}[rr]^{1 \times \Sigma_X}&& \cM(\bw_0) \times \cM(\bw)
\ar@{^{(}->}[u]
  }
\end{equation}
where the vertical arrows are inclusions of fixed points
and the representation $X$ is padded by 
zeros as necessary.

\section{Minuscule coweights}\label{s_minusc}

\subsection{}
Let $X$ be an algebraic variety. We call an action 
$$
\sigma: \C^\times \to \Aut(X) 
$$
minuscule, if the algebra $H^0(X,\cO_X)$ is generated by 
functions of $\sigma$-weight $\{-1,0,1\}$. Equivalently, 
there is an equivariant embedding 
$$
X_0 = \Spec H^0(X,\cO_X) \hookrightarrow V 
$$
where $V$ is a linear representation of $\sigma$ with 
weights in $\{-1,0,1\}$. This notion will play a crucial 
role below.

\subsection{}

\begin{Proposition}
The $\C^\times$-action corresponding to the tensor product of 
Nakajima varieties is minuscule. 
\end{Proposition}

\begin{proof}
It is enough to 
prove that 
$$
\C\left[\fZ\right]^{G_\bv}
$$
is generated by the functions of $\sigma$-weight in $\{0,\pm 1\}$. 
Since $G_\bv$ is reductive, 
the natural map 
$$
\C[\Rep_{\Qb}]^{G_\bv}\to \C\left[\fZ\right]^{G_\bv}
$$
is surjective. 

By the first fundamental theorem of invariant 
theory, see for example Section 9.5 in \cite{VinPop}, the 
$G_\bv$-invariants
are generated by all possible contraction of 
tensorial indices. Concretely this means either functions
of the form 
$$
\tr P_1 P_2 \cdots P_k 
$$
where $P_1,P_2,\dots, P_k$ is a closed chain of edges of $\Qb$ 
starting and ending at a $\bv$-vertex, or any matrix
coefficient of 
$$
P_1 P_2 \cdots P_k 
$$
where $P_1,P_2,\dots, P_k$ is a chain of edges going 
from one $\bw$-vertex to another. 
Clearly, the $\sigma$-weights of all these functions are
in $\{0,\pm 1\}$\,. 
\end{proof}

\chapter{Stable envelopes}\label{s_envelopes} 

Let a torus $\bA$ act on a nonsingular quasiprojective algebraic variety $X$ and let 
$\iota: X^\bA \to X$ denote the inclusion of the fixed locus. We have a natural map
$$
\iota^*: H^\hd_\bA(X) \to H^\hd_\bA(X^\bA) 
$$ 
of degree $0$. Our goal in this section is to construct a reasonably canonical map in the other direction
$$
\Stab_\fC:  H^\hd_\bA(X^\bA) \to H^\hd_\bA(X)
$$
that takes middle degree to middle degree. We will call $\Stab_\fC(\gamma)$ the \emph{stable envelope} of $\gamma$.
The main ingredients in its construction will be: 
\begin{itemize}
\item an $\bA$-invariant holomorphic symplectic form $\omega$ on $X$, 
\item a choice of a certain chamber $\fC\subset\mathfrak{a}=\Lie(\bA)$\,. 
\end{itemize}
\noindent
Stable envelopes appear to be useful in a broader context than strictly  
required for the purposes of the present paper.  We therefore discuss them in that greater
generality. For symplectic resolutions, a much simpler approach may be used, as we explain in Section \ref{s_sympl_res}.
In many examples, we expect the stable envelopes to specialize to well-known constructions.

We begin by explaining various conventions we use and recalling 
several basic constructions. 

\section{Assumptions and conventions}\label{s_Assump}

\subsection{Assumptions on $X$} 

We assume that $X$ is a nonsingular algebraic variety and 
$\omega \in H^0(\Omega^2 X)$ is a holomorphic symplectic form on $X$. 
In addition, we require a proper map 
\begin{equation}
 \label{blowdown}
\pi: X \to X_0
\end{equation}
to an affine variety $X_0$.

\subsection{Group actions}\label{s_convG} 

We denote by 
$$
\bA \subset \bT \subset \bG \to \Aut(X) 
$$
a pair of tori $\bA\subset\bT$ 
in  some reductive group $\bG$ acting on $X$. We denote 
by $\fa\subset\ft\subset\fg$ the corresponding 
Lie algebras. We assume: 
\begin{itemize}
\item  $\omega \subset H^0(\Omega^2_X)$ is an eigenvector of $\bG$, 
fixed by $\bA$; 
\item  the proper map $\pi$ is $\bG$-equivariant; 
\item  $X$ is a formal $\bT$-variety. 
\end{itemize}

See \cite{GKM} for a discussion of formality.  In particular, 
it implies $H^\hd_{\bT}(X)$ is free as a module over $H^\hd_{\bT}(\pt)$.  While this condition
is convenient, we expect it can be removed with a little care.

We denote by 
$$
\hbar \in \fg^*\,,
$$
the $\bG$-weight of $\omega$. By our assumption, 
$\bA$ is in the kernel of $\hbar$.

\begin{Example}\label{ex0} 
For $X=\cM(r,n)$, we take 
$$
\bG = GL(2) \times GL(r) 
$$
where the first factor acts on $\Pp^2$ keeping the line at 
infinity, while the second factor acts by changing 
the framing. We take $\bT$ to be the maximal torus of 
$\bG$ and 
$
\bA = \bT \cap GL(r) \,. 
$
The proper map $\pi$ is the map to the Uhlenbeck moduli space. 
\end{Example}

\begin{Example}\label{exquiv}
More generally,
for $X = \cM_{\theta,0}(\bv,\bw)$ with $\theta$ generic, 
we take $\bG$ as defined in section \ref{s_bG} and $\bT$  its maximal torus.
The proper map $\pi$ is the map
$$\pi: \cM_{\theta,0}(\bv,\bw) \rightarrow \cM_{0,0}(\bv,\bw).$$
Given a decomposition 
$$
\bw = \sum_{i=1}^r \bw^{(i)}\, ,
$$
we obtain a homomorphism 
$$
\bA = \{(z_1, \dots,z_r) \} \to G_\bw 
$$
given by 
$
\bw = \sum_{i=1}^r \bw^{(i)} \, z_i 
$
as in Section \ref{Nak_fix}.
\end{Example}

\subsection{Signs and adjoints}\label{s_signs_adj}

The varieties $X$ we will encounter in the paper have no odd 
cohomology, although the following discussion may be easily 
modified to include odd cohomology. 

When $X^\bT$ is proper, integration over $X$ 
$$
\gamma \mapsto \int_X \gamma \in \Q(\ft) 
$$
may be defined as an equivariant residue, making 
$H = H^\hd_\bT(X)$ a commutative Frobenius algebra over $\Q(\ft) $. 
In fact, it will prove very convenient to introduce 
the following sign twist in the Frobenius trace $\tau$
$$
\tau(\gamma) = (-1)^{\frac12 \dim X} \int_X \gamma \,. 
$$
Recall that $X$ is holomorphic symplectic, so $\dim X$ is even. 
For example, if $X=T^*Y$ and $[Y]$ is the class of the zero 
section, then 
$$
\tau\left([Y]^2\right) = \chi(Y) \,. 
$$

In this paper, we define adjoints using $\tau$. Concretely, 
this means the following. Consider a $\bT$-equivariant cycle, i.e. a
$\Q$-linear formal combination of invariant subvarieties
$$
Z  = \sum a_k Z_k \subset \prod_{i=1}^n X_i\,.
$$
Notice that we have abused notation to write a cycle as a subset of the ambient variety.

Fix a subset $S\subset \{1,\dots,n\}$. Then $Z$, viewed 
as a correspondence, defines a operator 
$$
\Theta_Z : H^\hd_\bT \left(\prod_{i\in S} X_i\right) \to 
H^\hd_\bT \left(\prod_{i\notin S} X_i\right)\otimes \Q(\ft)  \,, 
$$
see Section \ref{s_lagr_corr} for further discussion. For example, 
$Z$ could be the diagonal $\Delta \subset X \times X$ and then, 
for $S=\{1\}$, $\Theta_\Delta$ 
is the identity map. 

Using $\tau$, we may move factors $X_i$ from the source of the 
map $\Theta_Z$ to the target, and back. We call these new 
operators \emph{adjoint} to $\Theta_Z$ and denote them 
by $(\Theta_Z)^\tau$, to distinguish 
it from the ordinary permutations of factors. 
They acquire a sign $(-1)^p$, where 
$$
p= \frac12\sum_{i\in S'} \dim X_i - 
\frac12\sum_{i\in S} \dim X_i\,,
$$
and $S'$ is the source index set for the map $(\Theta_Z)^\tau$. 

For example, if $S'=\{1,2\}$ then 
$$
(\Theta_\Delta)^\tau (\gamma_1 \otimes \gamma_2) = 
(-1)^{\frac12 \dim X} \, \int_{X} \gamma_1 \cup \gamma_2 = 
\tau(\gamma_1 \cup \gamma_2) \in \Q(\ft)  \,.
$$

\section{Basic constructions}

\subsection{Chamber decomposition}

The cocharacters
$$
\sigma: \C^\times \to \bA
$$
form a lattice of rank equal to the rank of $\bA$. We denote 
$$
\mathfrak{a}_\R = \Cochar(\bA) \otimes_\Z \R \subset \mathfrak{a} \,.
$$
Each weight $\chi$ of $\bA$ defines a rational hyperplane in this vector space.

\begin{Definition}
The torus roots are the $\bA$-weights $\{\alpha_i\}$
occurring in the normal bundle to $X^\bA$.
\end{Definition}

The root hyperplanes partition $\mathfrak{a}_\R$ into finitely many (open) chambers
$$
\mathfrak{a}_\R \setminus \bigcup \alpha_i^\perp = \bigsqcup \fC_i \,. 
$$

\begin{Example}\label{ex1} 
In Example \ref{ex0}, we have 
$$
X^\bA = \bigsqcup_{n_1+\dots+n_r=n} \prod \Hilb(\C^2,n_i), 
$$
the normal weights $\alpha$ are the roots of $GL(r)$ 
$$
\mathfrak{a}\owns \diag (a_1,\dots,a_r) \mapsto a_i - a_j \,, 
$$
and the chambers $\fC$ are the usual Weyl chambers. 
\end{Example}

\begin{Example}\label{exquiv1}
Similarly, in Example \ref{exquiv}, we have
$$
\cM_{\theta,0}(\bv,\bw)^\bA = \bigsqcup_{\bv^{(1)}+\dots+\bv^{(r)}=\bv}\cM\left(\bv^{(1)},\bw^{(1)}\right) \times \cdots \times 
\cM\left(\bv^{(r)},\bw^{(r)}\right)
$$
by Proposition \ref{p_Nak_fix}
and the normal weights are again the roots of $GL(r)$.
\end{Example}

The stratification of $\mathfrak{a}_\R$ by root
hyperplanes coincides with 
the stratification by the dimensions of the fixed-point locus. In particular, 
if $\sigma$ does not lie on any hyperplane $\alpha_i^\perp$ then $X^{\sigma} = X^{\bA}$. 

\subsection{Attracting, or stable, manifolds} 

Let $\fC$ be a chamber as above. One says that 
a point $x\in X$ is $\fC$-stable if the limit 
$$
\lim_{z\to 0} \sigma(z) \cdot x \in X^{\bA}
$$
exists for one (equivalently, all) cocharacter $\sigma\in\fC$. The value of this limit 
is independent of the choice of $\sigma\in \fC$. We will denote it by $\lim_{\fC} x$. 

Given a subvariety $Y\subset X^\bA$, 
we denote by
$$
\LAttr_\fC(Y) = \left\{x\, \left| \, \textstyle{\lim_{\fC}} (x) \in Y\right.\right\} 
$$
 the set of points attracted to $Y$ by the 
cocharacters in $\fC$. We have the following: 

\begin{Lemma}\label{noncompactBB}
Let $Z$ be a connected component of $X^\bA$. Then 
$$
\textstyle{\lim_\fC}: \LAttr_\fC(Z) \to Z
$$
is an affine bundle. 
\end{Lemma}

\begin{Remark}
Note this affine bundle is $\bT$-equivariant. 
\end{Remark}

\begin{proof}
We apply the classical Bialynicki-Birula theorem to a smooth 
$\sigma$-equivariant projective compactification $X \subset \overline{X}$. We get a diagram 
$$
\xymatrix{
\LAttr_\fC(Z) \ar@{^{(}->}[r]\ar[d]_\lim& \LAttr_\fC(\bar Z) \ar[d]_{\overline{\lim}} \\
Z \ar@{^{(}->}[r] & \bar Z \\
}
$$
of $\sigma$-equivariant maps 
in which the horizontal arrows are open dense embeddings and $\overline{\lim}$ is an affine bundle. 
Since $\sigma$ acts with positive weights on the fibers of $\overline{\lim}$, any 
nonempty closed subset of the fiber contains the origin. Therefore, $\lim$ is also an affine bundle. 
\end{proof}

\begin{Example}
In Example \ref{ex1}, take $X=\cM(2,n)$, $\fC=\{a_1 > a_2\}$, and 
$$
Z = \left\{ \cF_1 \oplus \cF_2\, \big|\,  \cF_i \in \Hilb(\C^2,n_i) \right\} \,. 
$$
Then $\LAttr_\fC(Z)$ is a vector bundle with fiber $\Ext^1(\cF_2,\cF_1(-1))$, where
$\cF_1(-1)$ means the twist by minus the line at infinity of $\Pp^2$. 
\end{Example}

\subsection{Partial order by attraction}

The choice of a chamber $\fC$ determines a partial ordering on the set 
$$
\Fix=\pi_0(X^\bA)
$$
of connected components $Z$ of the fixed locus. This is a transitive closure of the 
relation 
$$
\overline{\LAttr_\fC(Z)} \cap Z' \ne \emptyset  \Rightarrow Z 
\succeq  Z' \,. 
$$
Using a projective compactification as in proof of Lemma \ref{noncompactBB}, one sees that 
this is indeed a partial order, that is 
$$
Z \preceq Z' \, \textup{and} \, Z' 
\preceq Z   \Rightarrow Z=Z' \,.
$$

\begin{Lemma}\label{partord}
For any component $Z$ of $X^\bA$ the set 
$$
\Branch(Z) = \bigsqcup_{ Z' \preceq Z} \LAttr_\fC(Z')
$$
is closed in $X$.
\end{Lemma}

\noindent 
We call $\Branch(Z)$ the \emph{full attracting set} of $Z$. 

\begin{proof}
Consider the map \eqref{blowdown} and choose an $\bA$-equivariant embedding 
$$
X_0 \hookrightarrow V 
$$
into a linear representation $V$ of $\bA$. Let $V_{\geq 0} \subset V$ denote 
the span of those weight subspaces that are non-negative on $\fC$. We have 
$$
\pi\left(\overline{\LAttr_\fC(Z)}\right) \subset X_0 \cap V_{\geq 0} 
$$
for any component $Z\subset X^\bA$.

Let $x$ lie in the closure of $\LAttr_\fC(Z)$. Then $\pi(x) \in V_{\geq 0}$ 
and the limit 
$$
z' = \textstyle{\lim_\fC} x \in \overline{\LAttr_\fC(Z)} \cap X^\bA 
$$
exists by the properness of $\pi$. Denoting by $Z'\in \Fix$ the component 
that contains $z'$ we see that $Z'\preceq Z$ and so we are done. 
\end{proof}

\subsection{The ample partial order}\label{s_ample_ord} 

It will be more convenient to work with a different 
partial order on $\Fix$ which is a priori finer, that is 
$$
Z \prec Z'  \Rightarrow Z < Z'\,,
$$
but is much easier to describe. 

Let $\sigma\in \fC$ be a cocharacter 
and let $C\cong \Pp^1$ be the closure of a $\sigma$-orbit. 
The degree 
$$
(\lambda,[C]) \in \Z\,, \quad \lambda\in \Pic(X)\,,
$$
may be computed by equivariant localization in terms of 
weights of $\lambda$ at the fixed points of $C$. This number
must be positive if $\lambda$ is ample. 

We therefore choose any $\bA$-linearization of an ample 
line bundle $\lambda$ and define 
\begin{equation}
Z > Z' \, \Leftrightarrow \, 
\left(\textup{weight} \, \lambda\Big|_{Z}  - 
\textup{weight} \, \lambda\Big|_{Z'}
\right)\Big|_{\fC} > 0  \,,
\label{ampl_ord}
\end{equation}
where $\textup{weight} \, \lambda\Big|_{Z} \in \fa^*$ is the 
weight of the $\bA$-action on the fiber of $\lambda$ restricted
to fixed point component $Z$. Note that the ambiguity in the 
choice of linearization cancels out of \eqref{ampl_ord}

See also Section \ref{s_superpolytope} below 
for a related discussion.

\begin{Example}\label{ex3} 
Recall that, by construction, Nakajima varieties come with a distinguished
ample class, namely 
$$
\theta = \sum \theta_i \, c_1(\cV_i) \,. 
$$
Consider the fixed points of the tensor product action 
\begin{equation}
Z_\eta = 
\cM_{\theta,\zeta}(\eta,\bw) \times 
\cM_{\theta,\zeta}(\bv-\eta,\bw') 
\subset 
\cM_{\theta,\zeta}(\bv,\bw+\bw') 
\label{Mtens2} 
\end{equation}
as in \eqref{Mtens}. By construction, 
$$
\textup{weight}\, c_1(\cV_i) \Big|_{Z_\eta} = \eta_i \,. 
$$
Therefore 
\begin{equation}
  \label{Nak_ord}
Z_\eta > Z_{\eta'} \quad  \Leftrightarrow \quad 
\theta \cdot \eta > \theta \cdot \eta' \,. 
\end{equation}
In particular, if $\theta_i >0$ for all $i$ then $Z_\emptyset=Z_0$ is 
minimal with respect to the ample order. 
\end{Example}

\subsection{Lagrangian correspondences}\label{s_lagr_corr} 

Given a holomorphic symplectic variety $M$ with symplectic form $\omega$, recall that a subvariety $Z \subset M$ is 
isotropic if  the restriction of $\omega$ to the smooth locus of $L$ vanishes.  It is Lagrangian
if it is also middle-dimensional.  We say that a cycle is Lagrangian if each component is Lagrangian.

Let $Y$ be another holomorphic symplectic variety on which 
 group $\bG$ acts with the same weight $\hbar$ of the 
symplectic form $\omega_Y$. 
Let
$$
L \subset X \times Y 
$$
be a $\bT$-invariant Lagrangian cycle with respect to 
the form $\omega_X - \omega_Y$. Recall that we use $\subset$ to denote cycles as well as subvarieties.

If $L$ is proper over 
$X$, it defines a map 
$$
\Theta_L: H^\hd_\bT(Y) \xrightarrow{\,\,p_2^*\,\,} H^\hd_\bT(L) 
\xrightarrow{\,\,(p_1)_*\,\,} H^\hd_\bT(X)
$$
As an equivariant 
residue, $\Theta_L$ may be defined  with a weaker properness assumption:
$\bT$ has to have proper fixed points in the
fibers of the push-forward. 

See, for example, \cite{CG} for a general discussion of 
operators defined by correspondences. In particular, 
$\Theta_L$ depends only 
on the class $[L]$ of $L$ in the $\bT$-equivariant Borel-Moore
homology of $X\times Y$. Also
$$
\Theta_{L_1} \circ \Theta_{L_2} = \Theta_{[L_1] \circ [L_2]} \,. 
$$
Here the convolution $L_1 \circ L_2$ of two cycles 
is defined by 
$$
[L_1] \circ [L_2] = (p_{13})_* \,\Delta^*([L_1] \times [L_2]) 
$$
where the maps 
$$
X \times Y \times Y \times Z  \xleftarrow{\,\,\Delta\,\,}
X \times Y  \times Z \xrightarrow{\,\,p_{13}\,\,}
X \times Z 
$$
are the inclusion of the diagonal and the projection, 
respectively.  Here $\Delta^*$ denotes Gysin pullback with respect to a regular embedding.
When the map $p_{13}$ is proper on the support of
$L_1 \times_{Y} L_2$, its image is isotropic.  As a consequence, 
the convolution $[L_1] \circ [L_2]$ is the cycle class of a $\bT$-invariant 
Lagrangian cycle in $X \times Z$. 

\subsection{Steinberg correspondences}\label{s_Stein_corr} 

Let $L\subset X \times Y$ be a Lagrangian correspondence as 
above. 
\begin{Definition}
A Steinberg correspondence 
is a Lagrangian correspondence
$$L \subset X \times Y$$
as above such that there exist proper equivariant maps 
$$
X \xrightarrow{\,\,\pi_{X}\,\,} V \xleftarrow{\,\,\pi_Y\,\,} Y
$$
to an affine $\bG$-variety $V$ such that 
$$
L \subset X \times_{V} Y \,. 
$$
\end{Definition}

The following easy lemma gives a sufficient condition for 
Steinberg correspondences to be closed under convolution.

\begin{Lemma}
Given Steinberg correspondences
$$
L_1 \subset X \times_{V_1} Y\,, \quad
L_2 \subset Y \times_{V_2} Z \,,
$$
the convolution $L_1\circ L_2$ is a Steinberg correspondence 
if there exists a commutative diagram of equivariant proper maps
\begin{equation}
\label{composable}
\xymatrix{
Y \ar[d]_{\pi_{Y,2}}\ar[r]^{\pi_{Y,1}}& V_1\ar[d]\\
V_2 \ar[r] &V
\\
}
\end{equation}
with $V$ affine.
\end{Lemma}
\begin{proof}
Both $X$ and $Z$ map admit proper, equivariant maps to $V$.
It is clear that the assumptions imply 
$$L_1 \circ L_2 \subset X\times_{V} Z.$$
\end{proof}

We say that two Steinberg correspondences are composable if they satisfy the sufficient condition
described above when they share a common factor.

\begin{Example}
Fix a quiver $Q$ and dimension vectors $\bv, \bv^{(i)}$ for $i = 1,\dots, n$, such that $\bv \geq \sum \bv^{(i)}$,
and similarly for $\bw, \bw^{(i)}$.
We have a proper map
$$\prod_{i=1}^{n} \cM_{\theta,\zeta}\left(\bv^{(i)},\bw^{(i)}\right) 
\rightarrow
\cM_{0,\zeta}\left(\sum \bv^{(i)}, \sum \bw^{(i)}\right)
\rightarrow \cM_{0,\zeta}(\bv, \bw)$$
where the first map is given by 
affinization and direct sum, while 
the second map is given by taking the direct sum with the zero representation.
We will only consider proper maps to affine varieties of this form or products of such maps.
As a result, if we have two such maps with the same domain, a commutative diagram of the form \eqref{composable}
always exists since the two targets can both be included into a still-larger $\cM_{0,\zeta}(\bv,\bw)$.
Therefore, the associated Steinberg correspondences will always be composable.
\end{Example}

Given a possibly disconnected variety $X$, if we have a collection
of composable Steinberg correspondences between components of $X$, we can consider the 
subalgebra of $\End H^\hd_\bT(X)$ that they span. When the context is clear, It 
will be called the \emph{Steinberg algebra} of $X$.

\section{Characterization of stable envelopes}\label{s_uniq_stable}

\subsection{Supports}

For the ease of reading formulas, we use restriction signs for the 
natural restriction maps in equivariant cohomology. Given a closed 
$\bT$-invariant subset $Y\subset X$ and a class $\gamma \in H^\hd_\bT(X)$ we say 
that $\gamma$ is supported on $Y$ if 
$$
\gamma\Big|_{H^\hd_\bT(X\setminus Y)}  = 0 \,. 
$$
Equivalently, $\supp \gamma \subset Y$ means
 that the Borel-Moore class $\gamma\cap [X]$ is pushed forward under
$Y\hookrightarrow X$.

\subsection{Polarization}\label{s_polar}

Let $Z\in \Fix$ be a component of $X^\bA$ and 
let $N_Z$ be the normal bundle to $Z$ in $X$. Any chamber $\fC$ 
gives a $\bT$-invariant decomposition
$$
N_Z = N_{+} \oplus N_{-} 
$$
into $\bA$-weights that are positive and negative on $\fC$, 
respectively. The symplectic form $\omega$ gives
\begin{equation}
\left(N_{+}\right)^\vee = N_- \otimes \hbar  \in K_{\bT}(Z) \,,
\label{dualN}
\end{equation}
where $\hbar$ denotes a trivial line bundle with 
the corresponding action of $\bT$.

Because $\hbar$ is trivial on 
$\bA$, the class 
\begin{equation}
  \bsi^2 = (-1)^{(\codim Z)/2} e(N_Z)\Big|_{H^\hd_\bA(\pt)} = 
\prod \alpha_i^2\,, 
\label{def_eps}
\end{equation}
is a perfect square. Here $\pm\alpha_i\in \fa^*$ are
the roots that occur in $N_Z$.  

\begin{Definition}
A choice of a square root $\bsi$ in \eqref{def_eps} will be 
called a polarization of $Z$ in $X$. The sign in $\pm e(N_-)$ 
agrees with polarization if $\pm e(N_-)$ 
restricts to $\bsi$ in $H^\hd_\bA(\pt)$. 
\end{Definition}

\begin{Example}\label{e_pol_T*} While 
polarization is a purely formal choice, 
geometrically natural choices save on signs. 

For example, if $X=T^*Y$ with $\bA$-action induced from 
$Y$, we can take $\bsi$ to be the product of 
nonzero $\bA$-weights in the fibers of $TX \to TY$. 

More generally, let a cocharacter 
$$
\sigma: \C^\times \to \bT
$$
be such that $(\hbar,\sigma)=-1$. This generalizes the scaling action 
of $\C^\times$ in the fibers of $T^*Y$. Then we can choose the weights 
in $\bsi$ as the $\sigma$-negative weights in the fiber
of $N_Z$ over some chosen $x\in Z^\sigma$. 
\end{Example}

\begin{Example}\label{s_pol_Nak}
We have a canonical polarization associated to Nakajima varieties as follows.
Recall from Section \ref{tangentsplitting} that we have a virtual splitting of the tangent bundle
$$T \Nak = \Th + \hbar^{-1}\otimes\left(\Th\right)^\vee.$$
Let $\epsilon$ denote the product, weighted by multiplicity, of the 
nonzero $\bA$-weights in the restriction of $(\Th)^\vee$ to some $x \in Z$.  
\end{Example}

\subsection{Degree in $\bA$}

Since $\bA$ acts trivially on $X^\bA$, we have 
$$
H^\hd_\bT(X^\bA) = H^\hd_{\bT/\bA}(X^\bA) \otimes_{\C[\ft/\fa]} \C[\ft] \,. 
$$
While there is no canonical splitting
\begin{equation}
  \label{split}
  \C[\ft] \cong \C[\ft/\fa] \otimes \C[\fa] 
\end{equation}
any such splitting leads to the same increasing filtration of $H^\hd_\bT(X^\bA)$ 
by the degree $\deg_\bA$ in $\C[\fa]$. Clearly, 
\begin{equation}
  \gr H^\hd_\bT(X^\bA) = H^\hd_{\bT/\bA}(X^\bA) \otimes \C[\fa] \,. 
\label{grH}
\end{equation}

\subsection{Characterization}

Choose a chamber $\fC\subset\fa$ and an polarization $\bsi$ of $X^\bA$.  The following theorem is the main result of this
section.
 
\begin{Theorem}\label{stablebasis}
There exists a unique map of $H_{\bT}^\hd(\pt)$-modules
$$
\Stab_{\fC,\bsi}: H_{\bT}^\hd(X^\bA) \rightarrow H_{\bT}^\hd(X)
$$ 
such that for any $Z\in \Fix$ and any $\gamma\in H^\hd_{\bT/\bA}(Z)$, the stable 
envelope $\Gamma = \Stab_{\fC,\sigma}(\gamma)$ satisfies: 
\begin{enumerate}
\item[\nui]
$\supp \Gamma \subset \Branch(Z)\,,$ 
\item[\nuii]
$\Gamma\big|_{Z} =  \pm e(N_-) \cup \gamma\,,$ according to polarization,  
\item[\nuiii]
$
\displaystyle 
\deg_\bA \Gamma \big|_{Z'} < \tfrac12 \codim Z' \,, 
$
 for any $Z' \prec  Z$\,. 
\end{enumerate}
\end{Theorem}

\begin{Remark}
The chamber and the polarization are independent parameters in 
the construction of $\Stab_{\fC,\bsi}$. The former being much 
more important than the latter, we abbreviate 
$$
\Stab_{\fC}=\Stab_{\fC,\bsi} \,,
$$
once some polarization $\bsi$ has been specified. 
\end{Remark}

\begin{Remark}
We will see $\Stab_{\fC}$ is given by a Lagrangian 
correspondence on $X \times X^\bA$, and, in particular, it maps 
middle degree to middle degree. 
\end{Remark}

The existence of $\Stab_\fC$ will be proven later. 
We now prove the uniqueness a map satisfying the conditions of the theorem. 

\begin{proof}
Let $\gamma \in H_{\bT}^\hd(X)$ be supported on a union of attracting sets and satisfy 
$$
\deg_\bA \iota^* \, \gamma < \tfrac12 \codim Z \,, 
$$
for any embedding $\iota: Z \hookrightarrow X$ of a fixed component. 
We claim this forces $\gamma=0$.

Pick a total ordering on $\Fix$ refining $\prec$ and choose $Z\in\Fix$ so that
$\gamma$ is supported on $\Branch(Z)$. We can factor $\iota = f_3 f_2 f_1$, where 
$$
Z \overset{f_1}\hookrightarrow \LAttr_\fC(Z) 
\overset{f_2}\hookrightarrow \Branch(Z) \overset{f_3}\hookrightarrow X \,. 
$$
Here $f_1$ is regular and $f_2$ is open. 
The support condition on $\gamma$ means that 
$$
\gamma \cap [X] = (f_3)_*\, \alpha 
$$
for a certain Borel-Moore homology class $\alpha$. 
Standard excess intersection arguments then show 
$$
\iota^*(\gamma) \cap [Z] = e(N_{-}) \cap f_1^{*} \, f_2^{*}\, \alpha \,. 
$$
The multiplication by $e(N_{-})$ is injective on \eqref{grH} 
and 
$$
\deg_\bA e(N_{-}) = \tfrac12\codim Z \,. 
$$
Because this exceeds the degree of the right-hand side, 
$f_1^{*} \, f_2^{*}\, \alpha =0$. 
Since $f_1^{*}$ is an isomorphism, this forces $f_2^{*}\alpha$ to vanish, 
meaning that $\gamma$ is supported on a smaller union 
of strata. Arguing inductively, we see $\gamma = 0$.

Now if $\Gamma_1,\Gamma_2\in H^\hd_\bT(X)$ are two classes satisfying
\nui--\nuiii\ then their difference satisfies the hypothesis above, 
hence vanishes. 
\end{proof}

\section{Lagrangian residues}\label{s_Res}

Let $L$ be an $\bA$-invariant Lagrangian and let 
$$
\iota: Z \hookrightarrow X 
$$
be an embedding of a component of $X^\bA$. 
The form $\iota^*\omega $ is 
symplectic and so we can talk about isotropic and Lagrangian subvarieties of $Z$. 

\begin{Lemma}
$L\cap Z$ is an isotropic subvariety of $Z$. 
\end{Lemma}

\begin{proof}
Let  $W$ be an irreducible component of $W$ of $L \cap Z$. 
For a general point $w \in W$, 
there exists a sequence of points $x_1, x_2, \dots$ in the smooth locus of $L$ approaching $w$
such that limit of $T_{x_k}L$ exists as $k \rightarrow \infty$ and contains the tangent space $T_{w}W$.  
This can be seen, for instance, by choosing a Whitney stratification of $L$ for which $L \cap Z$ 
is a union of strata.  Since the symplectic form on $Z$ is the restriction of the symplectic form on $X$, 
the lemma follows.  
\end{proof}

\noindent 
Now suppose an polarization $\bsi$ of $Z$ has been chosen. 

\begin{Lemma}
There is a unique Lagrangian cycle $\Res_Z L$ supported 
on $L\cap Z$ such that 
$$
\iota^*  [L] = \bsi \, [\Res_Z L] + \dots
$$
where dots stand for terms of smaller $\bA$-degree. 
\end{Lemma}

\begin{proof}
The class $\iota^*[L]$ is supported on a subvariety $L\cap Z$ of dimension at most 
$\frac{1}{2}\dim Z$.  Therefore, its $\bA$-degree can be at most 
$$
\codim_X L  - \codim_Z (L\cap Z) \le  \tfrac12 \codim_X Z\,.
$$
Assuming $L\cap Z$ is middle-dimensional, denote by $L_1,L_2, \dots $ its Lagrangian irreducible components. 
We have 
$$
\iota^*[L] = \sum [L_i]\cdot f_i + \dots 
$$
where $f_i \in H^\hd_\bA(\pt)$ is a homogeneous polynomial of degree $\tfrac12 \codim_X Z$ and 
dots stand for terms of smaller degree.

In order to calculate $f_i$, we shrink $X$ to a neighborhood of a smooth generic point of $L_i$. 
 Furthermore, we can degenerate to the normal cone of $Z$ inside $X$ and restrict to a transverse slice through a 
generic point of $L_i$.  After these simplifications, the following lemma finishes the proof. 
\end{proof}

\begin{Lemma}
Let $V= \C^{n}$ be a vector space equipped with the diagonal action of $\bA$ by characters
$\chi_1, \dots, \chi_n$.  Let $X = V \oplus V^{\vee}$ be the symplectic vector space equipped with the induced action of 
$\bA$ and suppose we have a Lagrangian $\bA$-invariant conical subvariety $L \subset X$.  
Then the residue of $[L]$ at the origin $0=Z \subset X$ is an integer multiple of $\bsi=\prod_{j=1}^n \chi_j$.
\end{Lemma}

\begin{proof} 
We embed $\bA$ in $T = (\C^\times)^{n+1}$, the maximal torus of $Sp(2n)\times 
\C^\times$, the stabilizer of the line $\C\omega\in \Omega^2(X)$. 
We can use $T$ to degenerate $L$ via a family of $\bA$-invariant conical subvarieties to a 
$T$-invariant conical subvariety and calculate the residue for this limit.  
Since $T$ scales $\omega$, this limit is still Lagrangian.  
On the other hand, the only such $T$-invariant subvarieties are unions of 
Lagrangian coordinate planes.  For a Lagrangian coordinate plane, it is clear that the residue is a product of 
the characters $\chi_j$ up to a sign.
\end{proof}

\begin{Lemma}\label{l_correct1}
For any $\bA$-invariant Lagrangian $L$ and any chamber $\fC$, 
there exists a Lagrangian cycle $L'$ supported
on $\Branch(Z)$ such that 
$$
\deg_\bA \iota^* ([L]-[L']) < \tfrac12 \codim Z \,. 
$$
\end{Lemma}

\begin{proof}
We can take $L'$ to be the closure of $\lim_{\fC}^{-1} (\pm \Res_Z L)$, 
counting multiplicity.  
\end{proof}

\begin{Lemma}\label{l_correctL} 
Let $L\subset X$ be an $\bA$-invariant Lagrangian subvariety 
supported on $\Branch(Z)$. Then there exists a unique Lagrangian cycle $L'$ such that 
$L'-L$ is supported on $\bigcup_{Z'\prec Z} \Branch(Z')$ and 
$$
\deg_\bA \, \iota_{Z'}^* [L'] < \tfrac12 \codim Z' 
$$
for any $Z'\prec Z$. 
\end{Lemma}

\begin{proof}
 The existence follows by induction from Lemma \ref{l_correct1}. The 
uniqueness is shown as in Section \ref{s_uniq_stable}. 
\end{proof}

In conclusion, we note that if $L$ is $\bT$-invariant, then so are
all other Lagrangians occurring in the above Lemmas.

\section{Proof of existence}

Consider the (possibly disconnected) $\bT$-variety $X \times X^{\bA}$ equipped with the antidiagonal symplectic form $(\omega,-\omega|_{X^{\bA}})$.
We construct $\Stab_{\fC}$ by exhibiting a correspondence between $X^{\bA}$ and $X$.

\begin{Proposition}\label{universalstableleaf}
There exists a $\bT$-invariant 
Lagrangian cycle $\textup{\textit{Stab}}$ $\cL_{\fC}$ on  $X \times X^{\bA}$, proper 
over $X$, 
with the following properties:
\begin{enumerate}
\item[\nui] For any $Z\in\Fix$, 
the restriction of $\cL_{\fC}$ to $X \times Z$ is supported on $\Branch(Z) \times Z\,;$
\item[\nuii]
the restriction of $[\cL_{\fC}]$ to $Z \times Z$ equals $\pm e(N_-) \cap [\Delta]$, according to polarization, where 
$\Delta$ is the diagonal$\,;$ 
\item[\nuiii] 
\label{lagrangianrestriction}for $Z' \prec Z$, the restriction of $[\cL_{\fC}]$ to $Z' \times Z$
has $\bA$-degree less than $\frac{1}{2}\codim Z'\,.$ 
\end{enumerate}
\end{Proposition}

\noindent 
This shows the existence of $\Stab_{\fC}$ by taking the map
$$
H_\bT^\hd(X^{\bA}) \rightarrow H_\bT^\hd(X)
$$
induced by the correspondence $\cL_{\fC}$. Properness over 
$X$ insures this map is well-defined. 

\begin{proof}

Fix some $Z$ and let $\pm L$ be the closure of the preimage of $\Delta$ under the map
$$
\LAttr_\fC Z \times Z   \to Z \times Z \,, 
$$
with sign as above. 
Then $L$ is a $\bA$-invariant Lagrangian supported on $Z \times \Branch(Z)$ which satisfies 
\nui\ and \nuii. 
Using Lemma \ref{l_correctL}, we can modify it on lower strata so that to achieve \nuiii. 
Repeating this for all $Z\in\Fix$, we obtain a Lagrangian cycle $\cL_{\fC}$ and it 
remains to check that its support is proper over $X$. 

As in the proof of Lemma \ref{partord}, choose a $\bA$-equivariant embedding 
$$
\pi: X_0 \hookrightarrow V
$$ 
into a linear representation of $\bA$ and let 
$V_0 \subset V_{\geq 0}$ be the subspaces formed by $\bA$-invariant and weights positive on $\fC$, 
respectively.  Let 
$$
\rho:  V_{\geq 0} \rightarrow V_0 
$$
be the natural projection. 
Consider the closed set $\pi^{-1}(V_{\geq 0}) \subset X$ (this is just
the union of all attracting manifolds), along with the morphism
$$
\rho\circ\pi: \pi^{-1}(V_{\geq 0}) \rightarrow V_0\,.
$$
By construction, the Lagrangian cycle $\cL_{\fC}$ lies in the fiber product
$$
 \pi^{-1}(V_{\geq 0})  \times_{V_0}  X^{\bA}\subset X \times X^\bA\,.
$$
Indeed, we construct $\cL_{\fC}$ by starting with the diagonal $\Delta \subset X^{\bA} \times X^{\bA}$
and taking attracting manifolds and closures.  The fiber product is closed with respect to both these operations.

On the other hand, the projection onto the first factor
$$
\pi^{-1}(V_{\geq 0})  \times_{V_0}  X^{\bA} \rightarrow X
$$
is proper:  since the map $\pi:  X \rightarrow V$ is proper, we can reduce the statement to the claim
that
$$
 V_{\geq 0}  \times_{V_0} V_{0} \rightarrow V
$$
is proper, which is obvious.

\end{proof}

We note the following corollary of the proof. 
It will play an essential role in proving various
properness statements later. 

\begin{Proposition}\label{p_leafX0}
Let $X_+$ denote the union of all attracting manifolds. Then 
$$
\cL_{\fC} \subset X_+ \times_{X_0} X^\bA \,. 
$$
\end{Proposition}

\begin{Remark}
Suppose $X = T^*Y$ where $Y$ is a smooth projective variety and assume the action of $\bA$ is induced from an action on $Y$ with isolated fixed points $\{p_k\}$.  Then a choice of chamber 
$\fC$ defines an $\bA$-invariant Bialynicki-Birula stratification of $Y$ by locally closed varieties $V_{p_{k}}$.  In this case, 
the stable envelope map $\Stab_\fC$ defines a collection of Lagrangian cycles on $X$.
These can be identified (up to a sign depending on the polarization) with the characteristic cycles of the constructible sheaves 
$(j_{k})_! \Q_{V_{p_{k}}}$ where $j_{k}$ denotes the inclusion into
$Y$. See, in particular, \cite{Su} for recent developments in this direction. 
\end{Remark}

\section{Torus restriction}

Let $\fC$ be a chamber and let 
$\fC'\subset \fC$ be a face of some dimension.  Consider 
$$
\fa' = \Span \fC' \subset \fa
$$
with associated subtorus $\bA'$.
The cone $\fC$ projects to a cone in $\fa/\fa'$ that we denote 
by $\fC/\fC'$. 

Let $\bsi$ be an polarization of $X^\bA\subset X$. We can factor 
$$
\bsi = \bsi' \, \bsi''
$$
into weights that are zero and nonzero on $\fa'$, respectively. The factors 
induce an polarization of $X^\bA\subset X^{\bA'}$ and $X^{\bA'}\subset X$, 
respectively. In the following lemma, we take these induced polarizations. 

\begin{Lemma}\label{l comm St} 
The diagram 
\begin{equation}
  \label{commSt}
  \xymatrix{
H^\hd\left(X^\bA\right) \ar[rr]^{\Stab_{\fC}} \ar[rd]_{\Stab_{\fC/\fC'}\,\,} && H^\hd(X)  \\
& H^\hd\left(X^{\bA'}\right) \ar[ru]_{\,\,\Stab_{\fC'}}\\
}
\end{equation}
is commutative. 
\end{Lemma}
\begin{proof}
This follows from the uniqueness of the stable envelopes.
Let $\cL_{\fC'}, \cL_{\fC/\fC'}$ be the Lagrangian correspondences constructed in Proposition \ref{universalstableleaf}, and consider their convolution
$$[\cL_{\fC,\fC'}] = [\cL_{\fC/\fC'}]\circ [\cL_{\fC'}]$$
which defines a Lagrangian cycle class in $X^{\bA} \times X$.

If we can show it satisfies the properties in Proposition \ref{universalstableleaf}, then uniqueness of $\Stab_{\fC}$ gives the result. In fact, using the definition of the chamber $\fC/\fC'$, most of the properties are immediate. For example, \nuiii\
 follows from the degree constraints of either $\cL_{\fC'}$ or $\cL_{\fC/\fC'}$.
\end{proof}

\section{Symplectic resolutions}\label{s_sympl_res}

\subsection{}

In this paper, we are mainly interested in equivariant symplectic resolutions, 
$$
X \rightarrow X_0 = \Spec H^0(\cO_X) \,,
$$ 
see  \cite{Kal} for a comprehensive discussion. For 
symplectic resolutions, stable envelopes are easier to 
construct and enjoy stronger properties.

In addition to Nakajima quiver varieties $\cM_{\theta,0}$ for $\theta$ generic,  examples
of symplectic resolutions include $T^*(G/P)$, where $P\subset G$ is a parabolic subgroup. 

\subsection{}

We begin with the universal deformation of the pair $(X,\omega)$ 
\begin{equation}\label{deform}
\xymatrix{
X \ar@{^{(}->}[r]^{\iota_0}\ar[d]& \tX \ar[d]_\phi \\
[0] \ar@{^{(}->}[r] & B \cong H^2(X,\C) \,,  \\
}
\end{equation}
in which the period map $\phi$ associates to a deformation 
$(X',\omega')$ the class of $\omega'$ in $H^2(X')=H^2(X)$. 
This universal deformation may be written down explicitly for 
Nakajima varieties and in all other examples, see  \cite{Kal} for further discussion. 

The deformation \eqref{deform} is $\bG$-equivariant, where 
$\bG$ acts on the vector space $B$ by the character $\hbar$. 
Therefore, the group 
$$
\bG_\omega = \Ker \hbar \supset \bA 
$$
acts on each fiber of $\phi$.

\subsection{}\label{s_coroot}

Suppose we are given a class 
$$
\alpha^\vee \in H_2(X,\Z)
$$
that is an effective curve class in some
fiber $(X',\omega') \ne (X,\omega)$. Then 
$$
\int_{\alpha'} \omega' = 0 
$$
and hence deformations with nonzero 
holomorphic curve classes belong to a union of hyperplanes in the base $B$. 
\begin{Definition}
A coroot hyperplane of $X$ is a hyperplane of $B$ along which the deformation of $X$ has nonzero holomorphic
curve classes. 
\end{Definition}
Over their complement 
$$
B^\circ = B \setminus \bigcup_{\textup{coroots}} {(\alpha^\vee)}^\perp
$$
the fiber of $\phi$ is affine.  It is an interesting question to 
find a geometric definition of coroots of $X$ themselves rather than just their associated hyperplanes.

\subsection{}
Consider the diagonal 
$$
\Delta^\circ \subset \tX^\circ \times_{B^\circ} \big(\tX^\circ\big)^\bA  \,,
$$
where $\tX^\circ=\phi^{-1}(B^\circ)$. Since the fibers over $B^\circ$ 
contain no holomorphic cycles, the inclusion 
\begin{equation}
\LAttr_\fC \Delta^\circ \hookrightarrow \tX^\circ \times_{B^\circ} \big(\tX^\circ\big)^\bA  
\label{family_B0}
\end{equation}
is a closed embedding and  defines a family of cycles over 
$B^\circ$. We denote by
$$
\widetilde{\cL}_\fC = \overline{\LAttr_\fC \Delta^\circ}
$$
it closure in $\tX \times_B \tX^\bA$. In particular, we can take 
the $\bA$-fixed points 
\begin{equation*}\label{fixedL}
\xymatrix{
\left(\widetilde{\cL}_\fC\right)^\bA \ar@{^{(}->}[r]
& \tX^\bA \times_B \tX^\bA \ar[d]_\phi \\
& B   \\
}
\end{equation*}

\begin{Proposition}\label{Stein_fiber} 
For any $b\in B$, the top-dimensional components of 
$$
\left(\widetilde{\cL}_\fC\right)^\bA  \, \cap \phi^{-1}(b)
$$
are Steinberg correspondences. 
\end{Proposition}

\begin{proof}
All fibers of $\phi$ are symplectic resolutions and we can 
find a universal proper $\bG$-equivariant map $\widetilde{\pi}$ 
$$
\xymatrix{
\tX \ar[r]^{\widetilde{\pi}} \ar[d]_\phi
& \widetilde{V}\ar[d] \\
B \ar[r]^{\textup{id}}& B   
}
$$
into a vector bundle $\widetilde{V}$ over $B$. The torus 
$\bA$ acts trivially on $B$ and we denote 
by $\widetilde{V}_{\ge 0}$ the subbundle formed by 
$\bA$-weights that are nonnegative on $\fC$. As in the 
proof of Proposition \ref{universalstableleaf}, one 
shows 
$$
\widetilde{\cL}_\fC \subset {\widetilde{\pi}}^{-1} \left(\widetilde{V}_{\ge 0}
\right) \times_{\widetilde{V}_{0}} \tX^\bA \,. 
$$
Therefore
$$
\left(\widetilde{\cL}_\fC\right)^\bA \subset 
\tX^\bA \times_{\widetilde{\pi}} \tX^\bA \,. 
$$
On the other hand, it is known that the $\phi$-fibers of 
$$
\tX \times_{\widetilde{\pi}} \tX \subset \tX \times_B \tX 
$$
are isotropic.\footnote{
This widely known and used statement may be deduced 
from the results of Kaledin \cite{Kaledin_Poisson} and 
Namikawa \cite{Nami}. Further details may be found in the 
forthcoming lecture notes of V.~Ginzburg on the subject.}
Therefore, their intersections 
with a symplectic subvariety $\tX^\bA \times  \tX^\bA$ are at 
most Lagrangian. Their Steinberg property is clear from the above. 
\end{proof}

\begin{Remark}
This Proposition gives an abundant source of Steinberg 
correspondences, as we will see below. 
\end{Remark}

\begin{Theorem}
The correspondence $\cL_\fC$ is the specialization of 
$\widetilde{\cL}_\fC$ to the central fiber, that is 
$$
\left[\cL_{\fC}
\right] = \iota_0^*\,
\big[\,\widetilde{\cL}_\fC\,\big] 
\in H^{BM}_{\bT}(X\times X^{\bA})\,.
$$
\end{Theorem}

\begin{proof}
It suffices to check 
the right-hand side satisfies the conditions of Proposition \ref{universalstableleaf}.  Properness is shown exactly as in the proof of
Proposition \ref{universalstableleaf}. 
Similarly, conditions \nui\ and \nuii\
follow from construction.

To show \nuiii\ we the consider inclusion 
$$
\iota: Z' \times Z \hookrightarrow \tX \times \tX^\bA\,, \quad 
Z \ne Z'\,, 
$$
of an off-diagonal component of $X^\bA \times  X^\bA$. 
By Proposition \ref{Stein_fiber} 
$$
\iota^*\,
\big[\,\widetilde{\cL}_\fC\,\big] = \sum f_i \, [L_i] + \dots \,, \quad 
\quad f_i \in H^{\codim Z'}_\bT(\pt) \,,
$$
where $L_i$ are the Lagrangian components of the intersection 
and dots stand for terms of smaller $\bA$-degree. The required degree
bound
$$
\deg_\bA f_i < \frac12 \codim Z'
$$
follows from a much stronger claim: all $f_i$ are divisible 
by $\hbar$. We state this as a separate result. 
\end{proof} 

For any $X$, not necessarily a symplectic resolutions, we can write
\begin{equation}
[\cL_\fC]\Big|_{X^\bA \times X^\bA} = 
\pm e(N_-) \cup \Delta + \,\textup{off-diagonal}
\label{diag_off}
\end{equation}
where the second term is a class supported on 
$$
\bigsqcup_{Z_1 \prec Z_2} Z_1 \times Z_2 \,, \quad Z_i \in \Fix \,.
$$

\begin{Theorem}\label{t_strong_vanish} 
For symplectic resolutions,
$$
[\cL_\fC]\Big|_{X^\bA \times X^\bA} = 
\pm e(N_-) \cup \Delta \mod \hbar H^\hd_\bT(X^\bA \times X^\bA) \,. 
$$
\end{Theorem}




\begin{proof}
Let $Z,Z'$ be two different components of $X^\bA$. We will show
the pull-back  of $\widetilde{\cL}$ by 
$$
\iota: Z' \times Z \hookrightarrow \tX \times_B \tX^\bA
$$
is divisible by $\hbar$, which also completes the proof of the last theorem. 
We choose a general line $\ell \subset B$
through the origin in the base of the deformation and denote by 
$\tX_\ell$ the restriction of $\tX$ to $\ell$. We may factor 
$\iota=\iota_2 \circ \iota_1$ where 
$$
Z' \times Z \xrightarrow{\,\,\iota_1\,\,} (\tX_\ell)^\bA \times_\ell (\tX_\ell)^\bA
\xrightarrow{\,\,\iota_2\,\,} \tX \times_B \tX^\bA\,. 
$$
Only the central fiber of $\tX_\ell$ contains holomorphic curves.  Therefore, 
if we consider the connected component $W$ of  $(\tX_\ell)^\bA \times_\ell (\tX_\ell)^\bA$
containing $Z \times Z'$,
the contribution of $W$ to $\iota_2^*  \big[\widetilde{\cL}\big]$ is supported over the origin, i.e.
$$
\supp_W \, \iota_2^*  \big[\widetilde{\cL}\big] \subset Z' \times Z. 
$$
Therefore $\iota^* \big[\widetilde{\cL}\big]$ factors through 
$$
\iota_1^* \circ (\iota_1)_* = \textup{multiplication by $\hbar$} \,. 
$$
\end{proof}

\chapter{Properties of $R$-matrices}\label{s_properties}

\section{Definition and braid relations} 

\subsection{} 

We fix some polarization $\bsi$ and consider 
the maps 
$$
\Stab_\fC: H^\hd_\ga(X^\bA) \to H^\hd_\ga(X) 
$$
parameterized by the chambers $\fC$. 
Here $\ga$ is a reductive group which commutes with $\bA$ and 
we denote $\gal = \Lie \ga$. 

The maps $\Stab_\fC$ 
become isomorphisms after inverting $e(N_-)$. Therefore 
we can make the following 

\begin{Definition}
$$
R_{\fC',\fC} = \Stab_{\fC'}^{-1} \, \circ \Stab_{\fC} \in 
\End\left(H^\hd_\ga(X^\bA)\right) \otimes 
\Q\left(\gal\right)\,. 
$$
\end{Definition}

\noindent

\subsection{Example}\label{s_R_Yang} 

Take $X=T^*\Pp^1$ with the action of $\bA = \C^\times$
induced from $\Pp^1$. We have 
$$
X^\bA = \{0, \infty \} 
$$
Let $u$ be the $\bA$-weight 
in $T_0\Pp^1$ and let $\C_\hbar^\times \subset \bG_\bA$ scale 
the cotangent fibers with weight $-\hbar$. Let the polarization 
$\bsi$ be given by the fibers. Then 
$$
\Stab_\fC(0) = [\Pp^1] + [F_\infty], \quad \Stab_\fC(\infty) = 
- [F_\infty] 
$$
for $\fC=\{u>0\}$ where 
\begin{alignat*}{2}
&[\Pp^1] &&= \textup{zero section} \,,\\
&[F_\infty] && = \textup{fiber over $\infty\in\Pp^1$} \,. 
\end{alignat*}
Similarly
$$
\Stab_{-\fC}(0) = - [F_0], \quad \Stab_{-\fC}(\infty) = 
 [\Pp^1] + [F_0] \,. 
$$
For $\{z_1,z_2\} = \{0,\infty\}$, we have 
$$
\Stab_{\pm \fC}(z_j)\Big|_{z_i} = 
\begin{pmatrix}
-u-\hbar & 0 \\
-\hbar &  u 
\end{pmatrix}\,, 
\begin{pmatrix}
-u & -\hbar \\
0 &  u -\hbar  
\end{pmatrix}\,. 
$$
Therefore, 
\begin{equation}
R(u) = \dfrac{1 - \frac{\hbar}u \, \mathbf{s}}{1 - \frac{\hbar}u} 
\label{R_Yang}
\end{equation}
where $\mathbf{s}$ is the permutation of $0$ and $\infty$. 
Up to proportionality, this is Yang's original $R$-matrix. 
It is normalized so that $R(u)=1$ on the invariants of $\mathbf{s}$.

\subsection{}

It will 
be convenient to represent rational functions appearing 
in $R_{\fC',\fC}$ as formal power series in inverse roots using some 
splitting \eqref{split} and 
$$
\frac1{\alpha+x} =
\frac1{\alpha} - \frac{x}{\alpha^2} + \frac{x^2}{\alpha^3}+\dots \,. 
$$
Here $\alpha\in \fa^*$ is a root, i.e.\ a weight 
appearing in the normal bundle to $X^\bA$,  and $x$ is the 
$(\ga/\bA)$-equivariant 
Chern root of the 
corresponding weight subspace of $N_-$. Since we only inverting 
$e(N_-)$, all 
denominators occurring in the $R$-matrices are of this form. 

One should keep in mind that this expansion depends
on a splitting \eqref{split} and reexpand accordingly 
if the splitting is changed.

For a different polarization, the $R$-matrices differ by 
conjugation by a diagonal $\pm1$ matrix. 

\subsection{Root $R$-matrices}\label{s_root_R}

Evidently, it is enough to 
consider $R$-matrices corresponding to a pair of chambers $\fC,\fC'$
separated by a wall $\alpha=0$. 
Here $\alpha$ is a root and we may assume
that $\alpha(\fC)>0$. Consider the subtorus 
$\bA_\alpha \subset \bA$ with Lie algebra $\fa_\alpha= \Ker \alpha$. 
We denote  
$$
X^\alpha 
= \, X^{\bA_\alpha}  \,. 
$$
For the $\bA/\bA_\alpha$-action on $X^{\alpha}$, there are
two chambers, namely $\alpha\gtrless 0$. We take the 
induced polarization of $X^\bA \subset X^\alpha$ and 
denote by 
$$
R_\alpha=R_{<0,>0}\in \End(H^\hd_\ga(X^\bA)) \otimes \Q(\gal/\fa_\alpha)
$$
the corresponding $R$-matrix. 

{}From Lemma \ref{l comm St} we have the following 

\begin{Corollary} If $\fC$ and $\fC'$ are separated by a wall $\alpha=0$ then 
$$
R_{\fC',\fC} = R_\alpha \,.
$$
\end{Corollary}

\noindent 
We call operators $R_\alpha$ the \emph{root $R$-matrices}.

\subsection{$R$-matrices for Nakajima varieties}\label{RforNakajima}

Given a quiver $Q$, vector $\bw$, and a generic choice of $\theta$,
we define 
\begin{equation}
\cM(\bw) = \bigsqcup_{v} \cM_{\zeta,0}(\bv,\bw) \,, 
\label{defcMw}
\end{equation}
where we dropped the moment map parameters on the left-hand side for brevity, and define
$$
H(\bw) = H^{\hd}_{\bG}(\cM(\bw))\,.
$$

Consider a tensor product of Nakajima varieties as 
in Section \ref{s_bw=}. There are two chambers
$$
\fC =\{u >0\} \,, \quad \fC'=\{u<0\}\,, 
$$
where $u$ is the weight of 
the defining representation of $\bA=\{z\}=\C^\times$. We denote
$$
R_{\bw',\bw''}(u) 
= R_{\fC',\fC} \in  \End(H(\bw') \otimes H(\bw'')) \otimes \Q(u)
$$
the corresponding $R$-matrix. 

\subsection{} \label{e_bR} 
More generally, a decomposition 
$$
\bw = \sum_{i=1}^n \bw^{(i)} 
$$
gives a homomorphism 
$$
\bA = \{(z_1, \dots,z_n) \} \to G_\bw 
$$
given by 
$
\bw = \sum_{i=1}^n \bw^{(i)} \, z_i 
$
as in Section \ref{Nak_fix}. By Proposition \ref{p_Nak_fix}
$$
\cM(\bw)^\bA = \cM\left(\bw^{(1)}\right) \times \cdots \times 
\cM\left(\bw^{(n)}\right)
$$
and hence 
$$
H^\hd_{\bG_\bA}\!\!\left(\cM(\bw)^\bA\right) = 
H\left(\bw^{(1)}\right) \otimes \cdots 
\otimes H\left(\bw^{(n)}\right) \,.
$$
The walls are the roots of $GL(n)$ 
$$
\alpha = a_i - a_j \,, \quad 1 \le i < j \le n \,, 
$$
and the corresponding fixed loci are of the form 
$$
\cM(\bw)^\alpha = \cM\left(\bw^{(i)}+\bw^{(j)}\right) \times \prod_{k\ne i,j}
\cM\left(\bw^{(k)}\right)\,, 
$$
where $\bA/\bA_\alpha$ acts only on the first factor. 
We conclude 
$$
R_\alpha = R_{\bw^{(i)},\bw^{(j)}}(a_i-a_j)_{ij}
$$
where the subscript means that it operates in the $i$th 
and $j$th tensor factors.

\subsection{Normalization}

{}From definitions, we have the following 

\begin{Proposition}\label{p_1u}
$$
R_\alpha = 1 + O\left(\alpha^{-1}\right)\,, \quad \alpha \to \infty \,. 
$$
\end{Proposition}

\noindent 
In other words, $R_\alpha$, as a formal power series in $\alpha^{-1}$ 
starts with the identity operator. 

For symplectic resolution, we deduce from Theorem \ref{t_strong_vanish}

\begin{Proposition}\label{p_1h}
$$
R_\alpha = 1 + O(\hbar)\,, \quad \hbar \to 0 \,. 
$$
\end{Proposition}
\noindent
In other words, $R_\alpha$ acts as identity on 
$H^\hd_\ga(X^\bA)/\hbar H^\hd_\ga(X^\bA)$.

\subsection{Braid relations}

Let $\fF\subset \fC$ be a codimension 2 facet and let 
$$
\fC = \fC_0, \fC_1, \dots, \fC_{2n} = \fC
$$
be the chambers containing $\fF$ as a facet, in cyclic order around $\fF$. 

\begin{Proposition}
\begin{equation}
\label{braid} 
R_{\fC_0,\fC_1} R_{\fC_1,\fC_2} \dots R_{\fC_{2n-1},\fC_{2n}} = 1 
 \end{equation}
\end{Proposition}

This relation, too obvious to be called a theorem, is of fundamental 
importance for much of what follows. 

\subsection{Example} \label{e_YB} 
In the setup of Section \ref{e_bR}, take 
$$
\fF =\{ a_1 = a_2 = a_3\} \,. 
$$
Then \eqref{braid} gives
\begin{multline}
  R_{12}(a_1-a_2) \, R_{13}(a_1-a_3) \, R_{23}(a_2-a_3) = \\
  R_{23}(a_2-a_3) \, R_{13}(a_1-a_3) \, R_{12}(a_1-a_2) \,,
\label{YB} 
\end{multline}
which is the Yang-Baxter equations with a spectral parameter.

\section{Changing the torus}\label{s_torus_change}

\subsection{} 

Suppose we have and inclusion of  tori 
$$
\bA_1 \subset \bA_2  
$$
where $\bA_2$ preserves the symplectic form. 
Clearly, 
$$
\textup{roots}\left(\bA_1\right) =  
\textup{roots}\left(\bA_2\right)\Big|_{\fa_1} 
\setminus \{0\}\,,
$$
and so every chamber $\fC_1 \subset \fa_1$ is contained in 
at least one closed chamber $\fC_2 \subset \fa_2$. 
{} From Lemma \ref{l comm St}, we deduce the following 

\begin{Proposition}
Let chambers $\fC_1,\fC'_1 \subset \fa_1$ be faces of 
$\fC_2,\fC'_2 \subset \fa_2$, respectively. Then the 
diagram 
$$
\xymatrix{
H^\hd_\gat\left(X^{\bA_2}\right) \ar[rr]^{\Stab_{\fC_2/\fC_1}}  
\ar[d]_{R_{\fC'_2,\fC_2}} && 
H^\hd_\gat\left(X^{\bA_1}\right) \ar[d]^{R_{\fC'_1,\fC_1}}  \\
H^\hd_\gat\left(X^{\bA_2}\right) \ar[rr]^{\Stab_{\fC'_2/\fC'_1}}  && 
H^\hd_\gat\left(X^{\bA_1}\right) 
}
$$
is commutative. 
\end{Proposition}

\noindent Here $\Stab_{\fC_2/\fC_1}$ really means 
$\Stab_{\fC_2/\fC_{2,1}}$, where $\fC_{2,1}\subset \fC_2$ 
is the minimal face that contains $\fC_1$.

\subsection{}

Note that there could be many walls between $\fC_2$ and $\fC'_2$ even 
when $\fC_1$ and $\fC'_1$ are adjacent. Thus enlarging the torus
leads to factorization of root $R$-matrices. 

\subsection{}

In practice, it convenient to reduce to the situation when 
$$
\dim \fa_1 = 1 \,, \quad \dim \fa_2 = 2\,, 
$$
by restricting to root $R$-matrices for $\bA_1$ and 
replacing $\fa_2$ by a generic line in $\fa_2/\fa_1$, 
if necessary. Denoting by $(u_1,u_2)$ the corresponding coordinates
in $\fa_2$, we can go between 
$$
\fC_2 = \{ u_1 \gg u_2 > 0\} \,, \quad 
\fC'_2 = \{ u_2 >0 \gg u_1\} 
$$
by crossing the walls in the decreasing order of $u_1/u_2$.

\subsection{Example} \label{e_RRR}
We continue with Example \ref{e_bR} and take 
\begin{align*}
\fa_1 & = \{(a_1,0,\dots,0) \} \,, \\
\fa_2 & = \fa_1 \oplus \C (0,a_2,\dots,a_n) \,. 
\end{align*}
To ensure that $\fa_2/\fa_1$ is generic in $\fa/\fa_1$, it is 
enough to take 
\begin{equation}
a_2 > a_3 > \dots > a_n \,. 
\label{cone_a2n}
\end{equation}
Then $X^{\fa_2}= X^\fa$, while 
$$
X^{\fa_1}= \cM\left(\bw^{(1)}\right) 
\times \cM\left(\bw-\bw^{(1)}\right)\,. 
$$
In $\fa_1$, we have two chambers
$$
\fC_1=\{a_1>0\}\,, \quad \fC'_1=\{0>a_1\} \,, 
$$
corresponding to 
$$
\fC_2 = \{a_1>a_2 > \dots > a_n\} \,, \quad 
\fC'_2 = \{a_2 > \dots > a_n > a_1 \} \,.  
$$
in $\fa_2$. Crossing from $\fC_2$ to $\fC'_2$, we get  
\begin{equation}
R_{\fC'_1,\fC_1} = R_{1,n}(a_1-a_n) 
\cdots R_{1,3}(a_1-a_3)\, R_{1,2}(a_1-a_2)
\label{RRR}
\end{equation}
in the stable basis of $H^\hd_\ga(X^{\fa_1})$ 
corresponding to the chamber \eqref{cone_a2n} in 
$\fa/\fa_1$. For a different choice of chamber, 
one reorders the factors accordingly.

\section{Covers and factorization of $R$-matrices}\label{s_covers}

\subsection{} 

It is interesting to elaborate on the factorization 
considered in Section \ref{s_torus_change} in the 
following special case. Let $Q$ be a quiver. We take 
two vertices $i,j\in I$ and 
$$
\bw = a \delta_i + \delta_j 
$$
where $a$ is a weight of $\bA_1 \cong \C^\times$. We have
$$
\cM(\bw)^{\bA_1} = \cM(\delta_i) \times \cM(\delta_j) \,. 
$$
The corresponding $R$ matrix
$$
R_{H_i,H_j}(a) \in \End(H_i \otimes H_j)\otimes
\Q(a) \,, \quad H_i = H(\delta_i)\,, 
$$
is one of the main building blocks of the theory. 

\subsection{} 

We take $\bA_2/\bA_1$ to be the maximal torus of $G'_\edge$ 
and denote by 
$$
\Gamma = \left(\bA_2/\bA_1\right)^\wedge \cong H_1(Q,\Z)
$$
its character group. As explained in Section \ref{s_abelian_cover}
$$
\cM(\bw)^{\bA_2} = \widetilde{\cM}(\delta_i) \times 
\widetilde{\cM}(\delta_j) \,, 
$$
where $\widetilde{\cM}$ are the quiver varieties 
associated to the universal abelian cover 
$\widetilde{Q}$ of the quiver $Q$. 

Here we lift 
vertices of $Q$ to vertices of $\widetilde{Q}$ that 
correspond to the trivial character of $\bA_2/\bA_1$. 
They form a fundamental domain for the action of $\Gamma$. 

\subsection{}

The walls in $\bA_2$ that we need to cross are of the form 
\begin{equation}
a = \gamma \,, \quad \gamma \in \Gamma\,,  
\label{wall_gamma}
\end{equation}
and the corresponding fixed loci are $\widetilde{\cM}(\bw_\gamma)$ 
where 
$$
\bw_\gamma = a \delta_{\gamma i} + \delta_{j} \,.
$$
Recall that $\Gamma$ acts freely on the vertices of $\widetilde{Q}$ 
and the $a \delta_{\gamma i}$ term in $\bw$ means 
that the corresponding framing arrow goes from a space of 
weight $a$ to a space of weight $\gamma$. On the wall
\eqref{wall_gamma} these 
weights match and we get fixed points.

\subsection{}\label{s_order_chamber}

To order the walls \eqref{wall_gamma}, we pick a generic 
vector $t\in \fa_2/\fa_1$ and order them in the decreasing 
order of $\gamma(t)$. Then 
\begin{equation}
R_{H_i,H_j}(a) = \overleftarrow{\prod_\gamma} 
\, \widetilde{R}_{H_{\gamma i},H_j}(a-\gamma)
\label{Rprod1}
\end{equation}
in the stable basis corresponding to $\fC_2 \owns t$ and 
the ordering of the product is such that we cross the wall 
with the larger value of $\gamma(t)$ first. 

Here  $\widetilde{R}$ is the $R$-matrix for the 
quiver $\widetilde{Q}$ and 
we use the embedding $\bA_2^\wedge \hookrightarrow \fa_2^*$
to write arguments of the $R$-matrices. 

The infinite product \eqref{Rprod1} is locally finite, that is, 
all but finitely many factors act trivially on any given cohomology 
group.

\subsection{}

The action of $\Gamma$ on $\widetilde{Q}$ extends to its
action on the corresponding Yangian $\widetilde{\bY}$, which 
will be defined and discussed in Chapter \ref{s_Yang}. 
It satisfies 
$$
\gamma(x) \big|_{H(\bw)} = x \big|_{H(\gamma^{-1} \bw)}\,, 
\quad x \in \widetilde{\bY}\,,
$$
where the action on framing vectors is by 
$$
\gamma \delta_i = \delta_{\gamma i} \,. 
$$
Note that varieties $\widetilde{\cM}(\bw)$ and 
$\widetilde{\cM}(\gamma^{-1}\bw)$ are naturally 
isomorphic and the matrix $\widetilde{R}$ is 
invariant under $\gamma\otimes\gamma$. 

Rewriting \eqref{Rprod1} in terms of this action, 
we obtain the following 

\begin{Theorem}\label{t_Rprod}
We have 
\begin{equation}
R_{H_i,H_j}(a) = \overleftarrow{\prod_{\gamma\in \Gamma}} 
\, (\gamma^{-1} \otimes 1) \cdot \widetilde{R}_{H_i,H_j}(a-\gamma)
\label{Rprod2}
\end{equation}
in the stable basis for the maximal symplectic torus in 
$G_\edge$, where the ordering of the factors 
corresponds to choice of a chamber as in Section 
\ref{s_order_chamber}. 
\end{Theorem}

\noindent

Factorization of this kind play an important role in 
the theory of quantum groups, see \cite{ESS}.

\subsection{Example}\label{s_coverAinf} 

Let $Q$ be the quiver with one vertex and one loop. Then 
$$
\widetilde{Q} = A_\infty\,, 
$$
on which the group $\Gamma \cong \Z$ acts by shifts. 
This action naturally extends to an action on 
$$
\quad \widetilde{\bY} = 
\bY(\mathfrak{gl}_\infty)\,.
$$
The $R$-matrix in basic representation of 
$\bY(\mathfrak{gl}_\infty)$ may be found, for example, 
by fusion of $R$-matrices for fundamental representations. 
This gives a certain infinite product formula for 
the $R$ matrix for $Q$, which is an object of 
significant interest.

\section{Adjoint operators}\label{s_R_adj} 

In this section, we assume $X^g$ is proper for some 
$g\in \ga$. As in Section \ref{s_signs_adj},
this defines the Poincar\'e pairing 
$$
(\gamma_1,\gamma_2)_X = \int_{X} \gamma_1 \cup \gamma_2  \in \Q(\gal) 
$$
on both $X$ and $X^\bA$, the sign-twisted trace map $\tau$, and 
the corresponding adjoints. 

In particular, the adjoint $\Stab_{\fC}^\tau$ of the map $\Stab_{\fC}$
is given by the correspondence 
$$
\cL_\fC^\tau = (-1)^{\frac12 \codim X^\bA} 
\left(\cL_{\fC}\right)_{21} \subset X^\bA  \times X  \,.
$$
Here $\codim: \Fix \to \Z$ denotes
the codimension of a component of $X^\bA$ and the subscript $21$
refers to a permutation of factors. 

Note that since $\cL_{\fC}$ is not proper over $X^\bA$, 
equivariant localization is required to define the adjoint 
 as an operator.

\begin{Theorem}\label{adjoint} For any polarization $\bsi$ and any 
chamber $\fC$, we have 
$$
\Stab_{-\fC}^\tau \circ\Stab_{\fC} = 1 \,. 
$$
\end{Theorem}

\begin{proof}

Let $\Delta: X \rightarrow X \times X$ be the diagonal map and
consider the cycle class 
$$
C=\Delta^{*}(\cL^\tau_{-\fC} \times \cL_\fC)
$$
on $X^{\bA} \times X \times X^{\bA},$
where we have pulled back along the internal $X \times X$ factor.  
By construction, 
\begin{equation}
 \Stab_{-\fC}^\tau \circ \Stab_{-\fC}  = (p_{13})_*(C)  \label{p13C}
\end{equation}
where $p_{13}$ is the projection 
along the middle factor. 

We claim $C$ is proper over $X^{\bA} \times X^{\bA}$.  
Indeed, as in the proof of Proposition \ref{universalstableleaf}, 
we have 
$$
\cL_\fC \subset X^\bA \times_{V_0} \pi^{-1}(V_{\ge 0}) \,.
$$
Since $V_{\ge 0} \cap V_{\le 0} = V_0$, we conclude 
$$
C \subset X^\bA \times_{V_0}  \pi^{-1}(V_0) \times_{V_0} X^\bA \,,
$$
whence the claim. Therefore, the composition \eqref{p13C} 
is defined in nonlocalized equivariant cohomology and,
 in particular, has no terms of 
negative degree in equivariant parameters. 

On the other hand, we may compute \eqref{p13C} by localization, that is, as a
sum of equivariant residues for all triples $(Z_1,Z_2,Z_3)\in \Fix^{\times 3}$. 
When 
$$
Z_1 = Z_2 = Z_3\,, 
$$
the stable and unstable Euler classes precisely 
compensate the denominator in 
the localization formula, giving the diagonal 
as a result. All other 
residues have negative $\bA$-degree and hence cancel out. 
\end{proof}

\begin{Corollary}\label{c_R_tau} 
We have 
$$
R_\alpha^\tau = \, R_\alpha 
$$
for any root $R$-matrix $R_\alpha$. 
\end{Corollary}

Note $R_\alpha$ is an operator from $H_\ga^\hd(X^\bA)$ to itself, 
so $R_\alpha^\tau$ coincides with the adjoint with respect to
the Poincar\'e pairing. 

\section{Unitarity}\label{s_unitarity}

\subsection{}

In the theory of quantum groups, an $R$-matrix 
$$
R(u) \in \End(V \otimes V) \otimes \Q(u) 
$$
is called unitary if it satisfies 
\begin{equation}
R_{21}(u) = R(-u)^{-1} \,,
\label{R12}
\end{equation}
where the subscript in $R_{21}(u)$ 
means that we permute the tensor factors. We will show 
that $R$-matrices for Nakajima varieties are unitary. 

\subsection{} 

Consider the following general setup. Let a group 
of the form  
$$
\ga = \bA \times \bG' 
$$
act on $X$, where $\bA$ is a torus preserving the 
symplectic form $\omega$. Define $\phi \in \Aut \ga$ by
$$
\phi \cdot (a,g') = (a^{-1},g') \,. 
$$
It gives a pull-back map 
$
\phi^* \in \End H^\hd_\ga(X) 
$
which is a homomorphism of algebras. In particular, $\phi^*$ is 
anti-linear over the base ring 
$$
\phi^*(a \gamma) = -a   \, \phi^*(\gamma) \,, \quad a\in \fa \,. 
$$
In the cohomology of the fixed locus
$$
H^\hd_\ga(X^\bA) = H^\hd_{G'}(X^\bA) \otimes \Q[\fa]
$$
the action of $\phi^*$ amounts to $a\mapsto -a$, $a\in \fa$. 

\subsection{}

Since weights positive on $\fC$ are precisely the weights negative
on $-\fC$, the following diagram commutes 
\begin{equation}\label{a-a}
\xymatrix{
H^\hd_\ga(X^\bA) \ar@{->}[rr]^{\Stab_\fC\,\,}\ar[d]_{a\mapsto -a}&& 
H^\hd_\ga(X)  \ar[d]^{\phi^*} \\
H^\hd_\ga(X^\bA)  \ar@{->}[rr]^{\Stab_{-\fC}} &&  H^\hd_\ga(X) \,.\\
}
\end{equation}
Note that $\Stab_\fC$ is literally the same correspondence 
as $\Stab_{-\fC}$ for the opposite action. 

In particular, for $\bA=\C^\times$ we conclude
\begin{equation}
R(-a) = R(a)^{-1} \,. 
\label{R-a}
\end{equation}

\subsection{}

For tensor products of Nakajima varieties, we have 
$$
\cM(\bw+\bw')^{\bA} = \cM(\bw) \times \cM(\bw') \,, 
\quad \bA = \C^\times\,, 
$$
Note, however, from Section \ref{s_bw=}
that the ordering of factors in the product above \emph{depends} 
on a lift
$$
\bA \to G_\bw
$$
and not just on the image of $\bA$ in $G_\bw$ modulo 
the kernel of the action. The two lifts
$$
z \bw + \bw' \quad  \textup{vs.}\quad   \bw + z^{-1} \bw' 
$$
where $z\in\C^\times$ 
give the same action, but different identification of the
fixed locus with the product.  
{}From \eqref{R-a}, the corresponding $R$-matrices are 
$$
R(u) = R(-u)_{21}^{-1} \,,
$$
where $u\in \Lie \C^\times$. We thus obtain the following 

\begin{Proposition}
The $R$-matrices for Nakajima varieties are unitary. 
\end{Proposition}

\section{Action of Steinberg correspondences}\label{s_tStein}

We consider the setup of Section \ref{s_Stein_corr}. 
The union of walls for $X$ and $Y$ defines a partition 
of $\fa$ into chambers and we let $\fC$ be one of those. Let 
$$
L \subset X \times_V Y 
$$ 
be a $\ga$-invariant Steinberg correspondence.

For any polarization of $\bA$-fixed loci, we denote by
$$
\bar\bsi = (-1)^{\codim/2} \, \bsi
$$
the opposite polarization. Assuming 
polarizations $\bsi_X,\bsi_Y$ of $X^\bA,Y^\bA$ have been 
fixed, we take 
$$
\bsi  = \bsi_X \, \bar{\bsi}_Y
$$
as a polarization of $X^\bA \times Y^\bA \subset X \times Y$. Using it,  
we define the residue
$$
L_\bA = \Res_{X^\bA \times Y^\bA} \, L \subset X^\bA \times Y^\bA
$$
as a Lagrangian cycle class supported on $L^\bA$, see 
Section \ref{s_Res}. As a fixed-point set of a Steinberg 
correspondence, $L^\bA$ is Steinberg and hence so is $L_\bA$.

\begin{Theorem}\label{t_Stein}
The diagram
\begin{equation}\label{LbA}
\xymatrix{
H^\hd_\ga(Y^\bA) \ar@{->}[rr]^{\Stab_\fC\,\,}\ar[d]_{\Theta_{L_\bA}}&& 
H^\hd_\ga(Y) \ar[d]_{\Theta_{L}} \\
H^\hd_\ga(X^\bA)  \ar@{->}[rr]^{\Stab_\fC} &&  H^\hd_\ga(X)\\
}
\end{equation}
is commutative for every $\fC$. In particular, the Steinberg 
correspondence 
$\Theta_{L_\bA}$ intertwines the $R$-matrices of $X$ and $Y$. 
\end{Theorem}

For solutions of the Yang-Baxter equation, an important 
invariant is their algebra of symmetries, that is, the 
commutant of $R(u)$ for all $u$. Theorem shows it contains
the Steinberg algebra of $X$ for our geometrically
constructed $R$-matrices.

\begin{proof}
We fix one chamber $\fC$ and define 
\begin{equation}
L' = \Stab_{-\fC,\bsi_X}^\tau \circ \, \Theta_L \circ
 \Stab_{\fC,\bsi_Y} \subset X^\bA \times Y^\bA\,. 
\label{e_LbA}
\end{equation}
By Theorem \ref{adjoint}, this makes the diagram \eqref{LbA} 
commute for one particular chamber $\fC$, after tensoring 
with $\Q(\gal)$. 

We claim the pushforward
along $X\times Y$ used in the definition of $L'$ is proper. 
This is shown as in the proof of Theorem \ref{adjoint}. Namely, 
we may assume $V$ is a linear representation of $\bA$. Let 
$$
(x_0,x,y,y_0) \in X^\bA \times X \times Y \times Y^\bA
$$
be such that 
$$
(x,x_0) \in \cL^X_{-\fC}, \quad (x,y)\in L\,, \quad (y,y_0)\in \cL^Y_\fC \,.
$$
It then follows that $x_0,x,y,y_0$ map to the same point 
of $V_0 = V^\bA$, implying the properness. 

Hence $L'$ is 
well-defined as a nonlocalized cycle class. It is 
$\ga$-invariant and Lagrangian, being a composition 
of such classes. It may be computed by equivariant localization with 
an arbitrary choice of equivariant parameters. 

In particular, we may chose the equivariant parameters 
to be at infinity of $\fa$. Taking into account the signs 
in adjoints, we have 
$$
[\cL^Y] = \bsi_Y \, [\Delta_{Y^\bA}] + \dots\,, \quad 
[\cL^X] = \bar{\bsi}_X \, [\Delta_{X^\bA}] + \dots\,, \quad 
$$
where dots stand for terms of smaller $\bA$-degree. Therefore, 
at infinity of $\fa$, only these diagonal terms contribute
and thus $L'$ is supported on $L^\bA$. By our construction,
$$
[L]\big|_{[L^\bA]} = \bsi_X \, \bar{\bsi}_Y \, \Res_{L^\bA} L + \dots \,.
$$
We see that polarizations exactly cancel the denominators in 
localization formula, thus  
$$
L' = L_\bA \,. 
$$ 
Since the original choice of $\fC$ was arbitrary, the theorem 
follows.
\end{proof}

\section{Vacuum matrix elements}
\label{s_vac}

\subsection{}

Let $Z\in\Fix$ be minimal with respect to the 
partial order defined by a chamber $\fC$.

\begin{Theorem}\label{t_vacuum} 
If $Z\in\Fix$ is minimal as above then 
$$
\left(R_{-\fC,\fC} \cdot \gamma_1,\gamma_2\right) = \int_{Z} \gamma_1 
\cup \gamma_2 \cup \frac{e(N_+ \otimes \hbar)}{e(N_+)} \,,
$$
where $N_\pm$ are the stable/unstable subbundles of the normal bundle
to $Z$ and $\gamma_i\in H^\hd_\ga(Z)$. 
\end{Theorem}

\noindent
In other words, the corresponding matrix elements of $R_{-\fC,\fC}$ equal 
the operator of classical multiplication by the class 
$$
\frac{e(N_+ \otimes \hbar)}{e(N_+)} = 
\frac{e(N_-)}{e(N_- \otimes \hbar)} \in H^\hd_\ga(Z)_{\textup{localized}}
$$

\begin{proof}
We use Theorem \ref{adjoint} and equivariant localization. By 
minimality of $Z$, the attracting set 
$$
\LAttr_\fC\left(\Delta_Z\right) \subset X \times Z
$$
is closed and hence is 
the relevant component of $\cL_\fC$. Further, $Z\times Z$ is 
the only component of $X^\bA \times Z$ that this attracting 
set intersects. The 
localization contributions give 
$$
(-1)^{\codim(Z)/2} \,\, \frac{e(N_-)^2}{e(N_Z)}  = 
\frac{e(N_+ \otimes \hbar)}{e(N_+)} \,. 
$$
\end{proof}

\subsection{}

Here $e(N_\pm)$ are equivariant Euler classes, in the sense that 
they account for the nontrivial action of $\bA$ on $e(N_\pm)$. 
Since $\bA$ acts trivially on the base $Z$, we may expand 
$e(N_\pm)$ in the characteristic classes of the same bundles
with trivial $\bA$-linearization. 

For example, if $\bA = \C^\times$ and it acts on 
$N_+$ by its defining representation then 
\begin{multline}
  \frac{e(N_-)}{e(N_- \otimes \hbar)} =
1 + \frac{\hbar}{u} \, \rk N_- + \\
+ \frac{\hbar}{u^2} 
\left(c_1(N_-) + \frac{\hbar}{2} \rk N_- (\rk N_-+1)\right)
+ O\left(\frac1{u^3}\right)\label{eN} \,, 
\end{multline}
where $u\in\fa^*$ is the weight of the defining representation.

\subsection{}\label{s_Nak_vac}

For example, consider the tensor product of Nakajima varieties
as in Example \ref{ex3} in Section \ref{s_ample_ord}. 
If $\theta>0$ then the minimal component in \eqref{Mtens2} is 
\begin{equation}
Z_\emptyset=\cM_{\theta,\zeta}(\bv,\bw) \hookrightarrow  
\cM_{\theta,\zeta}(\bv,\bw+\bw') \,,
\label{Zvac}
\end{equation}
which corresponds to 
$$
\eta= 0 
$$
in \eqref{Mtens2}. By formula \eqref{charN_}, we have
\begin{equation}
  N_-\Big|_{Z_\emptyset} = \bigoplus \cV_i^{\oplus \bw_i} \,. 
\label{N_tens}
\end{equation}
Recall that $\cM_{\theta,\zeta}(0,\bw)$ is a point.

\subsection{}

In particular, for moduli spaces of framed sheaves, this embedding 
takes the form 
$$
\cM(r'') \owns \cF \mapsto \cO^{r} \oplus \cF \in \cM(r+r'') \,. 
$$
Its normal bundle is 
$$
N_-=\Ext^1_{\Pp^2} (\cO^{r},\cF(-1)) = H^1_{\Pp^2} (\cF(-1))^{\oplus r} \,. 
$$
The bundle 
$$
\Taut = \cV_1 = H^1_{\Pp^2} (\cF(-1))
$$
is the tautological bundle on the moduli spaces of framed sheaves.

Theorem \ref{t_vacuum}, combined with 
\eqref{RRR}, gives an $R$-matrix formula for the operators 
of classical multiplication by characteristic classes of $N_-$. 
We will revisit this point below. 

\subsection{}\label{s_true_vac} 

For general $\theta$, the component \eqref{Zvac} is not 
minimal. We therefore adopt the following terminology. 

For all $\theta$, we will call $Z_\emptyset$ the 
\emph{vacuum} or the lowest weight component. We will 
call the minimal component the \emph{true vacuum} 
component. For Nakajima varieties it coincides with $Z_\emptyset$
if $\theta>0$. 

When the vacuum $Z_\emptyset$ is not the true vacuum, 
the relation between the vacuum matrix elements of 
the $R$-matrix and the operators of classical multiplication 
becomes more complicated. It will be explored in Section 
\ref{diag_R}.

\section{Classical $R$-matrices}

\subsection{}

In this section, we assume that $X$ is a symplectic resolution. 
Recall the root $R$-matrices and the subtori $\bA_\alpha$ 
introduced in Section \ref{s_root_R}. 
{}From Propositions \ref{p_1u} and \ref{p_1h}, it follows 
that 
\begin{equation}
R_\alpha = 1  + \frac{\hbar}{\alpha} \, r_\alpha + O(\alpha^{-2}) \,, 
\label{def_br}
\end{equation}
for a certain operator 
$$
r_\alpha \in \End(H^\hd_\ga(X^\bA)) \,. 
$$ 

\begin{Definition}
The operator $r_\alpha$ is called the \emph{classical} $R$-matrix.
\end{Definition}

\noindent 
 Note that $r_\alpha$  
does not depend on a choice of a splitting \eqref{split}.

\begin{Proposition}\label{class_Stein}
There is a Steinberg correspondence
$\br_\alpha\subset X^\bA \times X^\bA$ that defines
the operator $r_\alpha$. 
\end{Proposition}

\begin{proof} 
Let 
$$
\Stab_{>0}: H^\hd_\ga(X^\bA) \to H^\hd_\ga(X^\alpha), 
$$
the map corresponding to the chamber $\alpha > 0$. 
By Theorem \ref{adjoint}, 
$$
R_\alpha = \Stab^\tau_{>0}\circ \Stab_{>0} \,. 
$$
We compute this push-forward by 
$(\bA/\bA_\alpha)$-equivariant localization.  {}From Theorem 
\ref{t_strong_vanish}, we can write 
$$
\left[\Stab_{>0}\right]|_{X^\bA \times X^\bA} = 
\gamma_\textup{diag} + \hbar \, \gamma_\textup{off-diag}\,. 
$$
Further, by Proposition \ref{Stein_fiber},
$$
\gamma_\textup{off-diag}\Big|_{Z \times Z'} =
\alpha^{\frac12 \codim Z-1} \, [C_{Z,Z'}] + \dots
$$
for a certain Steinberg cycle $C_{Z,Z'} \subset Z \times Z'$. 
Here codimension is computed in $X^\alpha$ and dots stand 
for terms of smaller degree in $\alpha$. 

It follows that the quadratic in 
$\gamma_\textup{off-diag}$ term doesn't contribute
to $r_\alpha$, while 
terms linear in $\gamma_\textup{off-diag}$ contribute a 
Steinberg correspondence.  Same is obviously true for the 
diagonal term. 
\end{proof}

\subsection{}

Note from the proof of Proposition \ref{class_Stein} 
\begin{equation}
\br_\alpha = 
\Big(\sum_{k \in \Q_{>0}} \frac{\rk N^{[k \alpha]}}{k}\Big) \, \Delta 
+ \textup{off-diagonal} \,, 
\label{diag_br}
\end{equation}
where $N^{[k\alpha]}$ is the $\bA$-weight space of the normal 
bundle to $X^\bA$ with weight $k\alpha$. This is because 
the diagonal terms only occurs from the diagonal terms in 
the localization formula, that is, from the expansion of
$$
(-1)^{\frac12 \codim} \frac{e(N_-^\alpha)^2}
{e(N_-^\alpha) \, e(N_+^\alpha)}
= \frac{e(N_+^\alpha \otimes \hbar)}{e(N_+^\alpha)}\,,
$$
as in the proof of Theorem \ref{t_vacuum}. Here 
the codimension and the normal 
bundles are taken in $X^\alpha$.

\subsection{}\label{s_diag_r_Nak} 

In particular, for tensor product of Nakajima varieties 
the normal bundle to the fixed locus is identified in \eqref{charN_}. 
{}From \eqref{diag_br}, we can then identify the diagonal part 
of the classical $R$-matrix 
\begin{equation}
  \label{br_diag_Nak}
  \br_\textup{diag} = \sum \bw_i \otimes \bv_i + \sum \bv_i \otimes \bw_i - 
\sum \bC_{ij} \, \bv_i \otimes \bv_j \,. 
\end{equation}
Here $\bv_i$ denotes the operator of multiplication by $\bv_i\in \N$ and 
so on. 

\subsection{}

The classical $R$-matrices satisfy 
a classical version of the braid relation. Concretely, 
the terms of degree $-2$ in $a_1,a_2,a_3$ in the 
expansion of \eqref{YB} as $a_i - a_j \to \infty$ give
\begin{align}
  \left[\br_{12},\br_{13}+\br_{23}\right] & = 0  \notag \\
  \left[\br_{23},\br_{12}+\br_{13}\right] & = 0 \label{cYB1} \,,
\end{align}
which is equivalent to the equation 
\begin{equation}
  \left[\bbr_{12},\bbr_{13}\right]+
  \left[\bbr_{12},\bbr_{23}\right]+
  \left[\bbr_{13},\bbr_{23}\right] = 0\,, \quad 
\bbr_{ij} = \frac{\br_{ij}}{a_i-a_j} \,. 
\label{cYB} 
\end{equation}
This is know as  the \emph{classical} 
Yang-Baxter equation with spectral parameter, see 
e.g.\ Section 6.3 in \cite{ES}. 

For brevity, we call $\br$ and not $\bbr$, which contains
the exact same information, the classical $R$-matrix. In 
the conventional terminology \cite{ES}, $\bbr$ is known 
as the classical $R$-matrix for the Yangian.

\subsection{}\label{s_superpolytope} 

Our next goal is to show that the off-diagonal terms in 
\eqref{diag_br} are additive over the coroot hyperplanes of 
the symplectic resolution $X$. This additivity is best 
stated in the following language. 

Define a map 
$$
\bmu: \Fix \to \Pic(X)^* \otimes \fa^*
$$
as follows. Fix an $\bA$-linearization for a basis 
$D_1,D_2,\dots$ of $\Pic(X)$ modulo torsion and let
$$
\bmu(Z)(D) \in \fa^*
$$
be the character of $\bA$-action in $D\big|_Z$. 
If $D$ is ample, this is the moment map for the 
corresponding Fubini-Study symplectic $(1,1)$-form. 

A different choice of the linearization changes 
$\bmu$ by a translation. In particular the difference
$$
\bmu(Z) - \bmu(Z') \in \Pic(X)^* \otimes \fa^*
$$
is defined uniquely. If $C$ is an irreducible $\bA$-invariant 
curve joining $Z$ and $Z'$ then by localization 
\begin{equation}
\bmu(Z) - \bmu(Z') = [C] \otimes \textup{weight}(T_p C)\,,\quad p=C \cap Z \,.
\label{diff_bmu}
\end{equation}
Here $[C] \in H_2(X,\Z)$ defines an element of $\Pic(X)^*$ via
the natural pairing
$$
(C,D) = \deg D|_C \,.
$$

\subsection{}

Let $\beta \in H_2(X,\Z)$ be an effective class such that $\beta^\perp$ is a coroot hyperplane of $X$ and let 
$X_\beta$ be the general fiber over 
the coroot hyperplane $\beta^\perp$ in \eqref{deform}. 

For any root $\alpha$, $X_\beta$ has its own classical $R$-matrix
$\br_\alpha(X_\beta)$. The closure of $\br_\alpha(X_\beta)$ defines 
a Steinberg correspondence $\br_{\alpha,\beta}$ in 
the fibers of 
\begin{equation}
\label{deform_beta}
\xymatrix{
X^\bA \times X^\bA \ar@{^{(}->}[rr]\ar[d]&&
(\tX_\beta)^\bA \times (\tX_\beta)^\bA \ar[d] \\
0 \ar@{^{(}->}[rr] &&  \beta^\perp\,.  \\
}
\end{equation}
Here $\tX_\beta$ is the restriction of the universal 
deformation $\tX$ to the hyperplane $\beta^\perp$. We 
have the following 

\begin{Theorem}
Let $Z,Z'$ be two different components of $X^\bA$. If
$$
\bmu(Z)-\bmu(Z') \in \Q \, \beta \otimes \alpha 
$$
for some $\beta$ such that $\beta^\perp$ is a coroot hyperplane 
then 
$$
\br_\alpha\big|_{Z'\times Z} = \br_{\alpha,\beta}\big|_{Z'\times Z} \,. 
$$
Otherwise, $\br_\alpha\big|_{Z'\times Z}$ is empty. 
\end{Theorem}

\begin{proof}
We first note that for $\br_\alpha$ to be nonempty, $Z$ and $Z'$ 
must lie in the same component of $X^\alpha$. Therefore, there
must exist a chain of $\bA$-invariant rational curves 
with tangent weights proportional to $\alpha$ that joins $Z$ and $Z'$.
{}From \eqref{diff_bmu}, we conclude 
$$
\bmu(Z)-\bmu(Z') = \gamma \otimes \alpha 
$$
for some $\gamma\in H_2(X,\Q)$. 

To simplify the 
notation, we will assume that
$$
\dim B = 2 \,. 
$$
If $\dim B > 2$, we can 
pick a general 2-plane in the base $B$ of the universal 
deformation and restrict $\tX$ to it.

We denote by 
$$
\tX \times (\tX)^\bA \xleftarrow{\,\,\iota_3}
\widetilde{Z}'\times \widetilde{Z}  \xleftarrow{\,\,\iota_{2}} 
Z'\times Z
$$
the inclusion of an $\bA$-fixed component and the fiber 
over the origin $0\in B$, respectively. Recall that 
$\phi$ denotes the projection to $B$.  We claim 
\begin{equation}
\supp \, \iota_3^* \, \widetilde{\cL}
\subset 
\begin{cases}
\phi^{-1}\left(\gamma^\perp\right)\,, 
& \textup{$\gamma^\perp$ is a coroot hyperplane}\,,\\
\phi^{-1}(0)\,, & \textup{otherwise} \,, 
\end{cases}
\label{supp_iota3}
\end{equation}
where $\widetilde{\cL}$ is an in Section \ref{s_sympl_res}. 

Indeed over a general point of a divisor $\beta^\perp \subset B$, 
$\beta$ is the only effective cycle in $H_2(X)$. For the 
support to be nonempty, there must be a chain of curves of class 
$\beta$ joining $Z$ and $Z'$, whence 
$$
\bmu(Z)-\bmu(Z') = \beta \otimes \delta
$$
for some $\delta\in \fa^*$. This implies $\gamma\in \Q \beta$ and 
and $\delta\in\Q\alpha$. 

We can factor the inclusion $\iota_2$ as follows 
$$
\widetilde{Z}'\times \widetilde{Z}  \xleftarrow{\,\,\iota_{1}} 
\widetilde{Z}'_\beta\times \widetilde{Z}_\beta  \xleftarrow{\,\,\iota_{0}} 
Z'\times Z\,, 
$$
where $\widetilde{Z}_\beta$ denotes the restriction of 
$\widetilde{Z}$ to the divisor 
$\beta^\perp\subset B$. {}From \eqref{supp_iota3}, we conclude
\begin{equation}
\iota_3^* \,\widetilde{\cL}  = \sum f_i(a) \, {\iota_1}_* [L_i] + \dots \,,
\quad \deg_\bA f_i(a) = \tfrac12 \codim Z' - 1 \,,
\label{iota3cL}
\end{equation}
where 
$$
L_i \subset \widetilde{Z}'_\beta\times \widetilde{Z}_\beta
$$
are certain Steinberg correspondences and 
dots stand for classes that are either of smaller 
$\bA$-degree or in the image of ${\iota_2}_*$. 
Note that 
$$
\iota_2^* \circ {\iota_2}* = \textup{multiplication by $\hbar^2$}, 
$$
and therefore the dots in \eqref{iota3cL} do not contribute to 
classical $R$-matrices. By contrast, the leading term in 
\eqref{iota3cL} is what goes into the correspondence $\br_{\alpha,\beta}$. 
This concludes the proof. 
\end{proof}

\section{Diagonal matrix elements of $R$-matrices}\label{diag_R} 

\subsection{}

To simplify notation, we assume that $\bA\cong \C^\times$ 
and that the cocharacter $\sigma\in\fC$ gives this isomorphism. 
Let $\lambda\in \Pic(X)$ be ample and we linearize it so
that its weight is trivial on the vacuum components $Z_\emptyset$. 
We label all other components $Z_k$ of $X^\bA$ by a nonnegative
integer $k$ --- the weight of $\lambda$. 

By construction, our $R$-matrix comes with a 
a block Gauss decomposition of the form 
\begin{equation}
R = 
\begin{pmatrix}
U_{00} \\
U_{10} & U_{11} \\
U_{20} & U_{21} & U_{22} \\ 
& \ddots & \ddots & \ddots 
\end{pmatrix}^{-1} \, 
\begin{pmatrix}
S_{00} & S_{01} & S_{02} \\
 & S_{11} & S_{12} & \\
 &  & S_{22} & \\ 
&  &  & \ddots 
\end{pmatrix}\,,
\label{RGauss}
\end{equation}
where the blocks are indexed as above and 
$$
S,U: H^\hd_\ga(X^\bA) \to H^\hd_\ga(X^\bA) 
$$
is given by 
$$
S,U = \pm u^{-\codim/2} \, \Res \circ \Stab_{\pm \fC}\,,
$$
according to polarization, where 
$$
\Res: H^\hd_\ga(X) \to H^\hd_\ga(X^\bA) 
$$ 
is the restriction map. With this normalization 
\begin{equation}
S_{ij} = \delta_{ij} + O(u^{-1}) \,, \quad u\to\infty \,,
\label{Snorm}
\end{equation}
and similarly for $U_{ij}$. 

\subsection{}

Note that \eqref{RGauss} implies 
\begin{equation}
R_{00} = U_{00}^{-1} S_{00} 
\label{R00}
\end{equation}
which is the content of Theorem \ref{t_vacuum}. The proof of 
Theorem \ref{t_vacuum} shows 
\begin{equation}
  \label{USdiag}
  \begin{pmatrix}
    U_{00}^{-1} S_{00} \\
    & U_{11}^{-1} S_{11} \\
    && U_{22}^{-1} S_{22} \\
    &&& \ddots 
  \end{pmatrix}
= \frac{e(N_-)}{e(N_- \otimes \hbar)} \, \cup 
\end{equation}
as operator on $H^\hd_\ga (X^\bA)$, where $N_-$ is the unstable 
part of the normal bundle. 

\subsection{}

Similarly to \eqref{R00}, one  
computes, for example 
$$
R_{11} = U_{11}^{-1} S_{11} + R_{10} \, S_{00}^{-1} \, U_{00} \, R_{01} \,. 
$$
In general, the diagonal matrix elements $R_{kk}$ may be computed 
as follows. Define 
$$
\widetilde{U}_{ij} = U_{ii}^{-1} \, U_{ij}  
$$ 
and equate the $(k,i)$ matrix elements in 
$$
U \, R = S \,. 
$$
For $i=0,\dots,k-1$, we get the following system of block 
matrix equations 
\begin{equation}\label{tURS} 
\begin{pmatrix}
\widetilde{U}_{k0} & \dots  & \widetilde{U}_{k,k-1} 
\end{pmatrix}
\bigboxvoid 
=
-\begin{pmatrix}
R_{k,0} & \dots & R_{k,k-1}
\end{pmatrix}
\,, 
\end{equation}
where 
\begin{equation}
  \label{boxmat}
  \bigboxvoid  = 
\begin{pmatrix}
R_{00} & \dots   & R_{0,k-1} \\
\vdots &  & \vdots \\
R_{k-1,0} & \dots & R_{k-1,k-1}  
\end{pmatrix} \,,
\end{equation} 
while for $i=k$, we obtain 
$$
R_{kk} = U_{kk}^{-1} S_{kk}  - 
\begin{pmatrix}
\widetilde{U}_{k0} & \dots  & \widetilde{U}_{k,k-1} 
\end{pmatrix} 
\begin{pmatrix}
R_{0,k} \\ \vdots \\ R_{k-1,k} 
\end{pmatrix} \,. 
$$
Since
$$
\bigboxvoid  = 1+O(u^{-1})\,,
$$
the square matrix \eqref{boxmat} is invertible as a series in
$u^{-1}$.  
This proves the
following

\begin{Theorem}\label{t_Gauss} We have 
$$
R_{kk} =  \left. 
\frac{e(N_-)}{e(N_- \otimes \hbar)}\right|_{Z_k} + \dots 
$$
where dots stand for a universal noncommutative 
expression in the coefficients of the $1/u$-expansion of $R_{ij}$, $R_{ji}$,
$i<j\le k$, and of the operators 
$$
\left. 
\left(\frac{e(N_-)}{e(N_- \otimes \hbar)}\right)^{\pm 1} 
\right|_{Z_i}\,, 
\quad i < k \,. 
$$
These corrections are found from 
\begin{equation}
  \label{RtUR}
 R_{kk} = U_{kk}^{-1} S_{kk}  + \begin{pmatrix}
R_{k,0} & \dots & R_{k,k-1} 
\end{pmatrix} 
\, 
{\bigboxvoid}^{-1}  
\,
\begin{pmatrix}
R_{0,k} \\ \vdots \\ R_{k-1,k} 
\end{pmatrix} \,. 
\end{equation}
\end{Theorem}

\subsection{}

In particular, Theorem \ref{t_Gauss} gives a way to 
relate operators of classical multiplication to vacuum 
matrix elements of $R$-matrices in the case then the 
vacuum is not the true vacuum in the sense of Section 
\ref{s_true_vac}. 

\subsection{}
The relationship in Theorem \ref{t_Gauss} simplifies for 
operators of small cohomological degree because they 
appear in small coefficients of the $1/u$-expansion.  For example, 
from 
$$
R_{ij} = O(u^{-1}) \,, \quad  i\ne j\,, 
$$
we conclude the following

\begin{Proposition}\label{Gauss_u2} 
$$
U_{kk}^{-1} S_{kk}  = R_{kk} - \sum_{i<k} R_{ki} \, R_{ik}  + 
O(u^{-3}) \,. 
$$
\end{Proposition}

\subsection{}

For Nakajima varieties Proposition \ref{Gauss_u2} 
means the following. Recall the Example \ref{ex3} in 
Section \ref{s_ample_ord} and suppose $\theta\not>0$.
Then 
$$
Z_\eta < Z_\emptyset\,, \quad \theta \cdot \eta < 0  \,. 
$$
Denote
$$
H(\bw)_\eta = H^\hd_\ga(\cM_{\theta,\zeta}(\eta,\bw)) \,. 
$$
Consider the matrix element $R_{\eta,0}$  
 of the 
$R$-matrix 
$$
R_{\eta,0}: H(\bw)_0 \, \otimes H(\bw')_\bv \longrightarrow 
H(\bw)_\eta \, \otimes H(\bw')_{\bv-\eta} 
$$
and the 
operator $R_{0,\eta}$ going in the opposite direction.  
Then Proposition \ref{Gauss_u2} implies 
\begin{equation}
\frac{e\left(N^\emptyset_-\right)}{e\left(N^\emptyset_- \otimes \hbar\right)} = 
R_{00} - \sum_{\theta\cdot\eta <0} R_{0,\eta} R_{\eta,0} + 
O(u^{-3}) 
\label{R_theta}
\end{equation}
where 
\begin{equation*}
  N_-^\emptyset = \bigoplus \cV_i^{\oplus \bw_i} \,. 
\end{equation*}
is the unstable normal bundle to $Z_\emptyset$, as in \eqref{N_tens}.
 
Observe that in \eqref{R_theta} the sum is effectively over $\eta\le \bv$
simply because $H(\bw')_{\bv-\eta}=0$ if $\eta\not\le \bv$. It is 
convenient that we don't have to restrict the range of summation 
explicitly. 

\section{Flops and stable envelopes}

\subsection{}

Let $X$ be a symplectic resolution and let 
\begin{equation}\label{deform2}
\xymatrix{
X \ar@{^{(}->}[r]\ar[d]& \tX \ar[d] \\
0 \ar@{^{(}->}[r] & B  \,,  \\
}
\end{equation}
be its deformation. For our present goals, it suffices
to take $B$ a generic line in the base of \eqref{deform} in 
Section \ref{s_sympl_res}. By definition, a flop of $X$ is 
another family over the same base $B$
$$
\xymatrix{
X_\fl \ar@{^{(}->}[r]\ar[d]& \tX_\fl \ar[d] \\
0 \ar@{^{(}->}[r] & B  \,,  \\
}
$$
together with an isomorphism 
$$
\xymatrix{
\tX \setminus X \ar[r]^{\widetilde{F}}\ar[d]& \tX_\fl \setminus X_\fl\ar[d] \\
B \setminus \{0\} \ar[r]^{\textup{id}} & B \setminus \{0\}   \\
}
$$
of families over the punctured base. We require 
$\widetilde{F}$ to:
\begin{enumerate}
\item[1)] be equivariant with respect to all group actions, 
\item[2)] preserve the symplectic form, 
\item[3)] induce identity on the affine quotients. 
\end{enumerate}
For symplectic resolutions, 3) implies 2) because it implies 
the graph of $\widetilde{F}$ is Lagrangian in the product of fibers.  

An example is provided by 
the natural isomorphism 
$$
\cM_{\theta,t\zeta}(\bv,\bw) \cong \cM_{\theta',t\zeta}(\bv,\bw)
$$
where $\theta,\theta'$ are arbitrary, $t\in B \setminus \{0\} = 
\C^\times$, and $\zeta$ is generic.

\subsection{}

The closure of the graph of 
$\widetilde{F}$ defines a 
cycle in $\tX \times_B \tX_\fl$, the restriction of 
which to the origin defines a $\bG$-invariant 
Steinberg correspondence 
$$
F \subset X_\fl \times X \,. 
$$
For brevity, we denote the induced map 
$$
F: H^\hd_\bG(X) \overset\sim\to H^\hd_\bG(X_\fl) 
$$
by the same letter. 
This is an isomorphism because both families are topologically 
trivial. 

\subsection{} 

For example, if $Q$ is the quiver with one vertex and 
no edges, $(\bv,\bw)=(1,n)$ then this is the classical Mukai flop
of 
$$
\cM_{\theta,0}(\bv,\bw) = 
\begin{cases}
T^*\Pp(W^\vee)\,, &\theta >0 \,,\\ 
T^*\Pp(W)\,, &\theta <0 \,, 
\end{cases}
$$
where $W \cong \C^n$ is the framing space and $\Pp(W)$ is the 
projective space of lines through the origin in $W$. 
In this case 
\begin{equation}
F = \Pp(W^\vee) \times \Pp(W)
+ T^\perp \textup{Universal hyperplane} \,,
\label{flopM}
\end{equation}
where $T^\perp$ denotes the conormal bundle and 
$\Pp(W) \subset T^*\Pp(W)$ is the zero section. 
Note this cycle is $GL(W) \times \C^\times$-invariant. 

\subsection{}
Let $\bA \subset\bG$ be a torus preserving the symplectic form. 
Any such torus acts trivially on the base $B$. 
Since a flop is an $\bA$-equivariant isomorphism over $B\setminus\{0\}$, 
we have a natural bijection
$$ 
f: \Fix(X) \overset\sim\to \Fix(X_\fl) 
$$
of components of $\bA$-fixed loci. By taking fixed points, $F$
induces a certain flop (potentially trivial) 
$$
F_i \subset Z_{\fl,f(i)} \times Z_i 
$$
of each component
of $Z_i \subset X^\bA$. 

\subsection{}

Since flop is a Steinberg correspondence, Theorem \ref{t_Stein}
implies the following square commutes for any chamber 
$\fC$ 
\begin{equation}
\xymatrix{
H^\hd_{\bG_\bA}(X^\bA) \ar[rr]^{\Stab_\fC} \ar[d]^{F_\bA}&& 
H^\hd_{\bG_\bA}(X)\ar[d]^F \\
H^\hd_{\bG_\bA}(X_\fl^\bA) \ar[rr]^{\Stab_{\fC,\fl}} && 
H^\hd_{\bG_\bA}(X_\fl) \,. 
}\label{flop_A} 
\end{equation}
Here the cycle $F_\bA$ is residue of $F$, it is a Steinberg
cycle supported on $F^\bA$ with signs determined by the 
polarizations of the fixed loci. 

\begin{Lemma}\label{l_flop} 
The correspondence $F_\bA$ is the flop of $X^\bA$, up to signs 
determined by polarization. 
\end{Lemma}

\begin{proof}
By construction
$$
F \Big|_{Z_{\fl,j} \times Z_i} = 0 \,, \quad j\ne f(i) \,, 
$$
in $\bA$-equivariant cohomology. Therefore $F_\bA$ vanishes outside
the graph of $f$. On the graph of $f$, the statement holds by 
definition. 
\end{proof}

\subsection{}

In the example of the Mukai flop, consider the Lagrangian 
subvarieties 
$$
\sigma_U = T^\perp \Pp(U) \subset T^*\Pp(W)
$$
corresponding to linear subspaces $U\subset W$. In 
particular, $\sigma_W$ is the zero section while 
$\sigma_0 = \emptyset$. 
{}From \eqref{flopM}, one computes 
\begin{equation}
  \label{flopM2}
 F(\sigma_U) = 
\sigma_{U^\perp} - (-1)^{\dim U} \sigma_{W^\vee} \,. 
\end{equation}
The coefficient of $\sigma_{W^\vee}$ is the sum of  
$$
\sigma_W \cdot \sigma_U = (-1)^{\dim \Pp(U)} \chi(\Pp(U)) = 
(-1)^{\dim U-1} \, \dim U 
$$
and the analogous number for a hyperplane section of $U$. 

Let $\bA\subset GL(W)$ be a maximal torus with eigenbasis 
$e_1,\dots,e_n \in W$ and the corresponding fixed points
$x_i = \Pp(\C e_i) \in \Pp(W)$. We have
$$
\Stab(x_i) = \sigma_{U_i} + \sigma_{U_{i+1}} \,, 
\quad U_i = \Span(e_i,\dots,e_n) 
$$
for some choice of chamber and polarization. We see that 
$$
F (\Stab(x_i)) = \sigma_{U_i^\perp} + \sigma_{U_{i+1}^\perp}\,,
$$
where $U_{n+1}^\perp = W^\vee$. This is the stable basis 
for $\bA$ action on $T^*\Pp(W^\vee)$ for the same chamber 
and suitable polarization. 

The induced bijection of 
fixed loci is 
$$
f(x_i) = \Pp(\C \xi_{n-i+1}) 
$$
where $\{\xi_1,\dots,x_n\}$ is the dual basis of $W^\vee$. 

\subsection{}
Different cones in the space of the stability condition $\theta$
give different flops of a given Nakajima variety. Among them 
is the maximal flop 
$$
F_\textup{max} \subset \cM_{-\theta,\zeta}(\bv,\bw) \times \cM_{\theta,\zeta}(\bv,\bw)
$$
that corresponds to the opposite cone of stability conditions. 
For an arbitrary symplectic resolution $X$, one similarly 
expects to have a flop $F_\textup{max}$ 
that takes the ample cone of $X$ to its 
opposite. 

We next observe that for any chamber $\fC$, the map 
$$
\Stab_\fC: H^\hd_{\bG_\bA}(X^\bA) \to H^\hd_{\bG_\bA}(X)
$$
is characterized by its behavior near the diagonal and the opposite 
triangularity of the supports of $\Stab_\fC$ and 
$F_\textup{max} \Stab_\fC$.

\begin{Theorem}
The map $\Stab_\fC$ is uniquely determined by the 
the conditions {\nui}, {\nuii} in Theorem \ref{stablebasis}
together with a symmetric condition for its maximal 
flop 
\begin{equation*}
  \label{flop_supp}
  \supp F_\textup{max}\circ \Stab_{\fC}(Z_i)  \subset \Branch(Z_{\fl,f(i)})
\end{equation*}
\end{Theorem}



\begin{proof}
The above support condition is satisfied by 
\eqref{flop_A} and Lemma \ref{l_flop}. Since 
a maximal flop takes an ample class to minus an ample 
class while preserving $\bA$-weights, we have 
$$
 i > j \quad \Leftrightarrow \quad f(i) < f(j)  
$$
in the ample partial ordering, for any $\fC$. Thus 
the supports of $\Stab_\fC$ and $F_\textup{max}\Stab_\fC$ 
are triangular the opposite way. Hence
\begin{equation}
F_\textup{max} \Big|_{X_\fl^\bA \times X^\bA}  
= \Stab_{\fC,\fl} \Big|_{X_\fl^\bA \times X_\fl^\bA} 
\circ F_{\textup{max},\bA} \circ 
\left(\Stab_{\fC} \Big|_{X^\bA \times X^\bA} \right)^{-1} 
\label{Gauss_flop}
\end{equation}
is a Gauss factorization, and therefore unique. 
\end{proof}

\subsection{}

We see from \eqref{Gauss_flop} that flops give a way to package 
the information about stable envelopes that is somewhat 
different from $R$-matrices. This packaging has 
several convenient features, among them:
\begin{itemize}
\item flops are given by Steinberg correspondences, a very 
economical and geometric data, 
\item the maximal flop $F_\textup{max}$ can be factored into 
a product of flops that cross a single wall in 
the space of $\theta$'s,
\item additional constraints on $F_\textup{max}$ may be deduced from a 
noncanonical isomorphism 
$$
\cM_{\theta,\zeta}(\bv,\bw) \to  \cM_{-\theta,\zeta}(\bv,\bw)
$$
that replaces all quiver data by transposed with respect to some 
chosen bilinear form. 
\end{itemize}

\chapter{Yangians} \label{s_Yang} 

\section{Tensor products}\label{s_Tens} 

\subsection{}

Let $X$ satisfy the hypotheses of Section \ref{s_Assump}. By definition, 
we say that $X$ is a tensor product and write 
$$
X = X_1 \otimes \cdots \otimes X_n
$$
if the maximal torus $\bA \subset PGL(n)$ acts on $X$ preserving 
the symplectic form so that 
\begin{itemize}
\item[(1)] $X^\bA=X_1 \times \dots \times X_n$, 
\item[(2)] the roots of $X$ are the roots $\alpha_{ij}$ of $PGL(n)$,
\item[(3)] the corresponding fixed loci are of the form 
$$
X^{\alpha_{ij}} = X_{ij} \times \prod_{k\ne i,j} X_k 
$$
\end{itemize}

\noindent 
We view this definition as provisional; perhaps a better set of 
axioms will emerge later. Note that neither existence or 
uniqueness of tensor products 
is claimed. 

If one requires $X$ to have a unique, up to multiple, holomorphic 
symplectic form, then this rules out trivial nonuniqueness 
of the form 
$$
X \mapsto X \times \textup{vector representation of $\bA$} \,. 
$$

\subsection{}\label{s_cM_tensor}

In the case of quiver varieties, recall
$\cM(\bw)$ from Section \ref{RforNakajima}.
For any decomposition 
$$
\bw = \sum_{i=1}^n \bw_i 
$$
into nonzero terms, we have 
$$
\cM(\bw) = \bigotimes \cM(\bw_i) \,, 
$$
corresponding to the decomposition 
$$
\bw = \sum z_i \, \bw_i 
$$
as in Section \ref{s_bw=}. Here 
$$
(z_1,\dots,z_n) \in (\C^\times)^n = \bA \,.
$$

\subsection{}

For $X= X_1 \otimes \cdots \otimes X_n$, the construction of 
Chapters \ref{s_envelopes} and \ref{s_properties} gives 
a set of $R$-matrices
$$
R_{ij}(a_i-a_j) \in \End(F_1\otimes\cdots\otimes F_n)
\otimes\Q(\ft)\,,
\quad F_i=H^\hd_\ga(X_i) 
$$
satisfying the Yang-Baxter equation \eqref{YB}, a familiar 
setup in quantum integrable systems. 

\subsection{} 

Given an operator
$$
R_{12}(a_1-a_2) \in \End(F_1 \otimes F_2) \,, 
$$
its matrix elements in $F_1$ are operators on $F_2$. 
Our main interest is the algebra of operators thus 
obtained for Nakajima varieties. 
This algebra is an example of a Yangian.

\section{Construction of Yangians}

\subsection{} Yangians are Hopf algebras associated to 
rational solutions of the Yang-Baxter equation. There 
are several ways to describe a Yangian. 
For us, it is the so-called RTT=TTR formalism of \cite{FRT} 
that arises naturally. We briefly recall the basics. 

For simplicity, we limit the use of the categorical 
language, even 
though many construction and properties
are best stated in the language of tensor 
categories, see for example \cite{Soi}. 

\subsection{} 

Let $\bk\supset \Q$ be a commutative ring without 
zerodivisors. We write 
$$
\otimes = \otimes_\bk\,, \quad \End = \End_\bk
$$
for brevity. Let 
$\{F_i\}$ be a collection of free $\bk$-modules and let 
$$
R_{F_i,F_j}(u) \in \End(F_i \otimes F_j)(u) 
$$
be collection of operator-valued
rational functions of $u$ satisfying 
the Yang-Baxter equations \eqref{YBe}. We assume 
the normalization 
$$
R(\infty) = 1\,. 
$$
We also fix $\hbar\in\bk$ that divides $R(u)-1$. In geometric 
applications, this will be the weight of the symplectic 
form. 

\subsection{}

To this data, one associates a Hopf algebra 
$\bY$ over $\bk$ that acts on 
\begin{equation}
F_i(u) \overset{\textup{\tiny def}}= F_i \otimes \bk[u] \,.
\label{F(u)}
\end{equation}
and more generally on 
\begin{equation}
F_{i_1}(u_1) \otimes \cdots \otimes F_{i_n}(u_n) = 
F_{i_1} \otimes \cdots \otimes F_{i_n}\otimes \bk[u_1,\dots,u_n]
\label{F(u)times}
\end{equation}
This action commutes with multiplication by the $u_i$'s, so may be 
viewed as a family of $\bY$-modules indexed by $\mathbb{A}^n_\bk$. 

\subsection{}  

While $F_i[u]$ is a more logical notation for \eqref{F(u)}, 
the use of parentheses is traditional. The variable $u$ in 
\eqref{F(u)} is called the evaluation parameter, in 
reference to the following. 

By one of their many
 definitions, Yangians are Hopf algebra deformations 
of $\cU(\fg[u])$, where $\fg$ is a Lie algebra over $\bk$ and 
$\fg[u]$ is the Lie algebra of $\fg$-valued polynomials 
in $u$. The identity map 
$$
\fg[u] \to \fg\otimes \bk[u]
$$
may be viewed as family of evaluation homomorphisms 
$\fg[u]\to \fg$ and any $\fg$-module $F$ can be made 
a $\fg[u]$-module $F(u)$ by pull-back. 

\subsection{}\label{s_compl}  
A certain care is required if $\rk F_i = \infty$ for some $F_i$.  
We will always assume a grading 
$$
F_i = \bigoplus_{\alpha\in \Z^n} \big(F_i\big)_\alpha 
$$
such that all graded pieces are $\bk$-modules of finite rank. We 
further require that $\big(F_i\big)_\alpha\ne 0$ only for $\alpha$ in 
a translate of a certain nontrivial cone, which we will assume 
to be $(\Z_{\ge 0})^n$ for simplicity. 

The $R$-matrices will 
always have grading $0$. This makes  $\bY$ a graded 
algebra and  $F_i(u)$, with the grading induced from $F_i$, 
a graded module. The coproduct
\begin{equation}
\Delta: \bY \to \bY \, \widehat\otimes \, \bY
\label{defDel}
\end{equation}
to be defined below, takes values in the following completed 
tensor product.  
By definition, 
$$
\bY \, \widehat\otimes \, \bY  = \bigoplus_\alpha 
\left(\bY \, \widehat\otimes \, \bY \right)_\alpha 
$$
while 
$$
\sum_\beta y_{\alpha-\beta} \otimes y_{\beta} \in \left(\bY \, \widehat\otimes \, \bY \right)_\alpha 
$$
if $\beta$ ranges in a translate of $(\Z_{\ge 0})^n$. Such infinite sums 
act naturally on any $F_i(u_1) \otimes F_j (u_2)$. The iterates of 
$\Delta$ make \eqref{F(u)times} tensor products of \eqref{F(u)} 
as $\bY$-modules.

\subsection{Definition} \label{s_def_Y}
We define $\bY$ as the subalgebra
\begin{equation}
\bY \subset \prod_{i_1,\dots,i_n}  \End_{\bk[u_1,\dots,u_n]}
\left(F_{i_1}(u_1) \otimes \dots \otimes F_{i_n}(u_n) \right)
\label{YinProd}
\end{equation}
generated by the following operators. Let 
\begin{equation}
W = F_{1}(u_1) \, \otimes \, \dots \otimes F_{n}(u_n) 
\label{exW}
\end{equation}
be one of the spaces in \eqref{YinProd} where, for brevity, we 
write $F_k$ in place of $F_{i_k}$ to denote some element 
of the set $\{F_i\}$. Choose an additional $F_0\in \{F_i\}$ 
called an \emph{auxiliary} space and define 
\begin{equation}
R_{F_{0}(u), W} 
= R_{F_{0},F_{n}}(u-u_n) \, \cdots \, 
R_{F_{0},F_{1}}(u-u_1) \,. 
\label{Rtrain}
\end{equation}
Let 
$$
m(u)\in F_0 \otimes F_0^\vee \otimes \bk[u]
$$
be a polynomial in $u$ with values in operators in $F_0$ of finite rank.
Here 
$$
F_0^\vee = \Hom_\bk(F_0,\bk) 
$$
is the graded dual module.

Because $m(u)$ has finite rank and $\hbar$ divides
$R-1$, 
the following operator 
\begin{equation}
\bE(m(u)) = 
- \frac{1}{\hbar}\Res_{u=\infty} \, \tr_{F_0} m(u) \, 
R_{F_0(u),W} \, \in \End(W) 
\label{def_bE}
\end{equation}
is well-defined for all $W$ in \eqref{exW}. Since it comes 
from an expansion of rational 
functions of $u-u_i$ as $u\to\infty$, it depends polynomially 
on $u_1,\dots,u_n$. Thus, it defines an element 
of the right-hand side in \eqref{YinProd}.

By definition, 
$\bY$ is the $\bk$-subalgebra in \eqref{YinProd}
 generated by $1$ and \eqref{def_bE} for 
all $F_0$ and all $m(u)$. 
In English, the Yangian $\bY$ is the algebra generated by 
\begin{align*}
  & \textup{all coefficients of the $u\to\infty$ expansion of}\\ 
  & \textup{all matrix coefficients of the operators \eqref{Rtrain} for} \\
  & \textup{all auxilliary spaces $F_0$.} 
\end{align*}
Additionally, since all nontrivial matrix elements are divisible by 
$\hbar$, we divide by $\hbar$ in \eqref{def_bE}. 

\subsection{}

The product in  \eqref{YinProd} includes
the the factor $W=\bk$ corresponding to 
$$
\{i_1,i_2,\dots,i_n\} = \emptyset \,. 
$$
This $1$-dimensional $\bY$-module is the counit of the Yangian.

\subsection{}
After inverting $\hbar$, 
\eqref{def_bE} makes sense for any rational function $m(u)$ of 
$u$, in particular, 
\begin{equation}
\bE\left(m\, u^{-k} \right) = 
\begin{cases}
\hbar^{-1} \, \tr m\,, & k=1\,, \\
0\,, & k>1 \,. 
\end{cases}
\label{Eu^{-1}}
\end{equation}
While such operators are not in $\bY$, they will play a role 
in computation of commutation relations \eqref{top_comm_E} below.

\subsection{RTT=TTR equation}
By construction, \eqref{def_bE} extends to a surjection 
\begin{equation}
\bE: \textup{Tensor algebra} \left(
\bigoplus F_i \otimes F_i^\vee \otimes \bk[u] \right) 
\twoheadrightarrow \bY\label{defE}
\end{equation}
The Yang-Baxter equation shows it factors through the 
quotient by 
\begin{multline}
 \left( m_1(u_1) \otimes m_2(u_2)\right) \cdot R_{F_1 F_2}(u_1-u_2) - \\
R_{F_1
    F_2}(u_1-u_2) \cdot (m_2(u_2) \otimes m_1(u_1)) \,, \quad m_i(u) \in
  F_i \otimes F_i^\vee \otimes \bk[u] \,.
\label{eRTT}
\end{multline}%
This is known as the RTT=TTR relation. The letter T being 
overused in this paper, we substitute it in this context 
by $\bE$. 

The quotient of the tensor algebra by \eqref{eRTT} is of 
the same size as the symmetric algebra. This is still very 
big and below we will discuss how to write further 
relations in Yangians. 

\subsection{Filtration in the Yangian}
The Yangian $\bY$ is filtered by degree in $u$, that is, by 
defining 
$$
\deg \bE(m(u)) = \deg_u m(u) 
$$
on the generators of the Yangian. We set $\deg_u 1 = 0$. 

Equation \eqref{Eu^{-1}} shows this filtration does not 
extend to the algebra generated by these more general operators. 
Therefore, one has to be careful in situations where they appear. 
 
Since scalars cancel out of the RTT=TTR equation, it 
takes the form 
\begin{equation}
\left[\bE(m(u)),\bE(m'(v))\right] = 
\hbar \, \bE\left(\left[\frac{\br_{VV}}{u-v},m(u)\otimes m'(v)
\right]\right)+\dots 
\label{REE}
\end{equation}
where $\br$ is the classical $R$-matrix 
$$
R(u) = 1 + \frac{\hbar}{u} \, \br + O(u^{-2}) 
$$
and dots in \eqref{REE} come from the $O(u^{-2})$ term above.

Note that in the right-hand side of \eqref{REE} there are terms of 
the same degree as in the left-hand side. They 
come from the expansion 
$$
\frac{1}{u-v} = \frac{1}{u} + \frac{v}{u^2} + \frac{v^2}{u^3} + \dots 
$$
and \eqref{Eu^{-1}}, giving the right-hand side of the 
following formula \eqref{top_comm_E}. 

\begin{Proposition}\label{p_commtop} We have 
  \begin{equation}
\left[\bE(m \, u^i),\bE(m' \, u^j )\right] = 
\bE\left( (\tr \otimes 1) \left[\br_{VV},m \otimes m'\right] u^{i+j}\right) +
\dots \label{top_comm_E} 
\end{equation}
where dots stand for terms of smaller degree in $u$. 
\end{Proposition}

\begin{proof}
Were it not for  \eqref{Eu^{-1}}, 
the right-hand side of \eqref{REE} would have smaller total 
degree in $u$ and $v$ than $\deg_u m(u) + \deg_v m'(v)$. 

Each application of \eqref{Eu^{-1}} brings the total degree 
up by $1$. Note, however, that it can be applied only once
and with respect to the variable $u$, because all terms in \eqref{REE}
have nonnegative degree in $v$. Therefore, the dots in \eqref{REE} have  
total degree at most $\deg_u m(u) + \deg_v m'(v)-1$ and 
can be neglected. 
\end{proof}

\subsection{}

Note the commutation relation \eqref{top_comm_E}
has the form 
$$
\left[a \, u^i, b \, u^j\right] = [a,b] \, u^{i+j} \,, \quad a,b\in \fg \,,
$$
of the commutation relations in the Lie algebra of polynomials
$\fg[u]$ with values in a Lie algebra $\fg$. 

In fact, one of our 
goals is to show that for the Yangian associated to a quiver $Q$
$$
\gr \bY \cong \cU(\fg_Q[u]) 
$$
for a certain Lie algebra $\fg_Q$. Here $\gr \bY$ denotes the 
associated graded of $\bY$  for the filtration by degree in $u$.

\subsection{Coproduct}

The set of $W$ of the form \eqref{exW} is closed with respect to tensor product.
There is a corresponding projection 
$$
\prod_W \End W \to \prod_{W, W'} \End \left(W \otimes W' \right)\,.
$$
By applying this projection to $\bE(m(u))$,
it is easy to see that
it sends $\bY$ to the image of the map
\begin{equation}\label{tensorembedding}
\bY \, \widehat\otimes \, \bY \rightarrow \prod_{W, W'} \End \left(W \otimes W' \right)\,.
\end{equation}
The completion is needed because 
matrix elements of $R_{F_0,F_1\otimes F_2}$ are infinite 
sums of products of matrix elements of $R_{F_0,F_i}$ when 
$\dim F_0 = \infty$.

This defines a natural coproduct \eqref{defDel} on $\bY$ up
to an ambiguity arising from the kernel of \eqref{tensorembedding}.  
We will prove at the end of this chapter that $\bY$ is flat over $\bk$ and that, as a corollary,
the map \eqref{tensorembedding} is injective so this ambiguity does not arise. 
In the meantime, we only discuss the coproduct as evaluated on pairs of representations.

The coproduct is not commutative and in general 
$$
F_1(u_1) \otimes F_2(u_2) \not \cong F_2(u_2) \otimes F_1(u_1)
$$
as $\bY$-modules. However, 
$$
F_1(u_1) \otimes F_2(u_2) \otimes_{\bk[u_1,u_2]} 
\bk(u_1,u_2)
\cong F_2(u_2) \otimes F_1(u_1) \otimes_{\bk[u_1,u_2]} 
\bk(u_1,u_2)
$$
with the explicit intertwiner 
$$
R^\vee = (12) \, R_{F_1,F_2}(u_1-u_2)\,.  
$$
This follows at once from the Yang-Baxter equation.

\subsection{Translation automorphism}

All spaces $W$ in \eqref{exW} have an automorphism $\ttau_c$ 
that acts by 
$$
\ttau_c(u_i) = u_i + c\,, \quad i=1,2,\dots\,, 
$$
on the variables $u_i$ and as identity on $F_i$'s. It 
preserves $\bY$  because it amounts
to a reexpansion of $R(u-c)$ in inverse powers of $u$. 
We denote the corresponding automorphism of 
the Yangian also by $\ttau_c$.


\subsection{}\label{s_Yang_Nak} 
In the rest of this chapter, we specialize to the case 
of Nakajima varieties, see Section \ref{s_cM_tensor}. 
We fix a quiver $Q$ and set 
\begin{align}
  \bk &= H_{G_\edge}^\hd(\pt,\Q) \,, \notag \\
  F_i &= H_{G_\edge}^\hd (\cM(\delta_i),\Q) \,. 
\label{set_F_i} 
\end{align}
Here $\bw=\delta_i$ is the delta-function at some $i\in\Ib$. 
Note that in this case $G'_\bw =1$. The tensor product
construction will identify 
$$
H_{G_\bA}^\hd (\cM(\bw)^\bA) = \bigotimes_{i\in I} F_i(u_{i1}) \otimes 
\cdots \otimes 
F_i(u_{i\bw_i}) 
$$
where $\bA \subset G_\bw$ is a maximal torus and 
$$
\begin{pmatrix}
  u_{i1} \\
 & u_{i2} \\
 && \ddots \\
 &&& u_{i\bw_i} 
\end{pmatrix}
\in \mathfrak{gl}(W_i) \,, \quad i\in I \,, 
$$
are the equivariant parameters for the group $G_\bw$. 

The collection \eqref{set_F_i} can be enlarged by allowing 
arbitrary dimension vectors $\bw$ in place of 
$\delta_i$. This does not change the Yangian $\bY$ 
because, 
as we will see, $\bY$ already injects into the endomorphisms
of tensor products of $F_i(u_{ik})$.

\section{The Lie algebra $\fg_Q$} \label{s_fgQ} 

\subsection{}

Let $\fg_Q\subset \bY$ be the span of operator $\bE(m_0)$, where 
$m_0$ is constant polynomial in $u$. In other words, 
$\fg_Q$ is spanned by the matrix elements of the 
classical $R$-matrix $\br$. Formula 
\eqref{top_comm_E} shows $\fg_Q$ is a Lie algebra.
The following is clear

\begin{Proposition}
  All elements of $\xi\in \fg_Q$ are primitive, that is, 
$$
\Delta \xi = \xi \otimes 1 + 1 \otimes \xi \,,
$$
when evaluated on pairs of representations. In particular,
\begin{alignat}{2}
 \big[&\Delta \xi,&R \big] &= 0\,, \notag \\ 
\big[&\Delta \xi,& \br\big] &= 0\,,  
 \label{inv_tens_r}
\end{alignat}
that is, $\fg_Q$ commutes with $R$-matrices. 
\end{Proposition}

\noindent 
We expect that $\fg_Q$ is the Lie algebra of primitive 
elements of $\bY$.  

\subsection{}

As defined, $\fg_Q$ is a Lie algebra over $\bk$. We expect 
a natural isomorphism 
$$
\fg_Q = \left(\fg_Q\right)_\Q \otimes_\Q \bk
$$
for a certain Lie algebra over $\Q$. We think the required
$\Q$-structure may be constructed using the Decomposition 
Theorem.

\subsection{}

The identity  
$$
R(u)^{-1} = R(-u)_{12}
$$
from Section \ref{s_unitarity} implies the symmetry of $\br$, that is, 
$$
\br_{W,W'} = \br_{W',W}\,. 
$$
after identifying $W \otimes W'$ and $W' \otimes W$.


\subsection{}

It follows from formula \eqref{br_diag_Nak} that 
$$
\fhb_Q \subset \fg_Q
$$
where $\fhb_Q$ acts by linear functions of $\bv$ and $\bw$. 
Linear functions can be taken with $\bk$-coefficients or 
$\Q$-coefficients, and this defines   $\fhb_Q$ as $\bk$-submodule with 
a canonical $\Q$-submodule. All structures in $\fhb_Q$ 
are defined over $\Q$.

Recall the quadratic forms \eqref{ringel} with 
values in $K_\bG(\pt)$. Here we evaluate them
at $1\in \bG$, in other words, we use the 
nonequivariant Cartan matrix. 
The inverse of the nondegenerate form 
$(\,\cdot\,,\,\cdot\,)_{\Qb}$ from \eqref{ringel} defines
a bilinear form $(\,\cdot\,,\,\cdot\,)_{\fhb}$ on $\fhb_Q$. 
{}From \eqref{br_diag_Nak} we conclude 
$$
\br = \sum_{i\in I \sqcup \Ib} h_i \otimes h^i + \dots
$$
where 
$$
(h_i,h^j)_{\fhb} = \delta_{ij}
$$
and dots stand for
off-diagonal elements. Note that, with our conventions, 
$$
\dim \fhb_Q = 2 |I | \,. 
$$
While this looks unusual from the perpective of 
finite-dimensional Lie
theory (in which Cartan matrices are nondegenerate), 
this is very convenient and has been used before e.g.\
in \cite{FeiZel}. 

By construction, off-diagonal elements 
have a nonzero commutator with $\fhb_Q$ acting in one of the 
tensor factors. We deduce the following

\begin{Proposition}\label{hmaxcomm}
$\fhb_Q$ is a maximal commutative subalgebra of $\fg_Q$. 
\end{Proposition}

\subsection{}
For brevity, we write $\fhb = \fhb_Q$, $\fg=\fg_Q$. 
By Proposition \ref{hmaxcomm}, we can write 
\begin{equation}
  \fg = \fhb \oplus \bigoplus_{\eta\ne 0} \fg_\eta
\label{roots_g}
\end{equation}
where $\eta\in \Z^I$ and 
$\fg_\eta$ is spanned by $\xi$ such that
$$
\xi: H^\hd_\bG(\cM(\bw,\bv)) \to H^\hd_\bG(\cM(\bw,\bv+\eta)) \,. 
$$
The vectors $\eta$ such that $\fg_\eta \ne 0$ are called 
the roots of $\fg$. Clearly
\begin{equation}
  \label{[geta]}
  \left[ \fg_\alpha, \fg_\beta \right]
\subset \fg_{\alpha+\beta} \,. 
\end{equation}
We call a root $\eta$ positive if $\eta \in \N^I$. 



\subsection{} \label{s_h_eta} 

The decomposition \eqref{roots_g} parallels the 
root decomposition for Kac-Moody Lie algebras. As for a Kac-Moody Lie algebra, we define the coroot 
$$
h_\eta = \Cb \, \eta \in \fhb
$$
for every root $\eta$. These satisfy
\begin{equation}
  (\alpha,\beta)_\Qb = \alpha(h_\beta) = (h_\alpha,h_\beta)_{\fhb} \,. 
\label{prop_h_beta}
\end{equation}
\begin{Proposition}\label{p_comm_pair}
Let $\eta$ be a root and consider the commutator map 
$$
\fg_{\eta} \otimes \fg_{-\eta} \to \fhb \,.
$$
Its image is $\bk \, h_\eta$ and this gives an embedding 
$$
\fg_\eta \hookrightarrow \fg_{-\eta}^\vee = \Hom(\fg_{-\eta},\bk) \,. 
$$
\end{Proposition}

\noindent 
Later we will see that, in fact, this gives
an isomorphism $\fg_\eta \cong \fg_{-\eta}^\vee$. 

\begin{proof} 
Take $\xi \in \fg_\eta$ and consider the $(\eta,0)$-weight 
space in \eqref{inv_tens_r}. One of the terms is
$$
\left[\xi\otimes 1,\sum h_i \otimes h^i\right] = 
- \xi \otimes \sum_i h_i(\eta) \, h^i  = - \xi \otimes h_\eta \,. 
$$
We conclude
\begin{equation}
\big[1 \otimes \xi, \br_{\eta,-\eta} \big] = \xi \otimes h_\eta \,,
\label{xicommr}
\end{equation}
where $\br_{\eta,-\eta}$ denotes the corresponding weight component. 
Both claims follows from this. 
\end{proof}

\subsection{}
By construction, $\fg$ comes with modules $F_\bw$ containing 
vectors $\vacv{\bw}$ of lowest weight, that is, 
\begin{equation}
\fg_\eta \vacv{\bw} = 0 \,, \quad \eta \not> 0 \,.
\label{ann_vac}
\end{equation}
Recall that $\eta>0$ means $\eta\in\N^I$. Also 
$$
h  \vacv{\bw}  =  \bw(h) \vacv{\bw} \,, \quad  h \in \fhb 
$$
and $\vacv{\bw}$ is the unique, up to multiple, vector of 
weight $\bw$. We denote by $F_\bw(\eta)\subset F_\bw$ the 
subspace of weight $\bw+\eta$. The $\fg$-action gives 
maps 
\begin{equation}
  \label{fga}
  \fg_\eta \to F_\bw(\eta)\,, \quad \fg_{-\eta} \to F_\bw(\eta)^\vee 
\end{equation}
that take $\xi\in \fg_\eta$ to $\xi \vacv{\bw}$ and dually for 
$\fg_{-\eta}$. 

\begin{Proposition}\label{p_maps_inj} 
If $\eta\not<0$ and $\bw(h_\eta)\ne 0$ then the maps 
\eqref{fga} are injective. 
\end{Proposition}

\begin{proof}
Take $\xi\in \fg_\eta$ and $\xi' \in \fg_{-\eta}$. Then 
$$
\xi'\, \xi \vacv{\bw} = \left[\xi',\xi\right] \vacv{\bw} 
= \bw\left(\left[\xi',\xi\right]\right) \vacv{\bw} 
$$
where the step in the middle follows from \eqref{ann_vac}. 
Now the claim follows from Proposition \ref{p_comm_pair}. 
\end{proof}

\begin{Corollary}
All roots spaces are $\bk$-modules of finite rank. 
\end{Corollary}

\begin{Corollary}
All roots are either positive or negative.  
\end{Corollary}

\subsection{}\label{s_bPeta} 

The $(\eta,-\eta)$-weight component of $\br$ defines a map 
$$
F_\bw(0) \otimes F_\bw(\eta) \to F_\bw(\eta) \otimes F_\bw(0) \,. 
$$
Since $F_\bw(0) \cong \bk$, this gives an operator 
$$
\bP_\eta: F_\bw(\eta) \to F_\bw(\eta) \,.
$$

\begin{Proposition} 
  \begin{equation}
    \label{Psq}
      \bP_\eta^2 = - \bw(h_\eta) \, \bP_\eta \,. 
  \end{equation}
\end{Proposition}

\begin{proof}
  Follows from considering the map 
$$
F_\bw(0) \otimes F_\bw(0) \otimes F_\bw(\eta) 
\to F_\bw(\eta) \otimes F_\bw(0) \otimes  F_\bw(0)
$$
given by \eqref{cYB1}. 
\end{proof}

\begin{Proposition}
If $\eta>0$ and $\bw(h_\eta)\ne 0$ then image 
of \eqref{fga} is the image of $\bP_\eta$ and 
$\bP_\eta^\vee$, respectively. 
\end{Proposition}

\noindent Here $\bP_\eta^\vee$ denotes the transpose map
between the dual modules. 

\begin{proof}
Apply both sides of \eqref{xicommr} to $\vacv{\bw}\otimes\vacv{\bw}$. 
\end{proof}

\begin{Corollary}
The root subspaces $\fg_{\pm \eta}$ are dual projective modules
over $\bk$. The classical $\br$-matrix 
$$
\br_{\eta,-\eta} \in \fg_{\eta} \otimes \fg_{-\eta} 
$$
is the canonical element of this tensor product.  
\end{Corollary}

\begin{Corollary}
The commutator
pairing from Proposition \ref{p_comm_pair} is perfect.  
\end{Corollary}

\subsection{}

We summarize the preceding discussion in the following 

\begin{Theorem}\label{t_Lie} 
All roots of $\fg_Q$ are either positive or negative. 
All roots spaces are 
projective $\bk$-modules of finite rank. 
The Lie algebra $\fg_Q$ has an invariant 
bilinear form $(\,\cdot\,,\,\cdot\,)_{\fg}$ 
such that $\br$ is the corresponding invariant tensor. 
With respect to this form, $\fg_{-\eta} = \fg_{\eta}^\vee$. 
\end{Theorem}

\noindent 
Since for Nakajima varieties $\bk$ is a polynomial ring, 
the modules $\fg_\eta$ are free. 
Consequently, we can choose 
bases $\{e_{\alpha}^{(i)}\}$ of the root spaces so that
$$
(e_\alpha^{(i)},e_{\beta}^{(j)})_\fg = \delta_{\alpha,-\beta}\cdot\delta_{i,j}\,.
$$
Correspondingly, we write 
\begin{equation}
  \br = \sum h_i \otimes h^i + \sum_{\alpha\ne 0} 
\sum_i 
e_\alpha^{(i)} \otimes e_{-\alpha}^{(i)}\,. 
\label{formula_r}
\end{equation}
One should bear in mind, however, that it is the invariant tensor 
$\br$ that is canonically defined, while choosing bases of root spaces
is a matter of convenience.

\subsection{}

For future use, we record here the following easy lemma:
\begin{Lemma}\label{quadraticsteinberg}
For each root $\alpha \ne 0$, the quadratic operator
$$\sum_{i} e_{\alpha}^{(i)} e_{-\alpha}^{(i)}$$
acts via a Steinberg correspondence on
each $F_{\bw}(\bv)$.
\end{Lemma}
\begin{proof}
Since $[e_{\alpha}, e_{-\alpha}]$ acts via a scalar, it suffices to prove this for $\alpha > 0$.
Choose $\bw_0$ such that
$h_\alpha(\bw_0) \ne 0$.  
Up to a nonzero scalar, the claim then follows from considering the action 
of the composition of Steinberg operators
$$\br_{-\alpha, \alpha} \circ \br_{\alpha,-\alpha}$$
on 
$F_{\bw_0}(0) \otimes F_{\bw}(\bv)$. 
\end{proof}

\subsection{}

We note that the projector $\bP_\eta$ has a direct geometric 
meaning for Nakajima variety. It is given by a Steinberg 
correspondence 
$$
\bP_\eta \subset \cM(\bw,\eta) \times \cM(\bw,\eta) 
$$
supported on 
$$
\Stab \Big( \cM(\bw,\eta) \times \cM(\bw,0)\Big) \cap 
\cM(\bw,0) \times \cM(\bw,\eta)
$$
viewed as $\bA$-fixed loci in $\cM(2\bw,\eta)$.

\section{Operators of classical multiplication}

\subsection{}

In the Yangian $\bY$, we have the operators 
\begin{equation}
  \bE\left(\vacv{\bw} \! \vacd{\bw}  u^k\right) \,, \quad \bw\in \Z^I\,, 
\quad k=1,2,3,\dots \,, 
\label{bEvac}
\end{equation}
where 
$$
\vacv{\bw} \! \vacd{\bw} \in \End H^\hd_\bG(\cM(\bw))
$$
is the orthogonal projector onto the vacuum. 
Recall from Figure \ref{f_Baxter} that 
for any $g$ such that 
$$
\left[g \otimes g,R(u) \right] = 0
$$
the operators 
$$
\tr_{F_0} (g \otimes 1) R_{F_0,W}(u) \in \End(W) \otimes \Q(u) 
$$
commute for all $W$ and all values of $u$ as a 
consequence of the Yang-Baxter equation. In particular, 
for $g = \vacv{\bw} \! \vacd{\bw}$ this shows the 
operators \eqref{bEvac} commute.

\subsection{}

If $\theta>0$, the vector $\vacv{\bw}$ is the true vacuum 
in the sense of Section \ref{s_Nak_vac}. This implies that
the operators \eqref{bEvac} are operators 
of cup product by 
certain characteristic classes of the virtual bundle 
$$
(1-\hbar) \otimes N_- = (1-\hbar) \otimes \sum \bw_i \, \cV_i 
$$
where $N_-$ is the negative part of the normal bundle to 
the embedding 
$$
\cM(\bw'') \hookrightarrow \cM(\bw + \bw'') \,. 
$$
In particular, this gives another reason why these operators
commute.

It is also 
clear that the operators \eqref{bEvac} generate all characteristic
classes of $\cV_i$ in the case $\theta > 0$.

\subsection{}\label{gen_theta} 

For general $\theta$, the relation between the operators \eqref{bEvac}
and the operators of classical multiplication 
may determined along the lines of Theorem \ref{t_Gauss}. Since the 
general expression in Theorem \ref{t_Gauss} is rather complicated
and requires working in a certain completion of the Yangian, 
we will not do it here. 

For the operators of classical multiplication by divisors, which is 
what we need for the proof of the main result of the paper, the 
case of general $\theta$ will be considered in Section \ref{s_class_div}. 

\subsection{}\label{s_disc_cW}

In Proposition \ref{p_cW} below we will see 
the Yangian also contains the operators
of multiplication by characteristic classes of the bundles 
$\cW_i$. 

These bundles are trivial but carry nontrivial group action, so 
this gives
$$
\varprojlim_{\bw} H^\hd_{G_\bw}(\pt) \hookrightarrow \textup{center} (\bY) \,. 
$$

\subsection{}

We call the subalgebra 
\begin{equation}
  \label{HomNak}
\textsf{Classical} \subset \bY 
\subset  \prod_{\bv,\bw} \End H^\hd_\bG(\Nak) 
\,. 
\end{equation}
generated by the characteristic classes of $\{\cV_i,\cW_i\}$ 
the \emph{algebra of classical multiplication}. 
Recall we assume that $\theta>0$, otherwise a certain completion 
of the Yangian is required.

As already discussed, the algebra of classical 
multiplication is expected\footnote{This has now been 
established in \cite{mn}.}
to surject onto all operators of cup product in each factor 
of \eqref{HomNak}. The following weaker statement will be sufficient for
our purposes. Recall that $\ft$ denotes the Lie algebra
of a maximal torus in $\bG$. 

\begin{Proposition}\label{p_span_H} 
After tensoring with $\Q(\ft)$, the algebra of classical 
multiplication surjects onto all operators of cup products in each factor
of \eqref{HomNak}. 
\end{Proposition}

\begin{proof}
There is a $\C^\times$ action on $\Nak$ that scales all quiver 
data by the same scalar. After tensoring with $\Q(\ft)$, 
 we may replace 
the cohomology of $\Nak$ by the cohomology of $\Nak^{\C^\times}$. 
The structure sheaf of the 
$$
\textup{Diagonal} \subset \Nak \times \Nak 
$$
may be resolved by tautotological bundles $\cV_i$, see \cite{Nak98}. Since 
$\Nak^{\C^\times}$ is compact, it shows that its cohomology 
is spanned by characteristic classes of tautological bundles. 
\end{proof}

\section{The structure of the Yangian}\label{s_Y_str} 

\subsection{}

In this section we assume $\theta>0$ for simplicity. 
Our goal here is the following 

\begin{Theorem}\label{t_grY} 
The Yangian is generated by the Lie algebra $\fg_Q$ and 
the operators of classical multiplication. We have 
$$
\gr \bY \cong \cU(\fg_Q[u])
$$
with respect to the filtration by degree in $u$. 
\end{Theorem}

In the course of the proof, it will be convenient to choose 
a splitting of 
$$
\bE :  \bigoplus F_{i} \otimes F_{i}^\vee \to \fg_Q \to 0
$$
which exists because $\fg_Q$ is a projective $\bk$-module. We 
will write $\xi= \bE(\xi)$ using such splitting. 
A concrete splitting may be constructed using the projectors
$\bP_\eta$ from Section \ref{s_bPeta}.

\subsection{}

\begin{Proposition}\label{p_gen_grY} 
If $\bE(m)=0$ then  
$$
\bE(m \, u^k) \in \bY_{<k}  
$$
with where $\bY_{<k}\subset \bY$ is the corresponding 
filtration subspace.   
\end{Proposition}

\begin{proof}
Since $k=0$ this is a tautology, we take $k>0$. 

The map $\bE$ is $\fhb$-equivariant and we can 
assume that $m$ is an eigenvector of $\fhb$ of weight $\mu$. 
If $\mu\ne 0$ then 
\begin{equation}
\mu(h) \, \bE(m u^k) = 
\left[ \bE(h), \bE(m u^k) \right] = 
\left[ \bE(h u^k), \bE(m) \right] + \dots = \dots 
\label{trade_u}
\end{equation}
where the step in the middle is based on \eqref{top_comm_E}. 

If $\mu=0$ then $\bE(m u^k)$ is a linear combination of diagonal 
matrix elements of the $R$-matrix.  Theorem \ref{t_Gauss}
expresses diagonal matrix elements of the $R$-matrix
in terms of the off-diagonal ones and characteristic
classes of $N_-$.

All terms involving off-diagonal matrix elements in 
Theorem \ref{t_Gauss} have 
degree $<k$. This is because they are at least quadratic
the entries of the $R$-matrix and there is a degree shift 
from the expansion 
$$
R(u) = 1 + \sum_{n\ge 0} \frac{R_n}{u^{n+1}} 
$$
to the filtration in the Yangian: matrix coefficients of 
$R_n$ belong to $\bY_{\le n}$. 

Now consider the characteristic classes of $N_-$. We have 
$$
\frac{e(N_-)}{e(N_- \otimes \hbar)} = 1+ \hbar \sum_{n\ge 0} 
\frac{n! \ch_n N_- + \dots}{u^{n+1}}\,,
$$
where dots stand for characteristic classes of degree $< n$.
In particular, applying this to \eqref{charN_}, we get 
\begin{multline*}
  \frac{1}{k!} \bE(m u^k) = \sum_i (m, \bw_i) \, \ch_k \cV_i \, + 
\sum_i (m, \bv_i) \, \ch_k \cW_i \\ 
- \sum_{i,j} \bC_{i,j} (m,\bv_i) \, \ch_k \cV_j + \dots 
\end{multline*}
where the pairing with $\bv_i,\bw_i \in \fhb$ is the trace pairing 
and dots stand for elements in $\bY_{< k}$. Note that 
by induction all characteristic classes of $\cV_i$ and $\cW_i$ of 
degree $<k$ are in $\bY_{< k}$. 

If $\bE(m)=0$ then $(m, \bw_i)= (m, \bv_i)=0$ and this 
concludes the proof. 
\end{proof}

The following is a corollary of the proof. 

\begin{Proposition}\label{p_cW} 
If $\theta>0$, all characteristic classes of $\cV_i$ and $\cW_i$
lie in $\bY$ and this inclusion preserves degree. 
\end{Proposition}

The case of general $\theta$ may be treated using Theorem 
\ref{t_Gauss}. In this case, a certain completion of the 
Yangian is required.

\subsection{Proof of Theorem \ref{t_grY}} \label{s_pr_tgrY} 

By Proposition \ref{p_gen_grY} and 
\eqref{top_comm_E}, the operators $\bE(\xi u^i)$ for $\xi\in \fg_Q$
generate the Yangian and 
satisfy the relations in $\fg_Q[u]$ modulo lower degree terms. 
This gives a surjective map 
$$
\cU(\fg_Q[u]) \to \gr \bY \to 0  \,. 
$$
Its injectivity may be seen as follows. For any faithful 
representation of a Lie algebra 
$$
0 \to \fg \to \End(F) 
$$
the corresponding representation of the universal 
enveloping algebra in tensor powers of $F$ 
$$
0 \to \cU \fg \to \bigoplus \End\left(F^{\otimes n}\right) 
$$
is injective. Since the Yangian is defined as a subalgebra 
of endomorphisms of tensor products, it remains to check 
that the map 
$$
\fg_Q[u] \to \gr \bY 
$$
is injective, which is elementary. In fact, 
\begin{equation}
\bE(\xi u^i) \big|_{F(v)} = v^i \, \left(\xi\big|_{F}\right) + O(v^{i-1}) \,, 
\quad v\to \infty
\label{large_ev}
\end{equation}
where $v$ is the evaluation parameter for the representation 
$F(v)$ and we identify all $F(v)$ with $F=F(0)$ as linear spaces. 
Equation \eqref{large_ev} means that the Yangian degenerates into 
the loop algebra when all evaluation parameters are very large. 

The last claim of the Theorem, the fact the operators of classical 
multiplication and $\fg_Q$ generate the Yangian follows from 
\eqref{trade_u}.

\subsection{}

As a consequence of the above result, we see that 
$\gr \bY$ and thus $\bY$ are flat
as $\bk$-modules.
It follows that the map \eqref{tensorembedding} is injective.
Indeed, using flatness, it suffices to prove injectivity after tensoring with the fraction field $K$ of $\bk$ (which we denote by subscript for brevity).
We then have inclusions 
$$\bY_K \otimes \bY_K \rightarrow \prod_{W} \End W_K \otimes \prod_{W'} \End W'_K \rightarrow
\prod_{W,W'} \End(W_K) \otimes \End(W'_K).$$
Injectivity after completion then follows from this case by decomposing the kernel into bi-graded pieces.

As a corollary, 
the coproduct
$$\Delta: \bY \to \bY \, \widehat\otimes \, \bY$$
is well-defined.

\chapter{Further properties of the Yangian}\label{s_Further} 

\section{The core Yangian} 

\subsection{}
In this section we assume $\theta > 0$. 
By Proposition \ref{p_cW}, the Yangian 
$\bY$ contains all characteristic classes $\ch_k(\cW_i)$ of the bundles
$\cW_i$. Since $\cW_i$ are trivial, $\ch_k(\cW_i)$
add little geometric value and it may be desirable to have a smaller 
algebra $\bbY$ that does not contain them. The goal of this 
section is to define such \emph{core Yangian} 
$$
\bbY \subset \bY \otimes \bk \left[\bdel^{-1}\right]\,,
$$
where $\bdel\subset\bk$ is a certain equivariant constant 
that depends on the equivariant Cartan matrix $\bC$ of the 
quiver. In particular, if the nonequivariant Cartan matrix 
is invertible then $\bdel^{-1}\in \bk$ and $\bbY \subset \bY$. 

\subsection{}

Recall from Theorem \ref{t_Gauss} and 
from the proof of Proposition \ref{p_gen_grY} 
and that 
the characteristic classes of $\cV_i$ and $\cW_i$ come 
from the operator of cup product by 
$$
\frac{e(N_-)}{e(N_- \otimes \hbar)} \in H^\hd_{\ga}(\cM(\bw) 
\times \cM(\bw'))
$$
that appears in the diagonal matrix elements of the $R$-matrices. 
Here 
\begin{align}
 N_-  = &\sum \Hom(\cW_i,\cV'_i) + 
\sum \Hom(\cV_i,\cW'_i)\otimes \hbar^{-1} \notag \\
&- \sum \bC_{ij} \Hom(\cV_i,\cV'_j) \label{charN2_}
\end{align}
in the negative part of the normal bundle to $\cM(\bw) \times 
\cM(\bw')$ inside $\cM(\bw+\bw')$ and $\bC$ is the equivariant 
Cartan matrix.

\subsection{}

The basic idea for defining $\bbY$ is the following. 
Complete the square in \eqref{charN2_} as follows 
\begin{equation}
 N_-  = 
- \sum \bC_{ij} \Hom\left(\cVt_i,\cVt'_j\right)
 \label{charN3_}
+ \sum \left(\bC^{-1}\right)_{ij} \Hom\left(\cW_i,\cW'_j\right)
\otimes \hbar^{-1} 
\end{equation}
where 
$$
\bC^{-1} \in \Mat(|I|,K_\ga(\pt)_\textup{localized}) 
$$
is the inverse of the equivariant Cartan matrix and  
\begin{equation}
\cVt = \cV - \hbar^{-1} \otimes \bC^{-1} \cW \label{defcVt}
\end{equation}
as vectors in $K_{\ga}(\cM(\bw)\times \cM(\bw'))^I \otimes 
 K_\ga(\pt)_\textup{localized}$. In particular, the Chern 
character 
$$
\ch \cVt = \ch \cV - e^{-\hbar} \, (\ch \bC)^{-1} \cdot \ch \cW 
$$
is defined if $\bC$ is 
invertible\footnote{Recall from section \ref{weightconvention} that we embed group weights into 
Lie algebra weights. While convenient, this could be confusing, 
especially in the context of Chern character. For example, 
by this rule, $\ch \hbar = e^{\hbar}$.}. 
However, it may contain terms of negative cohomological 
degree if the nonequivariant Cartan matrix is not invertible, see
below. 

The main feature of \eqref{charN3_} is that its second term 
is a purely equivariant object and so its Euler class may 
be taken out as an overall factor from the $R$-matrix.
The diagonal matrix elements of 
the new $R$-matrix generate only $\ch_k \cVt$. 
This smaller algebra will be the 
desired core Yangian $\bbY$. 

We now proceed with the realization of the this plan. 

\subsection{}
Let $\bG$ be a complex reductive group and $f\in \C(\bG)$ 
a rational function on $\bG$. We define 
$$
\ch_k f  \in \C(\Lie \bG)
$$
by the series expansion 
$$
\sum_{k} x^k \, \ch_k f (\xi) =  f(\exp(x \xi))\,, 
\quad \xi \in \Lie \bG \,, \quad x \in \C \,. 
$$  
This has negative terms if $f$ is not regular at $1\in \bG$.

Functoriality of $\ch_k f$ 
with respect to homomorphisms 
$\phi: \bG' \to \bG$ may fail if $\phi(\Lie \bG')$ lands in
the pole divisor of the Chern character. Because of this, 
we work in $\bG$-equivariant 
$K$-theory and cohomology for some fixed group $\bG$ 
if the nonequivariant 
Cartan matrix is not invertible.

For the rest of 
this section, we fix a group $\bG$ such that 
$$
\ga \supset \bG \supset \C^\times_t  \,,
$$
where  $\C^\times_t$ is the group that scales all quiver 
data by the same number $t\in \C^\times$.

\subsection{}

\begin{Lemma}\label{l_ch-2}
The matrix $\bC$ is invertible in localized $\bG$-equivariant 
$K$-theory and 
$$
\ch_k \bC^{-1} = 0 \,, \quad k < -2 \,. 
$$
\end{Lemma}

\begin{proof}
For the first claim, it suffices 
to consider the case $\bG = \C^\times_t$. Then  
$$
\bC = 1 + t^2 - t \, (Q+Q^T)  
$$
where $Q$ is the nonequivariant adjacency matrix of the quiver $Q$. 
Clearly, this is invertible. As a real symmetric matrix,
$Q+Q^T$ is semisimple. This implies $\bC^{-1}$ has poles of order 
$\le 2$ for $\bG = \C^\times_t$. 

For general $\bG$, the matrix $e^{\hbar/2} \ch \bC$ is Hermitian 
when the equivariant parameters lie in  the Lie algebra
$$
\mathsf{g}_\textup{c} = 
\Lie \bG_\textup{compact} = \{\xi,\xi^* = - \xi\} 
$$
of the compact real form of $\bG$. Therefore, its eigenvectors 
and eigenvalues are analytic along any real-analytic arc through 
the origin in $\mathsf{g}_\textup{c}$. In particular, the orders
of the poles of $\left(\ch \bC\right)^{-1}$ along any arc 
are the orders of vanishing of 
the eigenvalues of $\ch \bC$ along the same arc.
The latter are determined by the 
coefficients of the characteristic polynomial, and, therefore, 
semicontinuous as a function of the arc. Since they are $\le 2$ 
for $\bG = \C^\times_t$, the Lemma follows. 

\end{proof}


\subsection{}

We define $\bdel$ as the lowest degree term in
the expansion 
$$
\det \ch \bC = \bdel + \dots  \,. 
$$
By construction 
$$
\ch_k {\bC^{-1}} \in \Mat(|I|,H^\hd_\bG(\pt)[\bdel^{-1}]) \,. 
$$
for all $k$. 

\subsection{}

Let $Q_d$ be the quiver with the adjacency matrix $Q+Q^T$, in 
other words, 
\begin{equation}
Q_d = \Qb \setminus \{ \textup{framing vertices} \} \,. \label{defQd} 
\end{equation}
Let $\Path(Q_d)$ denote the path algebra of $Q_d$ and let 
$$
\Pi(Q_d) = \Path(Q_d) \left/ \left( \sum\nolimits_{a\in Q_d} [a,a^*]\right)
\right. 
$$
 denote the preprojective algebra of $Q_d$. Here $a^*$ is the arrow 
in $Q_d$ opposite to an arrow $a\in Q$. 

The group $G_\edge$ acts naturally on $\Path(Q_d)$ and $\Pi(Q_d)$, 
this action is \emph{dual} to the defining action of $G_\edge$ on 
representations of these algebras. In other words, the action of
$G_\edge$ on the generators of $\Path(Q_d)$ is recorded in the matrix 
$\overline{\bC}$. In particular, the natural grading on $\Path(Q_d)$, 
in which every arrow has degree $1$ is given by minus the weight of 
the $\C^\times_t$-action. All these weight spaces are finite-dimensional. 

By construction, $\Path(Q_d)$ has orthogonal idempotents $e_i$, $i\in I$, 
namely paths of zero length that start and end at a vertex $i$. We 
set
$$
\Path(Q_d)_{ij} = e_i \Path(Q_d)\, e_j \,,
$$
and similarly for $\Pi(Q_d)$. 
It is known, see for example
\cite{MOV,EG}, that
\begin{equation} \label{charPath} 
  \textup{character} \,\, 
 \Pi(Q_d)_{ij} = \big(\,\overline{\bC}^{-1}\,\big)_{ji}  =
\hbar^{-1} \otimes \left(\bC^{-1}\right)_{ij} \,,
\end{equation}
provided $Q$ is not a quiver of ADE type. We recall that by our 
convention $\bC_{ji}$ records edges going from $j$ to $i$. 

Formula \eqref{charPath} provides the following geometric interpretation 
of the $K$-theory classes \eqref{defcVt}. 

\subsection{}

Recall that 
we assume $\theta>0$. This means that the
natural map of bundles over $\cM_{\theta,0}(\bv,\bw)$ 
$$
\bigoplus_{j\in I} \Path(Q_d)_{ij} \otimes \cW_j  \to \cV_i 
$$
is surjective for all $i\in I$. Choose a $\bG$-invariant 
linear map (not algebra homomorphism) 
$$
s: \Pi(Q_d) \hookrightarrow \Path(Q_d)
$$
splitting the canonical surjection in the other direction. 
The moment map equations for $\cM_{\theta,0}(\bv,\bw)$
equal the relations in $\Pi(Q_d)$ modulo terms in 
the image of $\cW_j$. Therefore 
\begin{equation}
\bigoplus_{j\in I} 
s\left(\Pi(Q_d)\right)_{ij} \otimes \cW_j  \to \cV_i \to 0 \,,
\label{PiWV}
\end{equation}
is still surjective. 

The grading by $\C^\times_t$ makes the 
class of $\Pi(Q_d)_{ij}$ well-defined in 
completed $\bG$-equivariant $K$-theory. 
{}From \eqref{charPath}, we have the following

\begin{Proposition}
If $Q$ is not of ADE type, $\theta>0$, and $\bG$ contains
$\C^\times_t$, 
then the $\bG$-equivariant $K$-class of $\cVt$ is minus the kernel 
in \eqref{PiWV}.  
\end{Proposition}

There should be a more general statement valid for all quivers
and all stability conditions. 


\subsection{Example}

Let $Q$ be the quiver with one vertex and one loop, that is 
the quiver with the adjacency matrix $Q=(1)$. Then 
$$
\Pi(Q_d) = \C\langle x, y \rangle / (xy - yx) = \C[x,y] \,.
$$
The variety 
$$
\cM_{1,0}(n,1) = \Hilb_n(\C^2) 
$$
is the Hilbert scheme of point of $\C^2$, that is, the moduli space
of ideals $I\subset \C[x,y]$ of codimension $n$. The tautological 
sequence 
$$
0 \to I \to \C[x,y] \to \C[x,y]/I \to 0 
$$
is precisely the sequence 
$$
0 \to \Ker \to \Pi(Q_d) \to \cV \to 0 \,. 
$$

\subsection{}\label{s_Gamma_gen} 

We defined the $K$-classes that appear in \eqref{charN3_}
and their Chern characters. We now consider the operator 
\begin{equation}
\frac{e(N_-)}{e(N_- \otimes \hbar)}
= \frac{c(N_-^\vee,u)}{c(N_-^\vee\otimes \hbar^{-1},u)}
\label{EtC}
\end{equation}%
where 
\begin{equation}
c(L,u) = u^{\rk L} + c_1(L) \, u^{\rk L -1} + \dots 
\label{chernpoly}
\end{equation}
is the Chern polynomial and the bundle arguments of the Chern 
polynomials  in \eqref{EtC} 
are taken with the trivial action of $u$. 

By definition, we set 
\begin{equation}
\log c(L,u) = \sum_k  \ch_k L \, \ln^{(k)} u \,, 
\quad \ln^{(k)} u = \left(\tfrac{d}{du}\right)^{k} \ln u 
\label{defcL}
\end{equation}
for any $K$-theory class $L$ whose Chern character is defined. 
Here $\ln^{(-1)} u  = u(\ln u -1)$ etc. 
This generalizes \eqref{chernpoly} and is the usual 
$\zeta$-regularization 
of infinite products given by $\Gamma$-functions, see 
for example \cite{Ruij,Soule}. 

In particular, this defines $e(\tN_-)\big/e(\tN_- \otimes \hbar)$ for 
\begin{equation*}
 \tN_-  = 
- \sum \bC_{ij} \Hom\left(\cVt_i,\cVt'_j\right) \,. 
 \label{chartN}
\end{equation*}
In fact, we will only need it for 
\begin{equation}
  \label{tN0}
\tN_- \big|_{\bv=\bv'=0} = - 
\hbar^{-1}\otimes 
\sum \left(\bC^{-1}\right)_{ij} \Hom\left(\cW_i,\cW'_j\right) \,.
\end{equation}
We set 
$$
\Gamma(\bw,\bw') = \left. \frac{e(\tN_-)}{e(\tN_- \otimes \hbar)} 
\, 
\right|_{\bv=\bv'=0} 
$$
and define the new matrix $\tR$ as a scalar multiple of the 
old $R$-matrix 
\begin{equation}
\tR = \Gamma(\bw,\bw') \, R  \,. 
\label{Rh_gen}
\end{equation}
Tautologically, it also satisfies the Yang-Baxter equation. 

The old $R$-matrix was normalized to act by $1$ on the
vacuum vector, while the new matrix $\tR$ acts by a certain 
multivariate $\Gamma$-function. An example 
of $\Gamma(\bw,\bw')$ is given in Section \ref{s_Gamma_M(r,n)}
below. The appearance of 
$\Gamma$-functions in normalization of $R$-matrices
is a well-known phenomenon in the theory of quantum groups, 
see for example \cite{KhT}. 
Here we have yet another angle from which it can be seen.

\subsection{}\label{gamma_sing} 

We modify the definitions of Section \ref{s_def_Y} as follows. 
For $W$ as in \eqref{exW}, define 
\begin{equation*}
\tR_{F_{0}(u), W} 
= \tR_{F_{0},F_{n}}(u-u_n) \, \cdots \, 
\tR_{F_{0},F_{1}}(u-u_1) \,. 
\label{rRtrain}
\end{equation*}
We can write
$$
\tR_{F_{0}(u), W} = e^{\hbar \gamma_\sing} \, \tR_{F_{0}(u), W, \reg}
$$
where $\tR_{F_{0}(u), W, \reg}$ has a $1/u$-expansion 
and $\hbar \gamma_\sing$ is the singular part of the 
$u\to\infty$ expansion of $\log \tR_{F_{0}(u), W}$. 
In particular, $\gamma_\sing$ is a scalar operator.

In fact, Lemma \ref{l_ch-2} implies 
$$
\ch_k \tN\otimes (1-\hbar) = 0 \,, \quad k < -1 \,. 
$$
Therefore 
\begin{equation}
\gamma_\sing = 
\bc_{-2}\,  \ln^{(-1)} u  + \bc_{-1} \, \ln u 
\label{lnGamma}
\end{equation}
for certain scalar operators 
$$
\bc_{-2}, \bc_{-1} \in \bk[\bdel^{-1}][u_1,\dots,u_n] 
$$
of equivariant degree $-2$ and $-1$, respectively. The 
dependence on $u_i$ comes from 
$$
\ln^{(-1)}(u-u_i) = \ln^{(-1)}(u)- u_i \ln u + \dots\,, 
$$
and is at most linear.

\begin{Definition}
The core Yangian 
$$
\bbY \subset \bY \otimes \bk[\bdel^{-1}] 
$$
is the algebra generated by the matrix coefficients of 
$\bc_{-2}$, $\bc_{-1}$, and all coefficients of the $1/u$ 
expansion of $\tR_{F_{0}(u), W, \reg}$. Inside $\bbY$ we 
have a Lie algebra 
$$
\fg'_Q \subset \bbY
$$
generated by $\bc_{-2}$, $\bc_{-1}$, and the $u^{-1}$
coefficient of $\tR_{F_{0}(u), W, \reg}$. 
\end{Definition}

\noindent
Arguing as in Section \ref{s_Y_str} we obtain the following 

\begin{Theorem}\label{t_coreY} 
The core Yangian $\bbY$ is generated by $\fg'_Q$ 
and the operators of cup product 
by $\ch_k \cVt_i$ for $k\ge 1$ and $i\in I$. 
\end{Theorem}

\section{Slices and intertwiners} \label{s_Slices_Int} 

\subsection{}

Consider the following setup. It will not be the most general, but 
will suffice for our purposes and will illustrate the 
general ideas. Consider $H^\hd_\bT(\cM(\bw))$, where 
$\bT \subset G_\bw \times G_\edge$ is a torus and 
$$
\bw = a_i \, \delta_i + a_j \, \delta_j \,. 
$$
Here $\delta_i$ and $\delta_j$ are delta functions at some 
vertices $i,j\in I$ and $a_i,a_j$ are weights of $\bT$. 

As explained in Section \ref{s_minusc}, 
the first fundamental theorem of 
invariant theory gives an embedding of $\cM_0(\bw)$ into a 
particular vector representation $V$ of $\bT$. 
The weights of 
this representation correspond to closed paths in 
\eqref{defQd} 
as well as paths that start and end at vertices in $\{i,j\}$.

\subsection{} 

Let $P$ be a path of the form 
$$
j \xrightarrow{P_1} \bullet \xrightarrow{P_2} \bullet 
\xrightarrow{\,\,} \cdots  \xrightarrow{\,\,} 
\bullet \xrightarrow{\,\,} i 
$$ 
where  dots represent 
vertices of $Q$ and $P_i$ are arrows from $\Qb$.  The 
weight of the corresponding $G_\bv$-invariant function 
$f_P\in \C[\cM_0(\bw)]$ 
is computed as follows 
$$
w_P = - \textup{weight} \, f_P = a_i - a_j  + \sum t_k
$$
where $t_k$ is the weight of the arrow $P_k$. We 
\emph{assume}
that $\bT$ is such that 
\begin{equation}
w_P \ne - \textup{weight} \, f_{P'}
\label{unique_w}
\end{equation}
for any other generator $f_{P'}$ of $\C[\cM_0(\bw)]$. This 
assumption is satisfied in examples from Sections \ref{s_slice_ex1} 
and 
\ref{s_slice_ex2}. 

Denote $\bT' = \Ker w_P$ and let $x_P \in \cM_0(\bw)^{\bT'}$
be the unique, up to multiple, nonzero fixed representation. 
By construction, $\bT$ scales $x_P$ with weight $w_P$. 
By our assumption 
\begin{equation}
  \label{line_x_p}
  \cM_0(\bw)^{\bT'} = \C x_P\,, 
\end{equation}
where  $\C x_P$ is 
the line through $x_P$.

\subsection{}

Let $\Sigma_P$ denote the slice at $x_P$
\begin{equation}
\Sigma_P : \cM(\bv',\bw') \times U \dasharrow
\cM(\bv,\bw) \,,\label{slice_at_x_p}
\end{equation}
where 
\begin{equation}\label{bvx_P}
\bv' = \bv - \dim x_P \,, \quad \bw' = \bw - \hbar \otimes  \bC \, \dim x_P  
\end{equation}
by Proposition \ref{p_slice_w} and 
$$
U \cong \C^{\dim \cM(\bv,\bw) - \dim \cM(\bv',\bw')}
$$
is a vector space factor with the $\bT'$-character given by \eqref{formk}. 
In particular, restricting to the origin in $U$ 
we obtain a map 
$$
\Sigma_{P,0} : \cM(\bv',\bw') \dasharrow
\cM(\bv,\bw)  
$$
which is regular in the neighborhood of the central fiber 
of $\cM(\bv',\bw')$ and hence defines a map
$$
\Sigma_{P,0}^*: H^\hd_{\bT'}(\cM(\bv,\bw)) \to 
H^\hd_{\bT'}(\cM(\bv',\bw'))  \,.
$$

\begin{Proposition}
The map $\Sigma_{P,0}^*$ is a $\bbY$-intertwiner. 
\end{Proposition}

\begin{proof}
Since slice is a Steinberg correspondence, the bottom arrow 
in the diagram \eqref{slice_tensor} intertwines the $R$-matrices
on both sides. The vector space $U$ contributes a scalar 
factor to the $R$-matrix, therefore $\Sigma_{P,0}^*$ intertwines
$R$-matrices, up to a multiple. To see that it intertwines $\tR$-matrices, 
it suffices to note that 
$$
\cVt' = \cV' - \hbar^{-1} \otimes \bC^{-1} \, \cW' = 
\cV - \hbar^{-1} \otimes \bC^{-1} \, \cW = \cVt 
$$
from \eqref{bvx_P} \,. 
\end{proof}

\subsection{}

Let $\tT$ be a torus in $G_{\bw'} \times G_\edge$
that contains $\bT'$ and a maximal torus $\bA' \subset G_{\bw'}$. 
For any chamber $\fC \subset \Lie \bA'$, 
we have a map 
$$
\Stab_\fC: \bigotimes H^\hd_{\tT}  (\cM(\delta_i))^{\otimes \bw'_i} 
\to H^\hd_{\tT}  (\cM(\bw'))
$$
which  becomes an isomorphism after 
tensoring with $\Q(\Lie \tT)$ and intertwines 
the action of both full and core Yangians. 
The order of tensor factors here is determined by the 
chamber $\fC$, see Section \ref{e_bR} . 

We denote $\bK = \Q(\Lie \bT')$
and denote by $a'_{kl}$ the $\bT'$-weights in 
$\bw' = \sum a'_{kl} \, \delta_k$. 

\begin{Proposition}\label{p_res_stab}
For any $\fC$, the map $\Stab_\fC$ restricts to isomorphism 
$$
\bigotimes_{k,l} F_k(a'_{kl}) \otimes \bK  
\,\,\xrightarrow{\,\sim\,} \,\, H^\hd_{\bT'}(\cM(\bw')) \otimes \bK
$$
of Yangian modules, where $F_k$ are as in \eqref{set_F_i}. 
\end{Proposition}

\noindent 
Here the evaluation parameters $a'_{kl}$ are as in Section \ref{s_Yang_Nak}
and the order of tensor factors as before. Note, in particular, 
the Proposition implies the tensor product on the left gives
isomorphic Yangian modules for any ordering of tensor factors. 

We begin with the following

\begin{Lemma}\label{l_Uweight} 
The torus $\bT'$ has a zero weight in $U$ and, therefore, 
a unique fixed point in $\cM_0(\bw')$. 
\end{Lemma}

\noindent 
The second claim here follows from the first because 
of \eqref{line_x_p}. 

\begin{proof}[Proof of Proposition \ref{p_res_stab}]
By Theorem \ref{adjoint}, the inverse map is given by 
$\Stab_{-\fC}^\tau$. The lemma shows $\cM(\bw')^{\bT'}$ 
is proper, therefore $\Stab_{-\fC}^\tau$ is 
well-defined in localized $\bT'$-equivariant cohomology. 
\end{proof}

\subsection{}

Now for $\cM(\bw)$ we want to do the same: first enlarge $\bT$ 
to include a maximal torus $\bA \cong (\C^\times)^2\subset G_\bw$ and then 
restrict to $\bT'$-equivariant cohomology. For $\bA$, there 
are only two chambers $\fC_{>}$ and $\fC_{<}$, 
corresponding to $a_i \gtrless a_j$. 
Let 
$$
 F_i(a_i) \otimes F_j(a_j)  
\xrightarrow{\Stab_{>}}  H^\hd_{\bT'}(\cM(\bw))
\xleftarrow{\Stab_{<}}  F_j(a_j) \otimes F_i(a_i) 
$$
be the corresponding maps. 

\begin{Proposition}
The map $\Stab_{>}$ becomes an isomorphism after tensoring with $\bK$. 
\end{Proposition}

\begin{proof}
The inverse map is given by $\Stab_{<}^\tau$. 
By construction the line \eqref{line_x_p}
has weight $w_P$ which is negative on $\fC_{<}$ and therefore 
transverse to the images of attracting manifolds. Thus $\Stab_{<}^\tau$
is well-defined in localized $\bT'$-equivariant cohomology. 
\end{proof}

Note that the analogous statement for $\Stab_{<}$ fails since 
the push-forward along $\C \, x_P$ is not defined in 
$\bT'$-equivariant cohomology. We have, however, the 
following 

\begin{Proposition}
The operator
$$
\left(\Sigma_{P,0} \circ \Stab_{>}\right)^\tau : 
H^\hd_{\bT'}(\cM(\bw')) \otimes \bK \to F_j(a_j) \otimes F_i(a_i)
\otimes \bK 
$$
is a well-defined $\bbY$-intertwiner. 
\end{Proposition}

\begin{proof}
The map is well-defined by 
 Lemma \ref{l_Uweight} since the image of $\Sigma_{P,0}$ 
is transverse to $\C x_P$. It is an intertwiner because
its transpose is.  
\end{proof}

\subsection{}

We summarize the preceding discussion as follows. Suppose 
$$
\cM_0(a_i \delta_i + a_j \delta_j)^{\bT'} = \C x_P\,,
$$
where $P$ is a path that starts at $j$ and ends at $i$. 
Define $a'_{kl}$ by the formula 
$$
\sum a'_{kl} \, \delta_k = 
a_i \delta_i + a_j \delta_j 
- \hbar \otimes  \bC \, \dim x_P\,,
$$
where $\dim x_P$ is a vector with values in $K_{\bT'}(\pt)$. 

\begin{Theorem}
The slice at $x_P$ gives rise to two $\bbY$-intertwiners: 
\begin{equation}\label{slice_inter1}
F_i(a_i) \otimes F_j(a_j) \otimes \bK \to 
\bigotimes F_k(a'_{kl}) \otimes \bK 
\end{equation}
and 
\begin{equation}
\bigotimes F_k(a'_{kl}) \otimes \bK \to F_j(a_j) \otimes F_i(a_i)
\otimes \bK \,, 
\label{slice_inter2}
\end{equation}
where the equivariant parameters are specialized to $\bT'$, 
$\bK = \Q(\Lie \bT')$, the order of the $F_k(a_{kl})$-factors 
is arbitrary in \eqref{slice_inter1} and reverse in 
\eqref{slice_inter2}. 
\end{Theorem}

\begin{proof}
The first map is given by 
$$
\Stab_{-\fC}^\tau \circ \Sigma_{P,0}^* \circ \Stab_{>}\,,
$$
for $\fC$ matching the order of factors. The second map 
is its transpose. 
\end{proof}

\section{The dual Yangian}

\subsection{}

We define the dual Yangian $\bY^*$ as the 
algebra generated by the operators 
\begin{equation}
\bE^*(m^*(v)) = 
\Res_{v=0} \, \tr_{F_0} m^*(v) \, 
R_{W,F_0(v)} \, \in \bY^*
\label{def_bEs}\,,
\end{equation}
for all $W$ of the form 
\begin{equation}
W = \bigotimes_{i=1}^n F_{i} \otimes \bk(u_i)_\infty \,,
\label{WY*} 
\end{equation}
and
$$
m^*(v)\in F_0 \otimes F_0^\vee \otimes v^{-1}\bk[v^{-1}]\,. 
$$
Here $\bk(u)_\infty$ denotes rational functions of $u$ 
regular at $u=\infty$.

In English, $\bY^*$ is generated by matrix elements of the same 
matrices $R(u-v)$ but expanded in ascending powers 
of $v$.  In particular, the operators $\bE^*$ depend 
rationally, not polynomially, on the evaluation 
parameters $u_i$. 

Note that the operators $\bE^*\left(m_{-1} v^{-1}\right)$ already give
all matrix elements of $R(u)$ and their orbits under shift 
automorphism span $\bY^*$.

\subsection{}

There is a natural pairing between $\bY$ and $\bY^*$ 
defined as follows. Let
$$
M(u) = m_1(u_1) \otimes \cdots \otimes m_k(u_k)
$$
be an element in the domain \eqref{defE} of the map $\bE$
and, similarly, let
$$
M^*(v) = m_1^*(v_1) \otimes \cdots \otimes m^*_l(v_l)
$$
lie in the domain of $\bE^*$. Let 
$$
R(u,v) = \overrightarrow{\prod_{\substack{1 \le i \le k
\\ 1 \le j \le l}}} R(u_i - v_j)
$$
be the corresponding $R$-matrix where $(i,j)$th term 
acts in the spaces with evaluation parameters $u_i$ and 
$v_j$ and the ordering of the 
$R$-matrices is as in \eqref{Rtrain}. We define 
\begin{align}
  \Big(\bE(M(u)), \bE^*(M^*(v)) \Big) &=
\left[\frac{1}{u_1 \dots v_l}\right] \, 
 \tr_{u,v} \left(M(u) \otimes M^*(v)\right) \, R(u,v) \notag 
 \\
&= \hbar^k  
\left[\frac{1}{v_1 \dots v_l}\right] \tr_{v} \, M^*(v) \, \bE(M(u))  
\notag \\
&=  
\left[\frac{1}{u_1 \dots u_k}\right] \tr_{u} \, M(u) \, \bE^*(M^*(v))  
\label{pairYY}
\end{align}
where coefficients are taken in the 
$u_i\to \infty$, $v_j\to 0$ expansion and 
the subscripts of traces indicate tensor 
factors in which they are taken. 

\subsection{} 

As defined, \eqref{pairYY} is a pairing between 
the domains of $\bE$ and $\bE^*$. It is clear, however, 
that the kernels on both sides are exactly 
the kernels of $\bE$ and $\bE^*$. 
In other words, we have the following 

\begin{Proposition} \label{p_ker_perp}
$$
\Ker \bE = \left(\bY^*\right)^\perp \,, \quad 
\Ker \bE^* = \left(\bY \right)^\perp \,. 
$$
\end{Proposition}

\subsection{}

By construction, \eqref{pairYY}  is a Hopf pairing, that is 
$$
(ab,c) = (a \otimes b, \Delta c) 
$$
and vice versa, where $(\,\cdot\,,\,\cdot\,)$ is extended to 
$$
\bY^{\otimes 2} \otimes \left(\bY^*\right)^{\otimes 2} \to \bk
$$
multiplicatively. Tautologically, this pairing 
stores the same information 
as the $R$-matrices.

\section{Intertwiners and relations}

\subsection{}

Let $W$ as in \eqref{WY*} be a $\bY^*$-module and let 
$$
C: W \to W'
$$
be a $\bk(u,u')$-linear map, where 
\begin{equation*}
W' = \bigotimes_{i'=1}^{n'} F_{i'} \otimes \bk(u'_{i'})_\infty \,,
\label{WY*'} 
\end{equation*}
be another $\bY^*$-module of the same form. Suppose that 
for certain values of $u$ and $u'$ the map $C$ becomes
a $\bY^*$-intertwiner, that is, 
$$
[y,C] \in \Hom(W,W') \otimes \bI
$$
for all $y \in \bY^*$ and a nontrivial ideal 
$$
\bI \subset \bk(u,u')_\infty
$$
in the local ring of the point 
$(u,u') = (\infty,\infty, \dots, \infty)$.

Note that $\bY^*$-intertwiners are  operators
that commutes with all $R$-matrices and, therefore, 
the same as $\bY$-intertwiners, up to extension of scalars. 
Intertwiners produce elements in $(\bY^*)^\perp$ and 
hence relations in $\bY$ as follows. 

\subsection{}

Let $\bI^\perp \subset \bk[u,u']$ denote the perpendicular of 
$\bI$ with respect the the residue pairing.

\begin{Proposition} For any $f\in \bI^\perp$ and any 
$$
m \in \bigotimes F_{i'}^\vee \otimes \bigotimes F_{i} \,. 
$$
we have a relation  
  \begin{equation}
    \label{relmC}
     \Res_{u'} \bE(f m C) = \Res_{u} \bE (f C m)  
  \end{equation}
in the Yangian $\bY$. 
\end{Proposition}

\noindent 
Here $\Res_u$ means taking the coefficient of $(u_1\cdots u_n)^{-1}$
in the $u_i\to\infty$ expansion. Also note that $m: W' \to W$ is 
an operator of finite rank, therefore both $mC$ and $Cm$ are 
in the domain of $\bE$. 

Note that in the product $f C$ under the 
$\bE$-sign in the left-hand side of \eqref{relmC} we should keep only the singular 
(that is, polynomial) terms in the $u_i\to \infty$ 
expansion because the residue in 
\eqref{def_bE}  vanishes for regular terms. Similarly for 
$u'_j \to \infty$ in the right-hand side of \eqref{relmC}.

\begin{proof}
For any $y\in \bY^*$ we have 
$$
\tr_{W} \, m \, C \, y  - \tr_{W'} \, C \, m \, y \in \bI 
$$
and therefore 
$$
\Res_u \, \Res_{u'} \left( \tr_{W} \, f\, m \, C \, y  - 
\tr_{W'} \, f\, C \, m \, y 
\right) = 0 \, .
$$
This is equivalent to \eqref{relmC}.
\end{proof}

\subsection{}\label{s_Q_slices} 

The whole discussion can be repeated for 
the core Yangian $\bbY$ in place of $\bY$. 
Since slices produce $\bbY$-intertwiners, 
the following question seems natural. 

\begin{Question}
Do all relations in Yangians come from slices ? 
\end{Question}

\section{Baxter subalgebras and Casimir connection}\label{s_Baxter} 

\subsection{}

Recall that $\fh \subset \fg_Q$ acts by linear 
functions of $\bv$ and let $\fH\cong(\C^\times)^I$ be the torus 
with Lie algebra $\fh$. Since 
$\fg_Q$ 
commutes with $R$-matrices, we have 
$$
\left[g \otimes g,R(u) \right] = 0
$$
for any $g\in \fH$. Recall from Section \ref{s_Baxter_sub} 
this implies the operators 
\begin{equation}
\bE_{F_0}( g \, u^k)  = 
\frac{1}{\hbar} 
\left[\frac{1}{u^{k+1}}\right] \, \tr_{F_0}  (g \otimes 1) \,  
R_{F_0(u),W} \,
\label{baxter_E}
\end{equation}
commute for all $k=0,1,\dots$ and all auxiliary spaces $F_0$
for which the trace $\tr_{F_0}$ is well defined. This 

In general, $F_0$ is not finite-dimensional and the 
trace in \eqref{baxter_E} is an infinite sum. However, it is 
well defined as a formal series in the variable 
$g\in \fH$ 
if $F_0$ satisfies the grading 
assumption from Section \ref{s_compl}. We denote by
$$
q^\bv \in \bk \fH^\wedge
$$
elements of the group $\bk$-algebra of the character group
$\fH^\wedge$. These functions of $g$ will be terms in 
our formal series. 
Introduce an algebra $\bY[[\fH^\wedge]]$ of formal series in $q^\bv$
with coefficients in $\bY$ by 
$$
\bY[[\fH^\wedge]] = 
\left\{\sum_{\bv \ge \bv_0} \mathsf{y}_{\bv} \, q^{\bv} \right\}  \,.
$$
Here $\mathsf{y}_{\bv} \in \bY$ and  
$\bv \ge \bv_0$ means $\bv-\bv_0 \in \Z_{\ge 0}^I$. We have 
\begin{equation}
\bE_{F_0}( g \, u^k) = 
\frac{1}{\hbar} 
\left[\frac{1}{u^{k+1}}\right] \sum_\bv q^\bv \, 
\tr_{(F_0)_\bv}  (g \otimes 1) \,  
R_{F_0(u),W} \,\in \bY[[\fH^\wedge]] 
\label{bax_gen}
\end{equation}
as a consequence of our grading assumption.

By definition, the subalgebra of $\bY[[\fH^\wedge]]$ generated over 
$\bk[[\fH^\wedge]]$ by the commuting 
operators \eqref{bax_gen} is called the \emph{Baxter subalgebra}. 
It is a formal family of commuting subalgebras of $\bY$. 

\subsection{}

Baxter subalgebras are graded with respect to the cohomological 
grading on the Yangian and 
$$
\deg_{\,\textup{coh}} \bE_{F_0}( g \, u^k)  = 2k  \,. 
$$
In particular, 
$$
\big(\textup{Baxter subalgebra}\big)_{\textup{coh degree $0$}} =
 \cU_\Q(\fhb)[[\fH^\wedge]] \,,
$$
where $\fhb \subset \fg_Q$ acts by linear functions of $\bv$ and $\bw$. 
Because $\bk$ has nontrivial cohomological grading, the 
universal enveloping algebra here is over 
$$
\Q = \big(\bk\big)_{\textup{coh degree $0$}} \,.
$$
Our goal now is to describe the degree $2$ part of Baxter 
subalgebra. It is spanned, over degree $0$ part, 
 by equivariant constants and $u^{-2}$ coefficients 
in \eqref{baxter_E}. 

\subsection{}

Evidently, only diagonal matrix coefficients contribute to the 
trace in \eqref{baxter_E}. 
 The $u^{-2}$-term of the diagonal matrix coefficients 
of $R$-matrices was computed in Proposition \ref{Gauss_u2}. 
The result can be stated as follows.
Let 
$$
\cM_{\theta,\zeta}(\bv,\bw) \times 
\cM_{\theta,\zeta}(\bv',\bw') \subset 
\cM_{\theta,\zeta}(\bv+\bv',\bw+ \bw')
$$
be a fixed component and let $R_{\bv,\bw,\bv',\bw'}$ be the 
corresponding diagonal block of the $R$-matrix. It 
follows from Proposition \ref{Gauss_u2} that 
\begin{equation}
\frac{1}{\hbar} \left[\frac{1}{u^2}\right] R_{\bv,\bw,\bv',\bw'}
 = (\bw - \bC \, \bv) \otimes \ch_1 \cV'
 + \hbar \sum_{\theta \cdot \alpha >0 } e_{\alpha} e_{-\alpha} 
\otimes e_{-\alpha} e_{\alpha}  + \dots \,,
\end{equation}
where $\ch_1 \cV'$ is a vector of $\ch_1 \cV'_i$, $i\in I$, 
$\bC$ is the nonequivariant Cartan matrix, and 
dots act by a scalar operator 
in $H^\hd_\bG(\cM(\bv',\bw'))$. 

\subsection{}

For $F_0$ as above define 
$$
\bchi(F_0) \in \fh[[\fH^\wedge]] 
$$
by requiring 
$$
\tr_{F_0} g \, h_\eta = \eta(\bchi(F_0)) 
$$
for all $\eta \in \fh^*$. Here $h_\eta = \Cb \eta \in \fhb$, see
Section \ref{s_h_eta}. 
Since $\tr_{F_0} g \, h_\eta$ depends linearly on $\eta$,  this 
is well defined. Clearly,  $\bchi(F_0)$ is linear in the 
$K$-theory class of $F_0$. 

\begin{Lemma}\label{l_tr_e_e}
  \begin{equation}
    \label{tr_e_e}
     \tr_{F_0} g \, e_\alpha \, e_{-\alpha} =  
-\alpha(\bchi(F_0)) \, \frac{q^\alpha}{1-q^{\alpha}} \,. 
  \end{equation}
\end{Lemma}

\noindent
The rational function in \eqref{tr_e_e} is to be expanded 
in one direction or another, depending on $\alpha \gtrless 0$, 
to represent an element of $\bk[[\fH^\wedge]]$.

\begin{proof}
Using 
\begin{equation}
[e_{\alpha},e_{-\alpha}] = h_\alpha \label{eaa0}
\end{equation}
we compute 
\begin{align*}
 \tr_{F_0} g \, e_\alpha \, e_{-\alpha} & = 
 \tr_{F_0} e_{-\alpha} \, g \, e_\alpha \\
& = q^\alpha \,  \tr_{F_0} g\, e_{-\alpha} \, e_\alpha\\ 
& = q^\alpha \,  \tr_{F_0} g\, e_{\alpha} \, e_{-\alpha} - 
q^\alpha \,  \tr_{F_0} g\, h_\alpha \,,
\end{align*}
whence the conclusion. 
\end{proof}

\subsection{}

{}From Lemma \ref{l_tr_e_e} we deduce the following 

\begin{Proposition} We have 
  \begin{equation}
    \bE_{F_0}(g u^2) = \bchi(F_0) \cdot \ch_1 \cV' -
\hbar \sum_{\theta\cdot\alpha >0} \alpha(\bchi(F_0)) \, 
\, \frac{q^\alpha}{1-q^{\alpha}} \, e_{-\alpha} \, e_{\alpha}  + \dots \,,
\label{Bax_Cas}
  \end{equation}
where dots stand for an element of $\cU(\fhb)[[\fH^\wedge]]$
\end{Proposition}

\noindent 
By Theorem \ref{t_Qlam} below this means that the degree
$2$ part of Baxter algebra is spanned by operators of quantum 
multiplication by $q$-dependent tautological divisors 
$$
\lambda = \bchi(F_0) \in \fh[[\fH^\wedge]] 
$$
and equivariant constants.

Using formula \eqref{c1div}, we can rearrange the terms in 
\eqref{Bax_Cas} as follows 
\begin{multline}
\bw \cdot \ch_1 \cV' -
\hbar \sum_{\theta\cdot\alpha >0} (\alpha,\bw) \, 
\, \frac{q^\alpha}{1-q^{\alpha}} \, e_{-\alpha} \, e_{\alpha}=
\\
\bE\left(\vacv{\bw} \! \vacd{\bw}  u^2\right) -
\hbar \sum_{\theta\cdot\alpha >0} 
(|\alpha|,\bw)  \,
\frac{q^{|\alpha|}}{1-q^{|\alpha|}} \, e_{-\alpha} e_{\alpha} + \dots \,,
\label{Bax_Cas_2}
\end{multline}
where dots stand terms from $\cU(\fhb)$ and 
$$
|\alpha| = 
\begin{cases}
\alpha\,, & \alpha >0 \,,\\
-\alpha\,, & \alpha <0 \,.
\end{cases}
$$
The second line in \eqref{Bax_Cas_2} is manifestly an element of 
$\bY[[\fH^\wedge]]$ while the sum over $\alpha$ in 
the first line 
converges in a different formal series completion 
--- the one corresponding 
to the effective cone of $\cM_{\theta,0}$.

\chapter{Quantum multiplication}
\section{Preliminaries}

We first set some notation regarding equivariant Gromov-Witten invariants.
Suppose we are given a smooth quasi-projective variety $X$ equipped with the action of a reductive group $\bG$.
For each effective curve class $\beta \in \mathrm{Eff}(X) \subset H_{2}(X,\Z)$, its associated $k$-point genus $0$ Gromov-Witten invariants are given by integrals
$$
\langle \gamma_1, \dots, \gamma_k \rangle_{0,k,\beta}^{X}
= \int_{[\overline{M}_{0,k}(X,\beta)]^{\mathrm{vir}}}
\ev^{*}\left(\gamma_1\boxtimes\dots \boxtimes \gamma_k\right).
$$
for $\gamma_i \in H^{\hd}_{\bG}(X,\Q)$.
Here, the integral is defined over the virtual fundamental class on the moduli space of $k$-pointed stable maps to $X$.
As always, if $X$ is noncompact (as in our case), the above expression can be defined via equivariant residue.  However, since the evaluation maps are proper, operators of quantum multiplication are defined without localization.

\section{Modified reduced operators}

\subsection{}

We recall some general results from \cite{BMO} for the quantum product 
for any equivariant symplectic resolution 
$$
X \rightarrow X_0 = \Spec H^0(\cO_X) \,.
$$

Due to the presence of the symplectic form $\omega$, the
moduli space of maps carries a \emph{reduced} virtual class in degree one larger than
the usual virtual dimension.
This reduced class determines the purely quantum contributions to all divisor operators via the relation
$$
\left(\gamma * \gamma_1,\gamma_2\right) = \left(\gamma \cup \gamma_1,\gamma_2\right) + \hbar
\sum_{\beta > 0} (\gamma \cdot \beta) q^\beta 
\int_{\left[\overline{M}_{0,2}(X,\beta)\right]_{\vir,\red}} 
\ev^*(\gamma_1 \boxtimes \gamma_2) \,. 
$$

Moreover, the pushforward
of the reduced virtual fundamental class under the evaluation map 
$$
\ev: \overline{M}_{0,2}(X,\beta) \to X^2\,, 
$$
is a $\Q$-linear combination of Steinberg correspondences 
of $X\times_{X_0} X$. In particular, it does not depend 
on equivariant parameters. We denote by
$$
\bQ_{2,\red} \in \End H^\hd_\bG(X)\otimes \Q[[\Eff(X)]]
$$
the purely quantum operator defined by the reduced class
$$
\left(\bQ_{2,\red} \cdot \gamma_1,\gamma_2\right) = \sum_{\beta > 0} q^\beta 
\int_{\left[\overline{M}_{0,2}(X,\beta)\right]_{\vir,\red}} 
\ev^*(\gamma_1 \boxtimes \gamma_2) \,. 
$$
This is a correspondence-valued element in the completion of the semigroup algebra of the effective cone of $X$, each
coefficient of which is a Steinberg correspondence for $X$.  

\subsection{}
Note that by \eqref{1*}
$$
\bQ_{2,\red} \cdot 1 = 0  \,. 
$$
This uniquely determines the coefficient of the diagonal in 
$\bQ_{2,\red}$ from the other terms.  It will be convenient to work modulo scalar
operator contributions to $\bQ_{2,\red}$ in this chapter; this relation allows us to fix this indeterminacy.

\subsection{}\label{hyp_t_reduc} 

Given
$$\bkap_X \in H^2(X,\Z/2),$$
we define the \emph{modified} quantum operator
$\bQ_{2,\red,\bkap}$ for $X$ by the substitution 
$$
q^\beta \mapsto (-1)^{(\kappa_X,\beta)} \, q^\beta \,, \quad \beta \in H_2(X,\Z) \,. 
$$
This is equivalent to changing the origin 
in the K\"ahler moduli space 
$H^2(X,\C)/2\pi i H^2(X,\Z)$ to $\pi i \kappa_X$.

Let a torus $\bA$ act on $X$ preserving the symplectic form 
and let $Y\subset X^\bA$ be a connected component. 
Assume we have chosen
$$
\bkap_Y \in H^2(Y,\Z/2)\,,
$$
such that
$$
c_1(N_+) \equiv \bkap_X \big|_Y + \bkap_Y \mod 2 \,
$$
where $N_+$ is the positive part of the normal bundle to $Y$ for 
some (equivalently, any) choice of the chamber $\fC\subset \fa = \Lie \bA$. 

For Nakajima varieties, the canonical theta characteristics
$\bkap$ were defined in \eqref{def_kappa} and connected
to the parity of $c_1(N_+)$ in \eqref{parity_N}.

\subsection{}

Our next goal is the following 

\begin{Theorem}\label{t_reduc} 
For $X$ and $Y$ as in Section \ref{hyp_t_reduc}, the diagram 
\begin{equation}\label{modif_stab}
\xymatrix{
H^\hd_\bT(Y) \ar@{->}[rr]^{\Stab_\fC\,\,}\ar[d]_{\bQ_{2,\red,\bkap_Y}}&& 
H^\hd_\bT(X) \ar[d]^{\bQ_{2,\red,\bkap_X}} \\
H^\hd_\bT(Y)  \ar@{<-}[rr]^{\Stab_{-\fC}^\tau} &&  H^\hd_\bT(X)\\
}
\end{equation}
is commutative for any $\fC$ and any polarization, 
after applying the map
$$\Q[[\Eff(Y)]] \rightarrow \Q[[\Eff(X)]]\,$$
to $\bQ_{2,\red,\bkap_Y}$ and working 
modulo scalar operators on $H^\hd_\bT(Y)$.
\end{Theorem}

\noindent 
Note that the bottom arrow in \eqref{modif_stab} 
is a priori defined only in localized equivariant 
cohomology. As a part of the proof, we will see that 
the composition of the top, right, and the bottom arrows 
in \eqref{modif_stab} is well-defined without 
localization. 

The proof of this theorem will require the discussion of 
broken and unbroken curves in equivariant localization. 
We recall 
the relevant definitions and results from \cite{OPhilb}.

\section{Broken curves}

\subsection{}

Let $f: C\to X$ be an $\bA$-fixed 
point of $\overline{M}_{0,2}(X,\beta)$
such that the 
domain $C$ is a chain of rational curves 
$$
C = C_1 \cup C_2 \cup \dots \cup C_k \,,
$$
with the two marked points $p_1,p_2$ lying on $C_1$ and $C_k$, respectively. 
 
If at every node of $C$ the $\bA$-weights
of the two branches are opposite and nonzero then we say that $f$ is an \emph{unbroken chain}. 
We say that $f$ \emph{connects} the 
points 
$$
x_0 = f(p_1), \,\, x_k=f(p_2) 
$$
of $X$ through the sequence of nodes 
$$
x_i=f(C_i \cap C_{i+1})\,, \quad i=1,\dots,k-1\,.
$$
Note that all of these points are fixed by $\bA$.

More generally, if $(C,f)$ is an $\bA$-fixed point of $\overline{M}_{0,2}(X,\beta)$,
we say that $f$ is an \emph{unbroken map} if it satisfies one of three
conditions:
\begin{enumerate}
\item $f$ arises from a map $f: C \to X^{\bA}$,
\item
$f$ is an unbroken chain, or
\item
the domain $C$ is a chain of rational curves
$$C = C_0\cup C_1\cup \cdots \cup C_k$$
such that $C_0$ is contracted by $f$, the marked points lie on $C_0$, and
the remaining components form an unbroken chain.
\end{enumerate}
Broken maps are $\bA$-fixed maps that do not satisfy one of these conditions.

In this last possibility, the contribution of these curves is block-diagonal with respect to $\bA$-fixed locus of $X$,
 i.e. scalar on each connected component $Y$, hence 
we will focus on the unbroken chains in what follows.

\subsection{}

We refer the reader to Section 3.8.3 of \cite{OPhilb} for the proof of 
the following. 

\begin{Theorem}[\cite{OPhilb}] 
Every map in a given connected component of $\overline{M}_{0,2}(X,\beta)^\bA$ is either broken 
or unbroken. Only unbroken components contribute to $\bQ_{2,\red}$ 
in $\bA$-equivariant localization. 
\end{Theorem}

\subsection{}

Let $f$ be an unbroken chain as before and let 
let $\cO(1)$ be a $\bA$-linearized ample line bundle on $X$. We may restrict 
it to fixed point $x_i$ to get elements of $\fa^*$. We have 
the following 

\begin{Lemma}\label{L_unbroken_monotone}
For an unbroken chain, the points 
$$
\cO(1)\Big|_{x_i} \in \fa^*\,, \quad i=0,\dots, k \,, 
$$
form a monotone sequence of distinct points of a real line.
\end{Lemma}

\begin{proof} We denote this sequence by $c_i$. Let $w$ denote the (nonzero) $\bA$-weight 
of $T_{p_1} C$. By the unbroken condition, the same weight occurs at all 
nodes and the weight of $T_{p_2} C$ is $-w$.  By localization, 
the terms of the sequence 
$$
\frac{c_0-c_1}{w},\dots,\frac{c_{k-1}-c_{k}}{w}
$$
are the degrees of $f^*\cO(1)$ restricted to $C_i$, hence positive 
integers. 
\end{proof} 

\subsection{} \label{s_no_loops} 

Lemma \ref{L_unbroken_monotone} is effective in ruling out 
unbroken loops. More generally, we have the 
following 

\begin{Lemma}\label{Lendpoints}
There are no $\bA$-fixed unbroken chains 
connecting two points in the same component $Y$ of $X^\bA$.
\end{Lemma}

\begin{proof}
The $\bA$-weight of $\cO(1)_y$ is constant for $y\in Y$, 
which contradicts the fact that points in Lemma \ref{L_unbroken_monotone}
are distinct.
\end{proof}

\section{Proof of Theorem \ref{t_reduc}}

\subsection{} 
Given $\beta \in H_{2}(X,\Z)$ and $\gamma_1,\gamma_2\in H^\hd_{\bT}(Y)$,
the statement to prove is 
\begin{multline}
  \sum_{\beta' \mapsto \beta} (-1)^{(\beta,\kappa_Y)+\frac12 \dim Y} \lang \gamma_1,\gamma_2 \rang^Y_{\beta',\red} 
=  \\ 
 (-1)^{(\beta,\kappa_X)+\frac12 \dim X} \Big\langle \Stab_\fC(\gamma_1),
\Stab_{-\fC}(\gamma_2) \Big\rangle^X_{\beta,\red} + c_{\beta}\,, 
\label{rest_t_reduc}
\end{multline}
where $c_{\beta}$ is a constant independent of the insertions $\gamma_1, \gamma_2$. 

The sign $(-1)^{\frac12 \codim_X Y}$ comes from the sign 
in the definition of the adjoint $\Stab_{-\fC}^\tau$.

\subsection{}
Recall that every coefficient of $\bQ_{2,\red,\bkap_X}$ is given by a Steinberg correspondence. 
As in the proof of Theorem \ref{t_Stein}, this implies the convolution 
$$
\Stab_{-\fC}^\tau \circ \, \bQ_{2,\red,\bkap_X} \circ \Stab_\fC \,. 
$$
is obtained by a proper push-forward. In particular, its coefficients can be determined by
any specialization of equivariant parameters.

This means we can compute the RHS of \eqref{rest_t_reduc} by 
$\bA$-equivariant localization, and study its limit after taking the equivariant 
parameters associated to $\fa$ to infinity, while setting $\hbar=0$ at
the same time.

\subsection{}

We only need to consider unbroken components of 
$\overline{M}_{0,2}(X,\beta)^\bA$ in equivariant localization. 

Since stable envelopes are proportional to fixed points modulo $\hbar$, 
setting $\hbar=0$ implies that only components where both marked points map to $Y$ will
give nonzero contribution.

If we fix a component of $\overline{M}_{0,2}(X,\beta)^\bA$ whose elements
consist of curves for which both marked points lie on a contracted component attached to an unbroken chain.
Since the evaluation map to $Y \times Y$ for this component factors through the diagonal,
the contribution of this component will give a scalar operator, so we can ignore it.
By Lemma \ref{Lendpoints}, unbroken chains do not contribute either,
so the only contributions come from stable maps which factor through $Y$.
Furthermore, since the localization contribution only depend on the equivariant normal bundle to $Y$ in $X$, we may replace $X$ by the total space $N$ of the normal bundle.

\subsection{}
For a vector bundle 
$$
p: N\to Y
$$
we have the following general 
result. Suppose $\bA$ acts on $N$ fiberwise and $N^\bA = Y$. 
We decompose 
$$
N = \bigoplus_\lambda N_\lambda
$$
according to the characters $\lambda\in \bA^\vee \subset \fa^*$. 

Given cohomology classes $\gamma_1, \dots, \gamma_k \in H^\hd(Y)$
we want to understand the asymptotic behavior of the Gromov-Witten invariant
$$\left\langle p^*(\gamma_1), \dots, p^*(\gamma_k) \right\rangle^{N}_{\beta,g}\in \Q(\fa^*)$$
defined via equivariant residue, as the variables in $\fa$ approach infinity.
Here, $g\ge 0$ is the domain genus and 
$\beta\in H_2(Y,\Z) = H_2(N,Z)$ is the degree of the map.

The residue invariant can be expressed in terms of the Gromov-Witten invariants of $Y$ by adding 
an Euler class insertion determined by $N$.  The following computation is then a standard application of
Riemann-Roch: 

\begin{Lemma}\label{Lultraeq}  We have the asymptotic behavior given by 
$$
\left\langle p^*(\gamma_1), \dots, p^*(\gamma_k) \right\rangle^{N}_{\beta,g,k}
\sim \prod_\lambda \lambda^{-(c_1(N_\lambda),\beta) - \rk N_\lambda (1-g)} \, 
\left\langle \gamma_1, \dots, \gamma_k \right\rangle^{Y}_{\beta,g,k}\,.
$$
\end{Lemma}

\subsection{}
We only need the $g=0$ case of the above lemma. Also
$$
N_\lambda = N^{\vee}_{-\lambda} 
$$
because of the symplectic form. Therefore, the prefactor in Lemma 
\ref{Lultraeq} becomes
$$
(-1)^{(c_1(N_+),\beta)} \big/ \det N  \,.
$$
Since 
$$
(\det N_+)^2 \big/ \det N = (-1)^{\frac12 \rk N} =
(-1)^{\frac12 \codim_X Y} 
$$
the equality \eqref{rest_t_reduc} follows. 

\section{Additivity}

\subsection{}

Suppose $Y$ as above factors 
$$
Y =Y_1 \times Y_2\,, 
$$
as it is the case in our main example \eqref{MtoM}. Then 
$\dim H^0(Y,\Omega^2)\ge 2$, leading to further constraints
on quantum cohomology of $Y$. 

\begin{Proposition}\label{red_add} 
$$
\bQ^{Y}_{2,\red} = \bQ^{Y_1}_{2,\red} \otimes 1 + 1 \otimes \bQ^{Y_2}_{2,\red} \,. 
$$
\end{Proposition}

\begin{proof}
Let $\beta=(\beta_1,\beta_2)$ according to 
$$
H_2(Y,\Z) = H_2(Y_1,\Z) \oplus H_2(Y_2,\Z) \,. 
$$
If $\beta_1\ne 0$ and $\beta_2\ne 0$ then the virtual fundamental class may 
be \emph{doubly reduced}, meaning that the reduced obstruction theory admits a further surjection to 
a trivial rank 1 bundle.  See Section 3.5 in \cite{OPhilb}. As a result, 
the corresponding reduced class vanishes. If $\beta_1=0$ or $\beta_2=0$, then 
the curve maps to a point in one of the factors, and the above 
additivity is obvious. 
\end{proof}

\subsection{}

Notice that additivity is not the same as primitivity.

In Proposition \ref{red_add}, we are
restricting to diagonal contributions; the off-diagonal terms will still be non-zero.

\chapter{Shift operators}\label{S_Int}

\section{Definition} 

\subsection{}

For any $X$ and any cocharacter
$$
\sigma:\C^\times \to \bG \subset \Aut(X)
$$
we can associate an $X$-bundle over $\Pp^1$ as follows 
\begin{equation}
  \Xt = (\C^2 \setminus 0) \times X \big/ \C^\times_\sigma
\label{defXt}
\end{equation}
where $\C^\times_\sigma$ acts on the first factor by scaling and on the second via 
the homomorphism  $\sigma$. 
This is just the classical operation of passing from a principal $\C^\times$-bundle
over $\Pp^1$ to the associated $X$-bundle. 

\subsection{}

Since $c_1(X)=0$,  we 
have\footnote{In the present discussion, one does not 
need to assume $X$ symplectic. It suffices to assume
that the canonical bundle of $X$ is a pure character, 
which we denote by $\hbar^{\frac12\dim X}$.}
\begin{equation}
  \label{c1tX}
  c_1\big(\Xt\big) = (2+ \tfrac{\dim X}{2}(\sigma,\hbar))\cdot \left[\textup{\sl Fiber}\right] \in H^2(\Xt,\Z)
\end{equation}
where fiber refers to the natural projection 
$$
p: \Xt  \to (\C^2 \setminus 0) \big/ \C^\times = \Pp^1\,,
$$
the inner product is the standard pairing of characters with 1-parameter
subgroups.

\subsection{} 

Let $\bG^\sigma$ be the centralizer of the image of $\sigma$ in $\bG$. 
We define 
\begin{equation}
  \label{defbGt}
  \bGt = \bG^\sigma \times \C^\times_\varepsilon \subset \Aut(\Xt) 
\end{equation}
where the second factor scales the second (by convention) coordinate of $\C^2$.
This group preserves the fibers $X_0,X_\infty$ of $p$ over $0,\infty\in\Pp^1$. 
More precisely, it fixes $X_0$ point-wise, but acts nontrivially on 
$X_\infty$.

We denote by $\varepsilon$ an element of $\Lie \C^\times_\varepsilon$. This is 
a new equivariant parameter which we have for $\Xt$.


\subsection{}

Any point of $X^\sigma$ gives a section of $p$ 
\begin{equation}
  \zeta_x: \Pp^1 \to \Pp^1 \times x \subset \Xt\,, \quad x \in X^\sigma \,. 
\label{sec}
\end{equation}
The homology class of this section gives an element 
$$
\zeta \in H^0(X^\sigma,\Z) \otimes H_2(\Xt,\Z) \,.
$$
More formally, for any $D\in H^2(\Xt,\Z)$ we define
\begin{equation}
  (D,\zeta) = \textup{proj}_* \, \textup{incl}^* (D) 
\in H^0(X^\sigma)\label{defzeta}
\end{equation}
where 
$$
X^\sigma \xleftarrow{\,\,\,\textup{proj}\,\,\,} X^\sigma \times \Pp^1 
\xrightarrow{\,\,\,\textup{incl}\,\,\,} \Xt 
$$
are the natural maps. 

\subsection{}

We have 
\begin{equation}\label{sectionclasses}
0 \to H_2(X,\Z) \to H_2(\Xt,\Z)
 \to H_2(\Pp^1)\cong \Z \to 0 \,.
\end{equation}
Any section $\zeta_x$ gives 
a noncanonical splitting of the above exact sequence.

In particular, the degrees $\beta\in H_2(\Xt,\Z)$ such 
that $p(\beta)=[\Pp^1]$ form a single $H_2(X,\Z)$-coset of sections. 
For $\beta$ in this coset, we consider
$$
M^{\sim}(\beta) = \ev^{-1} (X_0 \times X_\infty) \subset 
\overline{M}_{0,2}(\Xt,\beta) \,, 
$$
where $X_0,X_\infty\subset X$ are the fibers of $p$ over 
$0,\infty\in \Pp^1$.

\subsection{Example}\label{tildeP1} 

Take 
$$
X=T^*\Pp^1 = \cM_{1,0}(1,2)
$$
for the quiver 
$Q$ with one vertex and no arrows. This is the quotient of 
pairs 
$$
A =
\begin{pmatrix}
  a_1 & a_2
\end{pmatrix}
\in \Hom(\C^2,\C^1)\,,
\quad
B =
\begin{pmatrix}
  b_1 \\ b_2
\end{pmatrix}
\in \Hom(\C^1,\C^2)
$$
such that $AB=0$ and $A\ne 0$ by the action of $G_\bv = GL(1)$. Take 
$$
\sigma(z) = 
\begin{pmatrix}
z \\
& 1 
\end{pmatrix} \in G_\bw  \,.
$$
Then $X^\sigma = \{x_0,x_\infty\}$,  where 
$$
x_0 = \{a_2=0, B=0\}\,,  \quad x_\infty = \{a_1=0, B=0\}  \,. 
$$
The variety $\Xt$ is the relative 
cotangent bundle to the $\Pp^1$-bundle over $\Pp^1$ given by 
$$
\textup{Bl}_\textup{point}\,  \Pp^2 \to \Pp^1 \,. 
$$
We have 
$$
\left[\zeta_{x_0}\right] = \textup{line in $\Pp^2$} \,,
\quad 
\left[\zeta_{x_\infty}\right] = \textup{exceptional divisor}\,,
$$
and so $\left[\zeta_{x_0}\right] - \left[\zeta_{x_\infty}\right]$
is the generator $\left[\Pp^1\right]$ of $H_2(X,\Z)$.

\subsection{}\label{s_bS(sigma)}

We use the spaces $\Xt$ to define shift operators
$$\bS(\sigma):  H^\hd_{\bGt}(X_\infty) \longrightarrow H^\hd_{\bGt}(X_0) \otimes \Q[[\Eff(\Xt)]]$$
as follows.

Given $\gamma_1 \in H^\hd_{\bGt}(X_0)$ and $\gamma_2 \in H^\hd_{\bGt}(X_\infty)$,
we define the matrix element
\begin{equation}
  (\gamma_1,\bS(\sigma)\cdot \gamma_2)  = 
\sum_{p(\beta)=[\Pp^1]} q^\beta \int_{\left[M^{\sim}(\beta)\right]^{vir}} \ev^{-1}(\gamma_1
\times \gamma_2)\,.
\label{defSsig}
\end{equation}
%
%
By definition, the matrix coefficients \eqref{defSsig} take values 
in formal power series in $q^\beta$ with coefficients 
in the localized $\bGt$-equivariant cohomology of a point, although the operator $\bS(\sigma)$ itself is integral. 
In particular, \eqref{defSsig} depends on 
$$
\varepsilon \in \Lie(\C^\times_\varepsilon) \subset \Lie(\bGt) \,.
$$
Our eventual goal will be to find $\sigma,\gamma_1,\gamma_2$ such
that the integral in \eqref{defSsig} is proper of 
correct dimension, thus independent of all equivariant 
parameters.

\section{Intertwining property}

Given $D\in H^2_{\bGt}(\Xt)$, we set 
$$
\ddD\, q^\beta = \left(\int_\beta D \right)\, q^\beta\,, \quad \beta \in H_2(\Xt) \,.
$$
Note that this is nonequivariant, that is, depends only on the class of $D$ in 
nonequivariant cohomology. 

If we consider the restriction $D_0 = D\Big|_{X_0} \in H^2_{\bGt}(X_0)$, let
$\mathsf{M}_{D_{0}}(q)$ denote the operator of quantum multiplication by $D_0$, and similarly for $X_\infty$.

\begin{Proposition}\label{p_intertw} For any $D$ as above, 
the operator \eqref{defSsig} satisfies 
  \begin{equation}
    \label{intertw}
    \varepsilon \, \ddD \, \bS(\sigma) =\mathsf{M}_{D_{0}}(q)\circ \, \bS(\sigma) -  
\bS(\sigma) \, \circ\mathsf{M}_{D_{\infty}}(q) \,.
  \end{equation}
\end{Proposition}

\begin{proof}
For brevity, set $Y=X^\sigma$.

We compute $\bS(\sigma)$ by localization with respect to the 
$\C^\times_\varepsilon$-factor in \eqref{defSsig}. 
The domain of an  $\C^\times_\varepsilon$-fixed
map in $(C,f)\in M^{\sim}(\beta)$ is a union
$$
C=C_0 \cup C_1 \cup C_\infty
$$
where $f:C_0 \rightarrow X_0$ is a $\sigma$-fixed map, $f(C_\infty)\subset X_\infty$, and 
$C_1$ is of the form \eqref{sec}
$$
C_1 = \zeta_y
$$
for some point $y\in Y$.

Standard localization arguments (see e.g.\ Chapter 27 in \cite{Mirror})
give a factorization 
\begin{equation}
  \bS(\sigma) =  \Psi_0 \, \Psi_1\, \Psi_\infty \,,
\label{bSfac}
\end{equation}
with the following factors. We define 
\begin{equation}
  (\gamma_1,\Psi_0 \cdot \gamma_2)  = 
\sum_{\beta\in H_2(X_0,\Z)} q^\beta \int_{\left[\overline{M}_{0,2}(X_0,\beta)\right]^{vir}} 
\frac{\ev^*(\gamma_1
\times \gamma_2)}{\varepsilon-\psi_2}\,.
\label{defPsi0}
\end{equation}
Here $\psi_2$ 
is the cotangent line at the second marked point and the integral is computed 
in equivariant cohomology.
The unstable $\beta=0$ contributions to \eqref{defPsi0} are 
defined to give
\begin{equation}
  \Psi_0 = 1 + O(q^\beta), \quad \beta >0 \,.
\label{unstab}
\end{equation}
If we evaluate $\Psi_0$ via virtual localization, we obtain precisely the localization contribution of the
$C_0$. 

In the definition of $\Psi_\infty$, one replaces $\varepsilon-\psi_2$ by
$-\varepsilon-\psi_1$ and $X_0$ by $X_\infty$. 

Note the virtual class in \eqref{defPsi0} is the ordinary, 
nonreduced virtual fundamental class. The reduced virtual 
class gives 
\begin{equation}
  \Psi_0 = 1 + O(\hbar)\,, 
\label{Psihbar}
\end{equation}
if $X$ is holomorphic symplectic. 

The middle factor $\Psi_1$ is of the form 
$$
\Psi_1 = {\iota_0}_* \, q^{\zeta} \, \Gamma \, \iota_\infty^*
$$
where $\iota_0,\iota_\infty$ denote the inclusion of $Y$ into 
$X_0$ and $X_\infty$ respectively, 
$\Gamma$ is multiplication by a class in $H^\hd(Y)$ that absorbs 
the deformation and obstruction contributions of $C_1$.  The class 
$\zeta$ was defined in \eqref{defzeta}; $q^{\zeta}$ is a monomial which varies depending on the
connected component of $Y$. Note that by
localization 
\begin{equation}
  \label{[C_1]}
  \varepsilon \, (\zeta, D) = 
\iota_0^*(D) - \iota_\infty^*(D) \,.
\end{equation}

It is standard \cite{CoxKatz, Mirror} to abbreviate 
$$
\tau_k(\gamma_2) = \psi_2^k \, \ev^*(\gamma_2) \,. 
$$
The convention \eqref{unstab} means that
$$
\lang \gamma_1 \, \tau_{k}(\gamma_2) \rang_0 = \delta_{k+1} \, \int_{X} \gamma_1 \cup \gamma_2
$$
where angle brackets denote equivariant genus $0$ GW-invariants of $X$ and subscript refers to 
invariants of degree $\beta=0$. With this convention, the string and divisor 
equations yield
\begin{equation}
  \lang \gamma_1, \tau_k(\gamma_2), D \rang_\beta = 
 \left(\int_\beta D \right)\, \lang \gamma_1, \tau_k(\gamma_2) \rang_\beta +
\lang \gamma_1, \tau_{k-1}(\gamma_2 \cup D) \rang_\beta\,,
\label{sting_divisor}
\end{equation}
for all $k\ge 0$ and all degrees $\beta$. Similarly, the topological 
recursion relations (see e.g.\ Section 26.4 in 
\cite{Mirror}) 
read 
\begin{equation}
  \lang \gamma_1, \tau_k(\gamma_2), D \rang_\beta = 
\sum \lang \gamma_1, D, \eta\rang_{\beta'} \lang \eta^\vee \, \tau_{k-1}(\gamma_2)\rang_{\beta-\beta'} 
\label{trr}
\end{equation}
for all $k\ge 0$ and all degree splittings, where $\sum \eta \otimes \eta^\vee$ 
is the Poincar\'e dual of the diagonal in $X^2$. 

Combining \eqref{sting_divisor} with \eqref{trr} gives
  \begin{equation}
    \label{intertw1}
    \varepsilon \, \ddD \, \Psi_0 = \mathsf{M}_{D_{0}}(q)\circ \Psi_0 -  
\Psi_0 \circ \mathsf{M}_{D_{0}}(0)  
  \end{equation}
where $\mathsf{M}_{D_{0}}(0)$ denotes the operator of classical multiplication by $D_0$.
By the same 
reasoning 
  \begin{equation}
    \label{intertw2}
    \varepsilon \, \ddD \, \Psi_\infty = \mathsf{M}_{D_{\infty}}(0)\circ \Psi_\infty -  
\Psi_\infty\circ \mathsf{M}_{D_{\infty}}(q)\,.
  \end{equation}
Finally, \eqref{[C_1]} gives 
  \begin{equation}
    \label{intertw3}
    \varepsilon \, \ddD \, \Psi_1 = \mathsf{M}_{D_{0}}(0) \circ \Psi_1 -  
\Psi_1 \, \circ \mathsf{M}_{D_{\infty}}(0) \,.
  \end{equation}
The combination of \eqref{intertw1},  \eqref{intertw2}, and  \eqref{intertw3} 
completes the proof.  
\end{proof}

\section{Shift operators are quantum operators}
\label{bS_qm}

In this section, we extend Proposition \ref{p_intertw}; as a consequence, we will see that shift
operators are quantum operators after passing to the $\varepsilon =0$ limit.

Let 
$$M^{\sim}_{\bullet}(\beta) = 
 \ev_{1,2}^{-1} (X_0 \times X_\infty) \subset 
\overline{M}_{0,3}(\Xt,\beta)$$
denote the moduli space of twisted maps from last section, equipped with an extra marked point $\bullet$.

Given
$\gamma \in H^\hd_{\bGt}(X_0)$,
we define
the operator
$\bS_0(\sigma; \gamma)$
by
\begin{equation}
  (\gamma_1,\bS_0(\sigma;\gamma)\cdot \gamma_2)  = 
\sum_{p(\beta)=[\Pp^1]} q^\beta \int_{\left[M^{\sim}_{\bullet}(\beta)\right]^{vir}} \ev_{1,2}^{-1}(\gamma_1
\times \gamma_2)\cup \ev_{\bullet}^* (\iota_{0,*}\gamma)
\label{defSsig0}
\end{equation}

\begin{Lemma}\label{markedshiftformula}
We have the factorization
$$\bS_0(\sigma; \gamma) = \mathsf{M}_{\gamma}(q)\circ \bS(\sigma)$$
where $ \mathsf{M}_{\gamma}(q)$ 
denotes quantum multiplication operator for $\gamma$.
\end{Lemma}
\begin{proof}

We follow the approach of Proposition \ref{p_intertw}.
If we compute $\bS_0(\sigma;\gamma)$ by localization 
with respect to the $\C^\times_\varepsilon$-factor.  As before,
this gives a factorization
\begin{equation}
  \bS_0(\sigma;\gamma) =  \Psi^{\gamma}_0 \, \Psi_1\, \Psi_\infty \,,
\label{bS0fac}
\end{equation}
where the second and third factor are as before and the first factor is defined by
\begin{equation}
  (\gamma_1,\Psi^{\gamma}_0 \cdot \gamma_2)  = 
\varepsilon\cdot \sum_{\beta\in H_2(X)} q^\beta \int_{\left[\overline{M}_{0,3}(X,\beta)\right]^{vir}} 
\frac{\ev^*(\gamma_1
\times \gamma_2\times \gamma)}{\varepsilon-\psi_2}\,.
\label{defPsi0_}
\end{equation}

When we expand this expression, the leading term
with no $\psi_2$ is simply quantum multiplication by $\gamma$.  The terms with 
positive powers of $\psi_2$ can be expanded using
the topological recursion relation of \eqref{trr} to give
quantum multiplication by $\gamma$ composed with the $\beta > 0$ contribution to
$\Psi_0$.

The result is the factorization
$$
\Psi^{\gamma}_0 = \mathsf{M}_{\gamma}(q) \circ \Psi_0.
$$
Combining with \eqref{bSfac} gives the statement of the lemma.
\end{proof}

Given $\gamma \in H^\hd_{\bGt}(X_\infty)$,
we can define the operator $\bS_\infty(\sigma;\gamma)$ in the analogous manner, and we can derive the formula
$$\bS_\infty(\sigma; \gamma) = \bS(\sigma)\circ \mathsf{M}_{\gamma}(q)$$
in the same way.

If we restrict to $\bG^\sigma$-equivariant cohomology by setting $\varepsilon = 0$,
then for $\gamma \in H^\hd_{\bGt}(X)$, we have
$\iota_{0,\ast}\gamma = \iota_{\infty,\ast}\gamma$
after this specialization.  In particular, we have
$$\bS_0(\sigma;\gamma) = \bS_\infty(\sigma;\gamma)$$
after setting $\varepsilon = 0$.

As a corollary, we see that the shift operator
$\bS(\sigma)|_{\varepsilon=0}$ 
commutes with all quantum multiplication operators.  If we fix a 
splitting of \eqref{sectionclasses}, it can thus be identified 
with quantum multiplication by
$$\bS(\sigma)\Big|_{\varepsilon=0}(1) \in H^\hd_{\bG^\sigma}(X)\otimes \Q[[\Eff(X)]].$$

\chapter{Minuscule shifts and $R$-matrices}

\section{Setup}

\subsection{}\label{s_assume}

In this chapter, we consider shift operators 
$\bS(\sigma)$ satisfying the following 
additional assumptions: 
\begin{enumerate}
\item $X$ is a symplectic resolution, 
\item $\sigma$ preserves the symplectic form $\omega$, 
\item $\sigma$ is minuscule, 
\end{enumerate}
see Section \ref{s_minusc} for a discussion of 
the last condition.

\subsection{}
We define
\begin{equation}
  \Stab_{\pm}: H^{*}(X^\sigma) \to H^{*}(X) \,.\label{stabunstab}
\end{equation}
to be the stable envelope maps
 corresponding to the two chambers $\gtrless 0$ of 
$\Lie \C^\times_\sigma$ and an arbitrary choice of polarization. 
We will see a close relation between $\bS(\sigma)$
and 
\begin{equation}
R_\sigma = \Stab_{-}^{-1} \, \Stab_+  \label{defRsigma} \,. 
\end{equation}

\subsection{}
It follows from our assumptions and \eqref{c1tX} 
that 
$$
c_1\big(\Xt\big) =  2 \left[\textup{\sl Fiber}\right]
$$
and hence that 
\begin{equation}
  \virdim M^{\sim}(\beta) = \dim X \,.
\label{virdim}
\end{equation}
for all $\beta$ in \eqref{defSsig}. This means 
$\bS(\sigma)$ has cohomological degree $0$.

\subsection{}
\begin{Lemma}\label{l_smack} 
With the assumptions of Section \ref{s_assume},
all $\sigma$-weights
in the normal bundle to $X^\sigma$ are $\pm 1$. 
\end{Lemma}

\begin{proof}
Choose a proper map $X \to V$, where $V$ is a linear 
representation of $\sigma$ with weights in $\{0,\pm 1\}$. 
For any $x$, the $\sigma$-orbit of $x$ is either 
contracted by the map to $V$ or is mapped isomorphically to a 
line of weight $\pm 1$. 

If there is a weight $k\ne\pm1$ in the normal bundle to some 
component $Y$ of $X^\sigma$
then the corresponding normal directions are mapped to a point 
in $V$. Hence, their closure meets another component $Y'$ of $X^\sigma$, 
where same weight $k$ has to occur again. 
Using induction on $\prec$ and finiteness of 
the number of component of $X^\sigma$, we 
see that this is impossible. 
\end{proof}

\subsection{}\label{dis_zeta} 

Recall from \eqref{defzeta} that every component of $X^\sigma$
defines a curve class $H_2(\Xt,\Z)$.  If we fix a splitting of \eqref{sectionclasses}, we can project
to obtain curve classes in $H_2(X,\Z)$.  

A more convenient way of making this choice is as follows.
Choose 
a $\sigma$-linearization for a basis $\cL_1,\cL_2,\dots$ of 
$\Pic(X)$.  Given $x \in X^\sigma$, we define $\overline{\zeta_x} \in H_2(X,\Z)$
so that
$$
\int_{\overline{\zeta_x}} c_1(\cL_i)  = \deg \widetilde{\cL_i}\Big|_{\zeta_x}
$$
where $\zeta_x$ is the section \eqref{sec} and $\widetilde{\cL_i}$ 
is the lift of $\cL_i$ to $\Xt$ that uses the fixed 
linearization. A change of linearization adds a overall constant 
to $\overline{\zeta}$. 

For Nakajima varieties, the entire group of automorphisms $\bG$ 
acts naturally on all tautological bundles and their associated determinant bundles.
In this case, we can arrange for the linearization of $\Pic(X)$ to be compatible with this natural linearization.
If we know that these tautological divisor classes span $\Pic(X)$, this completely determines the linearization and
thus gives a preferred normalization of the map $\overline{\zeta}$.

\subsection{}\label{dis_sigma} 

In particular, take $\sigma$ to be the action 
corresponding to a tensor product of Nakajima varieties, that is to 
$$
\bw = z \, \bw' + \bw''\,, \quad \bv = z \, \bv' + \bv''\,,
$$
as in Section \ref{s_bw=}. Then 
$$
(\overline{\zeta},c_1(\cV_i)) = \bv'_i \,.
$$
In other words, connected components of $X^\sigma$ are 
distinguished by the value of $\bv'$ and 
\begin{equation}
\xymatrix{\overline{\zeta}\ar@{|->}[d]\ar@{|->}[r]
& H_2(\Nak,\Z)\ar[d]\\
\bv' \ar@{|->}[r] &\fh^*}
\label{zeta_bv}
\end{equation}
under the natural map on the right. 

For example, for $T^*\Pp^1$ as in Section \ref{tildeP1}, we get 
$$
q^{\overline \zeta} = 
\begin{pmatrix}
  q \\ 
& 1 
\end{pmatrix}
$$
in the basis $\{x_0,x_\infty\}$. 

\subsection{}

Recall that $\bkap$ is defined as $c_1(\Th)$ modulo $2$, where
$\Th$ is a half of tangent bundle as in \eqref{half_tangent}. 

\begin{Lemma}\label{l_parity} For minuscule $\sigma$
$$
(\zeta,\bkap) = \tfrac12 \codim X^\sigma  \mod 2 \,. 
$$
\end{Lemma}

\begin{proof}
Since $\sigma$ is minuscule, the $\sigma$-weights of $\Th$ when restricted to fixed loci lie in the set $\{0,\pm 1\}$.  
Moreover, the number of nonzero weights
equals $\tfrac12 \codim X^\sigma$. 
Therefore, the bundle $\widetilde{\Th}$ restricted
to $\zeta_x$ is a sum of $\cO(k)$, $k\in \{0,\pm 1\}$, and
the number of nontrivial terms in this sum equals 
$\tfrac12 \codim X^\sigma$. 
\end{proof}

\section{Properness}

\begin{Proposition}\label{p_proper} 
For $X$ and $\sigma$ be as above, the convolution 
\begin{equation}
\Stab_-^\tau \circ \, \bS(\sigma) \circ \Stab_-
\label{conv_Ssig}
\end{equation}
is proper and defines a Steinberg correspondence in 
$X^\sigma \times_{X_0} X^\sigma$.
\end{Proposition}

\begin{proof}
Since $\sigma$ is minuscule, we have proper $\C^\times_\sigma$-equivariant maps 
$$
X \to X_0 \subset V = \bigoplus_{|i|\le 1} V_i\,,
$$
where $\C^\times_\sigma$ acts on $V_i$ with weight $i$.
Applying \eqref{defXt}, we get a proper 
map 
$$
\Xt \to \Vs = \bigoplus_{|i|\le 1} V_i \otimes \cO(i)\,,
$$
to a vector bundle $\Vs$ over $\Pp^1$. Moreover, the fiberwise
image of $\Stab_-$ is contained in the subbundle 
$$
\Vs_{\le 0} = V_0\otimes \cO \oplus V_{-1} \otimes \cO(-1) \subset \Vs \,. 
$$

Now let
\begin{equation}
(x_4,x_3,x_2,x_1)\in X^\sigma \times X \times X \times X^\sigma 
\label{x4321}
\end{equation}
be a quadruple of points in the definition of the convolution \eqref{conv_Ssig}. 
Since $\pi(x_2),\pi(x_3) \in V_{\le 0}$ and $\cO(-1)$ has no sections, we must have 
$$
\pi(x_2),\pi(x_3) \in V_0 \,. 
$$
Moreover, by Proposition \ref{p_leafX0}, we have 
$$
\pi(x_4)=\pi(x_3)\,, \quad  \pi(x_2)=\pi(x_1) \,.
$$
Since a section of $\cO(i)$ is fixed by evaluation at $i+1$ points, 
we conclude 
$$
\pi(x_4)=\pi(x_3)=\pi(x_2)=\pi(x_1)
$$
and any section of $\Xt$ connecting $x_3$ and $x_2$ maps to the 
corresponding constant section of $\Vs$. Since the fibers of this 
projection are proper, the proposition follows. 
\end{proof}

\section{Computation of $\bS(\sigma)$}\label{s_S_R}

\subsection{}

Let
$$
R_{\sigma,\infty} \in \End H^\hd_{\bGt}(X^\sigma_\infty) \otimes 
\Q(\Lie \bGt)  
$$
be the $\bGt$-equivariant $R$-matrix for the fiber $X_\infty$. 
Since $\C^\times_\varepsilon$ acts on $X_\infty$ via 
the cocharacter $\sigma^{-1}$, 
$R_{\sigma,\infty}$ is obtained from the $\bG^\sigma$-equivariant 
$R$-matrix by the substitution 
$$
\xi \mapsto \xi - \varepsilon \sigma(1) \,, \quad \xi \in \Lie \bG^\sigma
\,.
$$

\begin{Theorem}\label{T_int_R} 
For $X$ and $\sigma$ as in Section \ref{s_assume},
and $\zeta$ as in Section \ref{dis_zeta}, 
\begin{equation} \label{S R} 
  \Stab_+^{-1} \, \bS(\sigma) \, \Stab_+ = (-1)^{(\zeta,\bkap_X)} \, q^\zeta \, R_{\sigma,\infty} \,. 
\end{equation}
\end{Theorem}

\begin{proof}
Using  \eqref{defRsigma} and Lemma \ref{l_parity}
we may restate this as commutativity of the following 
diagram 
$$
\xymatrix{
H^\hd_{\bGt}(X^\sigma_\infty) \ar[d]_{\Stab_-} 
\ar[rrr]^{(-1)^{\codim/2} \, q^\zeta }
&&& H^\hd_{\bGt}(X^\sigma_0) \\
H^\hd_{\bGt}(X_\infty)
\ar[rrr]^{\bS(\sigma)} &&& H^\hd_{\bGt}(X_0) \ar[u]_{\,\,\Stab^\tau_-}
}
$$
%
%
where $\codim/2$ denotes the locally constant function on $X^\sigma$ given by  taking half the codimension in $X$. 
By Proposition \ref{p_proper}, we may 
compute the composition 
$$
\Stab_-^\tau \circ \, \bS(\sigma) \circ \Stab_-
$$
with any choice of 
equivariant parameters. 

We choose $\hbar=0$ and 
$\varepsilon\to\infty$, where 
$\varepsilon$ is the equivariant parameter for the $\C^\times_\varepsilon$-action 
in  \eqref{defbGt}. In particular, since $\hbar=0$, 
stable envelopes are diagonal and we must have
$$
x_4=x_3,\quad x_2=x_1\,,
$$
in \eqref{x4321} above. Also $\hbar=0$ implies 
$$
\Psi_0 = \Psi_\infty =1 
$$
by \eqref{Psihbar} above. 

Next consider the operator $\Psi_1$. It counts constant 
sections $C_1$ of $\Xt$ corresponding to 
$$
x_4=x_3=x_2=x_1 \in X^\sigma \,. 
$$
By Lemma 
\ref{l_smack}, the normal bundle to $C_1$ is 
$$
\cN = N_1(1) \oplus N_{-1}(-1) \oplus T_{x_1}X^\sigma  
$$
where $N_{\pm 1}$ are $\sigma$-eigenspaces in the normal bundle
$N$ to $X^\sigma$ in $X$ and twists are by $\cO(i)$, $i=\pm 1$. 

It follows that $C_1\in  M^{\sim}(\zeta)$ is unobstructed with tangent space
$$
T_{C_1} M^{\sim}(\zeta) = \left(N_1\right)_0 \oplus \left(N_1\right)_\infty 
\oplus T_{x_1}X^\sigma  
$$
where the subscripts $0,\infty\in \Pp^1$ denote the fibers of 
$N_+(1)$ over the respective points. We observe that these 
correspond precisely to the normal directions to $\Stab_-(x_i)$. 

In the end, all contributions to the integral cancel except for 
the sign in the definition of adjoint $\Stab_-^\tau$. This sign 
gives $(-1)^{\codim/2}$. 
\end{proof}

\subsection{}

In particular, for $\sigma$ as in Section \ref{dis_sigma}, we have
$$
X^\sigma = \coprod \cM(\bv',\bw') \times \cM(\bv'',\bw'')
$$
and 
$$
q^{\overline{\zeta}} \mapsto q^{\bv'} = q^\bv \otimes 1 \,,
$$
after restricting to functions on $\fh$
as in \eqref{zeta_bv}. Our computation of $\bS(\sigma)$ 
together with the results of Section \ref{bS_qm} imply 
the following 

\begin{Corollary}\label{qR_qm}
For tensor products of Nakajima varieties, the operator 
$$
\left(q^\bv \otimes 1\right) R(u) 
$$
belongs to the algebra of modified operators of quantum 
multiplication. 
\end{Corollary}

\chapter{Quantum multiplication by divisors}

\section{Classical multiplication by divisors}\label{s_class_div}

\subsection{}\label{s_cl_div} 

A vector $\lambda \in \fh = \C^I$ corresponds to a divisor 
\begin{equation}
c_1(\lambda) = \sum \lambda^i \, c_1(\cV_i) 
\label{def_lam}
\end{equation}
which we identify with the corresponding cup 
product operator.  Using 
\begin{equation}
\left[e_\alpha,e_{-\alpha}\right] = h_\alpha \label{eaa}
\end{equation}
we obtain from \eqref{R_theta}
\begin{equation}
c_1(\bw) = \bE\left(\vacv{\bw} \! \vacd{\bw}  u^2\right) +
\hbar \sum_{\substack{\alpha>0\\
\theta\cdot\alpha<0}} \alpha(\bw) \, e_{\alpha} e_{-\alpha} + \dots 
\label{c1div}
\end{equation}
where the sum over roots is with multiplicity, 
$$
\alpha(\bw)  = \bw(h_\alpha) = \alpha\cdot \bw\,,
$$
and dots stand for a quadratic polynomial in $\bv$, that is, 
an element of $\cU(\fh)$ of degree at most two.

\subsection{}

Since $\bE\left(\vacv{\bw} \! \vacd{\bw}  u^2\right)$ 
comes from the $1/u^2$ coefficient of the 
$R$-matrix, its coproduct will involve itself and the 
$1/u$ coefficient, that is, the classical $\br$-matrix. 
{}From Theorem \ref{t_Lie}, we compute
\begin{multline}
  \Delta  \, \bE\left(\vacv{\bw} \! \vacd{\bw}  u^2\right) 
= \bE\left(\vacv{\bw} \! \vacd{\bw}  u^2\right) \otimes 1 + 1 \otimes 
\bE\left(\vacv{\bw} \! \vacd{\bw}  u^2\right) + \dots \\
+ \hbar \sum_{\alpha,\beta} \vacd{\bw} e_{-\beta} e_{\alpha} \vacv{\bw} 
\, e_{-\alpha} \otimes e_{\beta} \,, 
\end{multline}
where dots stand for terms in $\cU(\fh)^{\otimes 2}$. Using
\eqref{eaa}, we compute
$$
\vacd{\bw} e_{-\beta} e_{\alpha} \vacv{\bw} = 
\begin{cases}
- \alpha\cdot \bw \,, & \beta = \alpha > 0 \,,\\
0 \,, & \textup{otherwise}  \,. 
\end{cases}
$$
We deduce the following 

\begin{Theorem}\label{t_cl_divisor} 
We have 
$$
\Delta c_1(\lambda) = c_1(\lambda) \otimes 1 + 1 \otimes c_1(\lambda) 
- \hbar \sum_{\theta\cdot \alpha>0} 
\alpha(\lambda) \, 
 e_{-\alpha} \otimes e_{\alpha} + \dots 
$$
where 
the sum is over roots $\alpha$ of $\fg_Q$ with 
multiplicities and dots stand for terms in $\cU(\fh)^{\otimes 2}$. 
\end{Theorem}

\subsection{}

In particular, we have
\begin{align}
R(u)\, \Delta c_1(\lambda) \, R(u)^{-1} &= 
\Delta^{\textup{op}} c_1(\lambda) \notag  \\ 
&= \Delta \, c_1(\lambda) + \hbar 
\sum_{\theta\cdot\alpha>0} \alpha(\lambda) 
\left(e_{-\alpha} \otimes e_{\alpha}- 
e_{\alpha} \otimes e_{-\alpha}\right) \,. 
\label{commRc1}
\end{align}
%

\section{Quantum operators}

\subsection{}

We denote by $\bQ(\lambda)$ the operator of modified quantum 
multiplication by the divisor \eqref{def_lam}.
By construction 
$$
\bQ(\lambda) = c_1(\lambda) + \hbar 
\sum_{\beta} (-1)^{(\beta,\kappa)} \, q^\beta \, 
\lambda(\beta)\, \bQ_{2,\red}(\beta) \,,
$$
where $\beta\in H_2(\cM,\Z)$ is 
an effective curve class and 
$\bQ_{2,\red}(\beta)$ is the image of the corresponding 
reduced virtual class under the evaluation map. The quantum 
part of $\bQ(\lambda)$ is a linear combination of Steinberg 
correspondences. 

\subsection{}

\begin{Theorem}\label{t_Qlam}
We have 
$$
\bQ(\lambda) = c_1(\lambda) - 
\hbar \sum_{\theta\cdot 
\alpha>0} \alpha(\lambda) \, \frac{q^\alpha}{1-q^\alpha} \, 
\, e_{\alpha} e_{-\alpha} + \dots 
$$
where the sum is over roots of $\fg_Q$ with 
multiplicity and dots denote a scalar operator.  
\end{Theorem}

The scalar operator is fixed by the requirement that 
the purely quantum part of $\bQ(\lambda)$ annihilates 
the identity.

\begin{proof}
For brevity, we write $\bQ=\bQ(\lambda)$. 

Let $\Delta \bQ$ be the pullback of the 
operator $\bQ$ under the stable envelope map 
$$
H^\hd_\bG(\cM(\bw')) \otimes H^\hd_\bG(\cM(\bw''))
\to H^\hd_\bG(\cM(\bw'+\bw'')) \,. 
$$
We can decompose it 
$$
\Delta \bQ = \sum_\alpha \Delta_\alpha \bQ
$$
according to the 
weights of $\fh\otimes 1$. Here 
$$
\left[h\otimes 1,\Delta_\alpha \bQ\right] = 
\alpha(h) \, \Delta_\alpha \bQ \,. 
$$
In other words, $\Delta_\alpha \bQ$ increases $\bv'$ by $\alpha$. 
By Proposition \ref{red_add} and Theorem \ref{t_cl_divisor}, 
we have 
$$
\Delta_0 \bQ = \bQ \otimes 1 + 
1 \otimes \bQ \,. 
$$
By Corollary \ref{qR_qm}, 
$$
\left[ \left(q^\bv \otimes 1\right) R(u), \Delta \bQ\right] = 0 \,,
$$
which means 
$$
R(u) \, \Delta \bQ \, R(u)^{-1} = 
\sum q^{-\alpha} \, \Delta_\alpha \bQ  \,.
$$
The purely quantum part in $\bQ$ is a Steinberg correspondence, 
hence commutes with $R$-matrices. Taking into 
account the classical part, we get from \eqref{commRc1}
$$
\sum_{\alpha} (1-q^{-\alpha}) \,  \Delta_\alpha \bQ = 
\hbar  \sum_{\theta\cdot \alpha>0} \alpha(\lambda) 
\left(e_{\alpha} \otimes e_{-\alpha}- 
e_{-\alpha} \otimes e_{\alpha}\right) \,, 
$$
which uniquely determines all $\Delta_\alpha \bQ$ with $\alpha\ne 0$. 

Now consider 
$$
\bQ_\textup{remainder} = 
\sum_{\beta} (-1)^{(\beta,\kappa)} \, q^\beta \, \lambda(\beta)\,
\bQ_{2,\red}(\beta)
+ \sum_{\theta\cdot\alpha>0} \alpha(\lambda) \, \frac{q^\alpha}{1-q^\alpha} \, 
\, e_{\alpha} e_{-\alpha} \,.
$$

 By Lemma \ref{quadraticsteinberg}, this is a Steinberg correspondence.  Moreover, it commutes with 
$\fh$ and is primitive in the sense that 
$$
\Delta \bQ_\textup{remainder} = 
 \bQ_\textup{remainder} \otimes 1 + 1 \otimes  \bQ_\textup{remainder} \,. 
$$
The following Proposition finishes the proof. 
\end{proof}

\subsection{}

\begin{Proposition}\label{commwithh}
Let $\Theta$ be a family of Steinberg correspondences
$$
\Theta_{\bv,\bw} \subset \cM_{\theta,0}^{\times 2} 
$$
defined for all $\bv,\bw$. If it is primitive
$$
\Delta \Theta = \Theta \otimes 1 + 1 \otimes \Theta
$$
and commutes with $\fh$ then $\Theta \in \fhb_Q$. 
\end{Proposition}

\noindent 
Recall that $\fhb_Q$ acts by multiplication by linear
function of $\bv$ and $\bw$.   Again, by $\Delta$ in the above proposition,
we mean the pullback of $\Theta$ under the stable envelope map.

\begin{proof}
By hypothesis,  $\Theta$ preserves
the decomposition
$$
H^\hd_\bG(\cM(\bw)) = \bigoplus_{\bv} H^\hd_\bG(\cM(\bw,\bv))
$$
into $\fhb_Q$-weight subspaces. In particular, 
\begin{equation}
  \left[\Theta, \vacv{\bw}\!\vacd{\bw}\right] =0 \label{xivac}
\end{equation}
where $\vacv{\bw}\!\vacd{\bw}$ is the projector onto the $\bv=0$ part. 

Since $\Theta$ is a Steinberg correspondence on each component, 
$$
[R(u),\Delta \Theta] = 0 
$$
This and \eqref{xivac} implies 
$$
\left[\tr_{1} 
\big((\vacv{\bw}\!\vacd{\bw} \otimes 1)\circ R(u)\big) , \Theta\right] =0
$$
where the trace is over the first tensor factor and 
$\Theta$ acts in the second tensor factor. 
By the 
results of Section \ref{s_vac}, this means that $\Theta_{\bv,\bw}$ commutes
with operators of classical multiplication by all characteristic
classes of the tautological bundles.  Proposition \ref{p_span_H}
 implies that $\Theta_{\bv,\bw}$ is itself an operator of classical multiplication.
Since it has cohomological degree $0$, it can only be a multiple 
of the identity. The primitivity condition forces this multiple 
to be a linear function of $\bv$ and $\bw$. 
\end{proof}

\chapter{Cotangent bundles of Grassmannians} \label{s_Grassmann}

In this chapter, we illustrate the general theory for the 
simplest possible quiver $Q$ --- that with 
one vertex and no arrows. The corresponding 
Nakajima varieties are cotangent bundles of Grassmann
varieties.  

Grassmann varieties are among the oldest objects of study 
in algebraic geometry; in particular, 
their quantum cohomology has been described by many authors
from many different angles, see e.g.\ \cite{Bertram,Buch,BKT,GRTV,Kim,Mihalcea,
Rietsch,SiebertTian}. The modest 
goal of this chapter is to help the reader align his 
favorite point of view on Grassmannians with the direction 
of this paper.

\section{Quantum cohomology}

\subsection{Setup}

For the quiver $Q$ with one vertex and no arrows, the Nakajima 
quiver data is a pair of matrices
$$
\xymatrix{
\C^n \ar@/^/[r]^A & \C^k \ar@/^/[l]^B} 
$$
where $\C^n = \C^{\bw_1}$ is the framing space and 
$k = \bv_1$. Let $X$ be the corresponding quiver variety 
$$
X = \cM_{\theta,0}(k,n) = \{(A,B), AB=0\}/\!\!/_\theta \, GL(k) \,,
$$
where 
$$
\textup{stable points} = 
\begin{cases}
\rk A = k \,, \quad &\theta > 0 \,, \\
\rk B = k \,, \quad & \theta <0 \,. 
\end{cases}
$$
In either case, $X=\varnothing$ if $k> n$. The map
$$
(A,B) \mapsto L = \begin{cases}
\Ker A\,, \quad &\theta > 0 \,, \\
\Img  B\,, \quad & \theta <0 \,, 
\end{cases}
$$
makes $X$ a vector bundle over the Grassmannian  
$$
\Grl = \begin{cases}
Gr(n-k,n)\,, \quad &\theta > 0 \,, \\
Gr(k,n)\,, \quad & \theta <0 
\end{cases}
$$
of possible $L\subset \C^n$. The fiber of this vector 
bundle is $\Hom(\C^n/L,L)$, whence 
$$
X = T^* \Grl \,. 
$$
Of course, Grassmann varieties of complementary dimension are 
isomorphic, but this isomorphism is not canonical, in 
particular not $GL(n)$-equivariant. Here we are 
interested in $\bG$-equivariant quantum cohomology 
of $X$, where 
$$
\bG = GL(n) \times \C^\times_\hbar\,.
$$
The second factor in $\bG$ scales the cotangent directions with 
weight $-\hbar$. 

\subsection{Divisors} 

The tautological bundle $\cV=\cV_1$ is identified as follows 
$$
\cV = \begin{cases}
\C^n/L\,, \quad &\theta > 0 \,, \\
L\,, \quad & \theta <0 \,, 
\end{cases}
$$
that is, $\cV$ is the universal quotient bundle for $\theta>0$ and 
the universal subbundle for $\theta<0$. The line bundle 
$$
\cO(1) = \left(\Lambda^\textup{top} \cV\right)^{\sgn \theta}
$$ 
is the very ample generator of $\Pic X$. The corresponding 
projective embedding of the Grassmannian is classically known 
as the Pl\"ucker embedding. 

It is elementary to see that $c_1(\cO(1))$ generates $H^\hd_\bG(X)$. 
Therefore quantum multiplication by this class uniquely 
determines the algebra of quantum multiplication. 

\subsection{The affine quotient}

Let 
$$
\pi: X \to X_0
$$
be the affinization of $X$. Its target $X_0$ may be described 
in terms of square-zero matrices $D$, or differentials. Let 
$$
\fD=  \left\{D\,\big|\, D^2=0 \right\} \subset \End \C^n \,. 
$$ 
denote the set of square-zero matrices. 
It is stratified by $GL(n)$-orbits
\begin{equation}
\fD_{r} = \{\rk D = r \} \,, \quad  r=0,1,2,\dots,
\left\lfloor \frac n2  
\right\rfloor \,. 
\label{D_r}
\end{equation}
The map 
$$
(A,B) \mapsto D = BA 
$$
gives 
$$
X_0 \cong 
\fD_{\le r}\,, \quad r  = \min(k,n-k) \,. 
$$
The fibers of $\pi$ 
are Grassmann varieties, namely 
$$
\pi^{-1}(D) \cong \left\{L\,\big|\, \Img D \subset L \subset \Ker D 
\right\}\,. 
$$
In particular, $\pi^{-1}(0) = \Grl$.

\subsection{The Steinberg variety}

By definition, the Steinberg variety is 
$$
\fS = X \times_{X_0} X \,. 
$$
The stratification \eqref{D_r} gives a decomposition 
into irreducible components
$$
\fS = \bigcup_d \fS_d\,,
$$
where $\fS_d$ is the closure of $X \times_{\fD_{r-d}} X$. 
In particular, 
\begin{align*}
\fS_0 & = \textup{diagonal}\,, \\
\fS_r & = \Grl \times \Grl \,. 
\end{align*}
For us, the most important stratum is $\fS_1$. 

%

\subsection{Lines on $X$}

Let $\ell \in H_2(\Grl,\Z)$ be the effective generator. Curves
of class $\ell$ are lines in the Pl\"ucker embedding. 
Two 
points $L_1\ne L_2\in \Gr$ lie on a line $\ell_{L_1,L_2}$ if 
$$
\dim L_1 \cap L_2 = \dim L_1 - 1\,,
$$
in which case  
$$
\ell_{L_1,L_2} = \left\{L\,\big|\, L_1 \cap L_2 \subset L \subset 
L_1 + L_2\right\}\,. 
$$
Lines on $X$ are 
the lines in the fibers of $\pi$. Therefore $\fS_1$ 
is formed by pairs of points that lie on a line.

\subsection{Torus-fixed curves}\label{s_Gr_A}

Let $\bA\subset GL(n)$ be the diagonal torus. Since 
$X_0^\bA = \{0\}$, we have 
$$
X^\bA = \Grl^\bA\,. 
$$ 
This is a finite set formed by coordinate subspaces
$$
L_S = \bigoplus_{s\in S} \C e_{s} 
$$
where $\{e_1,\dots,e_n\}\in \C^n$ is the coordinate basis 
and $S\subset\{1,\dots,n\}$ ranges over subsets of 
cardinality $\dim L$. 

The set of reduced irreducible $\bA$-invariant
curves in $X$ is also finite, formed by lines $\ell_{S,S'}$ joining
fixed points $L_{S}$ and $L_{S'}$ with $|S \triangle  S'|=2$. 
Their tangent $\bA$-weights have the form 
$$
\pm(a_i - a_j)\,, \quad \{i,j\} = S \triangle  S' \,, 
$$
from which one concludes the following 

\begin{Lemma}\label{l_line}
The only unbroken $\bA$-fixed chains in $X$ are covers of lines
branched over fixed points. 
\end{Lemma}

\subsection{Quantum product by divisor}

For 
$d=1,2,\dots$ let 
$$
\bQ_d \subset H_\textup{middle}(X \times X)
$$
be the following Steinberg correspondence 
\begin{equation}
\bQ_d= d \, (-1)^{nd} \, 
\ev_* \left[\Mb_{0,2}(X,d\ell)\right]_\textup{virtual,reduced} \,. 
\label{pushvir}
\end{equation}
The sign $(-1)^{nd}$ is taken from the definition of modified 
Gromov-Witten invariants, that is, it comes from 
pairing $d\ell$ with 
$$
\kappa_X = c_1(\Grl) = n\, \cO(1) \,. 
$$
The factor of $d$ is introduced in \eqref{pushvir} so that 
\begin{equation}
\bQ_\textup{quantum}  = \hbar \, \sum_{d>0} q^{d\ell} \, \bQ_d
\label{hQd}
\end{equation}
is the modified purely quantum multiplication by $\cO(1)$.

\begin{Proposition}\label{p_Qd} For all $d>0$ we have  
  \begin{equation}
\bQ_d = \bQ_1  =  \pm \fS_1 \,. 
\label{QdQ1}
\end{equation}
\end{Proposition}

\begin{proof}
As a first step, we compute the push-forward \eqref{pushvir}
modulo terms supported on the diagonal. We do this 
by $\bA$-equivariant localization. 

Recall that 
only unbroken maps contribute to localization of 
reduced virtual classes. Suppose the marked points
of an unbroken map $f$ evaluate to distinct points of $X^\bA$. 
Then by Lemma \ref{l_line} $f$ has the form 
$$
f: \Pp^1 \xrightarrow{\quad z\mapsto z^d \quad} \ell_{S,S'} \subset \Grl\,,
$$
ramified over $L_S,L_{S'} \in \Grl^\bA$. 
In particular
$$
\Aut f = \Z / d \,,
$$
and hence $f$ contributes 
$$
-(-1)^{nd} \Euler' H^\hd (f^* TX)^{-1} \in \Q(\fa) \,, 
$$
to localization of $\bQ_d$. 
Here $\Euler'$ is the product of nonzero $\bA$-weights 
in the virtual $\bA$-module $H^\hd (f^* TX)$. 

To be precise, there 
are two zero weights in this module. One occurs in 
$H^0(f^* T\ell_{S,S'})$ and is taken out by the automorphism 
of a $2$-pointed $\Pp^1$. The other occurs in 
$H^1(f^* T^*\ell_{S,S'})$ and is taken out by passing 
to the reduced invariants. The minus sign appears 
because $T^*\ell_{S,S'}$ has weight $-\hbar$ under
the $\C^\times_\hbar$-action while 
in \eqref{hQd} we take out a factor of $\hbar$. 

Since 
$$
f^*TX = \cT \oplus \cT^*\,, \quad \cT = f^* T \Grl \,, 
$$
Lemma \ref{l_TT*} below shows 
\begin{equation}
  \label{locQd}
 \frac{\bQ_d\big|_p}{\Euler T_p \, X \times X} \, \, 
  = 
\frac{(-1)^{\dim \Grl}}
{\Euler T_p \, \Grl \times \Grl}
\end{equation}
for any off-diagonal $p\in \fS_1^\bA$, that is, for any 
$$
p = (L_{S},L_{S'}) \,, \quad |S \triangle S'| = 2 \,. 
$$
This proves \eqref{QdQ1} modulo 
a class supported on the diagonal. 

To show the contribution 
of the diagonal vanishes, it suffices to note that $\bQ_1$ 
annihilates the identity in cohomology and so does 
$\fS_1$. Indeed, the fibers of 
$\pi$ are positive-dimensional over $\fD_{r-1}$. 
\end{proof}

\begin{Lemma}\label{l_TT*} 
Let $\bA$ be a torus and let $\cT$ be an $\bA$-equivariant 
bundle on $\Pp^1$ 
without zero weights in the fibers $\cT_0,\cT_\infty$ over 
fixed points. Then 
$$
\Euler' H^\hd(\cT \oplus \cT^*)  = (-1)^{\deg \cT + \rk \cT+ \#z} 
\Euler \left(\cT_0 \oplus \cT_1\right)\,,
$$
where $\#z= \dim H^1(\cT \oplus \cT^*)^\bA$. 
\end{Lemma}

\noindent 
The sign in \eqref{QdQ1} is easily determined from \eqref{locQd}, but 
we will not need it in what follows. 

\section{The stable basis} 

\subsection{Tensor product structure} 

As usual, we define
$$
\cM(n) = \bigsqcup_{k} \cM_{\theta,0}(k,n) \,.
$$
The $\bA$-action makes $\cM(n)$ a tensor product 
$$
\cM(n) = \cM(1)^{\otimes n}\,, \quad \cM(1) = \textup{$2$ points} \,. 
$$
We write 
$$
H^\hd(\cM(1),\Q) = \Q \vacv{0} \oplus \Q \vacv{1} = \Q^2\,,
$$
where
$$
\bv  \vacv{k} = k \vacv{k} \,.
$$
Similarly, 
$$
H^\hd\left(\cM(n)^\bA\right) = \left(\Q^2\right)^{\otimes n} = 
\bigoplus_{S\subset \{1,\dots,n\}} 
\Q \vacv{S} 
$$
where we identify
$$
\textup{subsets of $\{1,\dots,n\}$}  
\leftrightarrow \{0,1\}^n 
$$
using indicator functions. In $\bG$-equivariant cohomology, 
we replace $\Q$ above with the equivariant
cohomology ring of a point.


\subsection{Polarization}\label{s_TG_pol}

Recall from Example \ref{s_pol_Nak} that we have a canonical 
choice for polarization of any Nakajima variety. In the case
at hand, this gives
\begin{align}
\Stab_\fC \vacv{S} \big|_{L_S} &= \Euler 
\Hom(\cV,\C^n  \ominus \cV) \notag 
\\
&= (\pm 1)^{\dim \Grl} 
 \Euler T_{L_S} \Grl \,,
\label{polTG}
 \end{align}
depending on the sign of $\theta$. 
Here the Euler class is the product of $\bA$-weights. 

Note the two possibilities in \eqref{polTG} differ by an 
overall scalar, which means that all geometric operators
act canonically in the stable basis. 

\subsection{Classical $\br$-matrix} 

We claim 
$$
\fg_Q = \gl(2)
$$
with its natural action on $\Q^2$ and, by tensor product, on 
$H^\hd\left(\cM(n)^\bA\right)$. Indeed, the classical 
$\br$-matrix is computed as follows in terms of the 
matrix units $e_{ij} \in \gl(2)$. 

\begin{Proposition} 
\begin{equation*}
  \br = e_{00} \otimes e_{11} + 
e_{11} \otimes e_{00} - e_{01} \otimes e_{10} 
- e_{10} \otimes e_{01} \,. 
\end{equation*}
\end{Proposition}

\begin{proof}
For $\cM(1) \otimes \cM(1)$ this was computed in 
Section \ref{s_R_Yang}. In general, it follows
by additivity of the classical $\br$-matrix. 
\end{proof}

\noindent 
Other ways to write the $\br$-matrix include 
\begin{align*}
  \br &= \bw \otimes \bw - \sum_{ij} e_{ij} \otimes e_{ji} \\
&=
 - e \otimes f - f\otimes e + \dots  \,, 
\end{align*}
where
$$
e = e_{10} \,, \quad f = e_{01}\,, \quad \bw = e_{00} + e_{11} \,,
$$
and dots stand for a diagonal operator.

\subsection{Quantum multiplication in stable basis}

Recall  the operators $\bQ_d$ are Steinberg correspondences. 
Therefore, 
by Theorem \ref{t_Stein}, 
their action in the stable basis does not  
depend on the choice of a chamber $\fC$ for $\bA$.

The following Proposition gives a direct verification of 
Theorem \ref{t_Qlam} for cotangent bundles of Grassmannians. 

\begin{Proposition} \label{p_Q_Gr} 
We have 
$$
\bQ_\textup{quantum} = \hbar \, \frac{q^\ell}{1-q^{\ell}} \, ef + \dots 
$$
where dots stand for a diagonal operator. 
\end{Proposition}

\begin{proof}
By Proposition \ref{p_Qd}, the statement to prove is 
$$
\bQ_{1} = ef + \dots \,. 
$$
Since $\dim \cM^\bA =0$, theorem \ref{adjoint} gives  
\begin{equation}
\bQ_{1,\bA} \vacv{S}
= \sum_{S'} (-1)^{\dim \Grl} 
\big(\Stab_{-\fC} \vacv{S'} \otimes \Stab_\fC \vacv{S}, 
\bQ_1 \big) \, \vacv{S'} 
\label{bQ1S}
\end{equation}%
The coefficient in \eqref{bQ1S} may be computed using 
\eqref{polTG} and \eqref{locQd} and recall that 
we can set $\hbar=0$ in this computation, which 
makes stable envelopes diagonal.  For either 
sign of $\theta$, this gives 
$$
\vacd{S'} \bQ_{1,\bA} \vacv{S} = 
\begin{cases}
0\,, \quad & |S\triangle S'|>2\,,\\
1\,, \quad & |S\triangle S'|=2\,,\\
\end{cases}
$$
proving the proposition. 
\end{proof}

\subsection{}

It is an interesting combinatorial and geometric question 
to compute the transition matrix between the stable basis 
and the fixed-point basis in $H^\hd_\bG(X)$. 

In the quantum integrable system language, the fixed-point 
basis corresponds to the eigenbasis at $q=0$, while the 
stable basis is the coordinate basis, that is, the spin 
basis of the spin chain for $X=T^*\Grl$. Thus, the question 
is equivalent to explicit diagonalization of the 
Hamiltonian at $q=0$. 

For the inhomogeneous XXX spin chain, the answer was known to 
Nekrasov and Shatashvili.  The corresponding symmetric 
functions are rational analogs of the interpolation Schur
functions. Just like Schur functions may be deformed 
to Macdonald polynomials associated to root systems of 
type $A$ and, more generally, to nonreduced 
$BC$ root systems, these rational interpolation Schur functions 
naturally lie in the family of special functions studied 
by E.~Rains in \cite{Rains}. 

In \cite{Daniel}, D.~Shenfeld shows how this identification 
is a example of the general abelianization procedure 
for stable bases.

\section{Yangian action}

\subsection{The Yangian of $\gl(2)$} 

Yangians of finite-dimensional  Lie algebras have been studied in 
great detail, see for example the exposition in 
\cite{ChariPr,ES,Molev,MNO}. 
We recall $\bY(\gl(2))$ is generated by 
countably many generators, the coefficients $\bE_{ij}^{(k)}$ 
in the generating series 
$$
\bE_{ij}(u) = \delta_{ij} + \sum_{k>0} \frac{\bE_{ij}^{(k)}}{u^k} \,,
\quad i,j\in \{1,2\} \,, 
$$
subject to the RTT=TTR relations. These relations are
written in terms of 
the matrix 
$$
\bE(u) =
\begin{pmatrix}
\bE_{11}(u) & \bE_{12}(u) \\  
\bE_{21}(u) & \bE_{22}(u) 
\end{pmatrix} \in \End \Q^2 \otimes \bY(\gl(2))[[u^{-1}]] 
$$
and have the form
\begin{equation}
R(u-v) \, \bE(u) \, \bE(v) = \bE(v) \, \bE(u) \, R(u-v)
\label{RTTgl2}  \,.
\end{equation}
The equality in \eqref{RTTgl2} is an equality in 
$$
\textup{\eqref{RTTgl2}}  
\in \End(\Q^2 \otimes \Q^2) \otimes \bY(\gl(2)) [[u^{-1},v^{-1}]] \,.
$$
The $R$-matrix in \eqref{RTTgl2} is 
$$
R(u) = \Big(1 - \dfrac{\mathbf{s}}{u}\Big) 
\Big/
\Big(1 - \dfrac{1}{u}\Big)
\in \End(\Q^2 \otimes \Q^2)[[u^{-1}]] \,,
$$
where $\mathbf{s}$ is the permutation of tensor factors. 
The scalar factor, which plays no role in \eqref{RTTgl2}, 
is chosen here so that $R(u)$ equals the $R$-matrix 
for $\cM(1) \otimes \cM(1)$ computed in \eqref{R_Yang}
for $\hbar=1$\,.

\subsection{Evaluation representation}

Consider the map $\bY(\gl(2)) \to \End \Q^2$ given by 
\begin{equation}
\bE_{ij}(u) \mapsto \Big(\delta_{ij} - 
\frac{e_{ji}}{u}\Big) \Big/
\Big(1 - \dfrac{1}{u}\Big) \,. 
\label{evrY}
\end{equation}
This takes $\bE(u)$ to $R(u)$ and is indeed a representation 
of $\bY(\gl(2))$ by the Yang-Baxter equation. We denote 
by $\Q^2(a)$ this representation precomposed with the 
translation automorphism of the Yangian. 
It is well-known, and can be seen as in Section \ref{s_pr_tgrY}, 
that 
\begin{equation}
\bigcap_{n=1}^\infty \, 
 \Ker \Q^2(a_1) \otimes \dots \otimes \Q^2(a_n) = 0  \,. 
\label{capker}
\end{equation}
Traditionally, a different representation of the Yangian, namely
$$
\bE_{ij}(u) \mapsto \delta_{ij} +
\frac{e_{ij}}{u} \,. 
$$
is called the evaluation representation. The two are 
related by a composition of automorphisms
$$
\bE(u) \mapsto \bE(-u)^T, \quad \bE(u) \mapsto f(u) \bE(u) 
$$
of $\bY(\gl(2))$, where the superscript $T$ denotes transposition
and $f(u)=1+O(u^{-1}) \in \Q[[u^{-1}]]$ is an arbitrary scalar
factor.

\subsection{Comparison of Yangians}

Let $\bY_Q$ denote the Yangian constructed in 
Chapter \ref{s_Yang}. 
This is an algebra over 
$
\bk=\Q[\hbar] \,. 
$
The Yangian $\bY_Q$ 
is graded by cohomological degree and $\hbar$ has cohomological 
degree $2$. Therefore, $\bY_Q$ is uniquely reconstructed, via the 
Rees algebra construction, from its specialization at $\hbar=1$, with
the induced filtration.  
We set $\hbar=1$ in what follows. 

\begin{Proposition}\label{p_YY2} 
$$
\bY(\gl(2)) \cong  \bY_Q 
$$
\end{Proposition}

\begin{proof}
Since the generators of $\bY_Q$ satisfy the RTT=TTR relation 
\eqref{eRTT}, we have a surjective
homomorphism $\bY(\gl(2)) \to \bY_Q$. 
Its injectivity follows from \eqref{capker}. 
\end{proof}

\subsection{The center of $\bY(\gl(2))$} 

For any Lie algebra $\fg$, we have
$$
\Center \cU(\fg[u]) = \cU(\Center(\fg)[u]) \,, 
$$
see e.g.\ Section 2.12 in \cite{MNO}. The center of 
$\cU(\gl(2)[u])$ deforms to the center $\bZ$ of $\bY(\gl(2))$, which 
is freely generated by the coefficients in the expansion 
$$
\qdet \bE(u) = 1 + \sum_{k>0} \qdet_k \, u^{-k} 
$$
of the quantum determinant 
$$
\qdet \bE(u) = \bE_{11}(u) \, \bE_{22}(u-1) - \bE_{21}(u) \, \bE_{12}(u-1) 
\,.
$$
The quantum determinant is group-like 
$$
\Delta \qdet \bE(u) = \qdet \bE(u) \otimes \qdet \bE(u)
$$
and
$$
\qdet \bE(u) \Big|_{\Q^2(a)} = \frac{u-a}{u-a-1} \,. 
$$
Whence the equality of ideals 
\begin{equation}
  \label{idealq}
  \Big(\qdet_k\Big)_{k>0} = \Big(\ch_k \cW \Big)_{k\ge 0} \subset \bZ \,. 
\end{equation}
\subsection{The core Yangian}

The nonequivariant Cartan matrix for $Q$ is $\bC = (2)$, which is 
invertible. Therefore 
$$
\bbY_Q \subset \bY_Q \,. 
$$
The classical $\br$ matrix for the core Yangian $\bbY_Q$ 
equals 
$$
\br - \bC^{-1} \, \bw \otimes \bw = - \tfrac12 \, h \otimes h - e\otimes f 
- f \otimes e \,,
$$
where 
$$
h = e_{11} - e_{00} \in \fsl(2)\,.
$$ 
This is the classical $\br$-matrix for $\fsl(2)$. This means 
the core Yangian $\bbY_Q$ is a filtered
deformation of $\cU(\fsl(2)[u])$.

\begin{Proposition}
$$
\bbY_Q \cong \bY(\fsl(2)) \,. 
$$
\end{Proposition}

\begin{proof}
Let $\fz\cong \gl(1)$ denote the center of $\gl(2)$. By 
deformation from 
$$
\cU(\gl(2)[u]) \cong \cU(\fsl(2)[u]) \otimes \cU(\fz[u]) 
$$
we get 
$$
\bbY_Q \otimes \bZ \cong \bY_Q \cong \bY(\gl(2)) 
\cong \bY(\fsl(2)) \otimes \bZ \,. 
$$
Taking the quotient by the ideal \eqref{idealq} 
gives the desired isomorphisms. 
\end{proof}

\subsection{}

Baxter subalgebras in $\bY(\gl(2))$ appeared in 
mathematical physics as quantum integrals of 
motion of the XXX spin chain with quasi-periodic
boundary conditions. 

Proposition \ref{p_Q_Gr} 
and Section \ref{s_Baxter} shows the operator 
$$
\bQ = c_1(\cO(1)) \cup \, + \, \bQ_\textup{quantum}
$$
of modified quantum multiplication by $c_1(\cO(1))$ lies 
in the Baxter subalgebra corresponding to 
$$
g = q^\bv \in GL(2) \,.
$$
Since the operator $c_1(\cO(1)) \cup$ in 
$H^\hd_\bG(X)$ has distinct eigenvalues, the algebra
of quantum multiplication equals the Baxter subalgebra in 
$\bY(\gl(2))$. This is 
one of the most basic examples in Nekrasov-Shatashvili 
theory.

\part{Instanton moduli}

\chapter{Classical $\br$-matrix and $\glh$. }
\label{s_g_fr} 

\section{Setup} 

\subsection{Moduli of framed sheaves} 

We now specialize our general discussion to the quiver $Q$ with 
one vertex and one loop. We take $\zeta=0$, $\theta>0$ and denote 
$$
(r,n) = (\bw_1,\bv_1) \,. 
$$
The corresponding Nakajima variety $\cM(r,n)$ is the moduli space of 
framed torsion-free sheaves $\cF$ with 
$$
\rk \cF = r \,, \quad c_2(\cF) = n \,,
$$
on $\Pp^2$, see \cite{NakBook}.  Framing means a choice 
of trivialization of $\cF$ along $\Pp^2\setminus \C^2$. 
It implies $c_1(\cF)=0$. As usual, we set 
$$
\cM(r) = \bigsqcup_n \cM(r,n)\,.
$$ 
In particular, 
$$
\cM(1) = \Hilb = \bigsqcup_n \Hilb_n
$$
is the Hilbert scheme of points of $\C^2$. 

Our goal in the rest of the paper is to make the general 
theory explicit in this very important special case. 

\subsection{Uhlenbeck space}

The affine variety 
$$
\cU(r,n) = \cM_{0,0}(r,n) 
$$
is the Uhlenbeck compactification of the moduli of framed instantons. 
The canonical map 
$$
\cM(r,n) \to \cU(r,n)
$$
takes a torsion free sheaf $\cF$ to the vector bundle $\cF^{\vee \vee}$ 
together with the support of $\cF^{\vee \vee}/\cF$, counting 
multiplicity.

\subsection{Group actions}\label{s_ADHM}

Concretely,
 $\Mrn$ is  the $GL(\C^n)$-quotient of the spaces of quadruples
$$
X_1,X_2: \C^n\to \C^n\,, \quad A:\C^r\to\C^n\,, \quad B:\C^n\to \C^r 
$$
satisfying the equation 
\begin{equation}
\label{ADHM}
[X_1,X_2] + A B = 0 
\end{equation} 
and stability condition: the image of $A$ must generate $\C^n$ 
under the action of $X_1$ and $X_2$.

The framing group 
$G_\bw = GL(r)$ acts by 
the automorphisms of $\C^r$ or by changing the framing in the 
sheaf description. The group $G_\edge = GL(2)$ acts by
$$
g\cdot
\begin{pmatrix}
 X_1 \\ X_2
 \end{pmatrix}
=
\begin{pmatrix}
 g_{11} \, X_1 + g_{12} \, X_2\\ 
 g_{21} \, X_1 + g_{22} \, X_2
 \end{pmatrix}\,, \quad g\cdot A =A\,, \quad g\cdot B = \det(g) \, B \,. 
$$
in the quiver description. In the sheaf description, it acts by 
automorphisms of $\Pp^2$ preserving $\C^2$. 

We fix a maximal torus  $\bA\subset GL(r)$ with 
$$
\fa = \Lie \bA  =\diag(a_1,\dots,a_r)  
$$
and take $\bG = \bA \times GL(2)$. Note that $\bA$ is central in $\bG$. 

\subsection{Fixed loci}

By Example \ref{ex1}, we have 
$$
\cM(r)^\bA = \cM(1)^{\times r}\,, \quad \cM(1) =  
\bigsqcup_n \Hilb_n \,. 
$$
In the sheaf description, $\cM(r)^\bA$ is the locus of direct sums
$$
\cM(r)^\bA  = \left\{ \bigoplus_{i=1}^r I_i \right\}
$$
of ideals  $I_i\subset\C[x_1,x_2]=\cO_{\C^2}$. 

\subsection{Polarization} 

Our general prescription for polarizations of Nakajima varieties
gives the following for instanton moduli.

Following Example \ref{e_pol_T*}, consider the $\C^\times$-action on 
$\C^2$ that scales one coordinate axis, say the $x_2$-axis. This 
scales $\omega$ with weight $-1$. One of the components $X^{\C^\times}$ is 
the following Quot-scheme 
$$
Q_n = 
\left\{\cF\,\big|\, x_2 \cO^{\oplus r} \subset \cF \subset \cO^{\oplus r}
\right\} \subset \Mrn \,.
$$
It is middle-dimensional. Since $\omega$ pairs $\C^\times$-weight 
spaces of total weight $1$, it is Lagrangian. In the quiver description, 
it is given by 
$$
X_2 = 0\,, \quad B = 0 \,,
$$
that is, by representations of one half of the quiver $\Qb$. 

As out polarization, we take weights that are normal to $Q_n$. 
Those are easily identified, giving 
\begin{equation}
\bsi = \prod_{i=1}^{r} \prod_{j\ne i} (a_j-a_i)^{n_i} 
\label{polarM}
\end{equation}
for the component 
$$
\Hilb_{n_1} \times \cdots \times \Hilb_{n_r} \subset \cM(r)^\bA \,.
$$

\subsection{$R$-matrices} 

Our general theory produces an $R$-matrix 
$$
\bR(u) \in \End\left(H^\hd_\ga(\Hilb)^{\otimes 2}\right) 
\otimes \Q(\gal) 
$$
which solves the Yang-Baxter equation with spectral parameter
$$
u = a_1 - a_ 2 \,. 
$$
Our goal now is to identify $\bR(u)$ and the 
corresponding Yangian.  We use the boldface letter
to denote this particular $R$-matrix. It will be 
characterized in terms of the Virasoro algebra in
Chapter \ref{s_fullR}.

\subsection{} 

As a first step, in Section \ref{s_classR} we show the 
corresponding classical $\br$-matrix is the $\br$-matrix
for $\glh$, modulo zero modes.
 The action of $\glh$ on the cohomology of Hilbert 
schemes was constructed by Nakajima \cite{NakHilb} and 
Grojnowski \cite{Groj}. Its extension 
to higher rank is due to Baranovsky \cite{Bar}.

\subsection{}

In principle, $\bR(u)$ may be computed from 
the $R$-matrix of $\bY(\mathfrak{gl}(\infty))$ using the factorization 
in Theorem \ref{t_Rprod}, see \cite{Smirnov}. In particular, 
the classical $\br$-matrix is very easy to determine 
in this approach. Here we take a different route to the 
same result.

\section{Baranovsky operators}\label{s_def_Bar} 

\subsection{}

We recall from \cite{Bar} the definition of Baranovsky operators $\beta_k$.  
For $k>0$, consider the locus 
\begin{equation}
 \fB \subset \Mr{n+k}  \times \C^2 \times \Mr{n}\label{defZ} 
\end{equation} 
of triples $(\cF',x,\cF)$ such that $\cF'\subset \cF$ and $\cF/\cF'$ is a length 
$k$ sheaf supported at $x$.
We have \cite{Bar} 
$$
\dim \fB = 2rn+rk+1\,,
$$
which is the middle dimension of the product. Note that $\fB$ is $\bG$-invariant.

\subsection{}

Next, $\fB$ is a Lagrangian 
Steinberg correspondence between the first factor and 
the other two, which can be seen as follows. We embed
$$
\C^2 \owns (c_1,c_2) \mapsto ((x_1-c_1)^k,x_2-c_2) \in \Hilb_k  \,.
$$
Note 
this pulls back the translation-invariant symplectic form. Consider the maps 
$$
\xymatrix{
\Hilb_k \times \Mrn  \,\ar@{^{(}->}[r] &\cM(r+1,n+k)\ar@{->}[d] &  
\ar@{_{(}->}[l]\, \Hilb_0 \times \Mr{n+k}  \\
& \cU(r+1,n+k) \,. 
}
$$
Here the horizontal arrows are formed by taking direct sums 
and the vertical is the canonical projection to the Uhlenbeck space. 
It is clear that points on the 
correspondence $\fB$ map to the same points of $\cU(r+1,n+k)$. 


\subsection{}

The correspondence $\fB$ defines a map
$$
\Theta_\fB: H^\hd_\bG(\C^2) \otimes H^\hd_\bG(\Mrn) \to H^\hd_\bG(\Mr{n+k}) \,.$$
We define the operators $\beta_{-k}$, $k>0$, as the matrix
elements of $\Theta_\fB$ with respect to the $\C^2$-factor, that is  
$$
\beta_{-k}(\gamma) \cdot \eta = \Theta_\fB(\gamma \otimes \eta) \,,
\quad \gamma\in H^\hd_\bG(\C^2) \,. 
$$

\subsection{}\label{s_sign_Baran}

For $k>0$, we define $\beta_{k}(\gamma)$ as the matrix 
elements of the adjoint operator
$$
\Theta_\fB^\tau
: H^\hd_\bG(\C^2) \otimes H^\hd_\bG(\Mr{n+k}) \to H^\hd_\bG(\Mr{n}) 
\otimes \bK\,,
$$
see Section \ref{s_signs_adj}. A larger base ring $\bK$ is required
because the adjoint correspondence is not proper and equivariant 
localization is needed to define it as an operator. We will see 
that 
$$
\bK = H^\hd_\bG(\pt) \left[\frac1{\det_{\C^2}}\right] 
$$
where 
$$
\textstyle{\det_{\C^2}} 
\begin{pmatrix}
t_1 \\ & t_2 
\end{pmatrix}
= t_1 t_2 \in \Q[\mathfrak{gl}(2)]
$$
is the determinant of the defining 
representation.

Also note that since 
we permute the source and target of $\Theta_\fB$, the operator 
$\Theta_\fB^\tau$ gets a sign, namely 
$$
(-1)^{rk} = (-1)^{\frac12 \dim \Mr{n+k} + \frac12 \dim \Mr{n}} \,. 
$$

\subsection{}

For $r=1$, Baranovsky operators specialize to 
the original Nakajima operators, up to 
normalization. We denote them by
$\alpha_k(\gamma)$.  
It is a theorem of Nakajima that these satisfy
\begin{equation}
\left[\alpha_k(\gamma_1),\alpha_l(\gamma_2)\right] =  k \, 
\delta_{k+l} \, \tau(\gamma_1 \cup \gamma_2) \,
\label{commNak}
\end{equation}
see \cite{NakHilb}. 
Recall from Section \ref{s_signs_adj} that $\tau$ involves 
a sign. Since $\gamma_i$ are cohomology classes on a surface, 
$$
\tau(\gamma) = - \int_{\C^2} \gamma \,, 
$$
where the integral is defined as an equivariant residue.
In particular
\begin{equation}
  \label{tau_1}
  \tau(1) = - \frac{1}{\det_{\C^2}} \,. 
\end{equation}

\subsection{}

Since $\beta_k(\gamma)$ is a Steinberg correspondence, 
there exist a Steinberg correspondence $\beta_k(\gamma)_\bA$
that makes the following diagram commute 
\begin{equation*}\label{beta_bA}
\xymatrix{
H^\hd_\bG(\cM(1))^{\otimes r}
 \ar@{->}[rr]^{\Stab_\fC\,\,}\ar[d]_{\beta_k(\gamma)_\bA}&& 
H^\hd_\bG(\cM(r)) \ar[d]^{\beta_k(\gamma)} \\
H^\hd_\bG(\cM(1))^{\otimes r}  \ar@{->}[rr]^{\Stab_\fC} &&  
H^\hd_\bG(\cM(r))\\
}
\end{equation*}
for every chamber $\fC$ and every $k<0$.  
Here we use that $\bA$ does not act on the 
$\C^2$ factor in \eqref{defZ}. By taking adjoints, the 
same holds for $k>0$ after tensoring with $\bK$.

\begin{Theorem}\label{t_Bar} We have 
$$
\beta_k(\gamma)_\bA  = \sum_{i=1}^r 1 \otimes \cdots \otimes \alpha_k(\gamma)
\otimes \dots \otimes 1 \,, 
$$
where $\alpha_k(\gamma)$ acts in the $i$th tensor factor. 
\end{Theorem}

\subsection{}

In particular, Theorem \ref{t_Bar} and 
commutation relations \eqref{commNak} imply
\begin{equation}
 \left[\beta_n(\gamma_1),\beta_m(\gamma_2)\right] = rn \, \delta_{n+m} 
\tau(\gamma_1 \cup \gamma_2)\,, \label{commBar} 
\end{equation} 
which is a theorem of Baranovsky, see \cite{Bar}.

\section{Proof of Theorem \ref{t_Bar}}


\subsection{}

By Theorem \ref{t_Stein}, the operator in question is given by 
a correspondence supported on $\fB^\bA$. {}From definitions
\begin{equation}
\fB^\bA = \Big\{\Big(\bigoplus_{i=1}^r I_i,x,\bigoplus_{i=1}^r J_i\Big)\Big\}
\label{dBA}
\end{equation}
where $I_i,J_i\in \cM(1)$, $I_i \subset J_i$, and 
$$
\supp J_i/I_i \subset \{x\} 
$$
for all $i$. The connected components of $\fB^\bA$ are classified 
by the second Chern classes of $I_i,J_i$, and their dimensions are
computed as follows 
$$
\dim = 2 + \sum_i \max \, (c_2(I_i)-c_2(J_i)-1,0)\,.
$$
In particular, 
$$
\fB^\bA = \bigcup_{i=1}^r \fB_1^{(i)} \cup \textup{lower dimension} \,,
$$
where $\fB_1^{(i)}$ denotes the corresponding correspondence 
for $r=1$ acting in the $i$th factor.

\subsection{}

The top dimensional components $ \fB_1^{(i)}$ 
are irreducible.  Therefore, to 
compute the Lagrangian residue of $\fB$, it suffices
to find  a smooth point $\fb$ of $\fB$ on each of them. 

By symmetry, we may assume $i=1$. In \eqref{dBA}, we take a point 
$\fb\in\fB^\bA$ such that 
$$
x = 0 \in \C^2\,, \quad I_1 = (x_1^k,x_2) \,, \quad J_1 = \cO\,,
$$
while $0\notin \supp I_i,\supp J_i$ for  $i>1$. 
Lemma \ref{lem_neigh} below gives 
a rational map 
$$
f:\Mr{k} \times \Mr{n}  \dashrightarrow \Mr{n+k} 
$$
which is an isomorphism in a neighborhood of $\fb$. Denoting 
the correspondence \eqref{defZ} by $\fB_{r,n,k}$, we have 
$$
f^*(\fB_{r,n,k}) = \fB_{r,0,k} \times \diag_{\Mrn}
$$
in a neighborhood of $\fb$. 

Note that polarizations in Theorem \ref{t_Stein} enter in
the combination $\bsi_X \, \bar\bsi_Y$. Therefore, the 
residue of the diagonal is always the diagonal and 
the computation is reduced to the case $n=0$. 

\subsection{}

The correspondence 
$$
\fB_{r,0,k} \subset \Mr{k} \times \C^2
$$
has the following quiver description:
$$
\fB_{r,0,k} = \{B=0,(X_1-x_1)^k=0,(X_2-x_2)^k=0\}\,, 
$$
where $x=(x_1,x_2)\in \C^2$. Our reference point $\fb$ on it
is given by  
$$
X_1=\begin{pmatrix}
     0 \\
     1 & 0\\ 
     0 & 1 & 0 \\ 
     & & \ddots & \ddots 
    \end{pmatrix} \,, \quad 
X_2 = 0 \,, \quad 
A = \begin{pmatrix}
     1 & 0 & 0& \dots\\
     0 & 0 & 0 \\ 
     0 & 0 & 0 \\
     \vdots & & & \ddots 
    \end{pmatrix} \,. 
$$

\begin{Lemma} \label{l_smooth} 
The variety $\fB_{r,0,k}$ is smooth at $\fb$ and its nonzero 
tangent $\bA$-weights are 
$$
(a_1 - a_i)^{\oplus n}\,, \quad  i=2,\dots,r \,. 
$$
\end{Lemma}

\begin{proof}
In a neighborhood, the operator $X_1-x_1$ will remain a
regular nilpotent and the $(1,1)$-entry of $A$ will remain nonzero, hence 
the triple $(X_1,X_2,A)$ may be brought to the normal form 
$$
X_1=\begin{pmatrix}
     x_1 \\
     1 & x_1\\ 
     0 & 1 & x_1 \\ 
     & & \ddots & \ddots 
    \end{pmatrix} \,,\quad 
X_2 = P(X_1)\,, \quad 
A = \begin{pmatrix}
     1 & * & *& \dots\\
     0 & * & * \\ 
     0 & * & * \\
     \vdots & & & \ddots 
    \end{pmatrix}
$$
by a unique element of $GL(n)$. Here $P$ is a polynomial of degree $<n$ and 
stars stand for arbitrary numbers. Thus a neighborhood of $\bF$ in $\fB$ 
 is isomorphic to $\C^{rn+1}$. The computation of the tangent weights 
is straightforward. 
\end{proof}

\subsection{}

In particular, we see that 
$$
T_\fb \fB \big/ T_\fb \fB^\bA  \cong T_\fb Q_k \big/ T_\fb Q_k^\bA
$$
as $\bA$-modules. 

Since $\fB$ is smooth at $\fb$, its Lagrangian residue is $\pm \fB^\bA$. 
Further, the normal weights to $\fB$ agree with the 
polarization \eqref{polarM}. This finishes the proof modulo the
following lemma used above.

\subsection{}

\begin{Lemma}\label{lem_neigh} 
Suppose the eigenvalues of $X_1$ may be partitioned 
$$
\textup{Eigenvalues}(X_1) = \bigsqcup E_i
$$
into a nontrivial disjoint union. Then a neighborhood 
of $(X_1,X_2,A,0)$ in $\Mrn$ is $GL(r)$-equivariantly isomorphic to an 
open set in $\prod \Mr{|E_i|}$. 
\end{Lemma}

\noindent 
We are grateful to the referee for pointing out that this 
statment, with a different proof, is the 
\emph{factorization property} of \cite{BFG}. 

\begin{proof}
For nearby $X_1$ we can still group 
the eigenvalues according to the same partition. We denote 
$$
P_i : \C^n \to \C^n \,, \quad i = 1,\dots,\ell(\mu)
$$
the corresponding spectral projectors. 

The projectors $P_i$ are canonically defined and, 
in particular, commute with the centralizer of $X_1$ in $GL(n)$. We 
thus may assume they project onto coordinate subspaces and replace
the $GL(n)$-quotient by $\prod GL(|E_i|)$. For each $i$, the quadruple 
$$
(Z_i,W_i,A_i,B_j) = (P_i \, X_1, P_i\, X_2 \,P_i, 
P_i \, A, B \, P_i) 
$$
solves \eqref{ADHM}. Because the starting point 
$(X_1,X_2,A,0)$ is stable, each of these blocks remains 
stable in a certain neighborhood. 
Thus we get a map to $\prod \Mr{|E_i|}$. 
Clearly, it is $GL(r)$-equivariant. 

The original data $(X_1,X_2,A,B)$ may be reconstructed as follows. Since 
$\sum P_i =1$, all we need is to recover 
$$
W_{ij} = P_i X_2 P_j 
$$
for $i\ne j$. It is a solution of 
$$
Z_{i} W_{ij} -  W_{ij} Z_{j} = - A_i B_j \,, 
$$
which exists and is unique because the spectra of $Z_i$ and $Z_j$ are 
disjoint.

\end{proof}

\section{Classical $\br$-matrix}\label{s_classR}

\subsection{}

Denote $\bk = H^\hd_\bG(\pt)$ and let 
$$
\vacv{} \in H^\hd_\bG(\cM(1,0)) \cong \bk
$$
be the identity element. We abbreviate
$$
\alpha_n = \alpha_n(1) 
$$
in this section. 
It is a theorem of Nakajima \cite{NakHilb} that the map
$$
\bk\big[\alpha_{-1},\alpha_{-2},\alpha_{-3},
\dots \big] \to  H^\hd_\bG(\cM(1))
$$
given by 
$$
f \mapsto f \vacv{} \,, 
$$
is an isomorphism. We will use it to identify its source
and target.

\subsection{}

We define 
$$
\bF = \bK\big[\alpha_{-1},\alpha_{-2},\alpha_{-3}\,.
\dots \big] 
$$
The operators  $\alpha_n$, $n>0$, act on $\bF$
satisfying \eqref{commNak} and annihilating the vector 
$$
\vac = 1 = \vacv{}\,.
$$
Following the tradition in quantum field theory, $\bF$ is 
called a Fock space. The operators $\alpha_n$ generate
a Heisenberg algebra in $\End(\bF)$.

\subsection{}

Consider 
$$
\bR(u) \in \End\left(\bF \otimes \bF\right) \otimes_\bK \Q(\gal) \,.
$$
By Theorem \ref{t_Bar}, it commutes with the operators
$$
\beta_n(1)_\bA = \alpha_n\otimes 1 + 
1 \otimes \alpha_n\,. 
$$
We define
$$
\alpha_n^{\pm} = \alpha_n \otimes 1 \pm 
1 \otimes \alpha_n \,. 
$$
These satisfy
\begin{equation}
\left[\alpha_k^\varepsilon,\alpha^{\eta}_l \right] =  
2 k \tau(1) \, 
\delta_{k+l} \, \delta_{\varepsilon,\eta} 
\,,
\label{commNak2}
\end{equation}
where $\varepsilon,\eta \in \{\pm\}$ as a consequence of \eqref{commNak}. 

We see the operators $\alpha_n^{\pm}$ generate
two new commuting Heisenberg subalgebras of $\bF\otimes \bF$ 
and $\bR$ commutes with one of them. 

\subsection{}

Using the operators $\alpha_n^{\pm}$, we can write 
\begin{equation}
\bF \otimes \bF = \bF^+ \otimes \bF^- \,. 
\label{twofact}
\end{equation}
We denote by $\End^{-}$ the image of $\End(\bF^-)$ in 
$\End(\bF^{\otimes 2})$,

\begin{Lemma} \label{l_F-} 
The operator $\bR$ belongs to $\End^-$.
\end{Lemma}

\begin{proof}
The operators $\alpha^{+}_n$ act irreducibly on $\bF^+$
and commute with $\bR$. 
\end{proof}

\subsection{}

\begin{Lemma}\label{l_End-}
An operator in $\End^{-}$ is uniquely determined
by its matrix elements in the subspace 
$$
\vac \otimes \bF  \subset \bF^{\otimes 2} \,. 
$$
\end{Lemma}

\begin{proof}
Let $A\in \End^-$ and suppose that
\begin{equation}
\big( A \cdot \vac 
\otimes v_1 , \vac  \otimes v_2
\big)  = 0 
\label{zeroA}
\end{equation}
for all $v_1,v_2\in \bF$, while 
\begin{equation}
\left( A \,  \prod \alpha^{-}_{-\mu_i} \, \vac\otimes \vac, 
 \prod \alpha^{-}_{-\nu_i} \, \vac\otimes \vac
\right)  \ne  0 
\label{nonzeroA}
\end{equation}
for some partitions $\mu,\nu$. We may further assume, 
the partitions $\mu,\nu$ in \eqref{nonzeroA} are chosen 
minimal with respect to $|\mu|,|\nu|$. Then taking 
$$
v_1 = \prod \alpha_{-\mu_i}  \, \vac \,, \quad 
v_2 = \prod \alpha_{-\nu_i}  \, \vac \,, 
$$
in \eqref{zeroA} and expanding 
$$
1\otimes \alpha_n=\tfrac12(\alpha^+_n - \alpha^-_n)
$$
we get a contradiction. 
\end{proof}

\subsection{}

The subspace 
$$
\vac \otimes \bF  \subset H^\hd_\bG(\cM(2)^\bA) \otimes \bK 
$$
is a vacuum subspace in the sense of Section \ref{s_vac}. 
By Theorem \ref{t_vacuum}, the corresponding matrix element 
of $\bR(u)$ is the operator of classical 
multiplication by 
\begin{equation}
\frac{e(N_-)}{e(N_- \otimes \hbar)} = 
1+\frac{\hbar \, \rk N_-}{a_1-a_2} + \dots 
\label{expN_1}
\end{equation}
where 
$$
- u = a_2 - a_1
$$
is the $\bA$-weight of $N_-$. By formula \eqref{charN_}\,, 
$$
\rk N_- \big|_{\cM(1,0) \times \cM(1,n)} = n \,.
$$
In the sheaf interpretation, the unstable 
normal bundle $N_-$ to 
$$
\cM(1) \owns I \mapsto \cO\oplus I \in \cM(2)
$$
is the tautological bundle of the Hilbert scheme 
$$
N_- \cong  \Taut = H^0(\cO/I) \,. 
$$
Hence $\rk N_-$ is indeed the number of points.

\subsection{}

Consider the operator 
$$
L_0 = - \sum_{k>0} \alpha_{-k}(1) \, \alpha_{k}(\pt) 
$$
where 
$$
\pt = [0] = \textstyle{\det_{\C^2}} \in H^2_\bG(\C^2)
$$
is the class of the origin. Note that since $\pt$ and $1$ are
proportional, they may be distributed arbitrarily between the 
two factors. From 
$$
\left[\alpha_{k}(\pt) , \alpha_{l}(1)\right] = - k \, \delta_{k+l} \,,
$$
one has the following 

\begin{Lemma} \label{l_L0} 
$L_0$ acts by multiplication by $n$ in $H^\hd_\bG(\cM(1,n))$. 
\end{Lemma}

\subsection{}

\begin{Theorem}\label{t_classical_r} 
 The classical $\br$-matrix for $\cM(1)\times \cM(1)$ equals
 \begin{equation}
\br =   - \sum_{n>0} \alpha^-_{-n}(1) \, \alpha^-_{n}(\pt) \,.
\label{brf}
\end{equation}
\end{Theorem}

\begin{proof}
This commutes with $\alpha_n^+$ and has correct vacuum matrix 
elements by Lemma \ref{l_L0}. We conclude by Lemma \ref{l_End-}. 
\end{proof}

\subsection{}

Expanding out \eqref{brf}, we get the following 
formula for the action of $\br$ on cohomology 
of 
$\cM(r_1) \times \cM(r_2)$ 
\begin{equation}
  \br = \bv \otimes \bw + 
\bw \otimes \bv + \sum_{k \ne 0} \beta_{-k}(1) \otimes \beta_k (\pt) 
\label{brMrn}
\end{equation}
where 
$$
\bw =r ,\quad  \bv = \ct
$$
act by multiplication by the 
rank and instanton charge, respectively, compare with 
\eqref{formula_r}. 

\subsection{}\label{s_g_M(r)}

We conclude 
$$
\fg_\bQ \cong \glh\otimes\bK \Big/ \textup{zero modes}\,,
$$
where zero modes (or constant loops) refer to central elements 
$\beta_0(\gamma)$. 
The brackets in this Lie 
algebra 
\begin{align}
  \left[\bv,\beta_n(\gamma)\right] &= - n \beta_n(\gamma) \,,
\notag\\ 
  [\beta_n(\gamma),\beta_m(\gamma')] &= \tau(\gamma
 \cup \gamma') 
\, 
  n \, \delta_{n+m}\, \bw \,, \label{commgl1}
\end{align}
are a special case of the relation \eqref{REE}.

\chapter{Free bosons}\label{s_free_bos}

\section{Fock spaces} 

\subsection{}

Anticipating application to algebraic surfaces other than 
$\C^2$, we will put the commutation relations \eqref{commgl1} 
in a more abstract framework, in which the insertions 
$\gamma$ take values in a general commutative
Frobenius algebra $\bbH$ over a ring $\bK$. 

To go 
back to framed sheaves on $\C^2$, one takes
\begin{equation}
\bbH = H^\hd_{\bG} (\C^2)\left[\tfrac1{\det_{\C^2}}\right] \,, \quad 
\bK = H^\hd_{\bG} (\pt)\left[\tfrac1{\det_{\C^2}}\right]
\label{exk}
\end{equation}
with the trace map 
$$
\tau: \bbH \to \bK 
$$
given by $\tau(\gamma)= - \int_{\C^2} \gamma$. We denote 
this Frobenius algebra $\bbH(\C^2)$. 

Most of the material in this section is completely standard and 
is recalled mainly for setting up the notation.

\subsection{Heisenberg algebras}

Let $\bbH$ be a free $\bK$-module 
with a nondegenerate symmetric bilinear form $(\, \cdot\,,\, \cdot \,)$.
Consider the space 
$$
\bbH[z^{\pm1}]  = \bbH \otimes \bK[z^{\pm1}] 
$$
of polynomial loops, that is, 
Laurent polynomials $f(z)$ with values in $\bbH$. 
This has a natural skew-symmetric form 
\begin{equation}
\{f,g\} = \int (df,g) \,, \quad 
\int = \oint \frac{dz}{2\pi i z} \,. 
\label{skewform}
\end{equation}
For example
$$
\{\gamma z^{n} , \eta  z^{-n} \} = n \, (\gamma,\eta) \,, \quad 
\gamma,\eta \in \bbH \,. 
$$
The form \eqref{skewform} makes
$\bbH[z^{\pm1}] \oplus \bK$ 
a Heisenberg Lie algebra. We denote by $\fHeis=\fHeis(\bbH)$
its universal enveloping algebra and denote by $\alpha_n(\gamma)\in 
\fHeis$ the image of $\gamma z^n$.

Note that $\fHeis$ has a center, generated by the identity 
and the zero modes $\alpha_0(\gamma)$, $\gamma\in \bbH$. 

\subsection{Translation automorphisms}\label{s_shift_aut}

The additive group of $\bbH$ acts 
on $\fHeis(\bbH)$ by automorphisms
$$
\ttau_\gamma \left(\alpha_n(\eta)\right) = \alpha_n(\eta) - 
\delta_{n,0} \, (\gamma,\eta) \,, \quad \gamma,\eta\in \bbH \,. 
$$
We denote
$$
\fHeis^\sim = \bK[\bbH_\textup{add}] \ltimes \fHeis \,, 
$$
where $\bK[\bbH_\textup{add}]$ denotes the group algebra 
of the additive group of $\bbH$. By definition, it is spanned
by linear combinations of $\ttau_\gamma$, $\gamma \in \bbH$. 

Introduce the corresponding Lie algebra 
elements
$$
\alsn(\gamma) =  \log \ttau_\gamma 
$$
which satisfy the relations 
$$
\left[\alpha_n(\gamma),\alsn(\eta)\right] = \delta_{n,0} \, 
(\gamma,\eta) \,.
$$

\subsection{Fields}\label{s_phi} 

The commutation relations in $\fHeis^\sim$ are best 
summarized using fields, or generating functions. Consider
\begin{equation}
\bphi(\gamma;z) = \alsn(\gamma) + \alpha_0(\gamma) \, \log z - \sum_{n\ne 0} 
\frac{\alpha_n(\gamma)}{n} \, z^{-n} 
\label{phi_field}
\end{equation}
where $z\in \C^\times$ is a variable. Then 
$$
\left[\bphi(\gamma;z) , \bphi(\eta;w)\right] =
(\gamma,\eta) \, \left(\log(z-w)_{|z|>|w|} - \log(w-z)_{|w|>|z|}\right) 
\,. 
$$
Here
$$
\log(z-w)_{|z|>|w|} = \log z - \sum_{n>0} \frac{(w/z)^n}{n} \,,
$$
is the series expansion in the region $|z|>|w|$. We will also 
consider 
\begin{equation}
\bal(\gamma;z) = \partial 
 \, \bphi(\gamma;z) = \sum_n \alpha_n(\gamma) \, z^{-n} \,, 
\label{alpha_field}
\end{equation}
where
$$
\partial = z \frac{\partial}{\partial z}  \,. 
$$
The coefficients of the fields \eqref{alpha_field}
generate $\fHeis(\bbH)$.

\subsection{Fock spaces}

The Fock representation of $\fHeis^\sim$ is generated by the 
vacuum vector $\vacv{0}$ such that 
$$
\alpha_n(\gamma) \vacv{0} = 0\,, \quad n\ge 0 \,. 
$$
We denote 
$$
\vacv{\eta} = \ttau_{-\eta}  \vacv{0} \,. 
$$
These satisfy 
$$
\alpha_0(\gamma) \vacv{\eta} = - (\gamma,\eta) \vacv{\eta} 
$$
and generate an irreducible $\fHeis$-module that we denote 
$\bF(\eta)$. We have 
$$
\bF(\eta) \cong \bF(0) \cong S^\hd\!\left(z^{-1} \bbH[z^{-1}]\right) 
\quad \textup{as vector spaces}\,,  
$$
the first isomorphism being the action of $\ttau_{\eta}$. 
The module structure of $\bF(\eta)$ varies with $\eta$, but only in how 
the center of $\fHeis$ acts. 

\subsection{Adjoints}\label{s_scalar_product} 

There is an anti-involution on $\fHeis_\zeta$ defined by 
$$
(\alpha_n(\gamma))^* = \alpha_{-n}(\gamma)\,, \quad 
\ttau_\eta^* = \ttau_{-\eta} \,,
$$
that is, 
$$
\bphi(\gamma,z)^* = - \bphi(\gamma,z^{-1}) \,. 
$$
The Fock representation has a unique inner product for 
which $\vacv{\eta}$ are orthonormal 
and the anti-involution $*$ coincides with taking the adjoint 
operator. We will use this inner product to define matrix 
elements of operators.

\subsection{Normally ordered products}

Consider a product $\bal(\gamma,z) \, \bal (\eta,w)$ of two fields. 
Its matrix elements are given by convergent series in the 
region $|z|> |w|$. At $z=w$ they have a singularity. 
This is regularized by commuting all annihilation operators
to the right. In other words, one defines the normally ordered
product by 
\begin{equation}
\bal(\gamma,z) \, \bal (\eta,w) = (\gamma,\eta) \, \frac{z w}{(z-w)^2} + 
\nord \bal (\gamma,z) \, \bal (\eta,w) \nord  \,,
\label{balOPE}
\end{equation}
where the first, singular,  term is to be expanded
in the region $|z|>|w|$. The normally ordered term in \eqref{balOPE}
is regular at $z=w$, in fact 
\begin{equation}
\big(\, \nord \bal (\gamma,z) \, \bal (\eta,w) \nord \, f_1, f_2
\big) \in \bK[z^{\pm 1},w^{\pm 1}] 
\label{matrix_element_norm_product} 
\end{equation}
for all $f_1,f_2$ in the Fock space. 

By linearity, we can say that the normally ordered product 
$\,\nord \bal (\gamma,z) \, \bal (\eta,z) \nord\,$ 
takes an element $\gamma\otimes \eta \in \bbH^{\otimes 2}$ 
as an argument. 

A generalization of \eqref{balOPE}, known as Wick's theorem, explains how 
to normally order any product of 
normally ordered monomials in $\bal(\gamma_i,z_i)$. See for example \cite{BF,DMS,KacVertex}.

\subsection{Grading}

Recall we assume $(\,\cdot\,,\,\cdot\,)$ to be nondegenerate and let
$g^{-1} \in \bbH^{\otimes 2}$ be the inverse quadratic form.  Then 
$\nord \bal^2\nord (g^{-1},z)$ is a well-defined 
operator-valued Laurent series, from which we can 
extract the constant term 
$\int \nord\bal^2\nord (g^{-1},z)$. 
The following computation is 
standard 

\begin{Lemma}\label{l_gr_Fock} Let the Fock space be graded by 
$$
\deg \vacv{\eta} = (\eta,\eta)/2\,, \quad \deg \alpha_{-n} = n \,. 
$$
Then $\frac12 \displaystyle\int \nord\bal^2\nord (g^{-1},z)$ is the the grading operator. 
\end{Lemma}

\noindent This is a generalization of Lemma \ref{l_L0}.

\section{Insertions and coproducts}\label{s_coprod}

\subsection{}
Note that in \eqref{matrix_element_norm_product} we evaluate
both operators at the same point $z=w\in \C^\times$, but they
still take two distinct cohomology insertions $\gamma$ and $\eta$, or, 
equivalently, a element of $\gamma\otimes\eta \in \bbH^{\otimes 2}$ 
as an argument. 

To write an operator with a single cohomology insertion, we 
need a coassociative coproduct
$$
\Delta: \bbH \to \bbH^{\otimes 2}\,, 
$$
and its iterates 
$$
\bbH \owns \gamma \mapsto \gamma^{\Delta n} \in \bbH^{\otimes n} \,. 
$$
We can then construct an operator 
$$
:\! \bal^n  \!: (\gamma,z) \overset{\textup{\tiny def}}= \,\, 
:\! \bal^n  \!: (\gamma^{\Delta n},z) 
$$
which depends on a single point $z\in\C^\times$ and also 
depends linearly on a single cohomology insertion $\gamma$.

\subsection{}

For example, for 
$
\bbH = H^\hd_{\bG} (\C^2)\left[\tfrac1{\det_{\C^2}}\right] 
$
we have 
$$
1^\Delta = - 1 \otimes \pt = - \pt \otimes 1 \,. 
$$
This is because the comultiplication, as adjoint to multiplication,
gets the sign $-1 = (-1)^{\frac12\dim\C^2}$. 
Therefore, the formula \eqref{brf} can be recast in the 
following form
\begin{equation}
  \label{br_vert}
  \br = \frac12 \int \nord (\bal^-)^2 \nord \, (1) \,,
\end{equation}
modulo zero modes $\alpha_0(\gamma)$.

\subsection{} 
Because of the Frobenius algebra structure on $\bbH$, 
Wick's formula for the operators 
$\nord\bal^n\nord(\gamma)$ takes the following particularly 
nice form. 

For any symmetric Frobenius algebra, there is a
canonical central element $\se\in \bbH$ such that 
$$
m(\gamma^\Delta) = \se \gamma 
$$
for all $\gamma\in \bbH$. Here $m: \bbH^{\otimes 2} \to \bbH$ is the 
multiplication map. This is associated with gluing a handle 
in the context of 2-dimensional topological quantum field
theories, see for example \cite{Kock}. One has 
$$
\tau(\se) = \rk_{\,\bK} \bbH \,.  
$$
In particular, if $\bbH=H^*(S)$ then this is 
the Euler characteristic of $S$ (recall we assume $\bbH$ is 
commutative for simplicity). 

\begin{Lemma}\label{l_Wick} 
\begin{multline}
    \label{Wick_Frob}
\nord\bal^n\nord(\gamma_1)(z_1) \, 
\nord\bal^m\nord(\gamma_2)(z_2)  
= \\ \sum_{k=0}^{\min(n,m)} c_k(z_1,z_2) \, 
\nord \bal^{n-k} (z_1) \bal^{m-k} (z_2)\nord 
\left(\gamma_1 \gamma_2 \se^{k-1}\right)\,,
\end{multline}
where 
\begin{equation}
c_k(z_1,z_2) = \frac{(-n)_k (-m)_k}{k!} 
\, \left(\frac{z_1 z_2}{(z_1-z_2)^2}\right)^k \,.
\label{comb_fact_Wick}
\end{equation}
\end{Lemma}

\noindent 
Here $(n)_k = n(n+1)\cdots(n+k-1)$ and the combinatorial factor 
in \eqref{comb_fact_Wick} is the number of ways to form $k$ pairs 
of elements from $\{1,\dots,n\}$ and $\{1,\dots,m\}$, respectively. 

Two terms in \eqref{Wick_Frob} require a special discussion. For
$k=0$, the insertion is defined to be 
$$
\gamma_1^{\Delta n} \otimes \gamma_2^{\Delta m} \in \bbH^{\otimes(n+m)} \,.
$$
For $n=m=k$, the whole term is defined to be 
$$
c_k(z_1,z_2) \, \tau(\gamma_1 \, \gamma_2 \, \se^{k-1}) \,. 
$$

\begin{proof}
This is an exercise in matching the Wick's formula with 
the graphical calculus for Frobenius algebras, as explained, for 
example, in \cite{Kock}. The tensor operations 
$$
\bbH^{\otimes 2} \to \bbH^{\otimes(n+m-2k)}
$$
that arise from Wick's formula, are interpreted graphically as 
surface of genus $k-1$ with two incoming and $n+m-2k$ 
outgoing holes, hence equal to 
$$
\gamma_1 \otimes \gamma_2 \to 
\left(\gamma_1 \gamma_2 \se^{k-1}\right)^{\Delta(n+m-2k)} \,.
$$
Note for $k=0$, the surface is disconnected, 
whence the need to consider this case separately. The other 
special case 
$n=m=k$ is the case of  no outgoing holes. In this case, there 
is only the scalar operator left in Wick's formula. 
\end{proof}

It is straightforward to generalize this Lemma to more 
than two normally ordered monomials.





\section{Virasoro algebra}

\subsection{}
For an arbitrary $\kappa \in \bbH$, define 
\begin{equation}
  \label{defvTK}
\vT(\gamma,\kappa) = \tfrac12 \nord
    \bal^2\nord (\gamma) 
 + 
    \, \partial \bal  (\gamma \kappa)  - \tfrac12 \tau(\gamma \, \kappa^2) \,. 
\end{equation}
This field generates a Virasoro-like subalgebra
of the Heisenberg algebra, known as the Feigin-Fuchs or 
background charge Virasoro algebra. The statement for an 
arbitrary $\bbH$ should also be considered known,  see for 
example the discussion in Section 5 of \cite{Lehn_Lectures}. 

\subsection{}

We denote by $\vL_n(\gamma,\kappa)$ the 
coefficients of $\vT(\gamma,\kappa)$, that is, 
$$
\vT(\gamma,\kappa) = \sum_{n\in\Z} \vL_n(\gamma,\kappa) \, z^{-n} \,. 
$$

\begin{Theorem}\label{t_vir_comm} The operators $\vL_n(\gamma,\kappa)$ satisfy 
\begin{multline}
  \left[\vL_n (\gamma_1), \vL_m(\gamma_2) \right] = \\
(n-m) \,
  \vL_{n+m}(\gamma_1 \, \gamma_2) + \tau(\gamma_1 \gamma_2 (\se-12 \kappa^2))\,
  \delta_{n+m} \frac{n^3-n}{12} \,. 
\label{comm_Virasoro} 
\end{multline}
\end{Theorem}

These are the familiar Virasoro relations adorned with cohomology 
labels. The element 
\begin{equation}
\cb= \se-12 \kappa^2\in \bbH\label{fcb}
\end{equation}
 plays the role of the 
central charge. 

\subsection{OPEs}

The most efficient way to encode the commutation relations for the 
operators $\vT$ is via the operator product expansion. This goes 
as follows. Let the fields
$$
A(z) = \sum_{n\in\Z} a_n \, z^{-n} \,, 
\quad 
B(z) = \sum_{n\in\Z} b_n \, z^{-n} \,, 
$$
satisfy a commutation relation of the form 
$$
\left[A(z),B(w)\right] = \sum_{k\ge 0} 
 C_k(w)\, 
\left(w
\, \frac{\partial}{\partial w}\right)^k \delta(z,w) \,, 
$$
where $\delta(z,w) = \sum_{n\in \Z} (z/w)^n$ and $C_k(w)$ 
are some fields like $A$ and $B$. Then 
\begin{align}
  A(z) B(w) &= \left[ A_-(z),B(w)\right] + 
\nord A(z) B(w) \nord \notag \\
&\sim \sum_k C_k(w) \,
  \left( w \frac{\partial}{\partial w}\right)^k \frac{w}{z-w} 
\label{simOPE} 
\end{align}
where $A_-(z)=\sum_{n>0} a_n \,  z^{-n}$ and 
$\sim$ means equality modulo terms that 
remain regular as $z\to w$. In particular, in \eqref{simOPE} we
dropped the normally ordered term. 

\subsection{Proof of Theorem \ref {t_vir_comm}}

Let
$$
G = \frac{\sqrt{z w}}{z-w} = 
\frac{1}{e^{x/2} - e^{-x/2}}\,, \quad x = \ln(z/w) \,, 
$$
denote one of the Green's functions of the 
$\bar\partial$ operator on the cylinder. Since we will 
only deal with expansions as $z \to w$, we may 
ignore the monodromy of $G$. 

\begin{Proposition}
The field $\vT$ satisfies the following OPE
\begin{multline}
\vT(\gamma_1)(z) \, \vT(\gamma_2)(w) \sim \\
\tfrac12 G^4 \, \tau(\gamma_1 \, \gamma_2 (\se-12 \kappa^2))+ 
2 G^2 \, \vT(\gamma_1 \gamma_2)(w) + 
G \, \partial \vT(\gamma_1 \gamma_2)(w) 
\,,\label{OPE_T}
\end{multline}
where $e\in H$ is the handle-gluing element. 
\end{Proposition}

\begin{proof}
Direct computation using Lemma \ref{l_Wick}. 
\end{proof}

\noindent
This proposition finishes the 
proof of Theorem \ref{t_vir_comm}.

\subsection{Lowest weight}

{}From definitions, we compute 
$$
\bL_n(\gamma,\kappa) \vacv{\eta} = 0 \,, \quad n>0 \,,
$$
while 
$$
\bL_0(\gamma,\kappa) \vacv{\eta} = 
\tfrac12 \, \tau(\gamma(\eta^2-\kappa^2)) \, \vacv{\eta}\,. 
$$
For $\gamma=1$ and $\kappa=0$ this specializes to Lemma
\ref{l_gr_Fock}. The element 
\begin{equation}
\db = \tfrac12 \, (\eta^2-\kappa^2) \in \bbH
\label{fdb}
\end{equation}
should thus be viewed as the conformal dimension of $\vacv{\eta}$, 
that is, the lowest weight of the Virasoro module $\bF(\eta)$.

\subsection{Irreducibility} 

{} We have the following standard 

\begin{Lemma}\label{l_irr}
The Virasoro module $\bF(\eta)$ is irreducible 
for generic $\eta$. 
\end{Lemma}

\begin{proof}
For $\eta\to \infty$, Virasoro algebra degenerates to Heisenberg 
algebra which acts irreducibly. 
\end{proof}

\section{Reflection operator}\label{s_reflect} 

\subsection{} 

Lemma \ref{l_irr} implies for generic $\eta$, $\bF(\eta)$ is 
a Verma module for Virasoro 
algebra  with central charge \eqref{fcb} and lowest weight 
\eqref{fdb}. Note, however, that the map 
$$
(\eta, \kappa) \mapsto (\db,\cb)
$$
is many-to-one, in particular, the 4 points $(\pm \eta,\pm \kappa)$ 
give isomorphic Virasoro modules for generic parameters. 
This implies the following 

\begin{Proposition}
For generic $\eta$ and 
any choice of signs, there exists a unique, up to multiple, 
operator $R_{\pm,\pm}$ that makes the following diagram 
\begin{equation*}\label{Rpmpm}
\xymatrix{
\bF(\eta)
 \ar@{->}[rr]^{\vT(\kappa)}\ar[d]_{R_{\pm,\pm}}&& 
\bF(\eta) \ar[d]^{R_{\pm,\pm}} \\
\bF(\pm \eta) \ar@{->}[rr]^{\vT(\pm \kappa)} &&  
\bF(\pm \eta)\\
}
\end{equation*}
commute. It depends rationally on $\eta,\kappa\in \bbH$. 
\end{Proposition}

\noindent
The first $\pm$ in $R_{\pm,\pm}$ is for $\eta$, the second 
--- for $\kappa$. 
The intertwiner $R_{\pm,\pm}$ is a rational function of $\eta,\kappa\in \bbH$
because it solves linear equations in which $\eta$ and $\kappa$ enter 
polynomially. We normalize it so that 
$$
R_{\varepsilon_1,\varepsilon_2} \vacv{\eta} = \vacv{\varepsilon_1 \eta} \,. 
$$

\subsection{}

In down-to-earth terms, 
$$
R_{\pm\pm} \, \prod \vL_{-\mu_i}(\gamma_i,\kappa) \, \vacv{\eta} =
\prod \vL_{-\mu_i}(\gamma_i,\pm \kappa) \, \vacv{\pm \eta} 
$$
for all partitions $\mu$. For generic $\eta$, these vectors
form a basis of $\bF(\pm \eta)$. 

In particular, $R_{\pm\pm}$ preserves the grading by $|\mu|$, hence 
is a direct sum of finite-dimensional operators.

\subsection{}

It is easy to see that 
\begin{equation}
R_{--} \, \alpha_n(\gamma) \,  R_{--}^{-1} = - \alpha_n(\gamma)  \,. 
\label{R--}
\end{equation}
Thus of the four operators $R_{\pm\pm}$ only one is really 
nontrivial. Also, we note
\begin{equation}\label{Reta0} 
\eta = 0 \Rightarrow R_{+-} = R_{--} \, .
\end{equation}

\subsection{} 

The operator $R_{-+}$ is known as the reflection 
operator in Liouville CFT, see \cite{Tesch}, while we will identify 
$\bR(u)$ with the operator $R_{+-}$ for 
$$
\bbH = H^\hd_{\bG} (\C^2)\left[\tfrac1{\det_{\C^2}}\right]
$$
in Chapter \ref{s_fullR}. 
Thus, the Liouville reflection operator will be identified 
with 
$$
\bR^\vee = (12) \, \bR \,. 
$$
The Yang-Baxter equation satisfied by $\bR(u)$ is a new 
and unexpected aspect of the theory.

\subsection{} 

In addition to the inner product discussed in Section 
\ref{s_scalar_product}, the action of the Virasoro 
algebra equips the Fock space with the Shapovalov 
inner product, such that
$$
\vL_n^\dagger = \vL_{-n}\,,
$$
where dagger denotes the adjoint operator with respect
to the Shapovalov product. 

We have
$$
R_{-+} \, \vL_n^\dagger  = \vL_n^* \, R_{-+}\,,
$$
therefore $R_{-+}$ is precisely the operator that relates
the two inner products. In particular, 
the determinant of the graded pieces 
of $\bR$ is very closely related to Kac determinant for 
Virasoro algebra, see \cite{Kac_det,FF}. We will see the classical 
results of Feigin and Fuchs on it from a new perspective 
in Chapter \ref{s_fullR}.

\chapter{The full $R$-matrix}\label{s_fullR} 

\section{Zero modes} 

In Section \ref{s_g_M(r)}, we identified the Lie algebra $\fg_Q$ 
for instanton moduli $\cM(r)$ with the algebra $\glh$ modulo the zero modes. 
On the other hand, we saw in Chapter \ref{s_free_bos} the 
convenience and importance of including the zero modes in 
the considerations. 

Later, a different normalization of 
$\bR(u)$ will be introduced which will reconcile these two 
points of view. For now, until the Section 
\ref{s_char_R}, we set zero modes to zero.

\section{Cup product by divisor}\label{s_cup_div}

\subsection{}

Generalizing the formula for $\br$, we define 
$$
\bPhi_n = \frac{1}{n!}  \int \nord\bal^n\!: (1) \,. 
$$
These are examples of 
Fourier coefficients of vertex operators,
see e.g.\ \cite{BF,KacVertex}. 
The following operator $\bOm$, while not a Fourier coefficient of a vertex 
operator, plays an important role in the 
theory. 

\subsection{}\label{s_|part|}

Define the operator $|\partial|$ by 
$$
|\partial| \cdot z^n = |n| \, z^n \,. 
$$
This is a composition of $\partial=z\frac{d}{dz}$ and the Hilbert transform. We define
$$
\bOm = \frac12 \int \nord \bal |\partial| \bal \nord (1)= 
\sum_{n>0} n \, \alpha_{-n} \alpha_n (1^\Delta)\,. 
$$

\subsection{}

The operator $\bOm$ appears in the following formula due to M.~ Lehn \cite{Lehn}. 
Recall that 
$$
\cO(1) = \Lambda^{\textup{top}} \Taut\,, \quad \Taut=\cV_1 = 
H^1_{\Pp^2}(\cF(-1))
$$
is the ample generator of the Picard group of $\cM(r)$.

\begin{Theorem}[\cite{Lehn}]\label{t_Lehn} The operator of cup 
product by $c_1(\cO(1))$  in $H^\hd_\bT(\Hilb)$ is given by
\begin{equation}
c_1(\cO(1)) \, \cup \, \dots = - \bPhi_3  + (a-\tfrac{1}{2} \, \hbar\,) \bPhi_2 
+ \tfrac{1}{2} \, \hbar\,
\bOm\,. 
\label{multc1}
\end{equation}
\end{Theorem}

Here $a$ is the weight of the framing torus $\bA\cong \C^\times$
that acts trivially 
on $\cM(1)$ itself, but nontrivially, namely with weight $a$, 
on the tautological bundle. Such an insignificant additional parameter
is usually suppressed and, in particular, it is not present 
in Lehn's formulation. 

Lehn's theorem may be also deduced from the factorization of $\bR(u)$ 
into $R$-matrices for $\bY(\mathfrak{gl}(\infty))$ given in 
Theorem \ref{t_Rprod}, see \cite{Smirnov}.

\subsection{}

Lehn's theorem identifies the operator of cup product by $c_1(\cO(1))$
with the second quantized trigonometric Calegero-Sutherland 
Hamiltonian, see for example \cite{CostGroj} for a comprehensive 
discussion. 

The explicit form of the Calogero-Sutherland operator 
in the basis of power-sum symmetric functions (that is, in the 
natural basis of the bosonic Fock space) was computed by 
Richard Stanley \cite{Stanley} and rediscovered many times since. 
The equivalence 
between Lehn's and Stanley's formulas was noticed, apparently, by many people,
\cite{Matsuo} being one of the early references, see the discussion 
in \cite{CostGroj}. 

Note that classes of torus-fixed points in $H^\hd_\bT(\Hilb)$ are trivially 
eigenfunctions of cup product operators and their identification 
with Jack polynomials, that is, CS eigenfunctions, was noted earlier, 
see in particular \cite{NakJack}. At about the same time, 
it was recognized by Mark Haiman that the more general Macdonald polynomials 
correspond to the classes of fixed points in the equivariant $K$-theory 
of Hilbert schemes, see for example \cite{Haiman}. 

\subsection{}

We will see the analogous integrable system for $\cM(r,n)$ is 
a coupled $r$-tuple of  Calogero-Sutherland systems. The coupling 
is triangular, so the spectrum is additive, which is obvious
from the geometric description of torus-fixed points. Independently 
of our work, the same quantum integrable system appeared in \cite{EPSS}. 

The algebra of operators of quantum multiplication gives a 
one-parameter deformation of cup product operators and thus 
a deformation of the Calogero-Sutherland quantum integrable 
system. It has been identified with the quantum 
Intermediate Long Wave equation \cite{NOILW}. In particular, 
this allows to determine the spectrum of the latter as well 
as to give an explicit construction of integrals of motion.

\subsection{}

Taking the expansion \eqref{expN_1} one step further, we get 
\begin{equation}
\frac{e(N_-)}{e(N_- \otimes \hbar)} = 
1+\frac{\hbar \, \rk }{u} +
\frac{\hbar \, c_1(\cO(1)) + \frac{1}{2} \, 
\hbar^2 \rk(\rk+1)}{u^2} + \dots 
\label{expN_}
\end{equation}
where $u=a_1-a_2$ and 
$c_1(\cO(1))$ is the operator from Theorem \ref{t_Lehn}
with $a=0$. This is because we already accounted for
the fact that $N_-$ has weight $-u$ 
with respect to the rank $2$ framing torus.


\subsection{}

\begin{Proposition}\label{p_R2} 
We have 
\begin{equation}
  \bR(u) = 1 + \frac{\hbar}{u} \,\bPhi^{-}_2 +
\frac{\hbar}{u^2} \, \bPhi^{-}_3 + 
\frac{\hbar^2}{2 u^2} \left( \bPhi^{-}_2 \right)^2 + O(u^{-3}) 
\,.  
\label{R2}
\end{equation}
\end{Proposition} 

\noindent 
Here and in what follows, $\bPhi^{-}_n$ denotes the result of 
substituting $\bal^{-}$ for $\bal$ in the definition of $\bPhi_n$.

\begin{proof}
Denote by $\sP$ 
the orthogonal projection onto $\vac\in\bF$. We compute
\begin{align}
\sP^{(1)} \,  \bPhi^{-}_3  \, \sP^{(1)}  & = -  \bPhi^{(2)}_3\,,
\notag \\ 
\sP^{(1)} \, \big( \bPhi^{-}_2 \big)^2 \, \sP^{(1)}  & = 
\big( \bPhi^{(2)}_2 \big)^2 + \bOm^{(2)} \,, 
\label{vacPhi2}
\end{align}
where upper indices like the one in $\sP^{(1)}$ denote an operator 
acting in the corresponding tensor factor of $\bF \otimes \bF$. 
It is very instructive to see 
how Fourier coefficients of 
vertex operator produce something which isn't one upon taking 
vacuum matrix elements. 

Now the result follows from comparing \eqref{multc1} with 
\eqref{expN_}. 
\end{proof}

\noindent
Note, for example, that 
$$
\bR(-u)_{12} = 1 - \frac{\hbar}{u} \, \bPhi^{-}_2 -
\frac{\hbar}{u^2} \, \bPhi^{-}_3 + 
\frac{\hbar^2}{2 u^2} \left( \bPhi^{-}_2 \right)^2 + O(u^{-3}) \,,
$$
because the permutation of tensor factors flips the sign of $\alpha^{-}$. 
This illustrates general results on unitarity of $R$-matrices, see 
Section \ref{s_unitarity}.

\subsection{}

We now consider an $(r+1)$-fold tensor power of $\bF$ and denote by 
$$
\bal^{(i)}\,, \quad i=0,\dots,r\,,
$$
the Heisenberg operators in the corresponding tensor 
factors. We denote
$$
\bPhi_n^{(ij)}  = \frac{1}{n!}  \int 
\nord(\bal^{(i)}-\bal^{(j)})^n\!: (1) 
$$
and
$$
\bOm^{(ij)}   = 
\sum_{n>0} n \, \alpha^{(i)}_{-n} \alpha^{(j)}_n (1^\Delta)\,. 
$$
In particular, $\bPhi_n^{(12)}= \bPhi_n^{-}$ and $\bOm^{(ii)}$
is the operator $\bOm$ acting in the $i$th tensor factor. 
Generalizing \eqref{vacPhi2}, we compute
\begin{equation}
\sP^{(0)} \,  \bPhi^{(0j)}_2  \bPhi^{(0i)}_2 \, \sP^{(0)}   = 
 \bPhi^{(j)}_2  \bPhi^{(i)}_2 + \bOm^{(ji)} \,. 
\label{vacPhi2ij}
\end{equation}

\subsection{}

We consider $X=\cM(r)$ and the action of the maximal torus 
$\bA$ of $GL(r)$ on it. Fix a chamber $\fC \subset \fa$ and 
denote by $\bQ_{cl}$ the operator that makes 
the following diagram commute 
$$
\xymatrix{
\bF^{\otimes r} \ar[rr]^{\hspace{-5mm}\Stab_\fC}\ar[d]_{\bQ_{cl}}&& 
H^\hd_\bG(\cM(r)) \otimes \bK \ar[d]^{\cup \, c_1(\cO(1))} \\
\bF^{\otimes r} \ar[rr]^{\hspace{-5mm}\Stab_\fC} && H^\hd_\bG(\cM(r)) 
\otimes \bK \,. \\
}
$$
Consider the following modified step function 
$$
\varrho(x) = 
\begin{cases}
  1\,, & x>0\,,\\
1/2\,, & x=0 \,, \\
0\,, & x<0 \,,
\end{cases}
$$
and define 
$$
\varrho_\fC(i,j)  = \varrho\left((a_i-a_j)|_\fC\right) \,. 
$$
\begin{Theorem}\label{t_Qclass} The operator $\bQ_{cl}$ is given by 
\begin{equation}
\bQ_{cl} =  \sum_{i=1}^r 
\left(  -
\bPhi_3^{(i)}  +  (a_i - \tfrac12 \hbar) \bPhi_2^{(i)}
\right) 
+ \hbar \sum_{i,j=1}^r \varrho_\fC(i,j) \, \bOm^{(j,i)} \,.
\label{Qclass}     
\end{equation}
\end{Theorem}

\noindent This is a special case of Theorem \ref{t_cl_divisor}. 
We recall the proof. 

\begin{proof}
Using Theorem \ref{t_vacuum} and equation  \eqref{expN_}, 
in particular, the operator $\bQ_{cl}$ may be 
computed from the $1/u^2$ coefficient of the $R$-matrix from 
Example \ref{e_RRR}. We substitute the formula from 
Proposition \ref{p_R2} and expand using \eqref{vacPhi2ij}. 
This gives the result. 
\end{proof}

\subsection{}
Note for the standard chamber $\fC$, we have 
\begin{equation}
\sum_{i,j=1}^r \varrho_\fC(i,j) \, \bOm^{(j,i)} = 
\frac14 \int \nord \bbeta \, |\partial| \, \bbeta \nord (1) +
\frac12 \sum_{i<j} \int \bal^{(i)} \, \partial \, \bal^{(j)} (1) 
\label{bOmpart} 
\end{equation}
where 
$$
 \bbeta = \bal^{(1)} + \dots + \bal^{(r)} \,. 
$$
For general $\fC$, the final sum in \eqref{bOmpart} is over all 
$i,j$ such that $a_i-a_j$ is positive on $\fC$. 

\section{$R$-matrix as a Virasoro intertwiner}\label{s_R_vir} 

\subsection{}\label{s_char_R}

\begin{Theorem}\label{t_Rvir} 
The operator $\bR(u)$ is obtained by substitution 
\begin{equation}
\bal = \frac1{\sqrt{2}} \, \bal^-\,, \quad 
\eta = \frac{u}{\sqrt{2}}\,, \quad 
\kappa  = \frac1{\sqrt{2}} \, \hbar \,.   \label{subsvir}
\end{equation}
into the Virasoro intertwiner $R_{+-}$ for 
$\bbH=\bbH(\C^2)$. 
\end{Theorem}

\noindent
Here $u=a_1-a_2$ and $\bbH(\C^2)$ is the Frobenius algebra \eqref{exk}. 
The square roots  in \eqref{subsvir} are needed because 
of the factor $2$ in \eqref{commNak2}. In other words, they are 
there because the vector $(1,-1)$ has length $\sqrt{2}$.

\subsection{Proof of Theorem \ref{t_Rvir}}\label{s_zero_modes}

{}From Lemma \ref{l_F-}, we know that $\bR(u)$ acts only 
in the $\bF^-$ factor in \eqref{twofact}.  To find out 
how it acts in $\bF^-$, we will use the intertwining 
relation with the operators $\bQ_{cl}$ 
 for the two chambers 
$$
a_1 \gtrless a_2\,. 
$$
We express $\bQ_{cl}$ in terms of $\bal^\pm$ and note 
that $\bal^+$ commutes with $\bR$. In particular, 
the first term in the right-hand side of \eqref{bOmpart}
commutes with $\bR$. Therefore we have, for $\fC_\pm=\{a_1\gtrless a_2\}$
\begin{multline}
  - 2 \bQ_{cl} = \dots + \frac14 \int \bal^+ \, \nord
    (\bal^{-})^2\nord (1) \, + \\
\int \bal^+ \left(
  \frac{a_2-a_1}{2} \, \bal^{-} \pm \frac{\hbar}{2}
    \, \partial \bal^{-} \right) (1) \,, 
\label{Qcl_Vir} 
\end{multline}
where dots stand for terms that commute with $\bR$. 

Since $\bR$ commutes
with $\bal^+$, it has to intertwine 
the coefficients of its modes in \eqref{Qcl_Vir}, 
therefore it has to intertwine 
the operators
\begin{equation}
\vT_\pm(\gamma) = \frac14 \nord
    (\bal^{-})^2\nord (\gamma) + 
\left(
  \frac{a_2-a_1}{2} \, \bal^{-} \pm \frac{\hbar}{2}
    \, \partial \bal^{-} \right)(\gamma)  + \dots 
\label{defvT_} 
\end{equation}
for all $\gamma\in H^\hd_\bG(\C^2)$. 
Here dots stand for a scalar operator that will be fixed in 
a minute. 

Strictly speaking, since $\bal^+$ does not include zero modes, 
the above argument shows $\bR$ intertwines all coefficients
of $\vT_\pm$ except the constant term $\int \vT_\pm(\gamma)$. 
However, this constant term  can be obtained as commutator of 
other coefficients of $\vT_\pm$, by Virasoro commutation 
relations.

We now compare \eqref{defvT_} with \eqref{defvTK}. The 
two operators become identical if we substitute 
\begin{equation*}
\bal = \frac1{\sqrt{2}} \, \bal^-\,, \quad 
\kappa  = \frac1{\sqrt{2}} \, \hbar \,,  
\end{equation*}
and make the zero mode present in \eqref{defvTK} act 
via the identification  
\begin{equation}
  H^\hd_\bG\left(\cM(2)^\bA\right) \otimes \bK \cong \bF(a_1) \otimes \bF(a_2) \,. 
 \label{zeroF2}
\end{equation}
This identification fixes the constant term left as dots in \eqref{defvT_}. 
Thus $\bR(u)$ is identified with $R_{+-}$ by the uniqueness of the latter.

\subsection{The determinant of $\bR(u)$}

By construction, $\bR(u)$ is a product of two 
triangular operators, namely of the composition 
$$
H^\hd_\bG(X^\bA) 
\xrightarrow{\,\, \Stab_\fC \,\,} 
H^\hd_\bG(X) 
\xrightarrow{\,\,\textup{Restriction} \,\,} 
H^\hd_\bG(X^\bA)\,,  
$$
and the inverse of the analogous composition for the other chamber. 
Each of these operators has simple diagonal parts, yielding a factorization 
for the determinant of the graded pieces of $\bR$. 
This gives an alternative derivation of the product formula 
for the determinant of the Shapovalov form \cite{Kac_det,FF}.

\subsection{}

{} From \eqref{R--} and \eqref{Reta0}, we conclude 
$$
\bR(0) = (12) 
$$
where $(12)$ is the permutation of the two factors. This is 
because
$$
(12) \,  \bal^- (12) = - \bal^- \,. 
$$

\section{The $1/u$ expansion of $\bR$}

\subsection{}

In this section we derive
an expansion of $\log \bR(u)$ in inverse 
powers of the spectral parameter $u$. 
We write 
$$
\vT(\gamma)_\pm = - \frac{u}{2} \, \bal^{-}(\gamma) + \vT'(\gamma)_\pm\,,
\quad \vT'_\pm =  \tfrac14 \nord
    (\bal^{-})^2\nord (\gamma)\pm  \tfrac{\hbar}2
    \, \partial \bal^{-} (\gamma) + \dots\,,
$$
where dots stand for a constant term that cancels out of 
the equation 
\begin{equation}
\bR \, \vT_{+}(\gamma) \, \bR^{-1} = \vT_-(\gamma) \,. 
\label{intertw_eq}
\end{equation}
We look for solutions in the form 
$$
\bR = \exp\left(\sum_{n>0} \frac{\br^{(n)}}{u^n} 
\right)
$$
where, in particular, 
$$
\br^{(1)} = \frac12 \int \nord (\bal^-)^2 \nord (\hbar)
$$
is, up to normalization, the familiar classical $R$-matrix. 
We denote by 
$
\bR^{(m)} =  \exp\left(\sum_{0<n\le m} \frac{\br^{(n)}}{u^n} 
\right)
$
the successive approximations. The recurrence relations for $n>1$ 
take the form 
\begin{equation}
\left[\br^{(n)}, \bal^{-}(\gamma)\right] = 2 
[u^{-n+1}]\, \exp(\ad(\log \bR^{(n-1)})) \cdot  \vT'_+(\gamma) \,. 
\label{eq_rn}
\end{equation}
where $[u^{-n+1}]$ denotes the coefficient of $u^{-n+1}$. These 
fix $\br^{(n)}$ uniquely up to an additive constant. The constant 
is determined by the requirement that $\br^{(n)}$ annihilates the vacuum 
vector. 

\subsection{}

Solving equations \eqref{eq_rn}, we obtain 
\begin{align}
\br^{(2)} = &\frac16 \int \nord (\bal^-)^3(\hbar) \nord \notag\\
\br^{(3)} = &\frac1{12} \int \nord (\bal^-)^4 \nord (\hbar)
 -
\frac1{12} \int \nord (\bal^-)^2 \nord (\hbar \se)
\label{logRu} \\
& \notag 
- \frac{1}{12}\int \nord \left(\partial \bal^-\right)^2 
\nord  (2\hbar^3+\hbar \se)\,\,\,,
\end{align}
%
where 
$$
\se= -\textstyle{\det_{\C^2}} \in H^\hd_\bG(\pt)
$$
 is the handle-gluing 
element. Of course, since our Frobenius algebra is $1$-dimensional, 
all cohomology insertions may be converted to coefficients
in the formula. 


\subsection{}

Further structures in this expansion will be discussed elsewhere. 
Here we only note the following. The normally ordered 
polynomials in the field $\bal^-$ and its derivatives
are, from definitions, \emph{vertex operators} in the Heisenberg
vertex algebra. Integrals of such operators are known 
as \emph{residues} of vertex operators. They act as 
infinitesimal automorphisms of the Heisenberg vertex algebra. 

\begin{Theorem}
The logarithm of $\bR$ is a residue of a vertex operator, that  
is
$$
\br^{(n)} = \int 
:P_n(\bal^-,\partial \bal^-, \partial^2 \bal^-,\dots; \hbar,\se): (1)\,,
$$
for some polynomials $P_n$. 
\end{Theorem}

\begin{proof}
The commutator of a vertex operator with a residue of 
a vertex operator is again a vertex operator. Therefore, 
by induction, the equation for $\br^{(n)}$ has the form 
\begin{equation*}
\left[\br^{(n)}, \bal^{-}(\gamma)\right] = 
\textup{vertex operator} \,. 
\end{equation*}
One can see explicitly that this equation is
solved by a residue of a vertex operator. 
\end{proof}

\subsection{}

Also note that in the grading such that 
$$
\deg \bal = \deg \hbar = 1 \,, \quad \deg e = 2 
$$
the polynomial $P_n$ is homogeneous of degree 
$$
\deg P_n = n+2 \,.
$$

\chapter{Quantum multiplication for $\Mrn$}

We can now return to the formulas for quantum multiplication for $\Mrn$ using the computations of the last chapters.

\section{Explicit formulas}

Let us first state explicitly the operator for modified quantum multiplication by $c_{1}(\cO(1))$.  We will express them in
terms of the Heisenberg operators $\alpha_{k}^{(i)}(\pt)$ and $\alpha_{-k}^{(i)}(1)$ for $k > 0$ and $1 \leq i \leq r$.
These satisfy the commutation relations
$$[\alpha_{k}^{(i)}(\pt), \alpha_{-k}^{(j)}(1)] = - \delta_{i,j} k = \delta_{i,j}\cdot k \cdot \tau(\pt).$$

Up to a scalar operator, we have
\begin{equation*}
\bQ = \textup{Cubic} + \textup{Quadratic} + \textup{Purely Quantum}
\end{equation*}
where we have decompose the contribution of classical multiplication into cubic and quadratic expressions in the Heisenberg generators.
The formula for the cubic term is
\begin{equation*}
\textup{Cubic}
= \sum_{i=1}^{r} -\frac{1}{2}\sum_{n, m > 0} \left(t_1 t_2 \alpha^{(i)}_{-n}(1)\alpha^{(i)}_{-m}(1) \alpha^{(i)}_{n+m}(\pt) + \alpha^{(i)}_{-n-m}(1) \alpha^{(i)}_{n}(\pt) \alpha^{(i)}_{m}(\pt)\right)\,.
\end{equation*}
The classical quadratic term is
\begin{align*}
\textup{Quadratic} =
- \sum_{i=1}^{r} &\sum_{n>0} (t_1 +t_2) \cdot(a_{i} + \frac{1-n}{2})\cdot \alpha^{(i)}_{-n}(1) \alpha^{(i)}_{n}(\pt)\\
&+ \sum_{i < j} \sum_{n > 0} (t_1 + t_2) \cdot n\cdot \alpha_{-n}^{(j)}(1) \alpha_{n}^{(i)}(\pt)\,.
\end{align*}
The purely quantum term is
\begin{equation*}
\textup{Purely quantum} = (t_1+t_2)\sum_{n> 0} \frac{n q^{n}}{1-q^{n}} \cdot \beta_{-n}(1)\beta_{n}(\pt)\,,
\end{equation*}
where
$$\beta_{-n}(1) = \sum_{i=1}^{r} \alpha_{-n}^{(i)}(1) \textup{ and } \beta_{n}(\pt) = \sum_{i=1}^{r} \alpha_{n}^{(i)}(\pt)$$
are the Baranovsky operators.

We can determine the scalar discrepancy as follows.
For $r > 1$, there is no correction required.
For $r=1$, we need to add the scalar term
$$-(t_1+t_2) \frac{q}{1-q} \sum_{n> 0} \alpha_{-n}(1)\alpha_{n}(\pt).$$

This follows from the evaluation of  $\bQ\cdot 1$ which comes via the following 
lemma.
\begin{Lemma}
We have the following vanishing statement:
\begin{equation}
\beta_k(\pt)\cdot 1  = 0\,,\textup{ if } k\geq 2,\textup{ or } k=1\,, r\geq 2.
\end{equation}
\end{Lemma}
\begin{proof}
The dimension of the fiber of the punctual Baranovsky correspondence in \eqref{defZ} over a generic point of
$\cM(r,n)$ is
$$r\cdot k - 1$$
which is positive under the hypotheses of the Lemma.  Therefore, the pushforward of the fundamental class under
this projection vanishes.  
\end{proof}

\section{Generation statement}

As a corollary, we can deduce the following:

\begin{Theorem}
The divisor $c_1(\cO(1))$ generates the quantum cohomology ring 
of $\Mrn$. 
\end{Theorem}

\begin{proof}
It suffices to show that $\bQ(q,t_1,t_2,a_1,\dots,a_r)$ has distinct eigenvalues for 
generic values of the parameters.

First, notice that by taking the substitution
$$t_1 = t, t_2 = t^{-1}, a_i = t a_i$$
and studying the limit 
$$\bQ_0 = \lim_{t \rightarrow \infty} \frac{1}{t} \bQ$$
as $t \rightarrow \infty$,
we can ignore the cubic term, and show the remaining operator
has distinct eigenvalues.

For $n \geq 1$, let
$$V_n = \bigoplus_{i=1}^{r} \Q e_{n}^{(i)}\,.$$
We have an identification
$$\bF^{\otimes r} = \mathrm{Sym}^*(\bigoplus_{n}V_{n})$$
characterized by 
sending $\vac^{\otimes r}$ to $1$ and requiring
$\alpha_{-n}^{(i)}(1)$ to act by multiplication by $e_{n}^{(i)}$ on the right-hand side.
In other words, if we think of the left-hand side as $r$-tuples of partitions, then
the right-hand side is the decomposition into parts of size $k$.

We can decompose $\bQ_0$ in terms of $V_n$ as follows.
Let
\begin{align*}
A_n(q,a_1, \dots, a_r) = 
-n \sum_{i=1}^{r} &\left(a_i +  \frac{1-n}{2}\right)E_{ii} + n^2\sum_{i < j} E_{ji}\\
&+ \frac{n^2 q^{n}}{1-q^{n}} \sum_{i,j} E_{ji}
\end{align*}
be a matrix valued function acting on $V_n$,
where 
$E_{ji}$ is the matrix with $1$ in position $(j,i)$ and $0$ elsewhere.
We extend $A_n$ by zero to an operator on $\bigoplus V_{n}$ and, by the Leibniz rule, 
to a derivation $D(A_n)$ on $\mathrm{Sym}^*(\bigoplus V_n)$.
Then it follows from our formulas that
$$\bQ_0 = \sum_{n} D(A_n).$$

In particular, the eigenvalues of $\bQ_0$ are non-negative linear combinations of
the eigenvalues of $A_n$.  The nondegeneracy of the spectrum of $\bQ$ is a consequence of the following lemma.
\end{proof}

\begin{Lemma}
For very general values of $a_1, \dots, a_r$ and $q$, there is no nontrivial finite linear relation
\begin{equation}\label{spectrumrelation}
\sum_{n,i} c_{n,i} \gamma_{n}^{(i)} = 0
\end{equation}
between the eigenvalues $\{\gamma_{n}^{(i)}\}$ of $A_n(q, a_1, \dots, a_r)$, with $c_{n,i} \in \Q$.
\end{Lemma}
\begin{proof}

Suppose otherwise.  Then there exists such a relation that is valid for all values of parameters for which the operators $A_n$ are well-defined.
Let $n$ be the largest index appearing in the relation with some nonzero coefficient $c_{n,i}$.  

Fix a base point $p = [q=0, a_1, \dots, a_r] \in \C \times \C^{\times r}$ so that the $a_i$ are distinct.  The eigenvalues of $A_n(p)$ are
$$\gamma_{n}^{(i)}(p) = -n\left(a_i + \frac{1-n}{2}\right)\,, \quad i= 1, \dots, r.$$
Let $U \subset \C \times \C^{\times r}$ be the complement of the discriminant loci for $A_j(q, a_i)$ with $j \leq n$; each $A_j$ has nondegenerate spectrum over $U$. Since $p \in U$, we know that $U$ is nonempty.

Let $\zeta = e^{2\pi i/n}$ be a primitive $n$-th root of unity.   Choose an analytic path $\Gamma :[0,1) \rightarrow U$ such
that $\Gamma(0) = p$, 
$$\lim_{s \rightarrow 1} \Gamma(s) = (\zeta, a_1', \dots, a_r'),$$
and that $\Gamma$ meets the hypersurface $q = \zeta$ transversely at this limit point.
As $q \rightarrow \zeta$, the last term in the formula for $A_n$ dominates the others.  
Since the matrix $\sum_{i,j} E_{ij}$ has eigenvalues $\{1, 0, \dots, 0\}$, it follows from perturbation theory of linear operators that 
one of the eigenvalues of $A_n$ goes to infinity on the order of $\frac{1}{|q-\zeta|}$ as $s \rightarrow 1$, 
while the others grow at a slower rate (or remain bounded).
Without loss of generality, we can assume that it is $\gamma_{n}^{(1)}$.
Furthermore, for $j < k$, the operator $A_{j}$ has a well-defined limit as $q \rightarrow \zeta$, so its eigenvalues remain bounded.

Therefore, if we take the relation \eqref{spectrumrelation} along the path $\Gamma$, $\gamma_{n}^{(1)}$ dominates the 
other terms, so this forces its coefficient to vanish: 
$$c_{n,1} = 0.$$

For $1 < i \leq r$,
if we choose a permutation $\sigma$ of $1, \dots, r$ that sends $1$ to $i$, then we can choose a path from $p$ to $\sigma(p)$, contained in the hyperplane $q=0$, and concatenate with the path $\sigma(\Gamma)$ starting from $\sigma(p)$.  Under this concatenation, the eigenvalue $\gamma_{n}^{(i)}$ is now the dominant term, so this forces
$$c_{n,i} = 0$$
for all $i$.  This is a contradiction, so no nontrivial relation exists.
\end{proof}

\chapter{Gamma functions} 
\section{The bundle $\cVt$}

\subsection{} 

Recall that the main ingredient in the construction of the 
core Yangian $\bbY$ is the Chern character of 
$$
\cVt = \cV - \hbar^{-1} \otimes \bC^{-1} \cW  \,. 
$$
We begin by identifying this $K$-theory class for the 
moduli spaces of framed sheaves. 

Let $t_1,t_2$ denote the weights of the $G_\edge= GL(2)$ action on $\C^2$. Then 
$\hbar = t_1^{-1} t_2^{-1}$, written multiplicatively, 
and the equivariant Cartan matrix equals 
$$
\bC = (1 - t_1)(1-t_2)\,,
$$
as already discussed in Section \ref{s_slice_ex2}. If 
$$
\bw = \sum a_i 
$$
is the character of the framing space then 
\begin{equation}
\hbar^{-1} \otimes \bC^{-1} \cW  = \frac{\sum a_i}
{(1-t_1^{-1})(1-t_2^{-1})}  = \textup{character} \,\, 
H^0(\C^2,\cO^{\otimes r})
\label{WhMrn}
\end{equation}
where $GL(2)$ acts on $\C^2$ and $G_\bw = GL(r)$ acts
by automorphisms of the trivial bundle $\cO^{\otimes r}$. 
In gauge theory,  
$G_\bw$ is known as the group of constant gauge transformations. 

This gives us the following interpretation of $\cVt$.

\subsection{} 

In the sheaf language, the tautological bundle $\cV$ 
is interpreted as the bundle with fiber 
$H^1(\Pp^2,\cF(-1))$ over $\cF\in \cM(r)$, where 
$(-1)$ denotes twisting down by the line at infinity. 
We claim 
$$
\cVt = - H^0(\C^2,\cF)
$$
in $K$-theory of $\cM(r)$. 
Indeed, consider 
the following exact sequence 
of sheaves on $\Pp^2$ 
$$
0 \to \cF(-1) \to \cF(+\infty) \to 
\bigoplus_{d\ge 0} \cO_{\Pp^1}(d)^{\oplus r} \to 0\,,  
$$
where $\Pp^1=\Pp^2 \setminus \C^2$ is the line at 
infinity. From the corresponding 
long exact sequence and its 
special case $\cF=\cO$, we obtain 
$$
0\to H^0(\C^2,\cF) \to 
 H^0(\C^2,\cO^{\oplus r})
\to \cV \to 0\,, 
$$
as desired.

\section{Barnes'  $\Gamma$-function} 

\subsection{}\label{s_Gamma_M(r,n)}

Moduli spaces of framed sheaves provide a nice example 
of the $\Gamma$-function regularization from Section 
\ref{s_Gamma_gen}. In particular, the bundle \eqref{tN0}
for $\bw=\bw'=1$ specializes to the negative of 
\eqref{WhMrn} with $r=1$ 
and $a_1=1$. 

We have 
$$
 \textup{character} \,\, 
H^0(\C^2,\cO)^\vee = \sum_{i,j \ge 0} a^{-1} t_1^i t_2^j 
$$
thus, symbolically,  
\begin{equation*}
  c(H^0(\C^2,\cO)^\vee,u)
  = \textup{\huge ``}\prod_{i,j\ge 0}
  (u-a+t_1 i + t_2 j ) \textup{\huge ''} \,. 
\end{equation*}
This is regularized using 
 Barnes' multiple $\Gamma$-function (specifically, double 
$\Gamma$-function), see \cite{Ruij} for a modern reference,
with the result that 
\begin{equation}
 c(H^0(\C^2,\cO)^\vee,u) = \Gamma(u-a| \, t_1,t_2)^{-1} \,. 
\label{cGamma}
\end{equation}
Note that the same regularization (and, essentially, for the 
same reason) appears as the perturbative 
part of Nekrasov partition functions, see \cite{NekrSW}.  

\subsection{}\label{s_Barnes_zeta} 

By definition, 
$$
\log  \Gamma(u|\,t_1,t_2) = 
\frac{\partial}{\partial s} \zeta(s,u|\,t_1, t_2) \Big|_{s=0} \,,
$$
where
$$
\zeta(s,u\,|\,t_1, t_2) = \frac{1}{\Gamma(s)} 
\int_0^\infty \frac{dz}{z} \, z^s \, \frac{e^{-u z}}
{(1-e^{-t_1 z})(1-e^{-t_2 z})} \,, \quad \Re s > 2 \,. 
$$
An asymptotic expansion of $\zeta(s,u\,|\,t_1, t_2)$ 
as $u \to +\infty$ may be obtained by expanding 
$$
\frac{1}{(1-e^{-t_1 z})(1-e^{-t_2 z})} = \sum_{k\ge -2} 
z^k \, \ch_k H^0(\C^2,\cO)
$$
and integrating term-wise to get
\begin{equation}
  \zeta(s,u\,|\,t_1, t_2) = \sum_{k\ge -2} \frac{\Gamma(s+k)}
{\Gamma(s) u^{s+k}} \, \ch_k H^0(\C^2,\cO) \,. 
\label{exp_zeta}
\end{equation}
Since 
$$
\frac{\partial}{\partial s} \left. 
\frac{\Gamma(s+k)}{\Gamma(s) \, u^{s+k}}  \right|_{s=0}
=(-1)^{k+1} \ln^{(k)} u \,, 
$$
 this verifies the agreement between \eqref{defcL} and 
\eqref{cGamma}. 

\section{The matrix $\bRh$}

\subsection{} 

For $\bw=a_1$ and $\bw'=a_2$ the 
$\Gamma$-factor from \eqref{Rh_gen} specializes to
\begin{multline}
  \Gamma(u\,|\,\bw,\bw') = \frac{c(H^0(\C^2,\cO)^\vee,u-\hbar)}
  {c(H^0(\C^2,\cO)^\vee,u)} = 
\\ = 
\frac{\Gamma(u|\,t_1,t_2)}
  {\Gamma(u+t_1+t_2\,|\,t_1,t_2)} = u \, \Gamma(u|\, t_1) 
\Gamma(u|\,
  t_2) \,, 
\end{multline}
where $u=a_1-a_2$ and $\Gamma(u\,|\,t_1)$ is the single 
Barnes's $\Gamma$-function, defined 
similarly\footnote{
It is 
related to Euler's $\Gamma$-function by 
$$
\Gamma(u \,|\, t_1)=
\frac{\exp((u/{t_1}-1/2) \ln t_1)}{\sqrt{2\pi}} \,
\Gamma(u/t_1)\,. 
$$}. 
We define $\bRh = \Gamma(u\,|\,\bw,\bw') \, \bR$.

\subsection{Zero modes and the singular part of $\bRh$} 

{} From \eqref{exp_zeta}, or the Stirling formula, we compute
\begin{multline}
  \frac{1}{\hbar} \ln \Gamma(u-a\,|\,\bw,\bw')  =  
  \tau(1) \, \ln^{(-1)}u - \tau(a) \, \ln u  \\
+
\left(\frac12\tau(a^2)  - \frac1{12}\right) \frac1{u} 
+ O\left(\frac1{u^2}\right) 
\label{lnGamma_Stir} 
\end{multline}
as $u\to\infty$. This gives the following identification of 
the central operators $\bc_{-2}$ and $\bc_{-1}$ from 
Section \ref{gamma_sing}. Write $M_{\varnothing,\varnothing}$
for the $\vacd{\bw} \, \cdot \, \vacv{\bw}$ vacuum 
matrix element of an operator 
$M$ corresponding 
to $\bw =1$. Then 
$$
\left(\bc_{-2}\right)_{\varnothing,\varnothing} = \tau(1) \, r \,,
\quad 
\left(\bc_{-1}\right)_{\varnothing,\varnothing} = \beta_0(1)\,, 
$$
where $r=\bv$ is the rank and 
$$
\beta_0   = \sum_{i=1}^r 1 \otimes \cdots \otimes \alpha_0
\otimes \dots \otimes 1 
$$
is the $0$th Baranovsky operator.  Here and in what follows 
we identify 
\begin{equation}
  \label{zeroFr}
H^\hd_\bG\left(\cM(r)^\bA\right) \otimes \bK \cong 
\bF(a_1) \otimes \cdots \otimes \bF(a_r) \,,
\end{equation}
generalizing \eqref{zeroF2} to arbitrary rank. Thus the zero modes appear
in the Yangian. 

Note by construction the operators 
$\left(\bc_{-i}\right)_{\varnothing,\varnothing}$ have 
the same span as the operators $\ch_i \cVt$ for $i\in \{-2,-1\}$.

\subsection{} 

We stress that in what follows we adopt the identification 
\eqref{zeroFr} and that, for now on, all formulas
involving $\bal$ include the zero modes.

\subsection{}

Similarly, consider the vacuum-vacuum matrix element of 
the regular part $\bRh_{\reg}$ of $\bRh$, as in Section \ref{gamma_sing}. 
The new terms coming from \eqref{lnGamma_Stir} give 
\begin{equation*}
\frac{1}{\hbar}\left[ \frac1u \right]  \left(\bRh_{\reg}
\right)_{\varnothing,\varnothing}  = \sum_{i=1}^r 
 1 \otimes \cdots \otimes \widehat{L}_0
\otimes \dots \otimes 1 
\end{equation*}
where 
$$
\widehat{L}_0 = 
\tfrac12\int \nord \bal^2 \nord (1) - \tfrac1{12}\,,
$$
where we keep the zero modes, 
compare with \eqref{br_vert}. Note the familiar $\zeta(-1)=-\frac1{12}$
term.  

\subsection{}

Recall the classical $\br$-matrix \eqref{brMrn} 
and note its matrix elements gave 
$$
\beta_{n}(1), \beta_{-n}(\pt) \in \bY(\glh)\,, \quad n > 0 \,. 
$$
Since the core Yangian is an algebra over $\bk[\bdel^{-1}]$ where 
$$
\bdel = t_1 t_2\,, 
$$
we have 
$$
\beta_{-n}(1) = \bdel^{-1} \,  \beta_{-n}(\pt) \in \bbY(\glh) \,.
$$

\chapter{Core Yangian modulo $\hbar$}

\section{Semiclassical $R$-matrix} 

\subsection{}

Since $\hbar=-t_1-t_2$ does not divide $\bdel = t_1 t_2$
we may study $\bbY$ modulo $\hbar$, which leads to great 
simplifications. 

Define the semiclassical $R$-matrix $\rsc$ by 
$$
\bRh(u) = 1 + \hbar \, \rsc(u)  +  O(\hbar^2) \,. 
$$
Modulo $\hbar$, the generators of $\bbY$ are primitive and 
act by matrix 
coefficients of $\rsc$. 

The Yang-Baxter equation becomes the 
classical Yang-Baxter equation for $\rsc$. It implies the 
generators of $\bbY/\hbar \bbY$ form a Lie algebra $\gsc$ and 
$$
\bbY/\hbar \bbY \cong \cU(\gsc) \,.
$$

\subsection{}

The Lie algebra $\gsc$ may be described explicitly by its 
action in the basis of stable envelopes of
$\cM(r)^\bA$, where 
$$
\bA \subset SL(2) \times GL(r) 
$$
is a maximal torus. Since $\cM(r)^\bA$ is finite, the 
classes of $\bA$-fixed points form an eigenbasis for 
operators of classical multiplication. 

In $\bA$-equivariant 
cohomology, stable envelopes are proportional to fixed points, 
and thus diagonalize operators of classical multiplication. 
Steinberg correspondences act nicely 
in this basis by the general principles explained 
in Section \ref{s_tStein}. 

\subsection{}

The fixed points of the maximal torus of $SL(2)$ on the Hilbert 
schemes are Nakajima varieties of type $A_\infty$, see 
in particular Section \ref{s_coverAinf}. 
We will see a close connection between $\gsc$ and the corresponding 
Lie algebra $\gl(\infty)$.

\section{Stable basis for $\Hilb_n$} 

\subsection{}

The stable basis for $\cM(1)=\Hilb$ is 
identified as follows. Let 
$$
\left\{
  \begin{pmatrix}
    z \\
& z^{-1} 
  \end{pmatrix}
\right \} \subset SL(2) 
$$
be the standard maximal torus. 
To  match standard symmetric functions 
conventions, we choose the $z\to\infty$ chamber, 
that is, 
$$
\fC = \{ u <0   \} \,,
$$
where $u = \log z$. 
The other choice may be obtained by a permutation 
of coordinates. 

A subscheme of $\C^2$ has a $z\to \infty$ limit if and only if
it is  set-theoretically 
supported on the $x_2$-axis 
$$
\ell_2 = \{x_1 = 0 \} \,.
$$ 
In particular,  the stable basis must be a $\Q$-linear combination of 
the Nakajima descendents of the $x_2$-axis 
$$
p_\mu = \prod \alpha_{-\mu_i}(\ell_2) \, \vacv{} \,.
$$
The notation is chosen to agree with the traditional map of 
the equivariant cohomology of the Hilbert scheme 
to symmetric function that takes
\begin{equation}
\alpha_{-k}(\ell_2) \mapsto \textup{multiplication by $p_k$} \,. 
\label{altop}
\end{equation}

\subsection{}

Recall the sign-twisted inner product on cohomology from Section 
\ref{s_signs_adj} and transport it to symmetric functions 
using \eqref{altop}. This gives  the Jack inner product on symmetric functions 
$$
\left[ p_k^\tau, p_l \right] = \delta_{kl} \, k \, (-t_1/t_2) 
$$
with parameter $-t_1/t_2$. In \cite{Macdonald}, this parameter is denoted 
$\alpha$. 

Gram-Schmidt orthogonalization 
of monomial symmetric function $m_\lambda$ 
with respect to this inner product gives, 
by definition, the basis of Jack symmetric functions. We define
$$
\bJ_\lambda = t_2^{|\lambda|} \,\cdot \,  \textup{integral Jack polynomial
as in \cite{Macdonald}}\,. 
$$
This is normalized so that 
\begin{equation}
\bJ_\lambda = \prod_{\square\in \lambda} (t_2(l(\square)+1) - t_1 \, a(\square))
\, m_\lambda + \dots \label{defbJl}
\end{equation}
and is a polynomial in $t_1,t_2$ of degree $|\lambda|$. Here 
$$
a(\square)=\lambda_i-j\,, \quad l(\square)=\lambda'_j-i
$$ 
denote the 
arm- and leg-length of a square $\square=(i,j)$ in the diagram $\lambda$. 
Note that the product in \eqref{defbJl} is the Euler class of 
$N_+$ at the monomial ideal 
\begin{equation}
\sI_\lambda = \left(x_1^{\lambda_i} x_2^{i-1}\right)_{i=1,2,\dots} \, 
\in \Hilb \,. 
\label{defsI}
\end{equation}

\subsection{}

The following is well-known and is a a consequence of the orthogonality 
of classes of fixed points $[\,\sI_\lambda]$ in cohomology 

\begin{Proposition}[\cite{NakJack,VassHilb,LQW}] The map \eqref{altop} sends $[\,\sI_\lambda]$ to 
$\bJ_\lambda$. 
\end{Proposition}

\subsection{}

Let us  polarize $\Hilb^\bA$ by the Euler class of $N_-$. We then 
have the following 

\begin{Proposition}
The map \eqref{altop} sends the stable envelope of 
$\,\sI_\lambda$  to the Schur function $s_\lambda$. 
\end{Proposition}

\begin{proof}
Schur functions are triangular with respect to $\bJ_\lambda$ and
proportional to them modulo $\hbar$. This shows stable envelopes
are proportional to Schur functions. By \eqref{defbJl} we have
$$
\bJ_\lambda = e(N_+) \, s_\lambda + \dots\,,
$$
which fixes the normalization. 
\end{proof}

\section{Differential operators on $\C^\times$ and $\gl(\infty)$}

\subsection{} 

Let $e_a$ denote the function 
\begin{equation}
e_a(x) = e^{ax} \,. 
\label{defea}
\end{equation}
Let $\ve\in\C^\times$ be a parameter and consider 
$$
\cD_\textup{assoc} = \C \lang D,e_{\pm\ve} \rang \,, \quad 
D = \frac{d}{dx} \,.
$$
It may be identified with 
differential operators on $\C^\times$ via the map $z = e_\ve$. 
The parameter $\ve$ may be scaled away but it will be 
convenient to keep it. We denote by 
$$
\cD = \left(\cD_\textup{assoc}\right)_\textup{Lie}
$$
the same algebra viewed as a Lie algebra.

The center of $\cD$ is spanned by $1\in \cD_\textup{assoc}$ 
which we denote by $D^0$ to avoid confusion.

\subsection{}

The natural 
action of $\cD$ on $e_s \, \C[e_{\pm\ve}]$, $s\in \C$, 
 gives a family of 
embeddings 
$$
\rho_s:\cD \hookrightarrow  \gl(\infty)
$$
into the Lie algebra $\gl(\infty)$ of all infinite matrices
with finitely many nonzero diagonals. Its image is the
unipotent Jordan block of the automorphism of $\gl(\infty)$ 
that corresponds to the shift of the Dynkin diagram. 

The diagram shift automorphism is the deck transformation of 
the universal cover of the quiver with one vertex and one 
loop. From this point of view, the description of $\cD$ 
as automorphism-finite vectors in $\gl(\infty)$ is intrinsic, 
while its identification with differential operators is less so.

The Lie algebra 
$\gl(\infty)$ has a central extension
$\widehat{\gl(\infty)}$
which may be pulled back to a central extension
\begin{equation}
0 \to \C \cb \to \cDh \to \cD \to 0  \,. 
\label{centD}
\end{equation}
This extension does not depend on  $s$.  

Representation theory of $\cDh$ was studied by Kac and Radul \cite{KacRadul} 
and many others. Here we will see the simplest representations:  
those obtained from the 
half-infinite wedge representations of $\gl(\infty)$.

\subsection{}

By construction, the representation $\pi_s = \bigwedge^{\infty/2} \rho_s$
is the $\cDh$ module with basis 
\begin{equation}
\vacv{\lambda;s} = \bigwedge_{i=1}^\infty e_{(\lambda_i -i)\, \ve +s} \,,
\label{veclam}
\end{equation}
where 
$$
\lambda = \lambda_1 \ge \lambda_2 \ge \dots \ge 0 
$$
is a partition. Usual rules of linear algebra give a
well-defined answer for the action of 
the off-diagonal elements of $\cD$ in this basis. For the 
diagonal elements, it is convenient to use the $\zeta$-regularization 
$$
\textup{\huge ``}\sum_{i=1}^\infty  
((\lambda_i-i)\, \ve +s)^k \textup{\huge ''} = k! \, [x^k] \, e_s
\, 
\sum_{i=1}^\infty  e_{(\lambda_i-i)\, \ve} \,,
$$
where $e_a=e_a(x)$ as in \eqref{defea}. 
Note 
$$
\sum_{i=1}^\infty  e_{(\lambda_i-i)\ve}  = 
\frac{1}{e_{\ve}-1} + 
\sum_{i=1}^\infty  \left[e_{(\lambda_i-i)\,\ve}-e_{-i\ve}\right]
$$
where the second term is a Laurent polynomial in $e_\ve$. In 
particular, 
$$
\pi_s(D^0) = \frac{s}{\ve}-\frac12 \,. 
$$
The central extension \eqref{centD} is normalized so that 
$$
\pi_s(\cb) = 1  \,.
$$

\subsection{}

For $\sI_\lambda$ as in \eqref{defsI} we have 
$$
\ch \cVt \Big|_{\, \sI_\lambda}  = - \frac{e^{a}}{1-e^{-t_1}} 
\sum_{i=1}^\infty e^{-\lambda_i t_1 - (i-1) \,t_2} 
$$
where $a$ is the framing weight and $t_1,t_2$ are the tangent 
weights of the two coordinate axes. We see that if
$$
t_1 = - t_2 = - \ve 
$$
then the map 
\begin{equation}
F(a) \owns \Stab \left[\sI_\lambda\right] \mapsto \vacv{\lambda; a+ \ve/2} 
\label{Ftwedge}
\end{equation}
identifies 
\begin{equation}
\ch \cVt  = \frac{\cb}{\ve(e^{\ve/2}-e^{-\ve/2})} +
\frac1{e^{\ve/2}-e^{-\ve/2}} \, \exp D \,. 
\label{chexpD}
\end{equation}
Here $\exp(D)$ is a generating function for the operators 
$D^k\in \cDh$, in other words 
$$
\pi_s\left(\exp D\right) = \sum_{k\ge 0} \frac{1}{k!} \,
 \pi_s\!\left(D^k\right) \ne \exp(\pi_s(D)) \,.  
$$


\subsection{}

Generalizing \eqref{chexpD}, we have

\begin{Proposition}
The identification \eqref{Ftwedge} gives 
$$
\gsc \cong \cDh \,. 
$$
\end{Proposition}

\begin{proof}
It remains to check that it takes
$$
\alpha_{-k}(\ell_2) \mapsto e_{\epsilon k} \in \cDh \,,
$$
which is easy. For example, mapping both sides
of \eqref{Ftwedge} to the Schur function $s_\lambda$, this 
becomes the classical rule for multiplication of Schur 
functions by power-sum functions. 
\end{proof}

\section{Pl\"ucker relations}

\subsection{} 

Let $\psi_a$ be the operator of wedge product by $e_a$
$$
\psi_a \, v = e_a \wedge v 
$$
and let $\psi_a^*$ be the adjoint operator with respect 
to inner product in which the vectors \eqref{veclam} are orthonormal. 
More canonically, the operators $\psi_a^*$ are associated
to bases of representations dual to $\rho_s$. 

\subsection{}

Consider the operator 
$$
\Omega = \sum_{a\in s+\Z \ve} \psi_a \otimes \psi^*_a 
$$
which depends only on the $\Z\ve$-coset of $s$. It 
defines a map 
$$
\Omega : \pi_s \otimes \pi_{s'} \mapsto \pi_{s+\ve} \otimes \pi_{s'-\ve} 
$$
provided
$$
s' \equiv s \! \! \mod \Z \ve \,.
$$
This map commutes with 
$\gl(\infty)$ and, hence, with $\cDh$. 

\subsection{}

Classically, $\Omega$ is used to describe the image of the 
natural embedding 
$$
GL(V) \hookrightarrow GL(\Lambda^\hd V) \,, 
$$
where $V$ a vector space, which for simplicity can be 
assumed to be finite-dimensional, see \cite{KacBook,MJD}. Matrix elements of 
$g\in GL(V)$ acting on $\Lambda^\hd V$ are the minors of $g$.

Commutation with $\Omega$ gives quadratic relations for minors 
of $g$, analogous to the better known Pl\"ucker relations 
among maximal minors of a rectangular matrix (that is, among 
the Pl\"ucker coordinates on the Grassmann variety). Here 
we use the term \emph{Pl\"ucker relations} in the broader sense.

\subsection{}

We denote by 
\begin{align*}
  \bE(\lambda,\mu; s, u) & = \vacd{\mu;s} \, \rsc(u) \vacv{\lambda;s}\\
  &=\sum_{k\ge -1} \bE(\lambda,\mu; s)_k \, \ln^{(k)}(u) \,.
\end{align*}
matrix elements of $\rsc$ in the first (by convention) tensor factor. 
Here $\bE(\lambda,\mu; s)_k \in \gsc$ and the singular 
central terms 
$$
\ch_{-2} \cVt = \frac{\cb}{\ve^2}\,, \quad 
\ch_{-1} \cVt = \frac{D^0}{\ve} 
$$
are only present if $\lambda=\mu$. By construction, 
$\bE(\lambda,\mu; s, u)$ only depend on $u+s$ in the sense 
that 
\begin{equation}
\forall t \quad \bE(\lambda,\mu; s+t, u-t)  = \bE(\lambda,\mu; s,u) \,.
\label{translt}
\end{equation}

\subsection{}
By construction, $\bE(\lambda,\mu; s)_k$ generate $\bbY/\hbar\bbY$ and 
all relations between these generators are linear. Among them are the 
 Pl\"ucker relations, which become linear 
 \begin{equation}
\left[\xi \otimes 1 + 1 \otimes \xi, \Omega\right] = 0 \,, \quad 
\xi \in \gsc \,, 
\label{linPl}
\end{equation}
at the Lie algebra level.

\begin{Proposition}\label{p_Pluck} 
Pl\"ucker relations and \eqref{translt} span all 
linear relations among matrix elements of $\rsc$. 
\end{Proposition}

\noindent 
This statement is a variation on the classical 
theme. For convenience, 
we give a proof.

\subsection{}
We divide the proof of Proposition \ref{p_Pluck} into a sequence of lemmas.

\begin{Lemma}\label{l_Pluck} 
Suppose $\psi_a^*  \vacv{\lambda} \ne 0$ and 
$\vacv{\mu} \ne  \psi_b \, \psi^*_a \, \vacv{\lambda}$ for all $b$. 
Then 
$$
\vacd{\mu} \xi \vacv{\lambda} =  \vacd{\mu} 
\psi_a \, \xi \, \psi^*_a \vacv{\lambda}
$$
for all $\xi\in \gsc$. 
\end{Lemma}

\noindent 
Note the hypothesis of the Lemma implies $\mu\ne\lambda$. 

\begin{proof}
Expand 
\begin{equation}
0= \vacd{\mu,\lambda} 1 \otimes \psi_a 
\,\, \Big[\xi \otimes 1 + 1 \otimes \xi, \Omega\Big] \,\, 
\psi^*_a \otimes 1 \, \vacv{\lambda,\lambda}   
\label{Plucc}
\end{equation}
where $\vacv{\mu,\lambda} = \vacv{\mu} \otimes \vacv{\lambda}$. 
\end{proof}

\begin{Corollary} Pl\"ucker relations imply 
 $$
 \bE(\lambda,\mu; s, u) \ne 0 \Longrightarrow \vacv{\mu} = 
\psi_b \, \psi^*_a \, \vacv{\lambda}
$$
for some $a,b\in s + \Z \ve$. 
\end{Corollary}

\noindent 
In the language of Chapter 
\ref{s_Grassmann}, this means the corresponding points 
of the half-infinite Grassmannian must lie on a line.

\begin{proof}
Otherwise, we can find $a$ in Lemma \ref{l_Pluck} 
such that $\vacd {\mu} \psi_a  \, =  0$. 
\end{proof}

\subsection{}

Let $\lambda\ne \mu$ lie on a line, which means that 
there exists $k,l\in \Z$ such that 
$$
\{k\} = \fS(\lambda) \setminus \fS(\mu)\,,
\quad
\{l\} = \fS(\mu) \setminus \fS(\lambda)\,,
$$
where $\fS(\lambda) = \{\lambda_i-i\} \subset \Z$. 
Using Lemma \ref{l_Pluck}, we can add or remove elements 
in $\fS(\lambda) \cap \fS(\mu)$, which means there exists
$\bE_{kl}(s,u)$ such that
$$
\bE(\lambda,\mu; s, u) = \pm \bE_{kl}(s, u)
$$
with the sign determined from the action of the operators
$\psi_a^*$ in the 
basis $\vacv{\lambda}$. 
Lemma \ref{l_Pluck} further implies  
\begin{align}
  \bE_{k+1,l+1}(s, u) &= \bE_{kl}(s+\ve, u) \notag \\
& = \bE_{kl}(s, u+\ve) \,, \label{shiftEkl} 
\end{align}
where the second step is based on \eqref{translt}. 

\subsection{}

Now consider diagonal matrix elements of $\rsc$. Here we have
the following

\begin{Lemma}\label{l_Pluck2} 
Suppose $\psi_a^*  \vacv{\lambda} \ne 0$ and 
$\psi_a^*  \vacv{\mu} \ne 0$. 
Then 
\begin{equation}
\vacd{\lambda} \xi \vacv{\lambda} - 
\vacd{\lambda} \psi_a \, \xi \, \psi^*_a  \vacv{\lambda} = 
\vacd{\mu} \xi \vacv{\mu} - 
\vacd{\mu} \psi_a \, \xi \, \psi^*_a  \vacv{\mu} 
\label{difflammu}
\end{equation}
for all $\xi\in \gsc$. 
\end{Lemma}

\begin{proof}
Expand 
%
$\displaystyle 
\vacd{\lambda,\mu} 1 \otimes \psi_a 
\,\, \Big[\xi \otimes 1 + 1 \otimes \xi, \Omega\Big] \,\, 
\psi^*_a \otimes 1 \, \vacv{\lambda,\mu}   \,. 
$
%
\end{proof}

We denote the 
difference of the matrix elements in  \eqref{difflammu} by $\bE_{kk}(s,u)$, 
where $a=k \ve + s$. For example, 
$$
\bE_{0,0}(s,u) = \bE(\varnothing,\varnothing; s+\ve, u) - 
\bE(\varnothing,\varnothing; s, u) \,. 
$$
One can choose a different parameter $s'\in s + \Z\ve$ for $\mu$ in 
\eqref{difflammu}
which shows the relation \eqref{shiftEkl} is valid for $k=l$. 

\subsection{}

Symbolically, Lemma \ref{l_Pluck2} and \eqref{shiftEkl} shows
$$
\bE(\lambda,\lambda; s, u)  = \textup{\huge ``} \sum_{k \in \fS(\lambda)} 
\bE_{00}(s,u+k\ve) \textup{\huge ''}\,. 
$$
A better way to write this relation is the following. 

For each partition $\lambda$, define 
$$
\crn_\lambda: \Z \to \{\pm 1,0\} 
$$
as the difference of the following indicator functions 
$$
\crn_\lambda = \sum_{\textup{inner corners $\square$}} 
\delta_{c(\square)} - \sum_{\textup{outer corners $\square$}} 
\delta_{c(\square)}\,.
$$
Here $c(\square)=j-i$ is the content of the square $\square = (i,j)$. 
This may also be defined using the identity 
$$
\sum_k \crn_\lambda(k) \, t^k = (t-1) \sum t^{\lambda_i -i}  \,. 
$$

\begin{Lemma}
$$
\bE(\lambda,\lambda; s, u) = 
\sum_k \crn_\lambda(k) \,\, \bE(\varnothing,\varnothing; s, u+k\ve) \,.
$$
\end{Lemma}

\begin{proof}
Follows from Lemma \ref{l_Pluck2} and \eqref{shiftEkl}. 
\end{proof}

\subsection{}

\begin{proof}[Proof of Proposition \ref{p_Pluck}] 
Previous lemmas reduce the matrix elements of $\rsc$ to shifts in $u$ 
of the operators 
$\bE(\varnothing,\varnothing; s, u)$ and 
$\bE_{k0}(s,u)$, $k\ne 0$. 

The algebra $\gsc$ is graded by eigenvalues 
of the adjoint action of $D$, this is the grading by the difference 
$k-l$ of $\bE_{kl}$. Each graded piece is further filtered by the 
degree in $u$, with $1$-dimensional factors. This shows there are no 
further linear relations among the coefficients of $\bE(\varnothing,\varnothing; s, u)$ and 
$\bE_{k0}(s,u)$, $k\ne 0$.
\end{proof}

\subsection{}

The factorization of Section \ref{s_coverAinf} gives the 
following formula for the semiclassical $R$-matrix  
\begin{equation}
\rsc  = \sum_{i,j,k\in\Z} \frac{E_{ij} \otimes E_{j+k,i+k}}{u-k \ve} 
\label{Rscsum}
\end{equation}
in terms of the classical 
$R$-matrix 
$$
\br_{\gl(\infty)}  = 
\sum_{i,j\in\Z} E_{ij} \otimes E_{ji} 
$$
for $\gl(\infty)$. The operator \eqref{Rscsum} acts in 
half-infinite wedge representations of $\gl(\infty)$ via
the $\zeta$-regularization discussed in Section \ref{s_Barnes_zeta}. 
%
%

It is instructive to retrace the steps of the above proof with 
this explicit formula.

\chapter{The Yangian of $\glh$}

\section{Generators of the core Yangian}

\subsection{}

By Theorem \ref{t_coreY}, $\bbY$ is generated by the Baranovsky 
operators $\beta_n$ and $\ch_k \cVt$ for $k=-2,-1,\dots$. 
Here, for brevity, we write $\beta_n = \beta_n(1)$.

The following theorem shows it suffices to add a 
 single operator $\ch_1 \cVt$ to the Baranovsky operators
to generate the Yangian.

\begin{Theorem}\label{t_gen_c1} 
The core Yangian $\bbY$ is generated by the Baranovsky operators
$\beta_{\pm 1}$, and the operator of cup product by 
$$
\bQh = \ch_1 \cVt 
$$
of cup product by $\ch_1$ of the bundle $\cVt = - H^0(\C^2,\cF)$\,. 
\end{Theorem}

\begin{proof}
Follows from the corresponding statement modulo $\hbar$. 
\end{proof}

\subsection{}

The generation statement can be made more effective 
using the the following geometric fact. Parallel 
results were proven by M.~Lehn for the cohomology of Hilbert schemes 
and by O.~Schiffmann and E.~Vasserot for the $K$-theory of $\cM(r)$. 

\begin{Proposition}\label{p_commutator} For any $k$ and $l$, 
\begin{equation}
\left[\ad(\bQh)^k \beta_1,\ad(\bQh)^l \beta_{-1}\right]
\label{commbeta1}
\end{equation}
is an operator of classical multiplication. 
\end{Proposition}

\begin{proof}
Recall that the Baranovsky operators $\beta_{\pm 1}$ are defined 
using the correspondence 
$$
\fB_1 =\{(\cG,x,\cF)\} \subset \cM(r,n+1) \times \C^2 \times \cM(n) 
$$
formed by exact sequences 
\begin{equation}
0 \to \cG \to \cF \to \cO_x \to 0 \,. 
\label{extOG}
\end{equation}
On this correspondence, we have a 
tautological line bundle $\cF/\cG$ and the 
action of $\ad(\bQh)$ introduces a factor of 
$$
c_1(\cF/\cG) =  - 
\ch_1(H^0(\C^2,\cG))+ \ch_1(H^0(\C^2,\cF))
 \in H^2_\bG(\fB_1)\,. 
$$
Therefore
\begin{multline*}
\left(\ad(\bQh)^k \beta_1(\gamma)\right) \circ 
\left(\ad(\bQh)^l \beta_{-1}(\gamma')\right)  = 
 \\
(-1)^{r}(\pi_{13})_* \left(
 (-c_1(\cF_1/\cG))^k \, c_1(\cF_2/\cG)^l \, 
\pi^*_{45}(\gamma\times \gamma') \right) 
\end{multline*}
where $\pi_{ij}$ are the projections to respective factors 
in the correspondence 
$$
\{(\cF_1,\cG,\cF_2,x_1,x_2)\} \subset \cM(r,n) \times \cM(r,n+1) 
\times \cM(r,n) \times \C^2 \times \C^2 
$$
in which $\cF_i/\cG \cong \cO_{x_i}$. The $(-1)^r$ factors comes from 
our sign conventions, see Section \ref{s_sign_Baran}.  

The product $\ad(\bQh)^l \beta_{-1} \ad(\bQh)^k \beta_1$ 
in the opposite order is, similarly, computed by pushing forward
$$
(-1)^r \, c_1(\cG'/\cF_1)^l (-c_1(\cG'/\cF_2))^k
$$
along the $\cG'$-factor in the correspondence defined by 
$$
\cG'/\cF_i = \cO_{x_{3-i}}\,,  \quad i=1,2\,. 
$$
We now note that outside of the diagonal $\cF_1 \cong \cF_2$ 
the two correspondences are canonically isomorphic, 
because necessarily 
$$
\cG' = \cF_1 + \cF_2\,, \quad \cG = \cF_1 \cap \cF_2\,,
$$
as subsheaves of the common double dual 
$\cF_1^{\vee\vee} = \cF_2^{\vee\vee}$. Clearly, 
$$
\cF_i/\cG \cong \cG'/\cF_{3-i}
$$
which identifies the integrands and shows the commutator 
\eqref{commbeta1} is supported on the diagonal $\cF_1\cong\cF_2$. 
This means it is an operator of classical multiplication. 
\end{proof}

\noindent From the proof above we have the following 

\begin{Corollary}
\begin{equation*}
\left[\ad(\bQh)^k \beta_1,\ad(\bQh)^l \beta_{-1}\right] = 
(-1)^k \left[\beta_1,\ad(\bQh)^{k+l} \beta_{-1}\right]
\end{equation*}
\end{Corollary}

\subsection{}

The commutator in Proposition 
\ref{p_commutator} can be explicitly identified. 
We do it using equivariant localization following \cite{SV1}. 

To set up equivariant localization, we need to 
identify the the normal bundle to $\fB_1$. 
We have the following 

\begin{Proposition}
The tangent bundle to $\fB_1$ fits into an exact
sequence of the form 
\begin{multline}
  0 \to T  \fB_1 \to 
T \cM(r,n+1) \oplus T \cM(r,n) 
 \to \\
\to \Ext^1(\cG,\cF(-1)) \to \C(-\hbar) \to 0 \,,
\label{tang_Bar} 
\end{multline}
where $\C(-\hbar)$ is the trivial bundle with equivariant 
weight $-\hbar$. 
\end{Proposition}

\noindent
In particular, $\fB_1$ is smooth, which is a special 
case of Theorem 5.7 in \cite{Nak98}. The sequence
\eqref{tang_Bar} may also be found there for more 
general Hecke correspondences among Nakajima 
varieties. 

\begin{proof}
Let 
$$
\xi = (\xi_\cG,\xi_\cF) \in \Ext^1(\cF,\cF(-1)) \oplus
  \Ext^1(\cG,\cG(-1))
$$
be a tangent vector to $\cM(r,n+1)\times \cM(r,n)$. 
A sheaf homomorphism (in our case, inclusion) 
$$
\phi: \cG \to \cF 
$$
deforms with $\xi$ to first order when the commutator 
\begin{equation}
[\xi,\phi]= \xi_\cF \, \phi - \phi \, \xi_\cG \in \Ext^1(\cG,\cF(-1))  
\end{equation}
vanishes.  Here 
$$
\Ext^i(\mathcal{A},\mathcal{B}) \otimes 
\Ext^j(\mathcal{B},\mathcal{C}) \to \Ext^{i+j} (\mathcal{A}, \mathcal{C})
$$
is the usual composition of $\Ext$ groups. Note that
$$
\rk \Ext^1(\cG,\cF(-1)) = 2rn+r\,, 
$$
while 
$$
\dim \cM(r,n+1) \times \cM(r,n) - \dim \fB_1 = 2rn+r-1 \,.
$$
In fact, the obstruction $[\xi,\phi]$ to deforming $\phi$ 
lies in the following 
corank 1 subbundle of $\Ext^1(\cG,\cF(-1))$.

For every deformation of $\cF$ there is some deformation
of $\cG\subset\cF$. This means  
the image of $\xi_\cF \mapsto \xi_\cF \, \phi$ lies in the image of 
$\xi_\cG \mapsto \phi \, \xi_\cG$ and hence the obstruction 
$[\xi,\phi]$ lies in the image of the 
first arrow in the following piece of the long exact sequence 
\begin{align*}
&\Ext^1(\cG,\cG(-1)) 
\to 
\Ext^1(\cG,\cF(-1)) 
\to 
\Ext^1(\cG,\cO_x) \to 
\\
\to & \Ext^2(\cG,\cG(-1))\,. 
\end{align*}
By Serre duality, 
$$
\Ext^2(\cG,\cG(-1)) = 0, 
$$
while 
$$
\Ext^1(\cG,\cO_x)^\vee \otimes \cO(-\hbar) = \Ext^1(\cO_x,\cG) \,. 
$$
We have
$$
\Ext^1(\cO_x,\cG)\big|_{\fB_1} \cong \C\,, 
$$
canonically trivialized by 
the class of the extension \eqref{extOG}. This gives the 
exact sequence stated. 
\end{proof}

\subsection{}

Suppose we are at a fixed point $(\cG,0,\cF)\in \fB_1$ of the torus action. 
Consider a free resolution of $\cF$ and its restriction to $\C^2$
\begin{equation}
0 \to \bigoplus \cO_{\C^2}(w_i) \to \bigoplus \cO_{\C^2}(v_i) \to 
\cF\Big|_{\C^2} \to 0 
\label{resolut_G}
\end{equation}
where $v_i,w_j \in (\Lie \bG)^*$ are the equivariant weight 
of the generators and relations. (Note that these include 
the framing weights.) 
We have 
$$
\ch \cF = \sum e^{v_i} - \sum e^{w_i} \,, 
$$
and 
$$
\ch \cG = \ch \cF - e^{v_k} (1- e^{-t_1}) (1-e^{-t_2}) 
$$
if the generator with weight $v_k$ surjects onto $\cF/\cG$. 
The characters of the $\Ext$-groups in \eqref{tang_Bar} 
are computed as follows 
$$
\ch \, \Ext^1(\cG,\cF(-1)) = 
\frac{(1-\overline{\ch \cG} \, \ch \cF)}{(1-e^{-t_1}) (1-e^{-t_2})}\,, 
$$
where 
bar denotes the dual representation, that is, $\overline{e^v} = 
e^{-v}$. 

Let $N_{(\cG,x,\cF)} \fB_1$ denote the normal bundle to 
the Baranovsky correspondence at the at the point $(\cG,x,\cF)$ 

\begin{Lemma}\label{l_N_T} We have 
  \begin{alignat*}{2}
  \ch \, N_{(\cG,x,\cF)} \fB_1 - 
 \ch \, T_{(\cG,x,\cF)} \fB_1 
&= e^{-\hbar - v_k} \, \ch \cG - e^{v_k} \, \overline{\ch \cG} 
&&- e^{-\hbar} + 1 \\
&= e^{-\hbar - v_k} \, \ch \cF - e^{v_k} \, \overline{\ch \cF}  
&&- e^{-\hbar} + 1  \,, 
\end{alignat*}
where $v_k$ is the weight of $\cG/\cF$. 
\end{Lemma}

\noindent Note that the trivial weight $1$ here cancels 
with the trivial weight that comes from the expansion 
of $e^{v_k} \, \overline{\ch \cG}$, and similarly for 
the weights $-e^{-\hbar}$. 

\begin{proof}
Direct computation from \eqref{tang_Bar} \,. 
\end{proof}

\begin{Proposition}
We have
\begin{equation}
 \left[\beta_1(\gamma_1),\frac1{u-\ad(\bQh)} \, \beta_{-1}(\gamma_2)\right]  
=\frac1{\hbar} 
\int_{\C^2} 
\left(1-\frac{c(\cF^\vee\otimes \hbar,u)}{c(\cF^\vee,u)}
\right) \gamma_1 \gamma_2 \,, 
\label{comm_bQh}
\end{equation}
where $\cF$ is the universal sheaf on $\cM(r) \times \C^2$, 
the right-hand side is viewed as operator of cup-product 
by this cohomology class in $H^\hd_\bG(\cM(r))$. 
\end{Proposition}

\noindent 
Note, for example, that the $1/u$ term here gives the familiar 
result 
$$
\left[\beta_1(\gamma_1),\beta_{-1}(\gamma_2)\right] = 
- \int_{\C^2} \gamma_1 \gamma_2 \, \rk \cF  = 
r \, \tau(\gamma_1 \gamma_2)  \,. 
$$
{}It is clear from Grothedieck-Riemann-Roch that the 
right-hand side of  \eqref{comm_bQh} generates the 
same algebra as $\ch_k \cVt$. 

\begin{proof}
We use equivariant localization. Let $\cF$ be a torus-fixed 
sheaf as in \eqref{resolut_G} and let 
$$
[\cF] \in H^{2\dim}_{\bG}(\cM(r,n))\,,  
$$
denote the class of this fixed point. The computation of 
$$
\left(\beta_1(\gamma_1) \circ \frac1{u-\ad(\bQh)} \, \beta_{-1}(\gamma_2) 
\cdot [\cF],[\cF] \right) \Big/ 
\left([\cF],[\cF] \right) 
$$
is given by summing $1/(u-v_i)$ over all generators of $\cF$ with a 
certain equivariant weight that accounts for the normal bundle 
to $\fB_1$ and for the tangent bundle to $\cM(r,n+1)\times \cM(r) 
\times (\C^2)^2$. This equivariant 
weight is determined from Lemma \ref{l_N_T}. 

The product in the opposite order involves summation over 
all relations $w_i$ in the resolution \eqref{resolut_G}, 
because they correspond to torus-fixed sheaves that contain 
$\cF$. The new generator has weight 
$w_i-\hbar$, therefore we sum $1/(u-w_i+\hbar)$ 
with a weight which is again computed from Lemma \ref{l_N_T}.

The resulting sum simplifies using the elementary
 identity 
\begin{multline}
  \sum_k \frac{\hbar}{u-v_k} \prod_{i\ne k} \frac{v_k-v_i+\hbar}{v_k-v_i}
  \prod_{i} \frac{v_k-w_i}{v_k-w_i+\hbar} \\ - \sum_k
  \frac{\hbar}{u-w_k+\hbar} \prod_{i\ne k}
  \frac{w_k-w_i-\hbar}{w_k-w_i} \prod_{i}
  \frac{w_k-v_i}{w_k-v_i-\hbar} =\\
\prod_i \frac{u-v_i+\hbar}{u-v_i}
\prod_i \frac{u-w_i}{u-w_i+\hbar} - 1 \,,
\end{multline}
which is proven by observing that it is 
 a partial fraction expansion in the variable $u$. 
(This identity also appears in \ \cite{SV1}.) 
Since 
$$
c(\cF^\vee,u) = \prod_i (u-v_i) \big/ \prod_i (u-w_i) 
$$
the result follows. 
\end{proof}

\begin{proof}[Another proof of Theorem \ref{t_gen_c1}] 
Follows from the above proposition and 
Theorem \ref{t_coreY}. 
\end{proof}

\section{Slices and screening operators}\label{s_Slice_screen}

\subsection{}

In Section \ref{s_Slices_Int}  
we constructed geometrically core Yangian intertwiners
from slices. In this section, we identify algebraically
the intertwiner corresponding to the slices from 
Section \ref{s_slice_ex2}. They turn out to be 
the well-known screening operators for Virasoro modules. 

By the boson-fermion correspondence, screening operators
specialize to Pl\"ucker relations modulo $\hbar$. 
Thus, by Proposition \ref{p_Pluck}, 
 they generate the relations in the core Yangian 
of $\glh$. Hence, for $\bbY(\glh)$, the answer to the question 
from Section \ref{s_Q_slices} is affirmative.

\subsection{}

We recall some basic notion, in the generality of Chapter
\ref{s_free_bos}.

A field $Y(\eta,z)=\sum_n Y_n(\eta) \, z^{-n}$ is called 
\emph{primary} 
of dimension $\lambda\in\bbH$ if it satisfies the OPE 
$$
\vT(\gamma,z) \, Y(\eta,w) \sim 
\frac{z w}{(z-w)^2} \, Y(\lambda \gamma \eta,w) + 
\frac{zw}{z-w} \, \frac{\partial}{\partial w} Y(\gamma \eta,w) \,.  
$$
Equivalently, 
$$
[\vL_n(\gamma),Y_m(\eta)] = Y_{m+n}\big((n \lambda - n- m ))
\, \gamma \eta\big)  \,.
$$
In particular, if $\lambda=1$ then the operator 
$$
Y_0(\eta) = \int Y(\eta,z) 
$$
commutes with all operators $\vL_n(\gamma)$.

\subsection{}

Define normally ordered exponential of a field $Y(\gamma,z)$ 
by 
$$
\nord \exp Y(z) \nord (\gamma) = 
\tau(\gamma)+ \nord Y(z) \nord (\gamma)+ 
\frac{1}{2}  \nord Y(z)^2 \nord (\gamma)+ \dots \,, 
$$
where terms of the form $ \nord Y(z)^n \nord (\gamma)$ 
are defined using the $n$-fold coproduct 
$\bbH\to \bbH^{\otimes n}$ as in Section \ref{s_coprod}. 

These satisfy the usual rules like 
$$
\frac{\partial}{\partial z} \,  
\nord \exp Y(z) \nord (\gamma)   = \, 
\nord 
\left(\frac{\partial}{\partial z} Y(z)\right) \, \exp Y(z) \nord (\gamma)\,. 
$$

\subsection{}

Let $\eta$ be an eigenvector of multiplication operators 
in $\bbH$. We define $\eta^\vee$ by
\begin{equation}
\gamma \eta = (\gamma, \eta^\vee) \, \eta \label{etaeig}\,,
\end{equation}
for all $\gamma\in\bbH$. Define
$$
\bV_\mu(z) = \, \nord \exp \mu \bphi^-(z) \nord (\eta)
$$
where 
$$
\bphi^-  = \bphi^{(1)} - \bphi^{(2)} 
$$
is the antiderivative of the field $\bal^-$, see 
 Section \ref{s_phi}. In particular, we have
\begin{equation}
\bal^-(\gamma,z) \, \bphi^-(\eta,w) \sim 
\frac{2 z}{z-w} \, (\gamma,\eta) + \dots 
\label{opeap}
\end{equation}

\subsection{} 

Since the operator $\bV_\mu$ involves $\alpha_{\log}$, it 
has nontrivial commutation relations with $\alpha^-_{0}$, 
namely 
$$
\left[\alpha^-_{0}(\gamma) , \bV_\mu(z) \right] = 
2 \mu (\gamma, \eta^\vee) \, \bV_\mu(z) \,. 
$$
This means 
$$
V_\mu : \bF(a_1) \otimes \bF(a_2) \to 
\bF(a_1-\mu \, \eta^\vee) \otimes \bF(a_2+ \mu \, \eta^\vee) \,. 
$$

\subsection{}

\begin{Proposition}\label{p_primar} 
If $\eta$ is an eigenvector of multiplication as in 
\eqref{etaeig} then the operator 
$$
z^{\mu^2(\se,\eta^\vee)} V_\mu(z)
$$
is primary for $\vT(z,K)$ of dimension 
$$
\lambda = \mu^2 \se - \mu K \,. 
$$
\end{Proposition}

\noindent 
Here $\se\in \bbH$ is the handle-gluing element. 

\begin{proof}
This is a standard computation that uses \eqref{opeap} 
and Lemma \ref{l_Wick}. 
\end{proof}

\subsection{} 
In particular, primary of dimension $1$ can give 
rise to Virasoro intertwiners. In the case 
\begin{equation}
\bbH = H^\hd_{\bG}(\C^2)
\left[\frac1{\det_\C^2}\right]\,, \quad K=\hbar = 
-t_1-t_2\,, \quad \eta^\vee = \se = 
-t_1 t_2 \,,
\label{exH}
\end{equation}
we have 
$$
\mu^2 \se - \mu K =1 \,\, \Rightarrow 
\,\, \mu = \frac{1}{t_1}, \frac{1}{t_2}\,.
$$
For the integral $\int z^{\mu^2(\se,\eta^\vee)} V_\mu(z)$ 
to be well-defined, the integrand has to have integral 
powers of $z$. The nonintegral powers of $z$ come from the 
$\log z$ term in $\bphi^-$, namely 
$$
e^{\mu \, \log z \, \alpha^-_0} 
 (\eta)  \, \Big|_{\bF(a_1) \otimes \bF(a_2)} = 
z^{-\mu (a_1-a_2,\eta^\vee)} \, \tau(\eta) \,.
$$
For the case \eqref{exH}, this integrality constrain becomes
$$
(\mu^2 \se- \mu(a_1-a_2),\eta^\vee) = - \frac{t_2}{t_1} 
-\frac{a_1-a_2}{t_1} = - n \in \Z\,, \quad \mu = \frac{1}{t_1}\,, 
$$
and similarly for $\mu=1/t_2$. 

\subsection{}

\begin{Theorem}
For every $n\in \Z$ the screening operator 
$$
\int z^{-\frac{t_2}{t_1}} \, \bV_{\frac1{t_1}}(z): 
\bF(a_2+n t_1 - t_2) \otimes \bF(a_2) \to 
\bF(a_2+ n t_1) \otimes \bF(a_2 - t_2) 
$$
is a map of $\bY$-modules. 
\end{Theorem}

\begin{proof}
The operator clearly commutes with the Baranovsky operators 
and intertwines the Virasoro operators $\vT_+(z)$ by Proposition 
\ref{p_primar}. Formula \eqref{Qcl_Vir} expresses the operator 
of classical multiplication by divisor in terms of the 
Baranovsky operators and $\vT_+(z)$, therefore the 
screening operator intertwines it as well. 
Now Theorem \ref{t_gen_c1} finishes the proof. 
\end{proof}



\subsection{}

Note, in particular, the screening 
operators annihilates the vacuum vector for $n<0$. 
This is reflected in the poles of the $\bR(u)$ at 
$$
u= \hbar,\hbar-t_1, \hbar-2t_1\,, \dots \,. 
$$

\chapter{Yangian and vertex algebras}

\section{The operator $\bQh$}

\subsection{}

Since the operator $\bQh$ plays an important role in 
Theorem \ref{t_gen_c1}, we give a formula for it that 
modifies the formula in  Theorem \ref{t_Qclass}.

More compact formulas are obtained for Chern 
character of $\cVt \otimes \hbar^{1/2}$, where 
$\ch \hbar^{1/2} = e^{\hbar/2}$. This is the familiar
twist by the square root of the canonical bundle 
(of $\C^2$, in this case). However, only the 
overall shape of the formula will be used below, not 
the details.

We define 
$$
\bOmh = \tfrac12 \int \nord \bbeta \, |\partial| \, \bbeta \nord (1) \,,
$$
as in  Section \ref{s_|part|} and denote by 
$$
\fC_{>} =  \{a_1 > \dots > a_r\}
$$
the standard chamber for $\bA$. 
The analog of Theorem \ref{t_Qclass} is the following

 \begin{Proposition}
Under the identification 
$$
 \bF(a_1) \otimes \dots \otimes \bF(a_r) 
 \xrightarrow{\,\,\Stab_{\fC_>}\,\,}
 H^\hd_\bG(\cM(r)) \otimes \bK \,, 
 $$
as in \eqref{zeroFr}, we have 
 \begin{multline}
 \ch_1\left(\cVt\otimes \hbar^{1/2}\right) = - \sum_{i} \tfrac16 \int 
\nord\left(\bal^{(i)}\right)^3\nord(1) +
\sum_{i} \tfrac1{24} \int 
\nord \bal^{(i)} \nord(\hbar^2+2\se)\\
+ \tfrac12 \hbar   
\sum_{i<j} \int \bal^{(i)} \, \partial \, \bal^{(j)} (1) + \tfrac12 \hbar \, \bOmh \,. 
 \label{ch_1Tauth}     
 \end{multline}
 \end{Proposition}

\noindent
For other chambers, one rearranges the $\sum_{i<j}$ term accordingly. 

 \begin{proof}
The inclusion of zero modes and the $\otimes \hbar^{1/2}$ 
twist 
removes the $\bPhi_2$-term from formula \eqref{Qclass}. 
 Therefore, the two sides of \eqref{ch_1Tauth} differ by a 
scalar operator that we can determine by evaluating on the vacuum 
vector. This is straightforward, using 
$$
- \ch_1 \frac{e^{a-\hbar/2}}{(1-e^{-t_1})(1-e^{-t_2})}
= \tfrac16 \tau(a^3) - \tfrac1{24} \tau((\hbar^2+2\se)a)\,,
$$ 
where $\hbar = - t_1 - t_2$, $\se = - t_1 t_2$. 
 \end{proof}

\subsection{}\label{s_vHeis} 

Let 
$$
\vHeisr \subset \End \bF(a_1) \otimes \dots \otimes \bF(a_r) 
$$
denote the algebra generated by all Fourier coefficients of 
vertex operators, that is, 
$$
\int z^n \, \nord P(\bal^{(i)}, \partial \bal^{(i)}, 
\partial^2 \bal^{(i)},\dots ) \nord \quad \in \vHeisr
$$
for any $n\in \Z$ and any normally ordered polynomial $P$ in 
the fields $\bal^{(i)}$, $i=1,\dots,r$, and their derivatives. 

\begin{Proposition}
The action of $\bbY$ on $\bF(a_1) \otimes \dots \otimes \bF(a_r)$ 
factors through a map 
\begin{equation}
\bbY \to 
\vHeisr\big[\bOmh\big] \,.
\label{bbYtovHeis}
\end{equation}
The map \eqref{bbYtovHeis} is equivariant with respect to 
the translation automorphism. 
\end{Proposition}

\noindent
Note the translation automorphism
$$
\ttau_c\left(\bal^{(i)}\right)  = \bal^{(i)} - \tau(c) 
$$
of the Heisenberg 
vertex algebra has a natural extension to $\vHeisr\big[\bOmh\big]$. 
This extension leaves $\bOmh$ invariant. 

\begin{proof}
This follows at once from formula \eqref{ch_1Tauth}  and 
Theorem \ref{t_gen_c1}. 
\end{proof}

\subsection{}

It may be curious to notice that the $\bOmh$ term disappears 
from the corresponding \emph{quantum} operator $\widehat{\bQ}$ 
upon averaging 
over all $|q|=1$ in the principal value sense.

\section{Yangian and $\cW$-algebras}\label{s_W} 

\subsection{}

Our goal in this section is to describe the image 
of the map $\eqref{bbYtovHeis}$ in terms of the so-called
$\cW$ vertex operator algebras. This provides a link 
to the ideas of Alday, Gaiotto, and Tachikawa 
\cite{AGT}, the existence of which was suggested to us 
by Nakajima and Tachikawa. 

The $\cW$-algebras first appeared
in mathematical physics as extended symmetry algebras of conformal field
theories, see for example \cite{Wsymm} for a survey. Following Feigin and 
Frenkel \cite{FeiFr1,FeiFr}, they may be described as explicit subalgebras of the 
Heisenberg vertex algebra. This is the description that we use here. 

\subsection{}

Let $\bK$ be a commutative ring and let 
$$
\bbH \cong \bK^r  
$$
be a free $\bK$-module of rank $r$ with a nondegenerate quadratic form. 
The setup is like in Chapter \ref{s_free_bos}, except neither 
product nor coproduct on $\bbH$ is required. 
One should view  $\bbH$ as a Cartan 
subalgebra of a reductive Lie algebra, with the restriction of a 
an invariant bilinear form. Here we need the form 
\begin{equation}
(x,x) =  \sum_1^r x_i^2 \,,
\label{Eucl}
\end{equation}
that corresponds to the Lie algebra $\gl(r)$. 

One defines the Heisenberg algebra $\fHeis(\bbH)$ as
in Chapter \ref{s_free_bos} and the algebra of Fourier coefficients 
of vertex operators
$$
\vHeis(\bbH)  \supset \fHeis(\bbH) 
$$
as in Section \ref{s_vHeis}. For any orthogonal decomposition 
$\bbH = \bbH_1 \oplus \bbH_2$, we have 
$$
\vHeis(\bbH) = \vHeis(\bbH_1) \, \widehat{\otimes} \, \vHeis(\bbH_2) \,,
$$
where the completion is the usual completion required to collect terms 
in a product of two series.

\subsection{}

Let 
$$
\eta=(0,\dots,1,-1,\dots,0)\,, 
$$
range over the simple positive roots of $\gl(r)$. For each 
$\eta$, consider the corresponding Heisenberg field
$$
\alpha_\eta(x) = 
\sum_{n} \left(\alpha^{(i)}_n - \alpha^{(i+1)}_n\right) \, x^{-n-1} \,.
$$ 
Here we denote the argument by $x$ to emphasize a small 
discrepancy between the conventions of Chapter  \ref{s_free_bos}
and standard CFT conventions. In Chapter
 \ref{s_free_bos}, 
the arguments of the fields were coordinates on 
$\C^\times$. Here $x$ is a 
coordinate on $\C$ and the exponents of $x$ are shifted by $1$, 
that is, by the conformal 
dimension of the field. 

Since $(\eta,\eta) = 2$, we have
$$
\alpha_\eta(x) \, \alpha_\eta(y) \sim \frac{2}{(x-y)^2}  \,, 
$$
and the field 
$$
T_\eta = \frac{1}{4} \, \nord \alpha_\eta^2 \nord
\,+ \,\frac{\kappa}{2} \,  \frac{\partial \alpha_\eta}{\partial x} 
$$
generates a Virasoro vertex algebra which we denote by  
$$
\fVir_\eta \subset \vHeis(\bK \eta) \,. 
$$
Here $\kappa$ is a parameter that enters the definition of the 
$\cW$-algebra. To match it to conventions in 
the literature, we note that 
$$
\kappa = \frac{\beta}{\sqrt{2}} - \frac{\sqrt{2}}{\beta}
$$
in the book \cite{BF} and that the central charge of $\fVir_\eta$
equals $1-6\kappa^2$.

\subsection{}

By definition, see for example \cite{BF}, a vertex operator algebra is 
a collection of operator-valued distributions, called \emph{vertex
operators}, satisfying certain axioms. In CFT, these correspond to 
local chiral operators and, as in any mathematical formulation of QFT, 
the locality of these operators is really the key axiom. A specific 
feature of 2-dimensional conformal field theories is the presence of 
the Virasoro algebra among its chiral operators. 

While the language of vertex operators is very rich and concise, 
for our current purposes it will be sufficient to work with 
the following classical algebraic structures associated to a vertex
algebra: 
\begin{itemize}
\item[---] the \emph{associative} algebra generated by the Fourier 
coefficients of vertex operators, such as $\vHeis(\bbH)$, 
$\cW(\gl(r))$, or $\fVir_\eta$, 
\item[---] the \emph{Lie algebra} generated by the Fourier 
coefficients of vertex operators with respect to the commutator, 
which will be indicated by a subscript like $\vHeis_\tLie(\bbH)$, 
$\cW_\tLie(\gl(r))$, or $\fVir_{\eta,\tLie}$
\end{itemize}
see Chapter 4 in \cite{BF}. Clearly, the latter generates the former. 

To describe the $\cW(\gl(r))$ as a subalgebra of $\vHeis(\bbH)$, 
we will use the following characterization due to Feigin and Frenkel. 
Recall that for each $\eta$ we have 
$$
\vHeis(\bbH) = \vHeis(\bK \eta)
\, \widehat{\otimes} \, \vHeis(\eta^\perp)\,. 
$$

\begin{Theorem}[\cite{FeiFr1,FeiFr}] The algebra $\cW_\tLie(\gl(r))$ is 
the intersection 
\begin{equation}
\cW_\tLie(\gl(r)) = \bigcap_\eta \, \fVir_{\eta,\tLie} 
\, \widehat{\otimes} \, \vHeis_\tLie(\eta^\perp)\,, 
\label{cWr}
\end{equation}
where $\eta$ ranges over the simple positive roots of $\gl(r)$.
\end{Theorem}

\noindent
The following outline of the argument was kindly provided by 
E.~Frenkel. 

\begin{proof}
The proof proceeds in 4 steps. 

First, for generic values of the parameter, the vertex 
$\cW$-algebra is equal to the intersection of the kernels of
the screening operators.
This is the most non-trivial step, proved in two ways: 
first, in Proposition 3 of \cite{FeiFr1} (this proof is reproduced in Theorem 15.4.12 of \cite{BF}) and second, in Theorem 4.6.9 of \cite{FeiFr}).

Second, in the case of $\mathfrak{sl}(2)$, 
the kernel of the screening operator is the Virasoro vertex algebra for generic
values of the parameter. This is proved in Proposition 4 of \cite{FeiFr1} (this proof is reproduced in 15.4.14 of \cite{BF}) and in Proposition
4.4.4 of \cite{FeiFr}.

Next, the kernel of the $i$-th screening operator is equal to the the tensor product of the
Virasoro vertex algebra along the $i$-th simple root and the Heisenberg vertex algebra orthogonal to the $i$-th root. The proof is given in the proof of Proposition 5 in \cite{FeiFr1} and in 15.4.15 of \cite{BF}.

Finally, the same results hold for the algebras of Fourier coefficients of vertex operators. This is proved in Proposition 2 of \cite{FeiFr1} and Theorem 4.6.11 of \cite{FeiFr}.
\end{proof}

We define 
$$
\cW(\fsl(r)) = \cW(\gl(r)) \cap \, \vHeis(\ffZ^\perp)
$$
where $\ffZ\subset \bbH \subset \gl(r)$ is the center. This implies 
$$
\cW(\gl(r)) = \vHeis(\ffZ) \, \widehat{\otimes} \, \cW(\fsl(r))\,. 
$$

\subsection{}

To compare this with our formulas, we take 
$$
\gamma = 1 \in H^\hd_\bG(\C^2)
$$
in the formula \eqref{defvT_}. Since $(1,1)_{H^\hd_\bG(\C^2)} = \tau(1)$, 
we have 
$$
\textup{our $\alpha_n(1)$}  = \sqrt{\tau(1)} \,\, \textup{standard $\alpha_n$} 
$$
where Heisenberg operators associated to the 
quadratic form \eqref{Eucl} are considered standard. Further, since 
$$
 \Delta 1 = \frac{1 \otimes 1}{\tau(1)}  
$$
we have in \eqref{defvT_}
\begin{equation}
\left[z^{-n}\right]\,  
\vT_+(1)  = \left[x^{-n-2}\right] 
T_\eta \, \Bigg|_{\textstyle \alpha_0 \mapsto \alpha_0 + 
\frac12 \kappa \eta} 
\label{vTTeta}
\end{equation}
where $\eta=(1,-1)$ is the root of $\gl(2)$ and 
\begin{equation}
\kappa  = \hbar \sqrt{\tau(1)}  = - \frac{t_1 + t_2}{\sqrt{-t_1 t_2}} \,.
\label{valkapp}
\end{equation}
The shift of zero modes 
\begin{equation}
\left(\alpha_0^{(1)},\alpha_0^{(2)}\right)  \mapsto 
\left(\alpha_0^{(1)}+\tfrac12 \, \hbar,\alpha_0^{(2)}- 
\tfrac12\, \hbar\right) 
\label{zms2}
\end{equation}
compensates for the difference between $\partial = z 
\frac{\partial}{\partial z}$ in $\vT_+$ and 
$\frac{\partial}{\partial x}$ in $T_\eta$. 

\subsection{}

Generalizing \eqref{zms2}, we incorporate the shift 
of the zero modes by 
$\kappa \rho$, where $\rho$ is the half-sum of positive
roots, in the definition of $\cW(\gl(r))$. This is an 
automorphism of the ambient Heisenberg vertex algebra.

\subsection{}

Nakajima varieties produce lowest weight Yangian modules. 
Any action of a graded algebra $A=\bigoplus A_n$ 
on a lowest weight module 
canonically extends to a certain completion $\overline{A}\supset A$. 
Neighborhoods of zero in this completions are left ideals 
generated by $\bigoplus_{n<-N}  A_n$. 

\begin{Proposition}
The action of $\bbY$ on $\bF(a_1) \otimes \dots \otimes \bF(a_r)$ 
factors through a map 
\begin{equation}
\bbY \to \overline{\vHeis(\ffZ)} \, \widehat{\otimes} \, \cW(\fsl(r))\,,
\label{YtoW}
\end{equation}
where $\overline{\vHeis(\ffZ)} \supset \vHeis(\ffZ)$ is a completion 
as above. 
\end{Proposition}

\begin{proof}
Extract the $\eta$-component from the 
operator \eqref{ch_1Tauth} as in Section \ref{s_zero_modes}. 
Using \eqref{ch_1Tauth}  and \eqref{vTTeta}, we conclude  
$$
\bQh \in \cW_\tLie(\gl(r)) + \bK \, \bOmh \otimes 1 \,, 
$$
where the second term is written with respect to the decomposition 
$\bbH = \ffZ \oplus \ffZ^\perp$. Therefore
$$
\bQh,\beta_{\pm1} 
\in \overline{\vHeis(\ffZ)} \, \widehat{\otimes} \, \cW(\fsl(r))
$$
and Theorem \ref{t_gen_c1} completes the proof. 
\end{proof}

\subsection{}
The proof of \eqref{cWr} by Feigin and Frenkel uses a screening 
operators characterization of $\fVir_\eta$. Those can be matched
to the screening operators of Section \ref{s_Slice_screen}. 

\subsection{}

\begin{Proposition}
The map 
\begin{equation}
  \label{barYtoW}
  \overline{\bbY} \to \overline{\cW(\gl(r))} 
\end{equation}
induced by \eqref{YtoW} is surjective. 
\end{Proposition}
\begin{proof}
Follows from the corresponding statement for $\hbar=\kappa=0$ 
proven by Frenkel, Kac, Radul, and Wang in \cite{FKRW}.  When 
$\hbar=0$, the nonlocal term $\bOm$ drops out and the surjectivity 
$$
\bbY/\hbar \bbY \cong \cU(\cDh) \to \cW(\gl(r))\big|_{\kappa=0} \to 0
$$
is true without completion, see \cite{FKRW}. Clearly, it implies the 
surjectivity after completion. 
\end{proof}

\subsection{}

One of the goals of \cite{AGT} is a characterization of 
interesting cohomology classes in terms of the $\cW$-action. 
For example, one can consider the vector of identities
$$
\mathbf{1} \in H^\hd_\bG(\cM(r)) 
$$
in the cohomology of each $\cM(r,n)$. In this direction, 
there is the following simple result. Define 
$$
\beta^{[k]}_n = \left(\ad \bQh\right)^k \cdot \beta_{n} \,. 
$$

\begin{Proposition}
The vector of identities $\mathbf{1}$  satisfies
$$
\beta^{[k]}_{n}(\pt) \cdot 
\mathbf{1} = 
\begin{cases}
0\,, &  k < rn-1 \\
- \mathbf{1} \,, &n=1, k=r-1\,. 
\end{cases}
$$
\end{Proposition}

\begin{proof}
The operator $\beta_{n}(\pt)$ is defined by a proper push-forward
with fibers of generic dimension $rn-1$, therefore it annihilates any 
cohomology class of degree less than $rn-1$. This proves
the first claim. 

If $n=1$ then generic
fibers are projective spaces $\Pp^{r-1}$ on which the 
generator $\ch_1 \cVt$ restricts to the hyperplane class
$c_1(\cO(1))$, up-to equivariant corrections. Therefore
\begin{multline}
\beta^{[r-1]}_{1}(\pt)   = 
(-1)^{r-1} \beta_1 \, \bQh^{r-1} \cdot \mathbf{1} = \\
- \left(\int_{\Pp^{r-1}} c_1(\cO(1))^{r-1}\right) \cdot \mathbf{1}
= - \mathbf{1}\,,
\end{multline}
where an extra $(-1)^r$ comes from the definition of $\beta_1 = 
\beta_{-1}^\tau$, see Section \ref{s_sign_Baran}. 

\end{proof}

\pagestyle{plain}

\end{document}